\documentclass[12pt]{article}
\usepackage[utf8]{inputenc}
\usepackage{amsfonts,graphicx,amsmath,amssymb,amsthm,geometry,
hyperref,color,nicefrac,enumerate,bbm,verbatim, url, mathrsfs, upgreek}
\usepackage{xcolor} 
\usepackage{todonotes}   
\newcommand{\smallsum}{\textstyle\sum}
\allowdisplaybreaks[1]
\newtheorem{lemma}{Lemma}[section]
\newtheorem{corollary}[lemma]{Corollary}
\newtheorem{proposition}[lemma]{Proposition}

\newtheorem{theorem}[lemma]{Theorem}

\newtheorem{setting}[lemma]{Setting}
\renewcommand{\O}{{ \bf O}}
\newcommand{\ff}{{ \bf F}}
\newcommand{\y}{{ \bf X}}

\renewcommand{\L}{\mathcal{L}}

\providecommand{\X}{{\ensuremath{\mathcal{X}}}}
\providecommand{\Y}{{\ensuremath{\mathcal{Y}}}}

\providecommand{\N}{{\ensuremath{\mathbbm{N}}}}
\providecommand{\R}{{\ensuremath{\mathbbm{R}}}}
\providecommand{\E}{{\ensuremath{\mathbb{E}}}}
\providecommand{\B}{{\ensuremath{\mathcal{B}}}}
\providecommand{\F}{{\ensuremath{\mathcal{F}}}}
\providecommand{\f}{{\ensuremath{\mathbb{F}}}}
\providecommand{\M}{{\ensuremath{\mathcal{M}}}}
\renewcommand{\H}{{\ensuremath{\mathbb{H}}}}
\newcommand{\U}{{\ensuremath{\mathbb{U}}}}
\renewcommand{\P}{{\ensuremath{\mathbb{P}}}}
\providecommand{\1}{{\ensuremath{\mathbbm{1}}}}
\renewcommand{\aa}{a}
\providecommand{\values}{{\ensuremath{\mathfrak{v}}}}
\providecommand{\N}{{\ensuremath{\mathbbm{N}}}}
\providecommand{\R}{{\ensuremath{\mathbbm{R}}}}

\providecommand{\E}{{\ensuremath{\mathbb{E}}}}
\renewcommand{\P}{{\ensuremath{\mathbb{P}}}}
\providecommand{\1}{{\ensuremath{\mathbbm{1}}}}
\providecommand{\HS}{{\ensuremath{\textup{HS}}}}

\allowdisplaybreaks 
\begin{document}
\title{Strong convergence rates on the whole probability \\
	 space for space-time discrete
numerical approximation \\ schemes for stochastic Burgers equations}
\author{Martin Hutzenthaler$^{1} $, Arnulf Jentzen$^{2, 3} $, Felix Lindner$^{4} $, and Primo\v{z} Pu\v{s}nik$^{5} $
	\bigskip 
	\\
	\small{$^1$ Faculty of Mathematics, University of Duisburg-Essen,}
	\\
	\small{Germany, e-mail: martin.hutzenthaler@uni-due.de} 
	\smallskip
	\\
	\small{$^2$ 
		Faculty of Mathematics and Computer Science,
	University of M\"unster,}
	\\
	\small{Germany,   
		e-mail: ajentzen@uni-muenster.de} 
		\smallskip
	\\
	\small{$^3$ 
		Department of Mathematics,
	ETH Zurich,}
	\\
	\small{Switzerland,
		e-mail: arnulf.jentzen@sam.math.ethz.ch}
	\smallskip
	\\
	\small{$^4$  
		Faculty of Mathematics and Natural Sciences,
	University of Kassel,}
	\\
	\small{Germany,
		e-mail: lindner@mathematik.uni-kassel.de}
		\smallskip 
	\\
	\small{$^5$ 
Department of Mathematics,
ETH Zurich,}
	\\
	\small{Switzerland,
		e-mail: primoz.pusnik@sam.math.ethz.ch}
}

\maketitle
\begin{abstract}
	The main result of this article establishes strong convergence rates on the whole probability space for 
	explicit space-time discrete numerical approximations
	for a class of stochastic evolution
	equations with possibly non-globally monotone coefficients such as
	additive trace-class noise driven stochastic Burgers equations. 
	The key idea in the proof of our main result is (i) to bring the classical Alekseev-Gr\"obner formula from deterministic analysis into play and (ii) to employ uniform
exponential moment estimates for the numerical approximations.
\end{abstract}

\tableofcontents
\newpage

\section{Introduction}
In this article we study the problem of 
establishing strong convergence rates for explicit space-time discrete 
approximations of semilinear stochastic 
evolution equations (SEEs) with non-globally monotone 
coefficients (see, e.g., Liu \& R\"ockner~\cite[(H2) in Chapter~4]{LiuRoeckner2015Book} 
for global monotonicity) such as stochastic Burgers equations. 
Proving strong convergence with rates for numerical approximations of SEEs with non-globally monotone coefficients is known to be challenging.
 In fact, there exist stochastic ordinary differential 
 equations (SODEs) with smooth and globally bounded but 
 non-globally monotone coefficients such that no 
 approximation method based on finitely many observations 
 of the driving Brownian motion can converge strongly 
 to their solutions faster than any given speed of 
 convergence (see~Jentzen et al.\ \cite[Theorem~1.3]{JentzenYaroslavtsevaMuller2016},
 Hairer et al.\ \cite[Theorem~1.3]{Haireretal2015}, 
 and also, e.g., \cite{GerencserJentzenSalimova2017published, 
 	HefterJentzen2019, 
 	MullerYaroslavtseva2017, Yaroslavtseva2016, 
 	MullerYaroslavtseva2017b}). 
 In addition,
 the classical Euler-Maruyama method, 
 the exponential Euler method, 
 and the linear-implicit 
 Euler method fail to converge strongly as well as weakly for some 
 SEEs with superlinearly growing coefficients (see, e.g.,
 Hutzenthaler
 et al.\ \cite[Theorem~2.1]{HutzenthalerJentzenKloeden2011}
 and
 Hutzenthaler
 et al.\
 \cite[Theorem~2.1]{HutzenthalerJentzenKloeden2013}
  for SODEs and 
  Beccari et al.\ \cite[Theorem~1.2]{BeccariHutzenthalerJentzenKurniawanLindnerSalimova2019}
  for 
  stochastic partial differential equations 
(SPDEs)).

Recently, a series of appropriately modified versions of 
the explicit Euler method have been introduced and 
proven to converge strongly for some SEEs with 
superlinearly growing coefficients 
(see, e.g.,
\cite{HutzenthalerJentzen2014PerturbationArxiv, HutzenthalerJentzen2014Memoires,  
	HutzenthalerJentzenKloeden2012, 
	Sabanis2013ECP, Sabanis2016, 
	TretyakovZhang2013, 
	WangGan2013}
for SODEs
and, e.g., \cite{BeckerGessJentzenKloeden2017, 
	BeckerJentzen2019, 
	GyongySabanisSiska2016, 
	HutzenthalerJentzenSalimova2016, 
	JentzenPusnik2019Published,
	JentzenSalimovaWelti2019, 
	MazzonettoSalimova2019} for SPDEs). 
These methods are easily implementable and tame 
the superlinearly growing terms in order to 
ensure strong convergence. 
Strong convergence rates for explicit time discrete and explicit space-time discrete numerical methods for SPDEs with a non-globally Lipschitz continuous but globally monotone nonlinearity have been derived in, e.g.,
Becker et al.\ \cite[Theorem~1.1]{BeckerGessJentzenKloeden2017},
Becker \& Jentzen~\cite[Theorem~1.1 and Corollaries~6.12 and~6.14]{BeckerJentzen2019},
Brehi\'er et al.\ \cite[Theorems~3.4 and~4.6 and Corollaries~3.9 and~3.10]{BrehierCuiHong2018},
Jentzen \& Pu\v{s}nik~\cite[Theorem~1.1 and Corollary~8.2]{JentzenPusnik2019Published},
and
Wang~\cite[Theorem~4.11]{Wang2018}.
Moreover, suitable nonlinear-implicit  
approximation schemes are known
to converge strongly in the case of several 
SEEs with superlinearly growing coefficients
(see, e.g., \cite{Higham2002, Hu1996} for SODEs 
and, e.g., 
\cite{BrzezniakCarelliProhl2013, 
	FengLiZhang2017,
	FurihataKovacsLarssonLindgren2017Published, 
	GyoengyMillet2005, 
	KovacsLarssonLindgren2013Published, 
	KovacsLarssonLindgren2015published, 
	MajeeProhl2018} 
for SPDEs).  
Strong convergence rates
 for  
temporal and spatio-temporal approximations 
 of SEEs with non-globally monotone coefficients
on suitable large subsets of the probability space 
(sometimes referred to as semi-strong convergence rates) 
have been established in, e.g.,  
Bessaih et al.\ \cite[Theorem~5.2]{BessaihBrzezniakMillet2014},
Carelli \& Prohl~\cite[Theorems~3.1, 3.2, and~4.2]{CarelliProhl2012}, 
and
Furihata et al.\
\cite[Theorem~5.3]{FurihataKovacsLarssonLindgren2017Published}.
These
semi-strong convergence rates can imply convergence in probability, but
they are not sufficient to prove strong convergence rates.
For completeness, we also refer to, e.g., 
\cite{ag06, 
	BloemkerJentzen2013,
	Breckner2000,
	CoxHutzenthalerJentzenNeervenWelti2016,
	BlomkerKamrani2017b, 
		Printems2001,
	ZhuZhu2017,  
	ZhuZhu2015Arxiv, 
	ZhuZhu2018} 
for results concerning convergence in probability with and without rates, pathwise convergence with rates, and strong convergence without rates for numerical approximations of SEEs with superlinearly growing coefficients. 
Weak convergence with rates for splitting approximations of 2D stochastic Navier-Stokes equations has been established  in D\"orsek~\cite[Corollary~4.3]{Doersek2012}.
In
Bessaih \& Millet~\cite[Theorem~4.6]{BessaihMillet2018}
strong 
convergence with rates is proven 
for 
fully drift-implicit Euler 
approximations
 in the case of 2D stochastic 
Navier-Stokes equations with additive trace-class noise 
by exploiting a rather specific property (see Bessaih \& Millet~\cite[(2.4) in Section~2]{BessaihMillet2018})
of the Navier-Stokes-nonlinearity
(see also 
Bessaih \& Millet~\cite[Theorems~3.6, 3.9, and~4.4 and Proposition~4.8]{BessaihMillet2018}
for further strong convergence results). 
These fully drift-implicit Euler approximations of 2D stochastic Navier-Stokes equations 
involve solutions of nonlinear equations that are not known to be unique 
and it is unknown how to approximate these solutions with positive convergence rates.
Strong convergences rates for nonlinear-implicit numerical schemes for SEEs with non-globally monotone coefficients have also been analyzed in
Cui \& Hong~\cite{CuiHong2018, CuiHong2019}
and Cui et al.\ \cite{CuiHongLiuZhou2019, CuiHongSun2019}
(cf.\ also, e.g., Cui et al.\ \cite{CuiHongLiu2017}
and Yang \& Zhang~\cite{YangZhang2017}).

To the best of our knowledge, 
there exist no results in the scientific literature 
establishing strong convergence with rates 
on the whole probability space for an 
explicit space-time discrete numerical method for an evolutionary 
SPDE with a non-globally monotone nonlinearity such as 
stochastic Burgers equations, 
stochastic Navier-Stokes equations, 
stochastic Kuramoto-Sivashinsky equations, 
Cahn-Hilliard-Cook equations, 
or stochastic nonlinear Schr\"odinger equations.
It is the key contribution of this work to partially 
solve this problem and to establish strong convergence 
rates for an appropriately tamed-truncated exponential 
Euler-type method for SPDEs with a possibly 
non-globally monotone nonlinearity and additive 
trace-class noise 
(see Theorem~\ref{theorem:Exact_to_numeric} below). 
In particular, in Corollary~\ref{corollary:BurgersFinal} below we derive 
strong convergence rates for 
explicit space-time discrete approximations 
of stochastic Burgers equations.  
A simplified version of Corollary~\ref{corollary:BurgersFinal} is presented in the following theorem.
\begin{theorem} 
	\label{theorem:BurgersFinal}
	\sloppy 
	Let
	$ ( H, \langle \cdot, \cdot \rangle _H, \left \| \cdot \right\|_H ) $ 
	be the $ \R $-Hilbert space
	of equivalence classes
	of Lebesgue-Borel square-integrable functions
	from $ (0,1) $ to $ \R $,
	let
	$ A \colon D( A ) \subseteq H \to H $ 
	be the
	Laplace operator with zero Dirichlet boundary conditions on $ H $,  
	let
	$ T \in (0,\infty) $, 
	$ c \in \R $, 
	$ \xi \in D(A) $,
	$ \beta \in (0, \nicefrac{1}{2} ] $,  
	$ B \in \HS(H, D((-A)^\beta) ) $,  
	$ (e_n)_{ n \in \N} \subseteq H $
	satisfy
	for all $ n \in \N $ that
	$ e_n( \cdot ) =\sqrt{2} \sin( n \pi (\cdot) ) $,
	let
	$ ( P_N )_{ N \in \N } \subseteq L(H) $
	satisfy for all
	$ N \in \N $,
	$ v \in H $ 
	that
	$ P_N(v) =\sum_{n=1}^N \langle e_n ,v \rangle_H e_n $, 
	let 
	$ F \colon D ( ( -A)^{ \nicefrac{1}{2} } ) \to H $ 
satisfy
	for all  
	$ v \in D ( ( -A)^{ \nicefrac{1}{2} } ) $
	that
	$ F( v ) =
	c\, 
	v' v 
	$,
	let
	$ ( \Omega, \F, \P ) $
	be a probability space, 
	let
	$ (W_t)_{t\in [0,T]} $
	be an
	$ \operatorname{Id}_H $-cylindrical 
	Wiener process,   
	let
	$ W^N \colon [0,T] \times \Omega \to P_N(H) $, $ N \in \N $,
	be stochastic processes which satisfy for all 
	$ N \in \N $,
	$ t \in [0,T] $ that 
	$ \P ( W_t^N = \int_0^t P_N B \, d W_s ) = 1 $,
	and
	let 
	$ \y^{ M, N }\colon [0, T] \times \Omega \to P_N(H) $, 
	$ M, N \in \N $,
	be   the 
	stochastic processes 
	which satisfy for all 
	$ M, N \in \N $,
	$ m \in \{ 0, 1, \ldots, M-1 \} $,   
	$ t \in ( \nicefrac{mT}{M}, \nicefrac{ (m+1) T }{M} ] $
	that
	$ \y^{M, N}_0 = P_N(\xi) $ 
	and
	\begin{equation}
	\begin{split}
	\label{eq:schemesIntro}
	\y_t^{M, N} 
	&  
	= 
	e^{(t- \nicefrac{m T}{M} )A} 
	\Big( 
	\y_{
		\nicefrac{m T}{M}
	}^{M, N} 
	\\
	&
	\quad
	+
	\,
	\1_{  \{ 
		1
		+
		\|   
		(-A)^{1/2}
		\y^{M, N}_{ \nicefrac{m T}{M} }  \|_{H  }^2 
		\leq 
		( \nicefrac{M}{T} )^{1/19}   \}}
	\Big[
	P_N
	F(
	\y_{ \nicefrac{m T}{M} }^{M, N}
	) \,
	(
	t 
	-
	( \nicefrac{m T}{M} )
	)  
	+
	\tfrac{ 
		  W^N_t - W^N_{  \nicefrac{m T}{M} }   
	}{
		1 + 
		\|  W^N_t - W^N_{  \nicefrac{m T}{M} }   
		\|_H^2
	} 
	\Big]
	\Big) 
	.
	\end{split}
	\end{equation}
	Then 
	\begin{enumerate}[(i)]
		\item \label{item:SolutionExistsIntro}
		there exists an up to indistinguishability
		unique stochastic process 
		$ X  \colon [0,T] \times \Omega \to  D( (-A)^{ \nicefrac{1}{2} } ) $
		with continuous sample paths
		(w.c.s.p.) 
		such that
		for all 
		$ t \in [0,T] $ 
		we have
		$ \P $-a.s.\ that
		\begin{equation}
		\label{eq:Eq}
		X_t 
		= 
		e^{tA} \xi 
		+
		\int_0^t e^{(t-s)A} F(X_s) \, ds 
		+
		\int_0^t e^{(t-s)A} B \, dW_s
		\end{equation} 
		and
		\item \label{item:AppConvergesIntro} 
		for every 
		$ \varepsilon, p \in (0, \infty) $
		there exists $ C \in \R $
		such that for all 
		$ M, N \in \N $
		we have that
		\begin{equation}
		\sup\nolimits_{ t \in [0,T] }
		\big( \E [ \| X_t - \y^{M, N}_t \|_H^p ]  \big)^{\nicefrac{1}{p}}
		\leq
		C
		\big(
		M^{ ( \varepsilon - \beta ) }
		+
		N^{ ( \varepsilon - 2 \beta ) }
		\big)
		.
		\end{equation}
	\end{enumerate}
\end{theorem}
Theorem~\ref{theorem:BurgersFinal}
is an immediate
consequence 
of Corollary~\ref{corollary:BurgersFinal}
in Section~\ref{section:Burgers}
below
(applies with
 $ T = T $,
 $ \varepsilon = \varepsilon $,
 $ c_0 = 1 $,
 $ c_1 = c $,
 $ \varsigma = \nicefrac{1}{19} $,
 $ p = \max \{ p, 1 \} $,
 $ \beta = \beta $,
 $ \gamma = \nicefrac{1}{2} $,
 $ H = H $,
 $ e_n = e_n $,
 $ A = A $, 
$ H_r = D( (-A)^r ) $,
$ B = B $,
$ \xi = \xi $, 
$ F = F $,
$ P_N = P_N $,
$ ( \Omega, \mathcal{F}, \P ) 
=
( \Omega, \mathcal{F}, \P ) $,
$ ( W_t )_{ t \in [0,T] } = ( W_t )_{ t \in [0,T] } $,
$ \y^{ \{ 0, T/M, \ldots, T \}, N }
= \y^{ M, N } $
for 
$ M, N, n \in \N $, 
$ \varepsilon, p \in (0, \infty) $,
$ r \in [0, \infty) $
in the setting of
Corollary~\ref{corollary:BurgersFinal})
and H\"older's inequality.
Corollary~\ref{corollary:BurgersFinal},
in turn,
is a consequence
of
Theorem~\ref{theorem:Exact_to_numeric}
in Subsection~\ref{subsection:StrongRates} below
(the main result of this work).
We note that if the diffusion coefficient $ B $ is a diagonal operator with respect to the orthonormal basis $ (e_n)_{ n \in \N} \subseteq H $, then the processes $ W^N $, $ N \in \N $, in Theorem~\ref{theorem:BurgersFinal} above are Wiener processes with computable covariance structure (cf.\ Corollary~\ref{corollary:Finite dimensional process Wn} below) and the approximation scheme~\eqref{eq:schemesIntro} is directly implementable up to an additional approximation error resulting from the numerical evaluations of 
Galerkin projections $ P_N $, $ N \in \N $.
We now briefly sketch the key ideas which we employ
to prove
Theorem~\ref{theorem:BurgersFinal}.
In the case of SPDEs with globally monotone 
nonlinearities one can, very roughly speaking, apply the It\^o formula to the squared
Hilbert space norm of the difference between the exact solution of the SPDE and its numerical approximation and, thereafter, employ the global  monotonicity property  
together with Gronwall's lemma and suitable 
uniform moment bounds for the 
solution and the numerical approximations 
to establish strong convergence rates.
This procedure, however, fails in the case of SPDEs with non-globally monotone coefficients.
We overcome this issue 
by bringing the classical Alekseev-Gr\"obner formula from deterministic numerical analysis (see, e.g., Hairer et al.\ \cite[Theorem~14.5]{HairerNonstiffProblems}) into play 
and by employing the fact that the 
considered approximation processes 
$ ( \y_t^{M, N} )_{ t \in [0,T] } $,
$ M, N \in \N $,  
(see~\eqref{eq:schemesIntro} above)
have
uniformly bounded exponential moments.
More specifically, we 
apply the extended version of the
Alexeev-Gr\"obner formula
in~\cite[Corollary~5.2]{JentzenLindnerPusnik2017a}
to 
a spatially semi-discrete version of the solution
$ ( X_t )_{ t \in [0,T] } $
of the considered SPDE
(see~\eqref{eq:Eq} above)
and its
numerical approximations
$ ( \y_t^{M, N} )_{ t \in [0,T] } $,
$ M, N \in \N $, 
(see~\eqref{eq:schemesIntro} above) 
in order to 
derive a suitable
error representation
(cf.\
Lemma~\ref{lemma:crucial_apply_Alekseev_Groebner}
below).
This allows us to estimate the strong approximation error 
by an appropriate integral expression
involving two main terms (cf.\ \eqref{eq:app_error} in  Corollary~\ref{corollary:main_error_estimate} below) which we analyze independently.
The first main term is, very roughly speaking,
the derivative 
of 
the
spatially semi-discrete version of 
$ (X_t)_{ t \in [0,T] } $
with respect to its initial value,
evaluated in a function of the
numerical approximations
$ ( \y_t^{M, N} )_{ t \in [0,T] } $,
$ M, N \in \N $, 
and the 
Wiener process $ (W_t)_{ t \in [0,T] } $. 
The second main term 
is a function of the numerical approximations
$ ( \y_t^{M, N} )_{ t \in [0,T] } $,
$ M, N \in \N $, 
and the 
Wiener process $ (W_t)_{ t \in [0,T] } $  
but does not involve
the 
spatially semi-discrete version of  
$ (X_t)_{ t \in [0,T]} $ 
(cf.\ 
Corollary~\ref{corollary:main_error_estimate}
below).
	A key step in establishing strong convergence rates is, 
	loosely speaking,  
	to obtain a uniform moment
	bound for
	the derivative 
	of the spatially semi-discrete version 
	of $ (X_t)_{ t \in [0,T] } $
	with respect to its initial value  
	in terms of an appropriate functional
	of the spatially semi-discrete version  of
	$ ( X_t )_{ t \in [0,T] } $ 
	and the numerical approximations 
	$ ( \y_t^{M, N} )_{ t \in [0,T] } $,
	$ M, N \in \N $
	(cf.\ Corollary~\ref{corollary:exp_bound} below).
	Applying a general result on exponential integrability
	from Cox et al.\ \cite[Corollary~2.4]{CoxHutzenthalerJentzen2013},
	this moment bound is then further estimated
	by appropriate exponential moments
	of the numerical approximations 
	$ ( \y_t^{M, N} )_{ t \in [0,T] } $,
	$ M, N \in \N $
	(cf.\ Lemma~\ref{lemma:inheritet} 
	below).
	The exponential moments established 
	in~\cite{JentzenLindnerPusnik2017b, JentzenPusnik2018Published}
	therefore yield a uniform upper bound
	for
	the first main term in the initial strong
	error estimate
	(cf.\
	Proposition~\ref{proposition:Exact_to_numeric_general},
	Corollary~\ref{Corollary:full_discrete_scheme_convergence},
	and the proof of
	Theorem~\ref{theorem:Exact_to_numeric}
	below).
	The fact that the 
	numerical approximations
	$ ( \y_t^{M, N} )_{ t \in [0,T] } $,
	$ M, N \in \N $, 
	enjoy sufficient regularity properties 
	(cf.\ Corollary~\ref{corollary:AprioriNumApp} and the regularity results
	in~\cite{JentzenLindnerPusnik2017b, 
		JentzenLindnerPusnik2017c})	
	ensures that the
	second main term 
	in the initial strong
	error estimate
	converges strongly with rates
	(cf.\ 
	Proposition~\ref{proposition:Exact_to_numeric_general}
	and
	the proof of
	Theorem~\ref{theorem:Exact_to_numeric}
	below).
	Combining the estimates for both main terms in the initial strong error estimate  finally 
	establishes strong convergence  
	rates for explicit space-time discrete approximations of the SPDE under consideration
	(cf.\ Theorem~\ref{theorem:Exact_to_numeric}
	and Corollaries~\ref{corollary:Exact_to_numeric}, 
	\ref{corollary:Burgers},
	and~\ref{corollary:BurgersFinal}
	below).

 Let us comment on the optimality of the convergence rates obtained in Theorem~\ref{theorem:BurgersFinal}.
 It is not clear to us whether the established strong convergence rates
 are essentially optimal or whether they can be substantially improved.
 In the simplified case $ c = 0 $, 
 where the nonlinearity is omitted and the stochastic Burgers equation
 in~\eqref{eq:Eq} reduces to a stochastic heat equation, lower bounds for strong and weak approximation errors
 are well understood
 (see, e.g.,  
 Becker et al.\ \cite{BeckerGessJentzenKloeden2018},
 Conus et al.\ \cite{ConusJentzenKurniawan2019},
 Davie \& Gaines~\cite{DavieGaines2001},
 Jentzen \& Kurniawan~\cite{KurniawanJentzen2015Arxiv}, 
  M\"{u}ller-Gronbach \& Ritter~\cite{MuellerGronbachRitter2007},
 M\"{u}ller-Gronbach et al.\ \cite{MuellerGronbachRitterWagner2008b, 
 	MuellerGronbachRitterWagner2008},
 and the references mentioned therein).
 In particular, e.g, 
 Becker et al.\ \cite[Theorem~1.1]{BeckerGessJentzenKloeden2018},  
 Conus et al.\ \cite[Lemma~7.2]{ConusJentzenKurniawan2019}, 
 Davie \& Gaines~\cite[Section~2.1]{DavieGaines2001},
 and
  M\"{u}ller-Gronbach et al.\ \cite[Theorem~4.2]{MuellerGronbachRitterWagner2008}
 indicate that the convergence rates in Theorem~\ref{theorem:BurgersFinal} 
 above 
 might not be optimal in the case $ c=0 $. 
 In the case $ c \neq 0 $, where the nonlinearity does not vanish, lower bounds for strong and weak approximation errors remain on open problem for future research.
The remainder of this article is organized as follows.
		In Subsection~\ref{subsection:Pathwise time approximation error estimates} 
		we apply the Alexeev-Gr\"obner formula from~\cite[Corollary~5.2]{JentzenLindnerPusnik2017a}
		and establish
		in 
		Lemma~\ref{lemma:main_error_estimate} below 
		a general
		pathwise estimate. 
		Combining this general pathwise estimate
		with suitable measurability results from the scientific literature allows us to establish in 
		Corollary~\ref{corollary:main_error_estimate}  
		in  
		Subsection~\ref{subsection:Measure} 
		below a strong 
		$ \mathcal{L}^p $ 
		estimate for the difference between
		the
		spatially semi-discrete version of the solution of the 
		considered SPDE
		and the considered numerical approximations.
		In Subsection~\ref{section:AprioriBounds}
		we employ
		Cox et al.\ \cite[Corollary~2.4]{CoxHutzenthalerJentzen2013}
		to provide
		an 
		appropriate a priori estimate for
		the derivative 
		of the spatially semi-discrete version of the solution of the considered SPDE 
		with respect to its initial value
		(see~\eqref{eq:UpperBound} in
		Lemma~\ref{lemma:inheritet} below).
		In Subsection~\ref{section:strongApriori}
		we 
		combine 
the results from
		Section~\ref{section:TemporaApp}
		and
		Subsection~\ref{section:AprioriBounds} 
		to obtain in Proposition~\ref{proposition:main_error_estimate} 
		a simplified upper
		bound for the
		strong error.
		In Subsection~\ref{section:AprioriBound}
		we establish suitable uniform moment bounds for the spatially semi-discrete version of the 
		considered SPDE
		which we then employ in Subsection~\ref{subsection:StrongConvergenceRates}  
		together with Proposition~\ref{proposition:main_error_estimate}
		to prove
		in Proposition~\ref{proposition:Exact_to_numeric_general}
		strong convergence with rates  		
		for space-time 
		discrete
		numerical approximations 
		with suitable integrability
		and regularity 
		properties
		for a large class of SPDEs.
In Subsection~\ref{subsection:StrongAprioriBoundNumeric} we show that the considered tamed-truncated numerical scheme 
enjoys appropriate integrability and measurability properties. These properties are then used together with 
Proposition~\ref{proposition:Exact_to_numeric_general} to establish in 
Theorem~\ref{theorem:Exact_to_numeric} 
in Subsection~\ref{subsection:StrongRates} below 
(see also Corollary~\ref{corollary:Exact_to_numeric})
 strong convergence rates for the considered tamed-truncated numerical scheme. 
In Section~\ref{section:Burgers}
we combine 
in
Corollaries~\ref{corollary:Burgers}   
and~\ref{corollary:BurgersFinal}
the results established in~\cite{JentzenLindnerPusnik2017c}
with
Corollary~\ref{corollary:Exact_to_numeric}
in this article 
to establish
strong convergence rates
in the case of 
additive trace-class noise driven stochastic Burgers equations.
\subsection{General setting}
\label{subsec:main Setting}
Throughout the article we frequently use the following setting.
\begin{setting} 
\label{setting:Notation}
\sloppy 
For all measurable spaces  
$ (\Omega_1, \F_1) $
and
$ (\Omega_2, \F_2) $ 
let $ \M( \F_1 , \F_2 ) $ 
be the set of all $ \F_1/\F_2 $-measurable functions,
for every set $ X $ let $ \mathcal{P}(X) $
be the power set of $ X $,
for every set $ X $ let 
$ \mathcal{P}_0(X) $ be the set given by 
$ \mathcal{P}_0(X) = \{ \theta \in \mathcal{P}(X) \colon \theta \text{ is a finite set} \} $, 
for every   
$ T \in (0,\infty) $ 
let $ \varpi_T $ 
be the set given by
$ \varpi_T
=
\{
\theta \in \mathcal{P}_0( [0,T] ) \colon
\{0, T\} \subseteq \theta \} $, 
for every $ T \in (0, \infty) $ let
$ \left | \cdot \right|_T \colon \varpi_T \to [0,T] $
satisfy 
for all
$ \theta \in \varpi_T $ that
\begin{equation}
| \theta  |_T
=
\max\!\Big\{
x \in (0,\infty)
\colon
\big(
\exists \, a, b \in \theta
\colon
\big[
x = b - a
\text{ and }
\theta \cap (a, \infty) \cap (-\infty, b) = \emptyset
\big]
\big)
\Big\},
\end{equation}
for every
$ \theta \in ( \cup_{T\in (0,\infty)} \varpi_T) $
let
$ \llcorner \cdot \lrcorner_\theta \colon [0,\infty) \to [0,\infty) $ 
satisfy 
for all
$ t\in (0,\infty) $
that 
$ \llcorner t \lrcorner_\theta
=  \max ( [0,t) \cap \theta ) $ 
and 
$ \llcorner 0 \lrcorner_\theta  = 0 $,
and
for every measure space
$ ( \Omega, \F, \mu) $,
every measurable space $ ( S, \mathcal{S} ) $,
every set $ R $,
and every function
$ f \colon \Omega \to R $
let
$ [f]_{\mu, \mathcal{S}} $
be the set given by
$ [f]_{\mu, \mathcal{S}} =
\{
g \in \M(\F,
\mathcal{S})
\colon
(
\exists\, A \in \F \colon
\mu(A) = 0  \text{ and } 
\{\omega \in \Omega \colon
f(\omega) \neq g(\omega)
\}
\subseteq A
)
\} $.
\end{setting} 
\begin{setting} 
\label{setting:main}
Assume Setting~\ref{setting:Notation},
let
$ ( H, \langle \cdot, \cdot \rangle _H, \left \| \cdot \right\|_H ) $
and
$ ( U, \langle \cdot, \cdot \rangle_U, \left\| \cdot \right \|_U ) $
be non-zero separable $ \R $-Hilbert spaces, 
let 
$ \H \subseteq H $
be an orthonormal basis of $ H $,  
let
$ \values \colon \H \to \R $
satisfy 
$ \sup_{h\in \H} \values_h < 0 $,
let
$ A \colon D( A ) \subseteq H \to H $ 
be the linear operator which satisfies
$ D(A) = \{
v \in H \colon \sum_{ h \in \H}  | \values_h \langle h, v \rangle_H  |^2 <  \infty \} $ 
and  
$ \forall \, v \in D(A) \colon 
A v = \sum_{ h \in \H} \values_h \langle h, v \rangle _H h $,
and
let
$ (H_r, \langle \cdot, \cdot \rangle_{H_r}, \left \| \cdot \right \|_{H_r} ) $, $ r\in \R $, be a family of interpolation spaces associated to $ -A $
(cf., e.g., \cite[Section~3.7]{SellYou2002}).
\end{setting}
Observe that the assumption in Setting~\ref{setting:main}   that 
$ (H_r, \langle \cdot, \cdot \rangle_{H_r}, \left \| \cdot \right \|_{H_r} ) $, 
$ r \in \R $,  
is a family of interpolation spaces associated to 
$ - A $ gives for all
$ r \in [0, \infty) $ 
that 
$ (H_r, \langle\cdot,\cdot\rangle_{H_r}, \left \| \cdot \right \|_{H_r} )
=
( D((-A)^r), 
\langle(-A)^r (\cdot),(-A)^r (\cdot )\rangle_H,
\left \| (-A)^r (\cdot) \right\|_H) $. 
\section{Time discretization error estimates based on an  Alexeev-Gr\"obner-type formula}
\label{section:TemporaApp}
\begin{setting}
\label{setting:Exact_vs_Numeric}
Assume Setting~\ref{setting:main},
assume
that
$ \dim(H) < \infty $,
let $ T \in (0,\infty) $, 
$ \theta \in \varpi_T $,
$ \xi \in H $, 
$ O \in \mathcal{C}( [0,T], H ) $,
$ \O \in \M( \B([0,T]) , \B(H) ) $,
$ F \in \mathcal{C}^1( H, H) $, 
let
$ \ff \colon H \to H $
be a function,
		for every
		$ s \in [0,T] $,
		$ x \in H $
		let
		$ X_{s,(\cdot)}^x
		=
		( X_{s,t}^x )_{ t \in [s,T] }
		\in \mathcal{C}([ s, T], H) 
		$
satisfy for all
		$ t \in [s,T] $
		that
		\begin{equation}
		\label{eq:Exact}
		X_{s,t}^x
		= 
		e^{ ( t - s ) A } x 
		+
		\int_s^t e^{ ( t - u ) A } F( X_{s,u}^x ) \, du + O_t - e^{(t-s)A}O_s
		,
		\end{equation}
		and
		let 
		$ \y \colon [0,T] \to H $
		satisfy
		for all $ t \in [0,T] $ that
		\begin{equation}
		\label{eq:Numeric}
		\y_t
		=
		e^{t A}
		\xi
		+
		\int_0^t
		e^{ ( t - \llcorner u \lrcorner_\theta ) A } 
		\ff ( \y_{ \llcorner u \lrcorner_\theta } ) \, du
		+
		\O_t
		.
		\end{equation}
\end{setting}
		Observe that for every topological space $ (X, \tau) $ we have that $ \B(X) $ is the smallest sigma-algebra on $ X $  which satisfies that $ \tau\subseteq \B(X) $.
\subsection{Pathwise temporal approximation error estimates}
\label{subsection:Pathwise time approximation error estimates}
In this subsection
we apply 
the extended Alekseev-Gr\"obner formula
in~\cite[Corollary~5.2]{JentzenLindnerPusnik2017a}
to express
the difference between
the 
exact solution
$ ( X_{0,t}^{ \xi + O_0 } )_{ t \in [0,T] } $
of 
the 
integral equation~\eqref{eq:Exact} above, started at time $ s=0 $ in $ x=\xi+O_0 $,
and the corresponding numerical approximation
$ (\y_t)_{ t \in [0,T]} $ 
in~\eqref{eq:Numeric} above
in terms of an appropriate integral
in Lemma~\ref{lemma:crucial_apply_Alekseev_Groebner} below. 
We then combine 
these auxiliary results with
Lemma~\ref{lemma:crucial_apply_Alekseev_Groebner}
and
Lemma~\ref{lemma:function_error}
to derive
an upper bound for the
approximation error in
Lemma~\ref{lemma:main_error_estimate}.
\begin{lemma}
\label{lemma:Simplified}
Assume Setting~\ref{setting:main}, 
assume that $ \dim(H) < \infty $, 
let 
$ T \in (0, \infty) $,
$ s \in [0,T] $,
$ x \in H $,
$ Z \in \M( \B([s,T]), \B(H)) $
satisfy 
$ \int_s^T \| Z_u \|_H \, du < \infty $,
and
let
$ Y \colon [s,T] \to H $
satisfy for all 
$ t \in [s,T] $ that
$ Y_t = e^{(t-s)A} x + \int_s^t e^{(t-u)A} Z_u \, du $.
Then 
\begin{enumerate}[(i)]
	\item \label{item:continuous} we have that $ Y \in \mathcal{C}([s,T], H) $
	and
	\item \label{item:rewrite} we have for all
	$ t \in [s,T] $ that
	$ Y_t = x + \int_s^t [ A Y_u  + Z_u] \, du $.
\end{enumerate}
\end{lemma}
\begin{proof}[Proof of Lemma~\ref{lemma:Simplified}]

Throughout this proof
assume w.l.o.g.\ that $ s \in [0,T) $.
Note that
the fact that
$ \dim (H) < \infty $
ensures that
for all 
$ t \in [s, T] $ we have that
$ \int_s^T \| e^{ (s-u) A} Z_u \|_H \, du < \infty $
and
\begin{equation}
\begin{split}
\label{eq:Integral form}
Y_t 
= 
e^{( t - s ) A}
\bigg( 
x
+
\int_s^t
e^{ (s-u) A} Z_u  \, du 
\bigg) 
.
\end{split}
\end{equation}
Moreover,
observe that the 
dominated convergence theorem gives that
\begin{equation}
\bigg( 
[s,T] \ni t 
\mapsto 
\int_s^t e^{(s-u)A} Z_u \, du 
\bigg) 
\in 
\mathcal{C}( [s, T], H ) 
.
\end{equation}
Combining~\eqref{eq:Integral form}
and the fact that
$ ( [s, T] \ni t \mapsto e^{(t-s)A} \in L(H) )
\in \mathcal{C}([s,T], L(H) ) $
therefore
justifies item~\eqref{item:continuous}.
Next note that~\eqref{eq:Integral form},
the fact that
$ [s, T] \times H \ni (t, h) \mapsto e^{(t-s) A} h \in H $
is continuously differentiable,
and, e.g., \cite[Corollary~2.8]{JentzenLindnerPusnik2017a}
(applies with
$ (V, \left \| \cdot \right \|_V ) = (H, \left \| \cdot \right \|_H) $,
$ (W, \left \| \cdot \right \|_W ) = (H, \left \| \cdot \right \|_H) $,
$ a = s $,
$ b = T $,
$ F =  ([s, T] \ni t \mapsto  
(
x + \int_s^t
e^{ (s-u) A } 
Z_u \, du ) \in  H ) $,
$ \phi = ( [s, T] \times H \ni (t, h) \mapsto e^{(t - s) A} h \in H ) $,
$ f = ([s, T] \ni u \mapsto ( e^{ (s-u) A } 
Z_u )
\in H
) $
in the setting of~\cite[Corollary~2.8]{JentzenLindnerPusnik2017a}) 
yield that
for all $ t \in [s, T] $ we have that
\begin{equation}
\begin{split} 
Y_t - x 
&
=
\int_s^t 
\bigg[ 
A e^{(u-s) A}
\bigg(
x + \int_s^u
e^{ (s-r) A } 
Z_r \, dr \bigg)
+
e^{(u-s)A} 
e^{ (s-u) A } 
Z_u
\bigg]
\, du
\\
&
=
\int_s^t 
[ 
A 
Y_u
+ 
Z_u
]
\, du
.
\end{split} 
\end{equation}
This justifies item~\eqref{item:rewrite}.
The proof of Lemma~\ref{lemma:Simplified}
is hereby completed.
\end{proof} 
\begin{lemma}
	\label{lemma:crucial_apply_Alekseev_Groebner}
Assume Setting~\ref{setting:Exact_vs_Numeric}.
	Then  
	\begin{enumerate}[(i)]
	\item \label{item:minor detail}
	we have that $ (\y - \O ) \in \mathcal{C}( [0,T], H ) $, 
	\item \label{item:differentiability} we have 
	that
	$ (
	\{ (u,v) \in [0,T]^2 \colon u \leq v \}
	\times H \ni (s,t,x)
	\mapsto X_{s,t}^x 
	\in 
	H
	)
	\in 
	\mathcal{C}^{0,0,1}
	( \{ (u,v) \in [0,T]^2 \colon u \leq v \}
	\times H,
	H ) $,
	\item \label{item:Integral exists}
	we have for all $ t \in [0,T] $ that 
	\begin{multline}  
	\big(
	[0,t] \ni s \mapsto
	\big[ 
	\tfrac{\partial}{\partial x}
	X_{s, t}^{ \y_s - \O_s + O_s}
	\big(
	e^{ ( s - \llcorner s \lrcorner_\theta ) A }
	\ff ( \y_{\llcorner s \lrcorner_\theta} )
	-
	F ( \y_s - \O_s + O_s ) 
	\big)
	\big] 
	\in H 
	\big)  
	\\
	\in  
	\mathcal{M}( \B([0,t]), \B(H) ) 
	, 
	\end{multline}
	\item \label{item:finite integral} 
	we have for all $ t \in [0,T] $ that 
	\begin{equation}
	\int_0^t
	\big\| 
	\tfrac{\partial}{\partial x}
	X_{s, t}^{ \y_s - \O_s + O_s}
	\big(
	e^{ ( s - \llcorner s \lrcorner_\theta ) A }
	\ff ( \y_{\llcorner s \lrcorner_\theta} )
	-
	F ( \y_s - \O_s + O_s ) 
	\big)
	\big\|_H
	\,
	ds
	<
	\infty 
	,
	\end{equation}
	and
	\item \label{item:MainFormula}
	we have for all 
	$ t \in [0,T] $ that
	\begin{equation}
	\begin{split}
	\!\!\!\!
	\y_t - X_{0, t}^{ \xi + O_0 }
	=
	\O_t - O_t
	+
	\int_0^t
	\tfrac{\partial}{\partial x}
	X_{s, t}^{ \y_s - \O_s + O_s}
	\big(
	e^{ ( s - \llcorner s \lrcorner_\theta ) A }
	\ff ( \y_{\llcorner s \lrcorner_\theta} )
	-
	F ( \y_s - \O_s + O_s ) 
	\big)
	\,
	ds
	.
	\end{split}
	\end{equation}
	\end{enumerate}
\end{lemma}
\begin{proof}[Proof of Lemma~\ref{lemma:crucial_apply_Alekseev_Groebner}]
Throughout this proof let
$ \lambda 
\colon
\mathcal{B}( [0, T] )
\rightarrow [0, T] $
be the Lebesgue-Borel
measure on
$ [0,T] $,
let
$ \Y \colon [0,T] \to H $
satisfy 
for all  
$ t \in [0, T] $ 
that
$ \Y_t = \y_t - \O_t $,
and let
$ \X^x_{s, ( \cdot ) } 
= 
( \X^x_{s,t} )_{ t \in [s,T] }
\colon [s,T] \to H $,
$ s \in [0,T] $,
$ x \in H $,  
satisfy for all 
$ s \in [0,T] $,
$ t \in [s,T] $,
$ x \in H $ 
that
$ \X_{s,t}^x = X_{s,t}^{x + O_s} - O_t $.
Note that~\eqref{eq:Exact} 
gives that 
for all
$ s \in [0,T] $,
$ t \in [s,T] $,
$ x \in H $ 
we have that
\begin{equation}
\begin{split} 
\label{eq:FirstHelp}
\X_{s,t}^x
=
e^{(t-s)A} x
+
\int_s^t
e^{(t-u)A}
F( \X_{s, u }^x + O_u ) \, du
.
\end{split}
\end{equation}
The fact that 
for all
$ s \in [0,T] $,
$ x \in H $
we have that
$  ([s,T] \ni t \mapsto F( \X^x_{s,t} + O_t) \in H ) 
\in \mathcal{C}( [s,T], H) $
and
item~\eqref{item:rewrite} of 
Lemma~\ref{lemma:Simplified}
(applies with 
$ T = T $,
$ s = s $,
$ x = x $,
$ Z = ([s,T] \ni t \mapsto F( \X^x_{s,t} + O_t) \in H ) $, 
$ Y = ( [s,T] \ni t \mapsto \X^x_{s, t} \in H) $
for 
$ s \in [0,T] $,
$ x \in H $
in the setting of
item~\eqref{item:rewrite} of 
Lemma~\ref{lemma:Simplified})
therefore
ensure that
for all 
$ s \in [0,T] $,
$ t \in [s,T] $,
$ x \in H $
we have
that 
\begin{equation}
\label{eq:strong_solution1}
\X_{s,t}^x
=
x
+
\int_s^t [ A \X_{s,u}^x +  F ( \X_{s,u}^x + O_u ) ] \, du
.
\end{equation}
Next note that~\eqref{eq:Numeric}
gives that
 for all
$ t \in [0,T] $ we have that
\begin{equation} 
\label{eq:kind of important}
\Y_t
=
e^{t A}
\xi
+
\int_0^t
e^{ ( t - \llcorner u \lrcorner_\theta ) A } 
\ff ( \Y_{ \llcorner u \lrcorner_\theta }
+
 \O_{ \llcorner u \lrcorner_\theta }
 ) \, du  
.
\end{equation}
In addition, observe that
the fact that
$ [0,T] \ni u \mapsto
e^{( u - \llcorner u \lrcorner_{ \theta } ) A } \in L( H ) $
is  bounded and left-continuous 
gives that
\begin{equation}
\label{eq:EnsuresL1 again}
( [0,T] \ni u \mapsto 
e^{( u - \llcorner u \lrcorner_\theta)A}
\ff ( \Y_{ \llcorner u \lrcorner_\theta } + \O_{ \llcorner u \lrcorner_\theta } ) 
\in H )
\in \mathcal{L}^1( \lambda ; H )  
.
\end{equation}
Combining~\eqref{eq:kind of important}
and
Lemma~\ref{lemma:Simplified}
(applies with 
$ T = T $,
$ s = 0 $,
$ x = \xi $,
$ Z = ( 
[0,T] \ni u
\mapsto 
e^{ ( u - \llcorner u \lrcorner_\theta ) A } 
\ff ( \Y_{ \llcorner u \lrcorner_\theta }
+
\O_{ \llcorner u \lrcorner_\theta }
)
\in H 
) $,
$ Y = \Y $
in the setting of
Lemma~\ref{lemma:Simplified})
therefore
verifies that
\begin{enumerate}[(a)] 
\item \label{item:continuity} we have that
$ \Y \in \mathcal{C}( [0,T], H) $ 
and
\item \label{item:change integral}
we have
for all $ t \in [0,T] $ that
\begin{equation}
\begin{split}
\label{eq:strong_solution2old}
\Y_t 
&
=
\xi
+ 
\int_0^t
[
A \Y_u
+
e^{ (u - \llcorner u  \lrcorner_\theta ) A }
\ff (  \Y_{ \llcorner u  \lrcorner_\theta }
+
\O_{\llcorner u  \lrcorner_\theta} )
]
\, du
.
\end{split}
\end{equation}
\end{enumerate}
Observe that
item~\eqref{item:continuity}
and the fact that
$ \Y = \y - \O $
justify item~\eqref{item:minor detail}.
Furthermore, note that~\eqref{eq:EnsuresL1 again}, 
the assumption that
$ O\in \mathcal{C} ([0,T],H) $,
the fact that $ F\in \mathcal{C} (H,H) $,
and item~\eqref{item:continuity} 
ensure that
\begin{equation}
\label{eq:EnsuresL1}
( [0,T] \ni u \mapsto 
e^{( u - \llcorner u \lrcorner_\theta)A}
\ff ( \Y_{ \llcorner u \lrcorner_\theta } + \O_{ \llcorner u \lrcorner_\theta } )
-
F( \Y_u + O_u )  
\in H )
\in \mathcal{L}^1( \lambda ; H )  
.
\end{equation}
In addition, 
observe that
the assumption that
$ \dim(H) < \infty $,
the fact that
$ O\in \mathcal{C} ([0,T],H) $,
the fact that
$ F \in \mathcal{C}(H, H) $,
and
item~\eqref{item:continuity} 
yield that
\begin{equation}
\label{eq:also used}
( [0,T] \ni u \mapsto
  A \Y_u  
+ 
  F( \Y_u + O_u )  
  \in H
)
\in \L^1(\lambda; H)
.
\end{equation}
This,
\eqref{eq:EnsuresL1}, 
and
item~\eqref{item:change integral} 
give
that
for all $ t \in [0,T] $ we have that
\begin{equation}
\begin{split}
\label{eq:strong_solution2}
\Y_t  
=
\xi
+ 
\int_0^t
[
A \Y_u
+ 
F (  \Y_u
+
O_u )
]
\, du
+
\int_0^t
[
e^{ (u - \llcorner u  \lrcorner_\theta ) A } 
\ff (  \Y_{ \llcorner u  \lrcorner_\theta }
+
\O_{\llcorner u  \lrcorner_\theta} )
-
F (  \Y_u
+
O_u )
]
\, du
.
\end{split}
\end{equation}
Combining~\eqref{eq:strong_solution1},
\eqref{eq:EnsuresL1}, 
\eqref{eq:also used},
the fact that $ ( [0,T] \times H \ni (u, h) \mapsto A h +  F(h + O_u ) \in H )
\in \mathcal{C}^{0, 1}(
[0,T] \times H, H ) $, 
and~\cite[Corollary~5.2]{JentzenLindnerPusnik2017a}
(applies with
$ V = H $, 
$ T = T $,
$ f = ( [0,T] \times H \ni (u, h) \mapsto A h +  F ( h + O_u ) \in H ) $,
$ Y = \Y $,
$ E = ( [0,T] \ni u \mapsto 
e^{( u - \llcorner u \lrcorner_\theta)A}
\ff ( \Y_{ \llcorner u \lrcorner_\theta } 
+ \O_{ \llcorner u \lrcorner_\theta } )
-
 F( \Y_u + O_u )  
\in H ) $,
$ X_{s,t}^x = \X_{s,t}^x $ 
for 
$ x \in H $,
$ t \in [s,T] $,
$ s \in [0,T] $
in the setting of~\cite[Corollary~5.2]{JentzenLindnerPusnik2017a})
hence
verifies 
that
\begin{enumerate}[(A)]
	\item \label{item:AlmostC1} we have that
$ (
\{ (u,v) \in [0,T]^2 \colon u \leq v \}
\times H \ni (s,t,x)
\mapsto \X_{s,t}^x 
\in 
H
)
\in  
\mathcal{C}^{0,0,1}
( \{ (u,v) \in [0,T]^2 \colon u \leq v \}
\times H,
H ) $,
\item \label{item:Measurable again}
we have for all
$ t \in [0,T] $ that
\begin{multline}
\big( 
[0, t] 
\ni s \mapsto 
\big[ 
\tfrac{\partial}{\partial x}
\X_{s, t}^{ \Y_s}
\big(
e^{ ( s - \llcorner s \lrcorner_\theta ) A }
\ff ( \Y_{\llcorner s \lrcorner_\theta} 
+
\O_{\llcorner s \lrcorner_\theta}
)
-
F ( \Y_s + O_s ) 
\big)
\big] 
\in H
\big) 
\\
\in 
\M( \B( [0,t] ), \B(H) )
,
\end{multline}
\item \label{item:FiniteIntegral} we have for all
$ t \in [0,T] $ that
\begin{equation}
\begin{split}
\int_0^t 
\big\| 
\tfrac{\partial}{\partial x}
\X_{s, t}^{ \Y_s}
\big(
e^{ ( s - \llcorner s \lrcorner_\theta ) A }
\ff ( \Y_{\llcorner s \lrcorner_\theta} 
+
\O_{\llcorner s \lrcorner_\theta}
)
-
F ( \Y_s + O_s ) 
\big)
\big\|_H 
\, ds
< \infty 
,
\end{split}
\end{equation}
and 
\item \label{item:MainIdentity} we have for all $ t \in [0,T] $ that
\begin{equation}
\begin{split}
\Y_t - \X_{0,t}^{ \Y_0 }
=
\int_0^t
\tfrac{\partial}{\partial x}
\X_{s, t}^{ \Y_s}
\big(
e^{ ( s - \llcorner s \lrcorner_\theta ) A }
\ff ( \Y_{\llcorner s \lrcorner_\theta} 
+
\O_{\llcorner s \lrcorner_\theta}
)
-
 F ( \Y_s + O_s ) 
\big)
\,
ds
.
\end{split}
\end{equation}
\end{enumerate} 
Observe that
the fact that
for all
$ s \in [0,T] $, 
$ t \in [s,T] $,
$ x \in H $
we have that 
$ X_{s,t}^{x}
=
\X_{s,t}^{x - O_s} + O_t $,
the 
assumption that 
$ O \in \mathcal{C}([0,T], H ) $,
and item~\eqref{item:AlmostC1}
justify 
item~\eqref{item:differentiability}.
Next note that
item~\eqref{item:Measurable again}, 
the fact that 
for all
$ s \in [0,T] $,
$ t \in [s,T] $ 
we have that
$ \frac{ \partial }{ \partial x } \X_{s,t}^{\Y_s} 
= 
\frac{ \partial }{ \partial x }
X_{s,t}^{\Y_s + O_s} $,
and 
the fact that
for all $ s \in [0,T] $ we have that
$ \Y_s = \y_s - \O_s $ 
give item~\eqref{item:Integral exists}.
In addition, observe that 
item~\eqref{item:FiniteIntegral}, 
the fact that 
for all 
$ s \in [0,T] $,
$ t \in [s,T] $ 
we have that
$ \frac{ \partial }{ \partial x } \X_{s,t}^{\Y_s} 
= 
\frac{ \partial }{ \partial x }
X_{s,t}^{\Y_s + O_s} $,
and 
the fact that
for all $ s \in [0,T] $ we have that
$ \Y_s = \y_s - \O_s $ 
yield
item~\eqref{item:finite integral}.
Moreover, note that 
item~\eqref{item:MainIdentity},
the fact that
for all
$ t \in [0,T] $ 
we have that
$ \X_{0,t}^\xi = X_{0,t}^{\xi + O_0} - O_t $,
the fact that 
for all 
$ s \in [0,T] $,
$ t \in [s,T] $  
we have that
$ \frac{ \partial }{ \partial x } \X_{s,t}^{\Y_s} 
= 
\frac{ \partial }{ \partial x }
X_{s,t}^{\Y_s + O_s} $,
and 
the fact that
for all $ s \in [0,T] $ we have that
$ \Y_s = \y_s - \O_s $
justify item~\eqref{item:MainFormula}. 
The proof
of Lemma~\ref{lemma:crucial_apply_Alekseev_Groebner}
is hereby completed.
\end{proof}
\begin{lemma}
	\label{lemma:function_error}
Assume Setting~\ref{setting:Exact_vs_Numeric},
	let $ C , c \in [1,\infty) $,  
	$ \gamma \in [0,1] $, 
	$ \delta \in [0, \gamma] $, 
	$ \iota \in [0, 1-\delta] $,
	$ \kappa \in \R $, 
	and 
	assume for all 
	$ x, y \in H $ 
	that 
	$
	\| F( x )- F( y )  \|_H 
	\leq 
	C  \| x - y \|_{H_\delta} ( 1 + \| x \|_{H_\kappa}^c + \| y \|_{H_\kappa}^c ) 
	$.
	Then we have for all $ t \in [0,T] $ that 
	\begin{equation}
	\begin{split}
	\label{eq:IsImplied}
	& 
	\| 
	e^{ ( t - \llcorner t \lrcorner_\theta ) A }
	\ff ( \y_{\llcorner t \lrcorner_\theta} )
	-
	 F ( \y_t - \O_t + O_t )  
	\|_H 
	\\
	&
	\leq 
	[| \theta |_T]^{\gamma - \delta}
	\| \ff ( \y_{\llcorner t \lrcorner_\theta } ) \|_{H_{\gamma - \delta}}
	+
	\| \ff ( \y_{\llcorner t \lrcorner_\theta } ) -  F ( \y_{\llcorner t \lrcorner_\theta } ) \|_H
	+ 
	C
	\Big( 
	[| \theta |_T]^{ \gamma - \delta }
	\| \xi \|_{ H_\gamma } 
	\\
	&
	\quad 
	+ 
	[ | \theta |_T ]^{ 1 - \delta } 
	\|
	\ff ( \y_{ \llcorner t \lrcorner_\theta } )
	\|_H 
	+ 
	[ |\theta|_T ]^\iota
	\smallint_0^{ \llcorner t \lrcorner_\theta }
	( \llcorner t \lrcorner_\theta - \llcorner s \lrcorner_\theta )^{-\delta-\iota} 
	\|
	\ff ( \y_{ \llcorner s \lrcorner_{ \theta } } )
	\|_H
	\, ds
	\\
	&
	\quad 
	+
	\| \O_t - \O_{ \llcorner t \lrcorner_\theta } \|_{H_\delta}
	+ 
	\| \O_t - O_t \|_{H_\delta}
	\Big)
	\big(
	1
	+
	\| \y_{\llcorner t \lrcorner_\theta} \|_{H_\kappa}^c
	+
	(
	\|
	\y_t
	\|_{H_\kappa}
	+
	\| \O_t - O_t \|_{H_\kappa}
	)^c
	\big)
	.
	\end{split}
	\end{equation}
\end{lemma}
\begin{proof}[Proof
	of Lemma~\ref{lemma:function_error}]
	Note that
	the triangle inequality yields
	that for all $ t \in [0,T] $ 
	we have that
	\begin{equation}
	\begin{split} 
	\label{eq:first_triangle}
	&
	\| 
	e^{(t -\llcorner t \lrcorner_\theta)A}  
	\ff ( \y_{\llcorner t \lrcorner_{\theta}} )
	-
	 F( \y_t - \O_t + O_t )
	\|_H
	\\
	&
	\leq
	\| 
	(
	e^{(t -\llcorner t \lrcorner_\theta)A}  
	-
	\operatorname{Id}_H
	)
	\ff ( \y_{\llcorner t \lrcorner_{\theta}} )
	\|_H 
	\\
	&
	\quad
	+
	\|  
	\ff ( \y_{\llcorner t \lrcorner_{\theta}} )
	-
	F ( \y_{\llcorner t \lrcorner_{\theta}} )
	\|_H 
	+
	\| 
 	F( \y_{\llcorner t \lrcorner_{\theta}} )
	-
	 F( \y_{ t } - \O_t + O_t ) 
	\|_H
	.
	\end{split}
	\end{equation}
	In addition, observe that for all
	$ t \in [0,T] $ we have that
	\begin{equation}
	\begin{split} 
	\label{eq:smoothing}
	&
	\| 
	(
	e^{(t -\llcorner t \lrcorner_\theta)A}  
	-
	\operatorname{Id}_H
	)
	\ff ( \y_{\llcorner t \lrcorner_{\theta}} )
	\|_H
	\leq
	\| 
	(-A)^{\delta - \gamma}
	(
	e^{(t -\llcorner t \lrcorner_\theta)A}  
	-
	\operatorname{Id}_H
	)
	\|_{L(H)}
	\|
	(-A)^{\gamma - \delta}
	\ff ( \y_{\llcorner t \lrcorner_{\theta}} )
	\|_H
	\\
	&
	\leq
	( t - \llcorner t \lrcorner_\theta )^{\gamma -  \delta}
	\|  \ff ( \y_{\llcorner t \lrcorner_\theta} ) \|_{ H_{\gamma - \delta} }
	\leq
	[| \theta |_T]^{\gamma - \delta} 
	\|  \ff ( \y_{\llcorner t \lrcorner_\theta} ) \|_{ H_{\gamma - \delta} }
	.
	\end{split}
	\end{equation}
	Moreover, note that for all 
	$ t \in [0,T] $ we have that
	\begin{equation}
	\begin{split}
	\label{eq:diff_fun}
	&
	\| 
	F( \y_{\llcorner t \lrcorner_{\theta}} )
	-
	F( \y_t - \O_t + O_t ) 
	\|_H
	\\
	&
	\leq
	C 
	\|
	\y_{\llcorner t \lrcorner_{\theta}}
	-
	\y_t 
	+
	\O_t
	-
	O_t
	\|_{H_\delta}
	\big(
	1
	+
	\|
	\y_{\llcorner t \lrcorner_{\theta}}
	\|_{H_\kappa}^c
	+
	\|
	\y_{ t } - \O_t + O_t
	\|_{H_\kappa}^c
	\big) 
	.
	\end{split}
	\end{equation}
	The triangle inequality hence yields that for all 
	$ t \in [0,T] $ we have that
	\begin{equation}
	\begin{split} 
	\label{eq:fun_diff2}
	&
	\| 
	F( \y_{\llcorner t \lrcorner_{\theta}} )
	-
	F( \y_t - \O_t + O_t ) 
	\|_H
	\\
	&
	\leq
	C 
	\big(
	\|
	\y_{\llcorner t \lrcorner_{\theta}}
	-
	\y_t
	\|_{H_\delta}
	+
	\| 
	\O_t 
	-
	O_t
	\|_{H_\delta}
	\big)
	\big(
	1
	+
	\|
	\y_{\llcorner t \lrcorner_{\theta}}
	\|_{H_\kappa}^c
	+
	(
	\|
	\y_{ t } 
	\|_{H_\kappa}
	+  
	\|
	\O_t - O_t
	\|_{H_\kappa}
	)^c
	\big) 
	.
	\end{split}
	\end{equation}
	In the next step we observe that for all 
	$ t \in [0,T] $ we have that
	\begin{equation}
	\begin{split}
	&
	\| \y_t - \y_{\llcorner t \lrcorner_{\theta}} \|_{ H_\delta }
	\leq
	\|e^{\llcorner t \lrcorner_\theta A} (e^{(t-\llcorner t\lrcorner_{\theta})A} - \operatorname{Id}_H )
	\xi
	\|_{ H_\delta }
	+
	\int_{\llcorner t \lrcorner_{\theta}}^t
	\|
	e^{(t-\llcorner s\lrcorner_{\theta})A}  
	\ff ( \y_{\llcorner s\lrcorner_{\theta}})
	\|_{ H_\delta }
	\,
	ds
	\\
	&
	\quad
	+
	\int_0^{ \llcorner t \lrcorner_\theta }
	\| ( e^{ ( t - \llcorner s \lrcorner_\theta ) A }
	-
	e^{ ( \llcorner t \lrcorner_\theta - \llcorner s \lrcorner_\theta ) A }
	)
	\ff ( \y_{ \llcorner s \lrcorner_{ \theta } } )
	\|_{ H_\delta } 
	\, ds
	+
	\| \O_t - \O_{ \llcorner t \lrcorner_\theta } \|_{ H_\delta }
	\\
	&
	\leq 
	\| ( - A)^{\delta - \gamma} (e^{(t-\llcorner t\lrcorner_{\theta})A} - \operatorname{Id}_H )
	\|_{L(H)}
	\| \xi \|_{H_\gamma}
	+
	\int_{\llcorner t \lrcorner_{\theta}}^t
	\|
	(-A)^\delta
	e^{(t-\llcorner t\lrcorner_{\theta})A}  
	\|_{L(H)}
	\|
	\ff ( \y_{\llcorner t\lrcorner_{\theta}})
	\|_H
	\,
	ds
	\\
	&
	\quad
	+
	\int_0^{ \llcorner t \lrcorner_\theta }
	\| 
	(-A)^{\delta + \iota}
	e^{ ( \llcorner t \lrcorner_\theta - \llcorner s \lrcorner_\theta ) A }
	\|_{ L(H) }
	\| 
	(-A)^{-\iota}
	(
	e^{ ( t - \llcorner t \lrcorner_\theta ) A }
	-
	\operatorname{Id}_H
	)
	\|_{ L(H) }
	\|
	\ff ( \y_{ \llcorner s \lrcorner_{ \theta } } )
	\|_H
	\, ds
	\\
	&
	\quad
	+
	\| \O_t - \O_{ \llcorner t \lrcorner_\theta } \|_{ H_\delta }
	\\
	&
	\leq 
	( t - \llcorner t \lrcorner_{\theta} )^{ \gamma - \delta }
	\| \xi \|_{H_\gamma}
	+
	(t-\llcorner t \lrcorner_\theta)^{1-\delta}
	\|
	\ff ( \y_{\llcorner t\lrcorner_{\theta}})
	\|_H
	\\
	&
	\quad
	+ 
	\int_0^{ \llcorner t \lrcorner_\theta }
	( \llcorner t \lrcorner_\theta - \llcorner s \lrcorner_\theta )^{-\delta-\iota}
	( t - \llcorner t \lrcorner_\theta )^\iota
	\|
	\ff ( \y_{ \llcorner s \lrcorner_{ \theta } } )
	\|_H
	\, ds
	+
	\| \O_t - \O_{ \llcorner t \lrcorner_\theta } \|_{ H_\delta }
	\\
	&
	\leq 
	[ |\theta|_T ]^{ \gamma - \delta }
	\| \xi \|_{H_\gamma}
	+
	[ |\theta|_T ]^{1-\delta}
	\|
	\ff ( \y_{\llcorner t\lrcorner_{\theta}})
	\|_H
	\\
	&
	\quad
	+
	[ |\theta|_T ]^\iota
	\int_0^{ \llcorner t \lrcorner_\theta }
	( \llcorner t \lrcorner_\theta - \llcorner s \lrcorner_\theta )^{-\delta-\iota} 
	\|
	\ff ( \y_{ \llcorner s \lrcorner_{ \theta } } )
	\|_H
	\, ds
	+
	\| \O_t - \O_{ \llcorner t \lrcorner_\theta } \|_{ H_\delta }
	.
	\end{split} 
	\end{equation} 
	Combining~\eqref{eq:first_triangle},
	\eqref{eq:smoothing},
	and~\eqref{eq:fun_diff2}
	therefore justifies~\eqref{eq:IsImplied}. 
	The proof of Lemma~\ref{lemma:function_error}
	is hereby completed.
\end{proof}
\begin{lemma}
	\label{lemma:main_error_estimate}
	Assume Setting~\ref{setting:Exact_vs_Numeric},
	let $ C, c \in [1,\infty) $,   
	$ \gamma \in [0, 1] $,
	$ \delta \in [0, \gamma ] $,  
	$ \iota \in [0, 1 - \delta] $,
	$ \kappa \in \R $,
	and 
	assume for all 
	$ x, y \in H $ 
	that 
	$
	\| F( x )- F( y )  \|_H 
	\leq 
	C  \| x - y \|_{H_\delta} ( 1 + \| x \|_{H_\kappa}^c  + \| y \|_{H_\kappa}^c ) 
	$.  
	Then 
\begin{enumerate}[(i)]
	\item \label{item:minor detail 2} we have that
	$ ( \y - \O ) \in \mathcal{C}( [0,T], H ) $,
\item \label{item:continuity2} we have that
$ (
\{ (u,v) \in [0,T]^2 \colon u \leq v \}
	\times H \ni (s,t,x)
	\mapsto X_{s,t}^x 
	\in 
	H
	)
	\in
	\mathcal{C}^{0,0,1}
	( \{ (u,v) \in [0,T]^2 \colon u \leq v \}
	\times H,
	H ) $,
and
\item \label{item: Basic estimate} we have for all 
	$ t \in [0,T] $ 
	that
	\begin{equation}
	\begin{split}
	&
	\|  \y_t - X_{0, t}^{ \xi + O_0 } \|_H  
	\leq
	\| \O_t - O_t \|_H
	+
	\int_0^t
	\big \| \tfrac{\partial}{\partial x}
	X_{s, t}^{ \y_s - \O_s + O_s }
	\big \|_{L(H)}
	\Big\{
	[ | \theta |_T ]^{\gamma - \delta}
	\| \ff ( \y_{\llcorner s \lrcorner_\theta } ) \|_{H_{\gamma-\delta}}
	\\
	& 
	\quad
	+
	\| \ff ( \y_{\llcorner s \lrcorner_\theta } ) - F ( \y_{\llcorner s \lrcorner_\theta } ) \|_H
	+
	C
	\Big( 
	[ | \theta |_T ]^{ \gamma - \delta }
	\| \xi \|_{H_\gamma}
	+ 
	[| \theta |_T]^{1-\delta} 
	\|
	\ff ( \y_{ \llcorner s \lrcorner_\theta } )
	\|_H
	\\
	&
	\quad 
	+ 
	[ |\theta|_T ]^\iota
	\smallint_0^{ \llcorner s \lrcorner_\theta }
	( \llcorner s \lrcorner_\theta - \llcorner u \lrcorner_\theta )^{-\delta-\iota} 
	\|
	\ff ( \y_{ \llcorner u \lrcorner_{ \theta } } )
	\|_H
	\, du
	+
	\| \O_s - \O_{ \llcorner s \lrcorner_\theta } \|_{H_\delta}
	\\
	&
	\quad
	+ 
	\| \O_s - O_s \|_{H_\delta}
	\Big)
	\big(
	1
	+
	\| \y_{\llcorner s \lrcorner_\theta} \|_{H_\kappa} 
	+
	\|
	\y_s
	\|_{H_\kappa}
	+
	\| \O_s - O_s \|_{H_\kappa}
	\big)^c
	\Big\}
	\,
	ds
	.
	\end{split}
	\end{equation}
	\end{enumerate}
\end{lemma}
\begin{proof}[Proof of Lemma~\ref{lemma:main_error_estimate}]
	 Observe that item~\eqref{item:minor detail}
	 of
	 Lemma~\ref{lemma:crucial_apply_Alekseev_Groebner}   
	 gives item~\eqref{item:minor detail 2}.
	 In addition, note that
	 item~\eqref{item:differentiability} 
	 of
	 Lemma~\ref{lemma:crucial_apply_Alekseev_Groebner}   
	 justifies
	 item~\eqref{item:continuity2}.
	 Moreover, observe that 
	 items~\eqref{item:Integral exists}
	 and~\eqref{item:MainFormula}
	 of
	 Lemma~\ref{lemma:crucial_apply_Alekseev_Groebner} 
	 and the triangle inequality
	 yield that
	 for all $ t \in [0,T] $ 
	 we have that 
	 \begin{equation}
	 \begin{split}
	 \label{eq:AlexG} 
	 \|
	 \y_t - X_{0, t}^{ \xi + O_0 }
	 \|_H
	 &
	 \leq 
	 \| \O_t - O_t \|_H
	 \\
	 &
	 \quad 
	 +
	 \int_0^t
	 \big\| 
	 \tfrac{\partial}{\partial x}
	 X_{s, t}^{ \y_s - \O_s + O_s }
	 \big(
	 e^{ ( s - \llcorner s \lrcorner_\theta ) A }
	 \ff ( \y_{\llcorner s \lrcorner_\theta} )
	 -
	 F ( \y_s - \O_s + O_s ) 
	 \big)
	 \big\|_H
	 \,
	 ds
	 .
	 \end{split}
	 \end{equation}
	 Lemma~\ref{lemma:function_error}
	 	 (applies with
	 	 $ C = C $,
	 	 $ c = c $, 
	 	 $ \gamma = \gamma $,
	 	 $ \delta = \delta $,
	 	 $ \iota = \iota $,
	 	 $ \kappa = \kappa $
	 	 in the setting of Lemma~\ref{lemma:function_error}) 
	 	and the fact that
	 	$ \forall \, a, b \in [0, \infty) $,
	 	$ c \in [1, \infty) \colon 
	 	1 + a^c + b^c \leq (1 + a +b)^c $ 
	therefore justify item~\eqref{item: Basic estimate}.
		The proof of Lemma~\ref{lemma:main_error_estimate}
		is hereby completed. 
\end{proof}
\subsection{Strong temporal approximation error estimates}
\label{subsection:Measure}
In this subsection we establish in Corollary~\ref{corollary:main_error_estimate} an upper moment bound for the difference between the solution of a stochastic version of 
the integral equation in~\eqref{eq:Exact} above (see~\eqref{eq:Exact2} below) and its numerical approximation (cf.\ \eqref{eq:Numeric2} below and~\eqref{eq:Numeric} above). To do so, we first recall in Lemma~\ref{Lemma:AliprantisBorder4.55}
(see, e.g., Aliprantis \& Border~\cite[Theorem~4.55]{AliprantisBorder2006}) 
an elementary fact on measurability properties of functions which we then employ together with Lemma~\ref{lemma:crucial_apply_Alekseev_Groebner} above to establish in Lemma~\ref{lemma:MeasurableDerivative} suitable regularity properties of the solution of the considered SODE (cf.\ \eqref{eq:introduce Xstx} below and~\eqref{eq:Exact} above). Combining Lemma~\ref{lemma:MeasurableDerivative} with the error estimate in Lemma~\ref{lemma:main_error_estimate} 
above enables us to derive Corollary~\ref{corollary:main_error_estimate}.
\begin{lemma}
	\label{Lemma:AliprantisBorder4.55}
	Let $  ( \Omega, \F ) $ be a measurable space,
	let $ ( X, d_X ) $ be a compact metric space,
	let
	$ (Y, d_Y) $ be a separable metric space,
	let
	$ \mathcal{C}(X,Y) $ be  
	endowed with the topology of uniform convergence,
	let
	$ f \colon X \times \Omega \to Y $
	be a function,
	assume for all 
	$ x \in X $ that
	$ \Omega \ni \omega \mapsto f( x, \omega ) \in Y $
	is $ \F / \B(Y) $-measurable, 
	and
	assume for all 
	$ \omega \in \Omega $
	that
	$ ( X \ni x \mapsto f( x, \omega) \in Y )
	\in \mathcal{C}( X, Y ) $.
	Then we have that
	$ \Omega \ni \omega \mapsto 
	( X \ni x \mapsto f( x, \omega ) \in Y )
	\in \mathcal{C}( X, Y ) $ 
	is
	$ \F / \B( \mathcal{C}(X, Y) ) $-measurable.
\end{lemma}
\begin{lemma}
	\label{lemma:MeasurableDerivative} 
	Assume Setting~\ref{setting:main}, 
	assume that
	$ \dim(H) < \infty $,
	let $ ( \Omega, \F, \P ) $
	be a probability space,
	let 
	$ T \in (0,\infty) $,  
	$ F \in \mathcal{C}^1( H, H) $,
	$ Y, Z \in \M( \B( [0,T]) \otimes \F, \B( H) ) $,
	let
	$ O \colon [0,T] \times \Omega \to H $
	be a stochastic process
	w.c.s.p.,
	and
	for every 
	$ s \in [0,T] $,
	$ x \in H $
	let
	$ X_{s,(\cdot)}^x
	=
	(X_{s,t}^x)_{ t \in [s,T] }
	\colon 
	[s,T] \times \Omega
	\to H $
	be a stochastic process
	w.c.s.p.\
	which satisfies 
	for all 
	$ t \in [s,T] $
	that  
	\begin{equation}
	\label{eq:introduce Xstx}
	X_{s,t}^x 
	= 
	e^{ ( t - s ) A } x 
	+
	\int_s^t e^{ ( t - u ) A } F( X_{s,u}^x  ) \, du + O_t - e^{(t-s)A} O_s
	.
	\end{equation}
	Then 
	\begin{enumerate}[(i)]
		\item \label{item:ContDiff} we have
		for all $ \omega \in \Omega $ 
		that
		$ ( \{ (u,v) \in [0,T]^2 \colon u \leq v \} \times H \ni (s,t,x) \mapsto X_{s,t}^x(\omega) \in H ) \in \mathcal{C}^{0,0,1}(\{ (u,v) \in [0,T]^2 \colon u \leq v \} \times H,H) $,
		\item \label{item:productMeasurability}
		we have that
		$ \big( \{ (u,v) \in [0,T]^2 \colon u \leq v \} \times \Omega
		\ni (s,t,\omega) 
		\mapsto  X_{s,t}^{ Y_s(\omega)}(\omega) \in H \big)
		\in \M( \mathcal{B}( \{ (u,v) \in [0,T]^2 \colon u \leq v \} ) \otimes \F, \mathcal{B}(H) ) $, 
		and
		\item \label{item:Mesurable}
		\sloppy 
		we have  
		that
		$ \big( \{ (u,v) \in [0,T]^2 \colon u \leq v \} \times \Omega \ni (s,t,\omega) \mapsto \frac{ \partial }{ \partial x } 
		X_{s,t}^{ Z_s(\omega) }( \omega ) \in L(H) \big) \in \M( \B( \{ (u,v) \in [0,T]^2 \colon u \leq v \} ) \otimes \F, \B(L(H))) $.
\end{enumerate}
\end{lemma}
\begin{proof}[Proof of Lemma~\ref{lemma:MeasurableDerivative}]
	Throughout this proof let 
	$ \angle_T = \{ (u,v) \in [0,T]^2 \colon u \leq v \} $,
	 let
	$ V = \mathcal{C}( \{ w \in H \colon \| w \|_H \leq 1 \} ,H) $,
	let
	$ \left \| \cdot \right \|_V \colon V \to [0, \infty) $
	satisfy for all 
	$ f \in V $ that
	\begin{equation} 
	\| f \|_V = \sup\nolimits_{ h \in \{ w \in H \colon \| w \|_H \leq 1 \} }
	\| f(h) \|_H 
	,
	\end{equation}
	and
	let
	$ \iota \colon L(H) \to V $
	satisfy for all 
	$ Q \in L(H) $ that
	\begin{equation}
	\iota( Q )
	=
	( \{ w \in H \colon \| w \|_H \leq 1 \} \ni h \mapsto 
	Q (h) \in H )
	.
	\end{equation}
	Note that item~\eqref{item:differentiability}
	of
	Lemma~\ref{lemma:crucial_apply_Alekseev_Groebner}
	(applies with
	$ T = T $, 
		$ O_t = O_t(\omega) $,
		$ F = F $,
	$ X_{s,t}^x = X_{s,t}^x(\omega) $
	for
	$ (s,t) \in \angle_T $,
	$ x \in H $,
	$ \omega \in \Omega $
	in the setting of
	item~\eqref{item:differentiability}
	of
	Lemma~\ref{lemma:crucial_apply_Alekseev_Groebner})
	justifies 
	item~\eqref{item:ContDiff}.
	This ensures that
	for all
	$  \omega \in \Omega $
	we have that
	\begin{equation} 
	( \angle_T \times H \ni (s,t,x) \mapsto X_{s,t}^x(\omega) \in H )
	\in \mathcal{C}( \angle_T \times H, H ) 
	.
	\end{equation}
	\sloppy 
	The fact that
	for all
	$ (s,t) \in \angle_T $, $ x \in H $
	we have that
	$ ( \Omega \ni \omega \mapsto X_{s,t}^x(\omega) \in H ) \in \M( \F, \mathcal{B}(H)) $ 
	and, e.g., \cite[Lemma~2.1]{JentzenLindnerPusnik2017c}
	(applies with
	$ \Omega = \Omega $,
	$ \F = \F $,
	$ X = \angle_T \times H $,
	$ d_X = ( [\angle_T \times H]^2 \ni ( (s_1, t_1, x_1), (s_2, t_2, x_2) )\mapsto [ |s_1 - s_2|^2 + |t_1 - t_2|^2 + \| x_1 - x_2 \|_H^2 ]^{\nicefrac{1}{2}} \in [ 0, \infty) ) $,
	$ Y = H $,
	$ d_Y = ( H^2 \ni ( x_1, x_2 )\mapsto \| x_1 - x_2 \|_H \in [ 0, \infty) ) $,
	$ f = ( \angle_T \times H \times \Omega \ni (s,t,x,\omega) 
	\mapsto X_{s,t}^x(\omega) \in H ) $
	in the setting of~\cite[Lemma~2.1]{JentzenLindnerPusnik2017c})
	hence
	yield that
	\begin{equation} 
	\label{eq:Joint_measurability}
	( \angle_T \times H \times \Omega \ni (s, t, x, \omega) \mapsto X_{s,t}^x(\omega) \in H ) 
	\in \M( 
	\mathcal{B}( \angle_T ) 
	\otimes 
	\mathcal{B}(H) 
	\otimes 
	\F,
	\mathcal{B}(H) ) 
	.
	\end{equation}
	The fact that
	$ (
	\angle_T \times \Omega 
	\ni (s,t,\omega) \mapsto
	(s, t, Y_s(\omega), \omega)
	\in \angle_T \times H \times \Omega
	)
	\in 
	\M
	(
	\B( \angle_T ) \otimes \F,
	\B( \angle_T ) \otimes 
	\B(H) 
	\otimes 
	\F
	) $
	therefore justifies item~\eqref{item:productMeasurability}.
	Furthermore, 
	observe that item~\eqref{item:ContDiff} 
	gives that
	for all 
	$ (s, t) \in \angle_T $,
	$ x \in H $,
	$ \omega \in \Omega $  
	we have that 
	\begin{equation}
	\begin{split}
	\label{eq:impliesMeasurable}
	&
	\limsup_{ r \searrow 0 }
	\Big\| 
	\Big( \{ w \in H \colon \| w \|_H \leq 1 \} \ni h \mapsto
	\tfrac{
		X_{s,t}^{x + r h}(\omega) - X_{s,t}^x(\omega) 
	}{ r }
\in H \Big)
	- 
	\iota 
	\big(
	\tfrac{ \partial }{ \partial x } X_{s,t}^x(\omega) 
	\big)
	\Big\|_V
	\\
	&
	=
	\limsup_{ r \searrow 0 }
	\bigg[ 
	\sup_{ h \in H, \| h \|_H \leq 1 }
	\Big\| 
	\tfrac{
	X_{s,t}^{x + r h}(\omega) - X_{s,t}^x (\omega)
}{ r }
	- 
	\big(
	\tfrac{ \partial }{ \partial x } X_{s,t}^x (\omega) 
	\big)
	h 
	\Big\|_H
	\bigg] 
	=
	0.
	\end{split}
	\end{equation}
	Moreover, note that
	Lemma~\ref{Lemma:AliprantisBorder4.55}
	(applies with
	$ \Omega = \Omega $,
	$ \F = \F $,
	$ X = \{ w \in H \colon \| w \|_H \leq 1 \} $,
	$ d_X = ( \{ w \in H \colon \| w \|_H \leq 1 \} \times  
	\{ w \in H \colon \| w \|_H \leq 1 \} \ni (x,y)
	\mapsto \| x - y \|_H \in [0, \infty) ) $,
	$ Y = H $,
	$ d_Y = ( H \times H \ni (x,y) \mapsto \| x-y\|_H \in [0, \infty) ) $,
	$ f = ( 
	 \{ w \in H \colon \| w \|_H \leq 1 \} \times \Omega  
	\ni (h, \omega) \mapsto X_{s,t}^{ x + r h }(\omega)  \in H ) $
	for
	$ (s,t) \in \angle_T $, $ x \in H $,
	$ r \in (0, \infty) $
	in the setting of Lemma~\ref{Lemma:AliprantisBorder4.55})
	assures that
	for all 
	$ (s,t) \in \angle_T $, 
	$ x \in H $,
	$ r \in (0, \infty) $
	we have that
	\begin{equation}
	\begin{split}
	\big(
	\Omega \ni \omega \mapsto
	( 
	\{ w \in H \colon \| w \|_H \leq 1 \}
	\ni h \mapsto X_{s,t}^{x + r h }( \omega )  \in H
	)
	\in V
	\big)
	\in 
	\M( \F, \B(V) )
	.
	\end{split}
	\end{equation}
	This and~\eqref{eq:impliesMeasurable}
	verify that
	for all 
	$ (s,t) \in \angle_T $,
	$ x \in H $
	we have that
	\begin{equation}
	\begin{split}
	\label{eq:Needed first}
	\big(
	\Omega \ni \omega \mapsto \iota 
	\big( 
	\tfrac{ \partial }{ \partial x } 
	X_{s,t}^{x }( \omega )  
	\big)  
	\in V
	\big)
	\in 
	\M( \F, \B(V) )
	.
	\end{split}
	\end{equation}
	Hence, we obtain that for all
	$ Q \in L(H) $,
	$ \varepsilon \in (0, \infty) $,
	$ (s, t) \in \angle_T $,
	$ x \in H $
	we have that
	\begin{equation}
	\begin{split}
	\label{eq:preimage 1}
	&
	\big \{
	\omega \in \Omega \colon  
	\big\|
	\iota 
	\big( 
	\tfrac{ \partial }{ \partial x } 
	X_{s,t}^x (\omega) 
	\big) 
	- 
	\iota ( Q ) 
	\big \|_V < \varepsilon
	\big \} 
	\in \F
	.
	\end{split}
	\end{equation}
	In addition, observe 
	that for all 
	$ Q_1, Q_2 \in L(H) $
	we have that
	\begin{equation}
	\begin{split} 
	\label{eq:stronger assertion}
	&
	\| Q_1 - Q_2 \|_{L(H)}
	=
	\sup\nolimits_{ h \in \{ w \in H \colon \| w \|_H \leq 1 \} }
	\| Q_1(h) - Q_2(h) \|_H
	\\&=
	\sup\nolimits_{ h \in \{ w \in H \colon \| w \|_H \leq 1 \} }
	\|  \iota(Q_1) (h) - \iota(Q_2) (h) \|_H
	=
	\| \iota(Q_1) - \iota(Q_2) \|_V
	.
	\end{split} 
	\end{equation}
	Combining this and~\eqref{eq:preimage 1}
	ensures that
	for all
	$ Q \in L(H) $,
	$ \varepsilon \in (0, \infty) $,
	$ (s, t) \in \angle_T $,
	$ x \in H $
	we have that
	\begin{equation}
	\begin{split}
	\label{eq:preimage 2}
	&
	\big \{
	\omega \in \Omega \colon 
	\big\| 
	\tfrac{ \partial }{ \partial x } 
	X_{s,t}^x (\omega)  
	- 
	Q
	\big \|_{ L(H) } < \varepsilon
	\big \} 
	=
	\big \{
	\omega \in \Omega \colon 
	\big\|
	\iota 
	\big( 
	\tfrac{ \partial }{ \partial x } 
	X_{s,t}^x (\omega) 
	\big) 
	- 
	\iota ( Q )
	\big \|_V < \varepsilon
	\big \} 
	\in \F
	.
	\end{split}
	\end{equation}
The fact that
$ L(H) $ is a separable metric space
and the fact that
the Borel-sigma algebra on a separable metric space is generated
by the set of open balls 
therefore
verify
that
	for all 
	$ (s,t) \in \angle_T $,
	$ x \in H $ 
	we have that
	\begin{equation}
	\begin{split}
	\label{eq:measurability 1}
	\big(
	\Omega \ni \omega \mapsto  
	\tfrac{ \partial }{ \partial x } 
	X_{s,t}^{x }( \omega )   
	\in L(H)
	\big)
	\in 
	\M( \F, \B( L( H ) ) )
	.
	\end{split}
	\end{equation}
	Moreover, note that
	item~\eqref{item:ContDiff}
	ensures that
	for all  
	$ \omega \in \Omega $
	we have that
	\begin{equation}
	\big(
	\angle_T \times H 
	\ni 
	(s,t,x) \mapsto  
	\tfrac{ \partial }{ \partial x } 
	X_{s,t}^{x }( \omega )   
	\in 
	L(H)
	\big)
	\in
	\mathcal{C}(\angle_T \times H, L(H))
	.
	\end{equation}
Combining this and~\eqref{eq:measurability 1} 
	with, e.g., \cite[Lemma~2.1]{JentzenLindnerPusnik2017c}
	(applies with
	$ \Omega = \Omega $,
	$ \F = \F $,
	$ X = \angle_T \times H $,
	$ d_X = ( [\angle_T \times H]^2 \ni ( (s_1, t_1, x_1), (s_2, t_2, x_2) )\mapsto [ |s_1 - s_2|^2 + |t_1 - t_2|^2 + \| x_1 - x_2 \|_H^2 ]^{\nicefrac{1}{2}} \in [ 0, \infty) ) $,
	$ Y = L( H ) $,
	$ d_Y = ( [L(H)]^2 \ni ( A_1, A_2 )\mapsto \| A_1 - A_2 \|_{L(H)} \in [ 0, \infty) ) $,
	$ f = ( \angle_T \times H \times \Omega \ni (s,t,x,\omega) 
	\mapsto \frac{ \partial }{ \partial x } X_{s,t}^x(\omega) \in L(H) ) $
	in the setting of~\cite[Lemma~2.1]{JentzenLindnerPusnik2017c})
	verifies that
	\begin{equation} 
	\big( 
	\angle_T \times H \times \Omega  
	\ni 
	(s, t, x, \omega) 
	\mapsto 
	\tfrac{ \partial }{ \partial x } X_{s,t}^x(\omega)
	\in 
	L(H) 
	\big) 
	\in 
	\M
	(
	\B( \angle_T ) \otimes \B( H ) \otimes \F,
	\B( L(H) ) 
	)
	.
	\end{equation} 
	The fact that 
	$ (
	\angle_T \times \Omega 
	\ni (s,t,\omega) \mapsto
	(s, t, Z_s(\omega), \omega)
	\in \angle_T \times H \times \Omega
	)
	\in 
	\M
	(
	\B( \angle_T ) \otimes \F,
	\B( \angle_T ) \otimes 
	\B(H) 
	\otimes 
	\F 
	) $
	hence justifies item~\eqref{item:Mesurable}.
	The proof of Lemma~\ref{lemma:MeasurableDerivative}
	is hereby completed.
\end{proof}
\begin{corollary}
	\label{corollary:main_error_estimate}
	Assume Setting~\ref{setting:main}, 
	assume that $ \dim(H) < \infty $,
	let $ ( \Omega, \F, \P ) $
	be a probability space,
	let 
	$ T \in (0,\infty) $, 
	$ \theta \in \varpi_T $, 
	$ C, c, p \in [1,\infty) $, 
	$ \gamma \in [0,1) $, 
	$ \delta \in [0, \gamma] $,   
	$ \iota \in [0, 1 - \delta) $,
	$ \kappa \in \R $,  
		$ \xi \in \M( \F, \B(H) ) $,
	$ F \in \mathcal{C}^1( H, H) $,   
	$ \ff \in \M( \B(H), \B(H) ) $,  
	$ \O \in \M( \B([0,T]) \otimes \F, \B(H)) $,
	let
	$ O \colon [0,T] \times \Omega \to H $
	be a stochastic process
	w.c.s.p.,
	assume for all
	$ x, y \in H $
	that  
	$
	\| F( x ) - F( y )  \|_H 
	\leq 
	C  \| x - y \|_{H_\delta} ( 1 + \| x \|_{H_\kappa}^c + \| y \|_{H_\kappa}^c ) 
	$,
	for every 
	$ s \in [0,T] $,
	$ x \in H $
	let 
	$ X_{s,(\cdot)}^x
	=
	(X_{s,t}^x)_{ t \in [s,T] }
	\colon 
	[s,T] \times \Omega
	\to H $
	be a stochastic process
	w.c.s.p.\
	which satisfies for all
	$ t \in [s,T] $
	that
	\begin{equation}
	\label{eq:Exact2}
	X_{s,t}^x 
	= 
	e^{ ( t - s ) A } x 
	+
	\int_s^t e^{ ( t - u ) A }  F( X_{s,u}^x  ) \, du + O_t - e^{(t-s)A} O_s
	,
	\end{equation}
	and 
	let
	$ \y \colon [0,T] \times \Omega \to H $
	satisfy 
	for all
	$ t \in [0,T] $ that
	\begin{equation}
	\label{eq:Numeric2}
	\y_t
	=
	e^{t A}
	\xi
	+
	\int_0^t
	e^{ (t - \llcorner u \lrcorner_\theta ) A } 
	\ff ( \y_{ \llcorner u \lrcorner_\theta } ) \, du
	+
	\O_t.
	\end{equation} 
	Then
	\begin{enumerate}[(i)] 
		\item \label{item:0_omega} we have that $ \y \in \M( \B([0,T]) \otimes \F, \B(H)) $,
		\item\label{item:1_omega} we have
		for all $ \omega \in \Omega $ 
	that
	$ ( \{ (u,v) \in [0,T]^2 \colon u \leq v \} \times H \ni (s,t,x) \mapsto X_{s,t}^x(\omega) \in H ) \in \mathcal{C}^{0,0,1}( \{ (u,v) \in [0,T]^2 \colon u \leq v \} \times H,H) $,
	\item \label{item:technicalMeasurability}
	\sloppy
	we have for all $ \zeta \in \M( \F, \B(H) ) $ that
	$ \big( \{ (u,v) \in [0,T]^2 \colon u \leq v \} \times \Omega
	\ni (s,t,\omega) 
	\mapsto  X_{s,t}^{ \zeta(\omega) + O_s(\omega) }(\omega) \in H \big)
	\in \M( \mathcal{B}( \{ (u,v) \in [0,T]^2 \colon u \leq v \} ) \otimes \F, \mathcal{B}(H) ) $,
	\item \label{item:MeasDer}
\sloppy
	we have for all
	$ \zeta \in \M( \F, \B(H) ) $ that
	$ \big( \{ (u,v) \in [0,T]^2 \colon u \leq v \} \times  \Omega
	\ni (s,t,\omega) 
	\mapsto 
	\tfrac{ \partial }{ \partial x } 
	X_{s,t}^{ \y_s(\omega) - \O_s(\omega) + O_s(\omega)
	+
	e^{sA} ( \zeta(\omega) - \xi(\omega) )	
	}(\omega) \in L(H) \big )
	\in \M( \mathcal{B}( \{ (u,v) \in [0,T]^2 \colon u \leq v \} ) \otimes \F $, $ \mathcal{B}(L(H)) )
	$,
	and
	\item\label{item:2_omega} we have
	for all 
	$ \zeta \in \M( \F, \B(H) ) $,
	$ t \in [0,T] $ 
	that  
	\begin{equation}
	\begin{split} 
	\label{eq:app_error}
	&
	\|  \y_t - X_{0, t}^{ \zeta + O_0 } \|_{ \L^p( \P; H ) } 
	\leq
	\| \O_t - O_t \|_{ \L^p( \P; H ) } 
	+
	\| 
	\xi - \zeta  
	\|_{ \L^p( \P; H ) } 
	\\
	&
	+
	\tfrac{ C \max \{ T, 1 \} }{ 1 - \delta - \iota }
	\int_0^t 
	\big \| \tfrac{\partial}{\partial x}
	X_{s, t}^{ \y_s - \O_s + O_s
		+
		e^{sA}( \zeta - \xi ) }
	\big \|_{ \L^{ 2 p } ( \P; L(H) ) }
	\Big\{
	[ | \theta |_T ]^{\gamma - \delta}
	\| \ff ( \y_{\llcorner s \lrcorner_\theta } ) \|_{ \L^{2p} (\P; H_{\gamma-\delta} )}
	\\
	&
	+
	\| \ff ( \y_{\llcorner s \lrcorner_\theta } ) - F ( \y_{\llcorner s \lrcorner_\theta } ) \|_{ \L^{ 2 p } ( \P; H ) }
	+ 
	\Big(  
	[ | \theta |_T ]^{1-\delta}
	\|
	\ff ( \y_{\llcorner s \lrcorner_{\theta}})
	\|_{\L^{4p}(\P; H )}
	\\
	&
	+ 
	[ |\theta|_T ]^\iota 
	\sup\nolimits_{ u \in [0,T] }
	\|
	\ff ( \y_u )
	\|_{\L^{4p}(\P; H )}
	+
	\| \O_s - \O_{\llcorner s \lrcorner_\theta} \|_{\L^{4p}(\P; H_\delta)}
 	+
 	[ | \theta |_T ]^{ \gamma - \delta }
 	\| \xi \|_{\L^{4p}(\P; H_\gamma)}
	\\
	& 
	+ 
	\| \O_s - O_s \|_{\L^{4p}(\P; H_\delta)}
	+
	\| 
	\xi - \zeta 
	\|_{\L^{4p}(\P; H_\delta)}
	\Big)
	\big[
	1
	+
	\| \y_{\llcorner s \lrcorner_\theta} \|_{ \L^{4pc} ( \P; H_\kappa ) }
	+
	\|
	\y_s
    \|_{ \L^{4pc} ( \P; H_\kappa ) }
	\\
	&
	+
	\| \O_s - O_s \|_{ \L^{4pc} ( \P; H_\kappa ) }
	+
	\| 
	\xi - \zeta  
	\|_{ \L^{4pc} ( \P; H_\kappa ) }
	\big]^c 
	\Big\}
	\,
	ds
	.
	\end{split}
	\end{equation}
\end{enumerate}
\end{corollary}
\begin{proof}[Proof of Corollary~\ref{corollary:main_error_estimate}]
	Observe that 
	item~\eqref{item:minor detail 2} of Lemma~\ref{lemma:main_error_estimate}
	(applies with
	$ T = T $,
	$ \theta = \theta $,
	$ \xi = \xi(\omega) $,
	$ O_s = O_s(\omega) $,
	$ \O_s = \O_s(\omega) $,
	$ F = F $,
	$ \ff = \ff $,
	$ X_{s,t}^x = X_{s,t}^x(\omega) $,
	$ \y_s = \y_s(\omega) $
	for
	$ \omega \in \Omega $,
	$ t \in [s, T] $,
	$ s \in [0, T] $,
	$ x \in H $
	in the setting of item~\eqref{item:minor detail 2} of Lemma~\ref{lemma:main_error_estimate})
	verifies that for all $ \omega \in \Omega $ we have that
	\begin{equation} 
	\label{eq:continuous}
	( [0,T] \ni t \mapsto \y_t(\omega) - \O_t(\omega) \in H )
	\in \mathcal{C}( [0,T], H ) 
	.
	\end{equation} 
	Moreover, note that~\eqref{eq:Numeric2},
	the fact that for all $ t \in [0,T] $ we have that
	$ ( \Omega \ni \omega \mapsto \O_t(\omega) \in H ) \in \M( \F, \B(H) ) $,
	 and the assumption that
	$ \xi \in \M(\F, \B(H)) $ 
	ensure that
	for all $ t \in [0,T] $ we have that
	\begin{equation}
	\label{eq:measurable}
( \Omega \ni \omega \mapsto \y_t(\omega) - \O_t(\omega) \in H )
	\in \M( \F, \B(H) )
	. 
	\end{equation}
	Combining this and~\eqref{eq:continuous}
	with, e.g., \cite[Lemma~2.1]{JentzenLindnerPusnik2017c} 
	(applies with
	$ \Omega = \Omega $,
	$ \F = \F $,
	$ X = [0,T] $,
	$ d_X = ( [0,T]^2 \ni (s,t) \mapsto |t-s| \in [0, \infty) ) $,
	$ Y = H $,
	$ d_Y = ( H \times H \ni (x,y) \mapsto \| x-y\|_H \in [0, \infty) ) $,
	$ f = \y - \O $
	in the setting of~\cite[Lemma~2.1]{JentzenLindnerPusnik2017c})
	ensures that
	\begin{equation}
( [0,T] \times \Omega \ni (t, \omega) \mapsto \y_t(\omega) - \O_t(\omega) \in H )
	\in \M( \B([0,T]) \otimes \F, \B(H) )
	.
	\end{equation}
	The assumption that $ \O \in \M( \B([0,T]) \otimes \F, \B(H) ) $
	therefore justifies item~\eqref{item:0_omega}.
	Next note that
	Lemma~\ref{lemma:MeasurableDerivative}
	(applies with
	$ ( \Omega, \F, \P ) = ( \Omega, \F, \P ) $,
	$ T = T $,
	$ F = F $, 
	$ Y_s = \zeta + O_s $,
	$ Z_s = \y_s - \O_s + O_s 
	+ e^{sA}( \zeta - \xi ) $,
	$ O_s = O_s $,
	$ X_{s,t}^x = X_{s,t}^x $
	for
	$ x \in H $,
	$ t \in [s,T] $,
	$ s \in [0,T] $,
	$ \zeta \in 
	\M( \F, \B(H) ) $ 
	in the setting of
	Lemma~\ref{lemma:MeasurableDerivative})
	justifies
	items~\eqref{item:1_omega}--\eqref{item:MeasDer}.
	In the next step we observe that
	for all
	$ s \in [0,T] $,
	$ t \in [s,T] $,
	$ x \in H $,
	$ \zeta \in 
	\M( \F, \B(H) ) $ 
	we have that
	\begin{equation}
	\label{eq:M1}
	X_{s,t}^{ x }
	= 
	e^{ ( t - s ) A } 
	x
	+
	\int_s^t e^{ ( t - u ) A } F( X_{s,u}^x   ) \, du 
	+ 
	( O_t + e^{tA} \zeta )
	- 
	e^{(t-s)A} 
	( O_s + e^{sA} \zeta ) 
	\end{equation}
	and
	\begin{equation}
	\label{eq:M2} 
	\y_t
	=
	\int_0^t
	e^{ ( t - \llcorner u \lrcorner_\theta ) A } 
	\ff ( \y_{ \llcorner u \lrcorner_\theta } ) \, du
	+
	\big[ 
	\O_t
	+
	e^{t A}
	\xi
	\big] 
	.
	\end{equation}
	Lemma~\ref{lemma:main_error_estimate} 
	(applies with
	$ T = T $, 
	$ \theta = \theta $,
	$ \xi = 0 $,
	$ O_s = O_s(\omega) + e^{sA} \zeta(\omega) $, 
	$ \O_s = \O_s(\omega) + e^{sA} \xi ( \omega ) $,
	$ F = F $, 
	$ \ff = \ff $, 
	$ X^x_{s,t} = X^x_{s,t}(\omega) $,
	$ \y_s = \y_s(\omega) $,
	$ C = C $, 
	$ c = c $,
	$ \gamma = \gamma $,  
	$ \delta = \delta $,
	$ \iota = \iota $,
	$ \kappa = \kappa $
	for 
	$ \omega \in \Omega $,
	$ x \in H $,
	$ t \in [s,T] $,
	$ s \in [0,T] $, 
	$ \zeta \in 
	\M( \F, \B(H) ) $ 
	in the setting of Lemma~\ref{lemma:main_error_estimate}) 
	therefore gives that
	for all 
	$ \zeta \in \M( \F, \B(H) ) $,
	$ t \in [0,T] $
	we have that 
	\begin{equation}
	\begin{split}
	&
	\|  \y_t - X_{0, t}^{ \zeta + O_0 } \|_{ \L^p( \P; H ) } 
	\\
	&
	\leq
	\| \O_t - O_t
	+ 
	e^{tA}( \xi - \zeta )
	 \|_{ \L^p( \P; H ) } 
	+
	\int_0^t
	\Big\|
	\big \| \tfrac{\partial}{\partial x}
	X_{s, t}^{ \y_s - \O_s + O_s
+ e^{ s A }( \zeta - \xi )	
 }
	\big \|_{L(H)}
	\\
	&
	\quad 
	\cdot
	\Big( 
	[ | \theta |_T ]^{\gamma - \delta}
	\| \ff ( \y_{\llcorner s \lrcorner_\theta } ) \|_{H_{\gamma-\delta}}
	+
	\| \ff ( \y_{\llcorner s \lrcorner_\theta } ) - F ( \y_{\llcorner s \lrcorner_\theta } ) \|_H
	\\
	&
	\quad 
	+
	C
	\Big(   
	[| \theta |_T]^{1-\delta}
	\|
	\ff ( \y_{ \llcorner s \lrcorner_\theta } )
	\|_H
	+ 
	[ |\theta|_T ]^\iota
	\smallint_0^{ \llcorner s \lrcorner_\theta }
	( \llcorner s \lrcorner_\theta - \llcorner u \lrcorner_\theta )^{-\delta-\iota} 
	\|
	\ff ( \y_{ \llcorner u \lrcorner_{ \theta } } )
	\|_H
	\, du
	\\
	&
	\quad 
	+
	\| \O_s 
	- 
	\O_{\llcorner s \lrcorner_\theta} 
	+ 
	(
	e^{sA}  
	-
	e^{ \llcorner s \lrcorner_\theta A }
	)
	\xi
	\|_{H_\delta}   
	+ 
	\| \O_s - O_s
	+
	e^{s A}( \xi - \zeta )
	 \|_{H_\delta}
	\Big) 
	\\
	&
	\quad 
	\cdot
	\big[
	1
	+
	\| \y_{\llcorner s \lrcorner_\theta} \|_{H_\kappa} 
	+
	\|
	\y_s
	\|_{H_\kappa}
	+
	\| \O_s - O_s 
	+
	e^{s A}( \xi - \zeta )
	\|_{H_\kappa}
	\big]^c 
	\Big)
	\Big \|_{ \L^p( \P; \R ) } 
	\,
	ds
	.
	\end{split}
	\end{equation}
	H\"older's inequality
	and the triangle inequality hence yield
	that  
	for all 
	$ \zeta \in \M( \F, \B(H) ) $,
	$ t \in [0,T] $
	we have that 
	\begin{equation}
	\begin{split}
	&
	\|  \y_t - X_{0, t}^{ \zeta + O_0  } \|_{ \L^p( \P; H ) } 
	\\
	&
	\leq
	\| 
	\O_t - O_t
	\|_{ \L^p( \P; H ) } 
	+ 
	\|
	e^{tA}
	(
	\xi - \zeta 
	)
	\|_{ \L^p( \P; H ) } 
	+
	\int_0^t 
	\big \| \tfrac{\partial}{\partial x}
	X_{s, t}^{ \y_s - \O_s + O_s
	+e^{sA}( \zeta - \xi ) }
	\big \|_{ \L^{ 2 p } ( \P; L(H) ) }
	\\
	&
	\quad 
	\cdot 
	\Big \| 
	[ | \theta |_T ]^{\gamma - \delta}
	\| \ff ( \y_{\llcorner s \lrcorner_\theta } ) 
	\|_{H_{\gamma - \delta}}
	+
	\| \ff ( \y_{\llcorner s \lrcorner_\theta } ) - F ( \y_{\llcorner s \lrcorner_\theta } ) \|_H
	\\
	& 
	\quad 
	+
	C
	\Big(  
	[ | \theta |_T ]^{1-\delta}
	\|
	\ff ( \y_{\llcorner s \lrcorner_{\theta}})
	\|_H 
	+ 
	[ |\theta|_T ]^\iota
	\smallint_0^{ \llcorner s \lrcorner_\theta }
	( \llcorner s \lrcorner_\theta - \llcorner u \lrcorner_\theta )^{-\delta-\iota} 
	\|
	\ff ( \y_{ \llcorner u \lrcorner_{ \theta } } )
	\|_H
	\, du
	\\
	&  
	\quad 
	+
	\| \O_s - \O_{\llcorner s \lrcorner_\theta} \|_{H_\delta}
	+
	\| 
	(
	e^{sA}  
	-
	e^{ \llcorner s \lrcorner_\theta A }
	)
	\xi
	\|_{H_\delta}   
	+ 
	\| \O_s - O_s \|_{H_\delta}
	+
	\| 
	e^{s A}( \xi - \zeta )
	\|_{H_\delta}
	\Big)
	\\
	&
	\quad 
	\cdot
	\big[
	1
	+
	\| \y_{\llcorner s \lrcorner_\theta} \|_{H_\kappa} 
	+
	\|
	\y_s
	\|_{H_\kappa}
	+
	\| \O_s - O_s \|_{H_\kappa}
	+
	\| 
	e^{s A}( \xi - \zeta )
	\|_{H_\kappa}
	\big]^c
	\Big \|_{ \L^{2p} ( \P; \R ) } 
	\,
	ds
	.
	\end{split}
	\end{equation}
	H\"older's inequality
	and the triangle inequality 
	therefore verify that
	for all
	$ \zeta \in \M( \F, \B(H) ) $,
	$ t \in [0,T] $ 
	we have that 
	\begin{equation}
	\begin{split}
	\label{eq:some_estimate}
	&
	\|  \y_t - X_{0, t}^{ \zeta + O_0 } \|_{ \L^p( \P; H ) } 
	\\
	&
	\leq
	\| \O_t - O_t \|_{ \L^p( \P; H ) } 
	+
	\| 
	\xi - \zeta  
	\|_{ \L^p( \P; H ) } 
	+
	\int_0^t 
	\big \| \tfrac{\partial}{\partial x}
	X_{s, t}^{ \y_s - \O_s + O_s
	+
	e^{sA}( \zeta - \xi ) }
	\big \|_{ \L^{ 2 p } ( \P; L(H) ) }
	\\
	&
	\quad 
	\cdot 
	\Big\{ 
	[ | \theta |_T ]^{\gamma - \delta}
	\| \ff ( \y_{\llcorner s \lrcorner_\theta } ) \|_{ \L^{2p} (\P; H_{\gamma-\delta} )}
	+
	\| \ff ( \y_{\llcorner s \lrcorner_\theta } ) -  F ( \y_{\llcorner s \lrcorner_\theta } ) \|_{ \L^{ 2 p } ( \P; H ) }
	\\
	&
	\quad    
	+
	C
	\Big\|  
	[ | \theta |_T ]^{1-\delta}
	\|
	\ff ( \y_{\llcorner s \lrcorner_{\theta}})
	\|_H 
	+ 
	[ |\theta|_T ]^\iota
	\smallint_0^{ \llcorner s \lrcorner_\theta }
	( \llcorner s \lrcorner_\theta - \llcorner u \lrcorner_\theta )^{-\delta-\iota} 
	\|
	\ff ( \y_{ \llcorner u \lrcorner_{ \theta } } )
	\|_H
	\, du
	\\
	&
	\quad 
	+
	\| \O_s - \O_{\llcorner s \lrcorner_\theta} \|_{H_\delta} 
	+
	\| 
	(
	e^{sA}  
	-
	e^{ \llcorner s \lrcorner_\theta A }
	)
	\xi 
	\|_{H_\delta}  
	+ 
	\| \O_s - O_s \|_{H_\delta}
	+
	\| 
	\xi - \zeta 
	\|_{H_\delta}
	\Big\|_{\L^{4p}(\P; \R)}
	\\
	&
	\quad 
	\cdot 
	\big\| 
	1
	+
	\| \y_{\llcorner s \lrcorner_\theta} \|_{H_\kappa} 
	+
	\|
	\y_s
	\|_{H_\kappa}
	+
	\| \O_s - O_s \|_{H_\kappa}
	+
	\| 
	\xi - \zeta 
	\|_{H_\kappa} 
	\big \|_{ \L^{4pc} ( \P; \R ) }^c 
	\Big\}
	\,
	ds
	.
	\end{split}
	\end{equation}
	In addition, note that 
	the fact that 
	$ \delta + \iota < 1 $
	assures that 
	for all 
	$ s \in [0,T] $ we have that
	\begin{equation}
	\begin{split}
	\label{eq:Something}
	& 
	\bigg\|
	\smallint_0^{ \llcorner s \lrcorner_\theta }
	( \llcorner s \lrcorner_\theta - \llcorner u \lrcorner_\theta )^{-\delta-\iota} 
	\|
	\ff ( \y_{ \llcorner u \lrcorner_{ \theta } } )
	\|_H
	\, du
	\bigg\|_{\L^{4p}(\P; \R) }
	\\
	&
	\leq 
	\smallint_0^{ \llcorner s \lrcorner_\theta }
	( \llcorner s \lrcorner_\theta - \llcorner u \lrcorner_\theta )^{-\delta-\iota} 
	\|
	\ff ( \y_{ \llcorner u \lrcorner_{ \theta } } )
	\|_{\L^{4p}(\P; H) }
	\, du 
	\\
	&
	\leq 
	\sup\nolimits_{ u \in [0,T] }
	\|
	\ff ( \y_u )
	\|_{\L^{4p}(\P; H) }
	\smallint_0^{ \llcorner s \lrcorner_\theta }
	( \llcorner s \lrcorner_\theta - u )^{-\delta-\iota} 
	\, du
	\\
	&
	= 
	\sup\nolimits_{ u \in [0,T] }
	\|
	\ff ( \y_u )
	\|_{\L^{4p}(\P; H) }
	\tfrac{ 
	( \llcorner s \lrcorner_\theta )^{1-\delta-\iota} 
}{ 1 - \delta - \iota }
	\\
&
\leq  
\sup\nolimits_{ u \in [0,T] }
\|
\ff ( \y_u )
\|_{\L^{4p}(\P; H) }
\tfrac{ 
	\max \{ T, 1 \} 
}{ 1 - \delta - \iota }
	.
	\end{split}
	\end{equation}
	Furthermore,
	observe that for all 
	$ s \in [0,T] $ we have that
	\begin{equation}
	\begin{split}
	&
	\| ( e^{sA} - e^{ \llcorner s \lrcorner_{ \theta } A } )
	\xi
	\|_{ H_\delta }
	=
	\|
	e^{ \llcorner s \lrcorner_\theta A }
	( e^{ ( s - \llcorner s \lrcorner_\theta ) A } 
	-
	\operatorname{Id}_H )
	\xi
	\|_{ H_\delta }
	\leq 
	\|
	(-A)^{\delta}
	( e^{ ( s - \llcorner s \lrcorner_\theta ) A } 
	-
	\operatorname{Id}_H )
	\xi
	\|_{H}
	\\
	&
	\leq
	\|
	(-A)^{\delta - \gamma}
	( e^{ ( s - \llcorner s \lrcorner_\theta ) A } 
	-
	\operatorname{Id}_H )
	\|_{L(H)}
	\|  
	\xi
	\|_{ H_\gamma }
	\leq
	( s - \llcorner s \lrcorner_\theta )^{\gamma - \delta}
	\|  
	\xi
	\|_{ H_\gamma }
	\leq
	[ | \theta |_T ]^{ \gamma - \delta }
	\|  
	\xi
	\|_{ H_\gamma }
	.
	\end{split}
	\end{equation}
	Combining this with~\eqref{eq:some_estimate}
	and~\eqref{eq:Something}
	justifies item~\eqref{item:2_omega}.
	The proof of Corollary~\ref{corollary:main_error_estimate}
	is hereby completed.
\end{proof}
\section[Moment bounds for the derivative process]{Moment bounds for the derivative process and resulting time discretization error estimates}
\label{section:Moment bounds for the derivative process}
\subsection{A priori bounds for the derivative process}
\label{section:AprioriBounds}
In this subsection we derive
in Lemma~\ref{lemma:inheritet}
an appropriate moment bound 
for the pathwise derivatives of the solution processes
$ ( X_{s,t}^x )_{ t \in [s,T] } $,
$ s \in [0,T] $,
$ x \in H $,
with respect to their initial conditions 
appearing in
item~\eqref{item:2_omega} of Corollary~\ref{corollary:main_error_estimate} above
(see~\eqref{eq:UpperBound} in Lemma~\ref{lemma:inheritet} below). 
We first demonstrate
in
Lemma~\ref{lemma:equivalent_monotonicity}
that the well known local monotonicity property
(see~\eqref{eq:LocMon} in Lemma~\ref{lemma:equivalent_monotonicity} below
and cf., e.g., Liu \& R\"ockner~\cite[(H$2'$) in Chapter~5]{LiuRoeckner2015Book})
together with
the continuous
Fr\'echet differentiability 
of the nonlinearity $ F $  
implies the 
property of $ F' $ 
that we are exploiting in this article
(see~\eqref{eq:Condition} in Lemma~\ref{lemma:equivalent_monotonicity} below).
In addition,
Proposition~\ref{proposition:Derivative_process_estimate}
(cf.\ Hairer \& Mattingly~\cite[(4.8) in Section~4.4]{HairerMattingly2006}) 
provides a suitable upper bound for the
derivative process
appearing in 
item~\eqref{item:2_omega} of Corollary~\ref{corollary:main_error_estimate}
(see~\eqref{eq:Bound} 
in Proposition~\ref{proposition:Derivative_process_estimate}  below).
Combining
Lemma~\ref{lemma:MeasurableDerivative}
and
Proposition~\ref{proposition:Derivative_process_estimate}
 implies 
Corollary~\ref{corollary:exp_bound} 
which 
we use together 
with   
Cox et al.\ \cite[Corollary 2.4]{CoxHutzenthalerJentzen2013}
as a tool 
to establish
in Lemma~\ref{lemma:inheritet}
the desired moment bound.	
\begin{lemma}
	\label{lemma:equivalent_monotonicity}
	Assume Setting~\ref{setting:main},
	let 
	$ \varepsilon, {\bf C}, \gamma \in [0,\infty) $, 
	$ F \in \mathcal{C}^1(H_\gamma, H) $,
	and
	assume for all 
	$ x, y \in  H_{ \max \{ \gamma, \nicefrac{1}{2} \} } $  that
	\begin{equation}
	\label{eq:LocMon}
	\langle F(x) - F(y), x - y \rangle_H
	\leq  ( \varepsilon \| x \|_{H_{\nicefrac{1}{2}}}^2 
	+ 
	{\bf C}
	  ) 
	\| x - y \|_H^2 
	+ \| x - y \|_{H_{\nicefrac{1}{2}}}^2.
	\end{equation}
	Then we have for all 
		$ x, y \in H_{ \max \{ \gamma, \nicefrac{1}{2} \} } $
		that
		\begin{equation} 
		\label{eq:Condition}
		\langle F'(x) y, y \rangle_H \leq 
		 ( \varepsilon \| x \|_{H_{\nicefrac{1}{2}}}^2 
		+ 
		{\bf C}
		  ) \| y \|_H^2 
		+ \|y \|_{H_{\nicefrac{1}{2}}}^2 
		.
		\end{equation}
\end{lemma}
\begin{proof}[Proof of Lemma~\ref{lemma:equivalent_monotonicity}]
	Observe that for all $ x \in H_{ \max \{ \gamma, \nicefrac{1}{2} \} } $, 
	$ y \in ( H_{ \max \{ \gamma, \nicefrac{1}{2} \} } 
	\backslash \{ 0 \} ) $
	we have that
	\begin{equation}
	\begin{split}
	\label{eq:first_implication}
	&
	\langle F'(x) y, y\rangle_H 
	=
	\big\langle 
	\lim\nolimits_{r \searrow 0} \tfrac{ F(x+ry) - F(x) }{r},
	y 
	\big\rangle_H
	=
	\lim_{r \searrow 0}
	\big\langle 
	\tfrac{ F(x+ry) - F(x) }{r},
	y 
	\big\rangle_H
	\\
	&
	=
	\lim_{ r \searrow 0 } 
	\big(
	\tfrac{1}{r^2}
	\langle F( x + r y ) - F( x ), r y \rangle_H
	\big)
	\\
	&
	\leq 
	\big( 
	( \varepsilon \| x \|_{H_{\nicefrac{1}{2} } }^2 +  {\bf C} )
	\| y \|_H^2
	+
	\| y \|_{H_{\nicefrac{1}{2}}}^2  
	\big)
	\limsup_{ r \searrow 0 }
	\Bigg[
	\frac{ 
		\frac{1}{r^2}
		\langle F( x + r y ) - F( x ), r y \rangle_H
	} { 
		( \varepsilon  \| x \|_{H_{\nicefrac{1}{2} } }^2 +  {\bf C} )
		\| y \|_H^2
		+
		\| y \|_{H_{\nicefrac{1}{2}}}^2   
	}
	\Bigg]
	\\
	&
	= 
		\big( 
	( \varepsilon \| x \|_{H_{\nicefrac{1}{2} } }^2 + {\bf C} )
	\| y \|_H^2
	+
	\| y \|_{H_{\nicefrac{1}{2}}}^2  
	\big)
	\limsup_{ r \searrow 0 }
	\Bigg[
	\frac{  
		\langle F( x + r y ) - F( x ), r y \rangle_H
	} { 
		( \varepsilon  \| x \|_{H_{\nicefrac{1}{2} } }^2 +  {\bf C} )
		\| r y \|_H^2
		+
		\| r y \|_{H_{\nicefrac{1}{2}}}^2   
	}
	\Bigg]
	\\
	&
	\leq
	\big( 
	( \varepsilon \| x \|_{H_{\nicefrac{1}{2} } }^2 +  {\bf C} )
	\| y \|_H^2
	+
	\| y \|_{H_{\nicefrac{1}{2}}}^2  
	\big) 
	\sup_{ r \in (0,1] }
	\Bigg[
	\frac{  
		\langle F( x + r y ) - F( x ), r y \rangle_H
	} { 
		( \varepsilon \| x \|_{H_{\nicefrac{1}{2} } }^2 +  {\bf C} )
		\| r y \|_H^2
		+
		\| r y \|_{H_{\nicefrac{1}{2}}}^2   
	}
	\Bigg]
	\\
	&
	\leq
	\big( 
	( \varepsilon  \| x \|_{H_{\nicefrac{1}{2} } }^2 + {\bf C} )
	\| y \|_H^2
	+
	\| y \|_{H_{\nicefrac{1}{2}}}^2  
	\big) 
	\sup_{ v \in H_{\max\{\gamma,\nicefrac{1}{2}\}} \backslash \{0\}  }
	\Bigg[
	\frac{  
		\langle F( x + v ) - F( x ), v \rangle_H
	} { 
		( \varepsilon \| x \|_{H_{\nicefrac{1}{2} } }^2 +  {\bf C} )
		\| v \|_H^2
		+
		\| v \|_{H_{\nicefrac{1}{2}}}^2   
	}
	\Bigg]
	.
	\end{split}
	\end{equation}
	Combining this and~\eqref{eq:LocMon}
	justifies~\eqref{eq:first_implication}.
	The proof of Lemma~\ref{lemma:equivalent_monotonicity}
	is hereby completed.
\end{proof}
\begin{proposition}
\label{proposition:Derivative_process_estimate}
Assume Setting~\ref{setting:main},
assume that
$ \dim(H) < \infty $,
let $ T \in (0,\infty) $,
$ \varepsilon, { \bf C } \in [0,\infty) $,
$ F \in \mathcal{C}^1  (H,
H) $,   
$ O \in \mathcal{C}( [0,T], H) $,
assume for all $ x, y \in H $ 
that
$ \langle F'(x) y, y \rangle_H \leq 
( \varepsilon \| x \|_{H_{\nicefrac{1}{2}}}^2 + { \bf C }  ) \| y \|_H^2 
+ \|y \|_{H_{\nicefrac{1}{2}}}^2 $,
and
for every 
$ s \in [0,T] $,
$ x \in H $
let
$ \X_{s,(\cdot)}^x
=
( \X_{s,t}^x)_{ t \in [s,T] }
\in \mathcal{C}([s,T], H) $ 
satisfy for all 
$ t \in [s,T] $
that
\begin{equation}
\label{eq:DefineX}
\X_{s, t}^x = x + \int_s^t
\big(
A \X_{s,u}^x  + F( \X_{s,u}^x + O_u ) 
\big)
\, 
du
.
\end{equation}
Then 
\begin{enumerate}[(i)]
\item \label{item:continuity3} 
we have that  
$ (
\{ (u,v) \in [0,T]^2 \colon u \leq v \} \times H
\ni
(s,t,x)
\mapsto
\X^x_{s,t}  
\in H 
)
\in \mathcal{C}^{0,0,1}
 ( \{ (u,v) \in [0,T]^2 \colon u \leq v \} \times H, H) $,
\item \label{item:DerivativeFormula}
 we have for all
 $ s \in [0,T] $,
 $ t \in [s,T] $,
  $ x, y \in H $
that
\begin{equation}
\begin{split}
\label{eq:derivative_under_integral}
\big(
\tfrac{\partial}{ \partial x }
\X_{s,t}^{x} 
\big) y
=
y
+
\int_s^t 
\big[
A 
\big(
\tfrac{\partial}{ \partial x }
\X_{s,u}^{x} 
\big) y + F'( \X_{s,u}^{x}  + O_u ) 
\big(
\tfrac{\partial}{ \partial x } \X_{s,u}^{x} 
\big) y
\big]
\,
du,
\end{split}
\end{equation}
and
\item \label{item:derivative_estimate}
we have for all
$ s \in [0,T] $,
$ t \in [s,T] $,
$ x \in H $
that
\begin{equation}
\begin{split}
\label{eq:Bound}
  \big \| 
\tfrac{\partial}{\partial x}  \X_{s, t}^{ x}   
  \big \|_{L(H)}
\leq
\exp\!
\bigg(  
\int_s^t
\big(
\varepsilon
\|   \X_{s,u}^{ x}  + O_u   \|_{H_{\nicefrac{1}{2}}}^2
+
{ \bf C }
\big)
\,
du
\bigg).
\end{split}
\end{equation}
\end{enumerate}
\end{proposition}
\begin{proof}[Proof of Proposition~\ref{proposition:Derivative_process_estimate}]
Note that the fact that $ ( [0,T] \times H \ni (u, h) \mapsto A h + F(h + O_u  ) \in H )
\in \mathcal{C}^{0, 1}(
[0,T] \times H, H
) $ 
and, e.g., \cite[items~(v) and~(vi) of Lemma~4.8]{JentzenLindnerPusnik2017a}
(applies with
$ V = H $,
$ T = T $, 
$ f = (
[0,T] \times H \ni (u, h) \mapsto
A h + F(h + O_u ) \in H 
) $,
$ X_{s, t}^x = \X_{s, t}^x $ 
for  
$ t \in [s,T] $,
$ s \in [0,T] $,
$ x \in H $
in the setting of~\cite[items~(v) and~(vi) of 
Lemma~4.8]{JentzenLindnerPusnik2017a})
justify items~\eqref{item:continuity3}
and~\eqref{item:DerivativeFormula}.  
Therefore,
we obtain
that
for all 
$ s \in [0,T] $,
$ t \in [s,T] $,
$ x, y \in H $
we have 
that
\begin{equation}
\begin{split}
\label{eq:some_impo_estimate}
&
\big \| \big(
\tfrac{\partial}{ \partial x } \X_{s,t}^{x}  
\big)
 y \big \|_H^2 
 - \|y \|_H^2
=
2\int_s^t
\big \langle
\big(
\tfrac{\partial}{ \partial x }  \X_{s,u}^x  
\big)  y,
A
\big(
\tfrac{\partial}{ \partial x } \X_{s,u}^x 
\big)  y
+
 F'(  \X_{s,u}^{ x} + O_u )   
\big(
\tfrac{\partial}{ \partial x } \X_{s,u}^x 
\big) y
\big \rangle_H 
\, 
du
\\
&
= 
2
\int_s^t
\Big[
\big \langle
 F'( \X_{s,u}^{ x} + O_u  )
  \big(
  \tfrac{\partial}{ \partial x } \X_{s,u}^{ x}  
  \big)y,
\big(
\tfrac{\partial}{ \partial x } \X_{s,u}^{ x}
\big) y 
\big \rangle_H
- 
\big\| 
\big(
\tfrac{\partial}{ \partial x }
\X_{s,u}^{ x}  \big) 
y 
\big\|_{H_{\nicefrac{1}{2}}}^2 
\Big] 
\, 
du
\\
&
\leq 
2  
\int_s^t
\Big[
\big(
\varepsilon
\|  \X_{s,u}^{x}  + O_u \|_{H_{\nicefrac{1}{2}}}^2
+
{ \bf C }
\big)
\big\|
 \big(
\tfrac{\partial}{ \partial x } \X_{s,u}^{ x}  
\big) y
\big\|_H^2
+
\big\|
\big(
\tfrac{\partial}{ \partial x }  \X_{s,u}^{ x} 
\big)y 
\big\|_{H_{\nicefrac{1}{2}}}^2 
-
\big\| 
\big(
\tfrac{\partial}{ \partial x }
\X_{s,u}^{x}  
\big) y
\big\|_{H_{\nicefrac{1}{2}}}^2
\Big] 
\, 
du
\\
&
= 
2 
\int_s^t
\Big[  
\big(
\varepsilon
\|  \X_{s,u}^{ x} + O_u   \|_{H_{\nicefrac{1}{2}}}^2
+
{ \bf C }
\big)
\big\|
\big( 
\tfrac{\partial}{ \partial x } \X_{s,u}^{x} 
  \big)
   y \big\|_H^2 
\Big] 
\,
du.
\end{split}
\end{equation}
Moreover,
note that the assumption that
$ \dim(H) < \infty $
assures that
for all  
$ s \in [0,T] $,
$ x \in H $
we have that 
\begin{equation} 
\big( [s,T] \ni u \mapsto \| \X_{s,u}^x 
+ O_u \|_{H_{\nicefrac{1}{2}}}^2 
\in 
[0, \infty)
\big) \in 
\mathcal{C}([s,T], [0,\infty)) 
.
\end{equation} 
Combining this, item~\eqref{item:continuity3},
and~\eqref{eq:some_impo_estimate}
with
Gronwall's lemma 
illustrates that
for all  
$ s \in [0,T] $,
$ t \in [s,T] $,
$ x, y \in H $
we have that
\begin{equation}
\begin{split}
\big\| \big(
\tfrac{\partial}{ \partial x } \X_{s,t}^{ x} 
\big) y 
\big\|_H
\leq
\| y \|_H
\exp\!
\bigg(  
\int_s^t
\big(
\varepsilon
\| \X_{s,u}^x  + O_u \|_{H_{\nicefrac{1}{2}}}^2 
+ 
{ \bf C }
\big)
\,
du
\bigg).
\end{split}
\end{equation}
The proof of Proposition~\ref{proposition:Derivative_process_estimate} is hereby completed. 
\end{proof}
\begin{corollary}
	\label{corollary:exp_bound}
Assume Setting~\ref{setting:main}, 
assume that $ \dim(H) < \infty $, 
let $ ( \Omega, \F, \P ) $
be a probability space,
let $ T \in (0,\infty) $,
$ \varepsilon, {\bf C} \in [0,\infty) $,
$ p \in [1,\infty) $,
$ F \in \mathcal{C}^1  (H,
H) $,    
$ Y \in \M( \B([0,T]) \otimes \F, \B(H)) $,
let 
$ O \colon [0,T] \times \Omega \to H $
be a stochastic process
w.c.s.p.,
assume for all $ x, y \in H $ 
that 
$ \langle F'(x) y, y \rangle_H \leq 
( \varepsilon \| x \|_{H_{\nicefrac{1}{2}}}^2 + { \bf C }  ) \| y \|_H^2 
+ \|y \|_{H_{\nicefrac{1}{2}}}^2 $,
and
for every 
$ s \in [0,T] $,
$ x \in H $
let
	$ X_{s,(\cdot)}^x
=
(X_{s,t}^x)_{ t \in [s,T] }
\colon 
[s,T] \times \Omega
\to H $
be a stochastic process
w.c.s.p.\ 
which satisfies for all   
$ t \in [s,T] $
that
\begin{equation} 
X_{s,t}^{x }
=
e^{ (t-s) A }   x
+
\int_s^t
e^{(t-u)A}
 F( X_{s,u}^{x } )
\, du
+
O_t - e^{(t-s)A} O_s
.
\end{equation}
Then 
\begin{enumerate}[(i)]
	\item \label{item:continuity4}
	we have
	for all 
$ \omega \in \Omega $
 that
$ (
  \{ (u, v) \in [0,T]^2 \colon u \leq v \} \times H
\ni
(s,t,x)
\mapsto
X^x_{s,t}(\omega) 
\in H 
 )
\in \mathcal{C}^{0,0,1}
( \{ (u, v) \in [0,T]^2 \colon u \leq v \} \times H, H) $,
\item \label{item:second_joint_measurable}
\sloppy 
we have   
that 
$ \big( \{ (u, v) \in [0,T]^2 \colon u \leq v \} \times \Omega \ni (s,t,\omega)
\mapsto 
\frac{ \partial }{ \partial x } 
X_{s,t}^{Y_s(\omega) 
}(\omega) \in L( H ) 
\big)
\in \M( \B( \{ (u, v) \in [0,T]^2 \colon u \leq v \} ) \otimes \F, \B( L( H ) ) ) $,
\item \label{item:measurability RHS}
we have  
that 
$ \big( \{ (u, v) \in [0,T]^2 \colon u \leq v \} \times \Omega \ni (s,t,\omega) 
\mapsto X_{s,t}^{Y_s(\omega)  
}(\omega) \in H_{ \nicefrac{1}{2} } \big)
\in \M( \B( \{ (u, v) \in [0,T]^2 \colon u \leq v \} ) \otimes \F, \B( H_{ \nicefrac{1}{2} } ) ) $,
and
\item \label{item:L1estimate} we have for all 
$ s \in [0,T] $,
$ t \in [s,T] $
that
\begin{equation}
\begin{split}
\E \Big[ 
\big \| \tfrac{\partial}{\partial x} 
X_{s,t}^{ Y_s }  
\big\|_{L(H)}^p
\Big] 
\leq 
\E 
\bigg[ 
\exp\!
\bigg(
p 
\int_s^t 
 \big( 
 \varepsilon \|    X_{s,u}^{ Y_s 
 }  \|_{H_{\nicefrac{1}{2}}}^2
+ { \bf C } \big)
\,
du
\bigg)
\bigg].
\end{split}
\end{equation}
\end{enumerate}
\end{corollary}
\begin{proof}[Proof of Corollary~\ref{corollary:exp_bound}]
Throughout this proof let
$ \X_{s,(\cdot)}^x
=
(\X_{s,t}^x)_{ t \in [s,T] }
\colon 
[s,T] \times \Omega
\to H $,
$ s \in [0,T] $,
$ x \in H $,
satisfy for all 
$ s \in [0,T] $,
$ t \in [s,T] $,
$ \omega \in \Omega $,
$ x \in H $ 
that
\begin{equation} 
\label{eq:DefineNew} 
\X_{s,t}^x(\omega) = X_{s,t}^{x + O_s(\omega)}(\omega) - O_t(\omega) 
.
\end{equation}
Observe that items~\eqref{item:ContDiff} and~\eqref{item:Mesurable} of
Lemma~\ref{lemma:MeasurableDerivative}
(applies with
$ ( \Omega, \F, \P ) = ( \Omega, \F, \P ) $,
$ T = T $,
$ F = F $,
$ Z_s = Y_s $,
$ O_s = O_s $,
$ X_{s,t}^x = X_{s,t}^x $
for
$ t \in [s,T] $,
$ s \in [0,T] $,
$ x \in H $
in the setting of
items~\eqref{item:ContDiff}
and~\eqref{item:Mesurable} of
Lemma~\ref{lemma:MeasurableDerivative})
justify
items~\eqref{item:continuity4}
and~\eqref{item:second_joint_measurable}.
Furthermore, note that item~\eqref{item:productMeasurability} 
of 
Lemma~\ref{lemma:MeasurableDerivative}
(applies with
$ ( \Omega, \F, \P ) = ( \Omega, \F, \P ) $,
$ T = T $,
$ F = F $,
$ Y_s = Y_s $,
$ O_s = O_s $,
$ X_{s,t}^x = X_{s,t}^x $
for
$ t \in [s,T] $,
$ s \in [0,T] $,
$ x \in H $
in the setting of
item~\eqref{item:productMeasurability}  of
Lemma~\ref{lemma:MeasurableDerivative})
gives
that 
\begin{multline} 
\big( \{ (u, v) \in [0,T]^2 \colon u \leq v \} \times \Omega \ni (s,t,\omega) 
\mapsto X_{s,t}^{Y_s(\omega) 
}(\omega) \in H \big)
\\
\in \M( \B( \{ (u, v) \in [0,T]^2 \colon u \leq v \} ) \otimes \F, \B( H) ) 
.
\end{multline} 
The assumption that $ \dim(H) < \infty $
hence
justifies item~\eqref{item:measurability RHS}.
Next observe that~\eqref{eq:DefineNew} 
and the fact that
for all
$ s \in [0,T] $,
$ t \in [s,T] $,
$ \omega \in \Omega $,
$ x \in H $ 
we have that
\begin{equation} 
X_{s,t}^{x + O_s(\omega)}(\omega)
=
e^{ (t-s) A }   ( x + O_s(\omega) )
+
\int_s^t
e^{(t-u)A}
 F( X_{s,u}^{x + O_s(\omega) } (\omega) )
\, du
+
O_t(\omega) - e^{(t-s)A} O_s(\omega) 
\end{equation}
verify that
for all 
$ s \in [0,T] $,
$ t \in [s,T] $,
$ \omega \in \Omega $,
$ x \in H $
we have that
\begin{equation}
\begin{split}
\X_{s,t}^x(\omega)
=
e^{(t-s)A} 
x
+
\int_s^t
e^{ (t-u) A } F( \X_{s,u}^x(\omega) + O_u(\omega) )
\, du
.
\end{split}
\end{equation}
The fact that
$ F \in \mathcal{C}(H,H) $,
the fact that
$ \forall \, s \in [0,T] $, 
$ \omega \in \Omega 
\colon  
( [s,T] \ni t \mapsto O_t(\omega) \in H )
\in \mathcal{C}( [s,T], H) $,
the fact that
$ \forall \, s \in [0,T] $, 
$ \omega \in \Omega $, 
$ x \in H 
\colon 
( [s,T] \ni t \mapsto \X_{s,t}^x( \omega) \in H
)
\in \mathcal{C}( [s,T], H) $, 
and
Lemma~\ref{lemma:Simplified}
(applies with
$ T = T $,
$ s = s $,
$ x = x $,
$ Z = ( [s,T] \ni t \mapsto F( \X_{s, t}^x(\omega) + O_t(\omega) ) \in H ) $,
$ Y = ( [s,T] \ni t \mapsto \X_{s,t}^x(\omega) \in H ) $
for 
$ s \in [0,T] $,
$ \omega \in \Omega $,
$ x \in H $
in the setting of Lemma~\ref{lemma:Simplified})
therefore
ensure that
for all $ s \in [0,T] $,
$ t \in [s,T] $,
$ \omega \in \Omega $,
$ x \in H $
we have that
\begin{equation}
\label{eq:necessary}
\X_{s,t}^x(\omega)
= 
x
+
\int_s^t
[
A \X_{s,u}^x(\omega)
+
F( \X_{s,u}^x(\omega) + O_u(\omega) )
]
\, du
.
\end{equation}
Item~\eqref{item:continuity3} 
of
Proposition~\ref{proposition:Derivative_process_estimate}
(applies with 
$ T = T $,
$ \varepsilon = \varepsilon $,
$ { \bf C } = { \bf C } $,
$ F = F $,
$ O_t = O_t( \omega ) $,
$ \X_{s,t}^x = \X_{s,t}^x(\omega) $ 
for 
$ t \in [s,T] $,
$ s \in [0,T] $,
$ \omega \in \Omega $,
$ x \in H $
in the setting of
item~\eqref{item:continuity3} 
of
Proposition~\ref{proposition:Derivative_process_estimate})
hence
verifies that 
for all $ \omega \in \Omega $ we have that
\begin{multline} 
(
\{ (u, v) \in [0,T]^2 \colon u \leq v \} \times H
\ni
(s,t,x)
\mapsto
\X^x_{s,t}(\omega) 
\in H 
)
\\
\in \mathcal{C}^{0,0,1}
( \{ (u, v) \in [0,T]^2 \colon u \leq v \} \times H, H) 
.
\end{multline} 
Moreover, 
observe that~\eqref{eq:necessary}
and 
item~\eqref{item:derivative_estimate} 
of
Proposition~\ref{proposition:Derivative_process_estimate}
(applies with 
$ T = T $,
$ \varepsilon = \varepsilon $,
$ { \bf C } = { \bf C } $,
$ F = F $,
$ O_t = O_t( \omega ) $,
$ \X_{s,t}^x = \X_{s,t}^x(\omega) $ 
for 
$ t \in [s,T] $,
$ s \in [0,T] $,
$ \omega \in \Omega $,
$ x \in H $
in the setting of
item~\eqref{item:derivative_estimate} 
of
Proposition~\ref{proposition:Derivative_process_estimate}) 
ensure that  for all 
	$ s \in [0,T] $,
	$ t \in [s,T] $,
	$ \omega \in \Omega $
	we have
	that
	\begin{equation}
	\begin{split}
	\big\| 
	\tfrac{\partial}{\partial x}  \X_{s, t}^{ Y_s (\omega) - O_s(\omega) }  (\omega)
	\big\|_{L(H)}
	\leq
	\exp
	\!
	\bigg(   
	\int_s^t
	\big(
	\varepsilon
	\|  \X_{s,u}^{ Y_s (\omega) - O_s(\omega) } (\omega) + O_u (\omega)  \|_{H_{\nicefrac{1}{2}}}^2
	+
	{ \bf C }
	\big)
	\,
	du
	\bigg).
	\end{split}
	\end{equation}
This and~\eqref{eq:DefineNew} 
yield that
for all 
$ s \in [0,T] $,
$ t \in [s,T] $,
$ \omega \in \Omega $
we have that
\begin{equation}
\begin{split}
\label{eq:EstimateNoE2}
\big\| 
\tfrac{\partial}{\partial x}  X_{s, t}^{ Y_s (\omega)  }  (\omega)
\big\|_{L(H)}
\leq
\exp
\!
\bigg(  
\int_s^t
\big(
\varepsilon
\|  X_{s,u}^{ Y_s (\omega) } (\omega)    \|_{H_{\nicefrac{1}{2}}}^2
+
{ \bf C }
\big)
\,
du
\bigg).
\end{split}
\end{equation}
Combining 
this
and 
items~\eqref{item:second_joint_measurable}
and~\eqref{item:measurability RHS}
justifies item~\eqref{item:L1estimate}.
The proof of  
Corollary~\ref{corollary:exp_bound}
is hereby completed.
\end{proof}
\begin{lemma}
	\label{lemma:Ito formula}
	Assume Setting~\ref{setting:main},
	assume that $ \dim(H) < \infty $,
	let 
	$ T \in (0,\infty) $,  
	$ s \in [0,T] $,  
	$ B \in \HS(U,H) $,
	let $ ( \Omega, \F, \P ) $ 
	be a probability space,
	let $ (W_t)_{ t \in [0,T]} $
	be an
	$ \operatorname{Id}_U $-cylindrical
	Wiener process, 
	let
	$ \xi \in \M( \F, \mathcal{B}(H) ) $,
	$ Z\in\mathcal M(\B([s,T])\otimes\F,\B(H)) $ satisfy for all 
	$ \omega \in \Omega $ 
	that 
	$ \int_s^T \|Z_u(\omega)\|_H \, du < \infty $,
	and
	let 
	$ Y \colon [s,T] \times \Omega \to H $
	and
	$ O \colon [0,T] \times \Omega \to H $
	be stochastic processes
w.c.s.p.\ 
	which satisfy for all $ t \in [s,T] $
	that
	$ [ O_t ]_{\P, \mathcal{B}(H ) } 
	= 
	\int_0^t e^{ ( t - u ) A } B \, dW_u $
	and
	\begin{equation}
	\begin{split}
	\label{eq:a.s.}
	\P\bigg(
	Y_t = e^{(t-s)A} \xi 
	+ 
	\int_s^t e^{(t-u)A} Z_u \, du
	+
	O_t - e^{(t-s)A} O_s
	\bigg)
	=
	1
	.
	\end{split} 
	\end{equation}
	Then 
	we have for all $ t \in [s, T] $ that
	\begin{equation}
	\begin{split} 
	\label{eq:finaly established}
	[ Y_t ]_{\P, \mathcal{B}(H)} 
	=
	\bigg[
	\xi 
	+ 
	\int_s^t [ A Y_u + Z_u ] \, du
	\bigg]_{\P, \mathcal{B}(H)} 
	+
	\int_s^t B \, dW_u 
	.
	\end{split}
	\end{equation}
\end{lemma}
\begin{proof}[Proof of Lemma~\ref{lemma:Ito formula}]
Throughout this proof let
$ \Sigma 
=
\{ \omega \in \Omega 
\colon 
(
\forall \, t \in [s,T] \colon 
Y_t(\omega) = e^{(t-s)A} \xi(\omega) 
+ 
\int_s^t e^{(t-u)A} Z_u(\omega) \, du
+
O_t(\omega) - e^{(t-s)A} O_s(\omega)
)
\} $.
Observe that
item~\eqref{item:continuous} of
Lemma~\ref{lemma:Simplified}
(applies with
$ T = T $,
$ s = s $,
$ x = \xi(\omega) $,
$ Z_t = Z_t(\omega) $,
$ Y_t = 
e^{(t-s)A} \xi(\omega) 
+ 
\int_s^t e^{(t-u)A} Z_u(\omega) \, du $
for 
$ t \in [s, T] $,
$ \omega \in \Sigma $
in the setting of
item~\eqref{item:continuous} of 
Lemma~\ref{lemma:Simplified})
verifies that
for all 
$ \omega \in \Omega $
we have that
\begin{equation} 
\bigg( 
[s,T] \ni t
\mapsto  
e^{(t-s)A} \xi(\omega) 
+ 
\int_s^t e^{(t-u)A} Z_u(\omega) \, du \in H 
\bigg) 
\in \mathcal{C}( [s,T], H )
.
\end{equation}  
The fact that 
$ O $ and $ Y $ have continuous sample paths 
and~\eqref{eq:a.s.}
therefore yield that 
\begin{equation}
\label{eq:full Sigma}
\P( \Sigma ) = 1
.
\end{equation}
Next note that the assumption that $ \dim(H) < \infty $ ensures that
for all $ t \in [s,T] $ we have that
\begin{equation} 
\begin{split} 
\label{eq: Ito 3}
[ e^{-(t-s)A}
O_t ]_{\P, \B(H) }
& =
\int_0^t
e^{(s-u)A} 
B
\, dW_u 
= 
\int_0^s
e^{(s-u)A} 
B
\, dW_u
+
\int_s^t
e^{(s-u)A} 
B
\, dW_u 
\\
&
=
[ O_s ]_{\P, \B(H)} 
+
\int_s^t
e^{(s-u)A} 
B
\, dW_u 
.
\end{split} 
\end{equation} 
This gives that for all $ t \in [s,T] $ we have that
\begin{equation}
\int_s^t
e^{(s-u)A} 
B
\, dW_u 
=
[ e^{- ( t - s ) A} O_t -   O_s ]_{\P, \B(H)}
.
\end{equation}
Combining~\eqref{eq: Ito 3},
the fact that 
$ [s,T] \times H \ni (t,x)
\mapsto e^{(t-s)A} x \in H $
is twice continuously differentiable, 
and It\^o's formula
hence
yields that for all $ t \in [s,T] $
we have that
\begin{equation}
\begin{split}
[ O_t ]_{\P, \B(H)}
&
=
[
e^{(t-s)A}
O_s
]_{\P, \B(H) }
+
e^{(t-s)A}
\int_s^t
e^{(s-u)A} 
B
\, dW_u 
\\
&
=
\bigg[ 
e^{(t-s)A}
O_s
+
\int_s^t 
A
e^{(u - s )A}
( 
e^{-(u-s)A} O_u - O_s
) 
\, 
du
\bigg ]_{\P, \B(H)}
+
\int_s^t e^{(u-s)A} 
e^{(s-u)A} B \, dW_u
\\
&
=
\bigg [ 
e^{(t-s)A}
O_s
+
\int_s^t 
A ( 
O_u - e^{(u-s)A} O_s
) \, du
\bigg ]_{\P, \B(H)}
+
\int_s^t B \, dW_u
.
\end{split}
\end{equation}
This gives that for all $ t \in [s,T] $ we have that
\begin{equation}
\begin{split}
\label{eq: first part}
[ O_t - e^{(t-s)A} O_s ]_{ \P, \B(H) }
=
\bigg[
\int_s^t A ( O_u - e^{(u-s)A} O_s) \, du
\bigg]_{\P, \B(H)}
+
\int_s^t B \, dW_u
.
\end{split}
\end{equation}
Moreover, observe that 
item~\eqref{item:rewrite} of
Lemma~\ref{lemma:Simplified}
(applies with
$ T = T $,
$ s = s $,
$ x = \xi(\omega) $,
$ Z_t = Z_t(\omega) $,
$ Y_t = 
Y_t(\omega) - ( O_t(\omega) - e^{(t-s)A} O_s(\omega) ) $
for 
$ t \in [s, T] $,
$ \omega \in \Sigma $
in the setting of
item~\eqref{item:rewrite} of 
Lemma~\ref{lemma:Simplified})
verifies that
for all
$ t \in [s,T] $,
$ \omega \in \Sigma $ 
we have 
that
\begin{equation}
\label{eq: second part}
Y_t(\omega) - ( O_t(\omega) - e^{(t-s)A} O_s(\omega) )
=
\xi(\omega)
+
\int_s^t
\big[ 
A
\big(
Y_u(\omega) - ( O_u(\omega) - e^{(u-s)A} O_s(\omega) )
\big)
+
Z_u(\omega)
\big] 
\, du
.
\end{equation}
Combining~\eqref{eq:full Sigma} and~\eqref{eq: first part}
therefore justifies~\eqref{eq:finaly established}.
The proof of Lemma~\ref{lemma:Ito formula} is hereby completed.
\end{proof}
\begin{lemma}
	\label{lemma:inheritet}	
	Assume Setting~\ref{setting:main}, 
	assume that $ \dim(H) < \infty $, 
	let 
	$ T \in (0,\infty) $,  
	$ a, b, { \bf C}, \rho \in [0,\infty) $,   
	$ p \in [1,\infty) $, 
	$ B \in \HS(U,H) $,    
	$  \varepsilon \in [0, ( \nicefrac{ 2 \rho }{ p } )
	\exp( - 2 (  b + \rho \| B \|_{\HS(U,H)}^2 ) T ) ] $,
	$ F \in \mathcal{C}^1( H, H ) $,
	assume
	for all      
	$ x, y \in H $
	that  
	$ \langle x, F(x) \rangle_H  
	\leq
	a + b \| x \|_H^2 $  
	and
	$ \langle F'(x) y, y \rangle_H \leq 
	( \varepsilon \| x \|_{H_{\nicefrac{1}{2}}}^2 + { \bf C}  ) \| y \|_H^2 
	+ \|y \|_{H_{\nicefrac{1}{2}}}^2 $, 
	let $ ( \Omega, \F, \P ) $
	be a probability space with a normal filtration
	$ ( \f_t )_{t \in [0,T]} $,
	let $ (W_t)_{ t \in [0,T]} $
	be an
	$ \operatorname{Id}_U $-cylindrical
	$ ( \f_t )_{t\in [0,T]} $-Wiener process,  
	let
	$ Y \colon [0,T] \times \Omega \to H $
	and
	$ O \colon [0,T] \times \Omega \to H $
	be 
	$ ( \f_t )_{t\in [0,T]} $-adapted
	stochastic processes
	w.c.s.p.,
	assume
	for all     
	$ t \in [0,T] $
	that 
	$ [ O_t ]_{\P, \mathcal{B}(H ) } = \int_0^t e^{(t- u)A} B \, dW_u $,
	and
	for every
	$ s \in [0,T] $,
	$ x \in H $
	let 
		$ X_{s,(\cdot)}^x
	=
	(X_{s,t}^x)_{ t \in [s,T] }
	\colon 
	[s,T] \times \Omega
	\to H $ 
	be an $ ( \f_t )_{ t \in [s,T] } $-adapted 
	stochastic process 
	w.c.s.p.\
	which satisfies
	for all
	$ t \in [s,T] $ that
	\begin{equation} 
	\label{eq:DefineX2}
	X_{s,t}^{x }
	=
	e^{ (t-s) A }   x
	+
	\int_s^t
	e^{(t-u)A}
	F( X_{s,u}^{x } )
	\, du
	+
	O_t - e^{(t-s)A} O_s
	.
	\end{equation} 
	Then 
	\begin{enumerate}[(i)]
		\item \label{item:differentiable5} 
		we have for all 
		$ \omega \in \Omega $
		that
		$ (
		\{ (u, v) \in [0,T]^2 \colon u \leq v \} \times H
		\ni
		(s,t,x)
		\mapsto
		X^x_{s,t}(\omega) 
		\in H 
		)
		\in \mathcal{C}^{0,0,1}
		( \{ (u, v) \in [0,T]^2 \colon u \leq v \} \times H, H) $,
		\item \label{item:second_joint_measurable2}
		\sloppy 
		we have that 
		$ \big( 
		\{ (u, v) \in [0,T]^2 \colon u \leq v \} \times \Omega \ni (s,t,\omega)
		\mapsto 
		\frac{ \partial }{ \partial x } 
		X_{s,t}^{Y_s(\omega)  }(\omega) \in L( H ) \big )
		\in \M( \B( \{ (u, v) \in [0,T]^2 \colon u \leq v \} ) \otimes \F, \B( L( H ) ) ) $,
		and
		\item we have for all  
		$ s \in [0,T] $,
		$ t \in [s,T] $ 
		that 
		\begin{equation}
		\begin{split} 
		\label{eq:UpperBound}
		&  
		\E 
		\Big[ 
		\big \| \tfrac{\partial}{\partial x} 
		X_{s,t}^{ Y_s }  
		\big\|_{L(H)}^p
		\Big]
		\leq
		\exp
		\!
		\big(  
		\big( 
		p { \bf C} 
		+   
			\rho
			( 2a + \| B \|_{\HS(U,H)}^2 ) 
			\big)
			(t-s)
		\big)
		\E
		\big[
		e^{ \rho \| Y_s \|_H^2 }
		\big]
		. 
		\end{split}
		\end{equation}
	\end{enumerate}
\end{lemma}
\begin{proof}[Proof of Lemma~\ref{lemma:inheritet}]
	Throughout this proof 
	let 
	$ \mathbb{B} \in L(H, U) $
	satisfy for all
	$ v \in H $, $ u \in U $ that
	$ \langle B u, v \rangle_H
	=
	\langle u, \mathbb{B} v \rangle_U $,
	let
	$ R \colon U \to [ \operatorname{ker}(B) ]^\perp $ 
	be the orthogonal projection
	of $ U $ on $ [ \operatorname{ker}(B) ]^\perp $,
	let
	$ d = \dim(H) $,
	$ m = \dim( [ \operatorname{ker}( B ) ]^\perp) $,
	and
	let 
	$ \iota \colon H \to \R^d $
	and
	$ \kappa \colon R(U) \to \R^m $
	be isometric isomorphisms.
	Observe that the assumption that
	for all
	$ x, y \in H $
	we have that
	$ \langle F'(x) y, y \rangle_H \leq 
	( \varepsilon \| x \|_{H_{\nicefrac{1}{2}}}^2 + { \bf C}  ) \| y \|_H^2 
	+ \|y \|_{H_{\nicefrac{1}{2}}}^2 $
	and
	items~\eqref{item:continuity4} and~\eqref{item:second_joint_measurable}
	of
	Corollary~\ref{corollary:exp_bound}
	(applies with 
	$ ( \Omega, \F, \P  ) =
	( \Omega, \F, \P ) $, 
	$ T = T $,
	$ \varepsilon = \varepsilon $,
	$ { \bf C } = { \bf C} $, 
	$ p = p $,
	$ F = F $,
	$ Y_s = Y_s $, 
	$ O_s = O_s $,
	$ X_{s, t}^x = X_{s, t}^{x } $	 
	for    
	$ t \in [s,T] $,
	$ s \in [0,T] $,
	$ x \in H $
	in the setting of 
	items~\eqref{item:continuity4} and~\eqref{item:second_joint_measurable}
	of
	Corollary~\ref{corollary:exp_bound}) 
	justify items~\eqref{item:differentiable5}     and~\eqref{item:second_joint_measurable2}.
	Moreover, note that
	 the assumption that
	for all
	$ x, y \in H $
	we have that
	$ \langle F'(x) y, y \rangle_H \leq 
	( \varepsilon \| x \|_{H_{\nicefrac{1}{2}}}^2 + { \bf C}  ) \| y \|_H^2 
	+ \|y \|_{H_{\nicefrac{1}{2}}}^2 $
	and
	items~\eqref{item:measurability RHS}
	and~\eqref{item:L1estimate}
	of
	Corollary~\ref{corollary:exp_bound}
	(applies with 
	$ ( \Omega, \F, \P  ) =
	( \Omega, \F, \P ) $, 
	$ T = T $,
	$ \varepsilon = \varepsilon $,
	$ { \bf C } = { \bf C} $, 
	$ p = p $,
	$ F = F $,
	$ Y_s = Y_s $, 
	$ O_s = O_s $,
	$ X_{s, t}^x = X_{s, t}^{x } $	 
	for    
	$ t \in [s,T] $,
	$ s \in [0,T] $,
	$ x \in H $
	in the setting of 
	items~\eqref{item:measurability RHS}
	and~\eqref{item:L1estimate}
	of
	Corollary~\ref{corollary:exp_bound}) 
	verify that for all 
	$ s \in [0,T] $,
	$ t \in [s,T] $
	we have
	that
	\begin{equation}
	\begin{split} 
	\label{eq:first_estimate1}
	\E \Big[ 
	\big \| \tfrac{\partial}{\partial x} 
	X_{s,t}^{ Y_s  }  
	\big\|_{L(H)}^p
	\Big]
	\leq  
	e^{p { \bf C} (t-s)}
	\E\bigg[
	\exp\!\bigg(
	p  \varepsilon \int_s^t  
	\|
	X_{s,u}^{ Y_s }
	\|_{H_{\nicefrac{1}{2}}}^2 
	\,
	du
	\bigg)
	\bigg]
	.
	\end{split}
	\end{equation}
	In the next step we intend to apply Cox et al.\ \cite[Corollary~2.4]{CoxHutzenthalerJentzen2013} in order to derive an a priori bound for the right-hand side of~\eqref{eq:first_estimate1}.
	For this note that
	the assumption that
	for all $ x \in H $
	we have that
	$ \langle x, F(x) \rangle_H  
	\leq
	a + b \| x \|_H^2 $
	gives that for all
	$ x \in H $ we have that
	\begin{equation}
	\begin{split}
	\label{eq:generator_assumption}
	&
	2 \rho \langle x,  A x + F(x) \rangle_H
	+
	\rho \|  B  \|_{\HS(U,H)}^2
	+
	2 \rho^2 \| \mathbb{B} x\|_U^2
	\\
	&
	\leq
	- 2 \rho \| x \|_{H_{\nicefrac{1}{2}}}^2 
	+
	2 \rho \langle x, F(x) \rangle_H
	+
	\rho \|  B   \|_{HS(U,H)}^2
	+
	2 \rho^2 \|   B   \|_{\HS(U, H)}^2 
	\| x \|_H^2 
	\\
	&
	\leq 
	- 2 \rho \| x \|_{H_{\nicefrac{1}{2}}}^2 
	+
	2 \rho a + 2 \rho b \| x \|_H^2
	+
	\rho \| B \|_{\HS(U,H)}^2
	+
	2 \rho^2 \| B \|_{\HS(U,H)}^2 
	\| x \|_H^2
	\\
	&
	=
	- 
	2 \rho
	\| x \|_{H_{\nicefrac{1}{2}}}^2
	+
	\rho (2 a + \| B \|_{\HS(U,H)}^2)
	+
	2 \rho ( b + \rho \| B \|_{\HS(U,H)}^2 )
	\| x \|_H^2			
	.
	\end{split}
	\end{equation}
	Next note that
	Lemma~\ref{lemma:Ito formula}
	(applies with
	$ T = T $,
	$ s = s $,
		$ B = B $, 
	$ ( \Omega, \F, \P ) =
	 ( \Omega, \F, \P ) $,
	$ (W_t)_{ t \in [0,T] } = (W_t)_{ t \in [0,T] } $,
	$ \xi = Y_s $,
	$ Z_{s + t } =  F( X_{s, s+t}^{Y_s} ) $,
	$ Y_{ s + t} = X_{s, s + t}^{Y_s} $,
	$ O = O $
	for  
	$ t \in [0, T-s] $,
	$ s \in [0,T] $
	in the setting of 
	Lemma~\ref{lemma:Ito formula})
	ensures that for all
	$ s \in [0, T] $,
	$ t \in [0, T - s] $ we have that
	\begin{equation}
	[ X_{s, s+t}^{Y_s} ]_{\P, \B(H)}
	=
	[
	Y_s
	]_{\P, \B(H)}
	+
	\bigg[ 
	\int_s^{s + t} 
	\big[
	A
	X_{s,u}^{Y_s} 
	+
	F( X_{s,u}^{Y_s} ) 
	\big]
	\, du
	\bigg ]_{\P, \B(H)}
	+
	\int_s^{s + t} B \, dW_u
	.
	\end{equation}
	Moreover, observe that 
	the assumption that
	$ \dim(H) < \infty $
	ensures that
	$ \dim( [ \operatorname{ker}(B) ]^\perp ) < \infty $
	and
	$ R \in \HS(U) $.
	This gives that there
	exists
	a stochastic process 
	$ \mathbb{W} \colon [0,T] \times \Omega \to R(U) $  
	w.c.s.p.\ 
	which satisfies for all $ t \in [0,T] $ that
	\begin{equation} 
	\label{eq:some BM}
	[ \mathbb{W}_t ]_{\P, \mathcal{B}( R(U) ) } = \int_0^t  R \, dW_s 
	.
	\end{equation}
	Observe that~\eqref{eq:some BM} 
	gives that for all 
	$ s \in [0,T] $,
	$ t \in [0, T-s] $
	we have that
	\begin{equation} 
	\begin{split}
	\label{eq:BrownianFits}
	\int_s^{s + t}
	B \, dW_u
	=
	\int_s^{s + t}
	B R \, dW_u
	=
	\int_s^{s+t}
	( B|_{R(U)}  ) 
	\, d \mathbb{W}_u
	=
	[ 
	( B|_{R(U)} ) 
	( \mathbb{W}_{s + t} - \mathbb{W}_s ) 
	]_{\P, \mathcal{B}( H )}
	.
	\end{split}
	\end{equation}
	In addition, note that, e.g., 
	\cite[Lemma~3.2]{JentzenPusnik2018Published}
	(applies with 
	$ H = R(U) $,
	$ U = U $,
	$ T = T $,
	$ Q = \operatorname{Id}_U $,
	$ R = \operatorname{Id}_{R(U)} $,
	$ ( \Omega, \F,
	\P, ( \F_t )_{ t \in [0,T]} )
	=
	( \Omega, \F,
	\P, ( \f_t )_{ t \in [0,T]} ) $,
	$ ( W_t )_{ t \in [0,T] } = ( W_t )_{ t \in [0,T] } $,
	$ (\mathcal{G}_t )_{t\in [0,T]}
	=
	(\f_t )_{t\in [0,T]} $,
	$ (  \tilde W_t  )_{ t \in [0,T] }
	=
	( \mathbb{W}_t )_{ t \in [0,T] } $
	in the setting of~\cite[Lemma~3.2]{JentzenPusnik2018Published})
	verifies that 
	$ ( \mathbb{W}_t )_{ t \in [0,T] } $	
	is an $ \operatorname{Id}_{R(U)} $-standard
	$ ( \f_t )_{ t \in [0,T]} $-Wiener process.
	Combining this, 
	\eqref{eq:generator_assumption},
	and~\eqref{eq:BrownianFits}
	with
	Cox et al.\ \cite[Corollary~2.4]{CoxHutzenthalerJentzen2013}
		(applies with
	$ d = \dim(H) $,
	$ m = \dim( [ \operatorname{ker}( B ) ]^\perp) $,
	$ T = T-s $,
	$ O = \R^d $,
	$ \mu = ( \R^d \ni x \mapsto (\iota \circ A \circ \iota^{-1} )( x ) 
	+ (\iota \circ F \circ \iota^{-1} )(x) \in \R^d ) $, 
	$ \sigma = ( \R^d \ni x \mapsto \iota \circ (B|_{R(U)}) \circ  \kappa^{-1} \in \HS(\R^m, \R^d ) ) $,
	$ ( \Omega, \F, \P  ) 
	= ( \Omega, \F, \P ) $,
	$ \F_u = \f_{s + u} $,
	$ W_u = \kappa( \mathbb{W}_{s + u} - \mathbb{W}_s ) $, 
	$ \alpha = 2 b + 2 \rho \| B \|_{\HS(U,H)}^2 $,
	$ U = ( \R^d \ni x \mapsto \rho \| \iota^{-1}(x) \|_H^2 \in \R ) $,
	$ \bar U = ( [0, T-s]  \times \R^d \ni (r, x) \mapsto
	2 \rho \| \iota^{-1}(x) \|_{H_{\nicefrac{1}{2}}}^2
	-
	\rho ( 2 a + \| B \|_{\HS(U,H)}^2 )	
	\in \R ) $,
	$ \tau = ( \Omega \ni \omega \mapsto t - s \in [0, T - s] ) $,
	$ X_u = \iota \circ X_{s, s + u}^{  Y_s } $ 
	for 
	$ u \in [0, T-s] $,
	$ t \in [s, T] $, 
	$ s \in [0, T) $
	in the setting of Cox et al.\ \cite[Corollary~2.4]{CoxHutzenthalerJentzen2013}) 
	yields that for all 
	$ s \in [0,T] $,
	$ t \in [s,T] $
	we have that
	\begin{equation}
	\begin{split} 
	&
	\E\bigg[
	\exp\!
	\bigg(
	\rho
	e^{
		- 2 (
		b + \rho \| B \|_{\HS(U,H)}^2)
		( t - s )
	}
	\| X^{ Y_s }_{s,t} \|_H^2
	\\
	&
	\quad
	+
	\int_s^t
	e^{
		- 2 (
		b + \rho \| B \|_{\HS(U,H)}^2)
		( u - s )
	}
	\big (
	2 \rho \| X_{s,u}^{ Y_s } 
	\|_{ H_{ \nicefrac{1}{2} } }^2 - \rho(2a + \| B \|_{\HS(U,H)}^2 ) 
	\big )
	\, du
	\bigg)
	\bigg] 
	\leq
	\E
	\big[
	e^{ \rho \| Y_s \|_H^2 }
	\big]
	. 
	\end{split}
	\end{equation}
	This 
	gives that for all   
	$ s \in [0,T] $,
	$ t \in [s,T] $
	we have that
	\begin{equation}
	\begin{split} 
	&
	\E\bigg[
	\exp\!
	\bigg(
	\rho
	e^{
		- 2 (
		b + \rho \| B \|_{\HS(U,H)}^2)
		( t - s )
	}
	\| X^{ Y_s }_{s,t} \|_H^2
	+
	2
	\rho
	\int_s^t
	e^{
		- 2 (
		b + \rho \| B \|_{\HS(U,H)}^2)
		( u - s )
	} 
	\| X_{s,u}^{ Y_s } \|_{ H_{ \nicefrac{1}{2} } }^2   
	\, du
	\bigg)
	\bigg] 
	\\
	&
	\leq
	\exp\!\bigg(
	\rho
	( 2 a + \| B \|_{\HS(U, H)}^2 ) 
	\int_s^t
	e^{
		- 2 (
		b + \rho \| B \|_{\HS(U,H)}^2)
		( u - s )
	}
	\, du
	\bigg)
	\E
	\big[
	e^{ \rho \| Y_s \|_H^2 }
	\big]
	. 
	\end{split}
	\end{equation}
	Therefore, we obtain that for all 
	$ s \in [0,T] $,
	$ t \in [s,T] $
	we have 
	that
	\begin{equation}
	\begin{split} 
	&
	\E\bigg[
	\exp\!
	\bigg(
	2
	\rho
	e^{
		- 2 (
		b + \rho \| B \|_{\HS(U,H)}^2)
		T
	}
	\int_s^t 
	\| X_{s,u}^{ Y_s } \|_{ H_{ \nicefrac{1}{2} } }^2  
	\, du
	\bigg)
	\bigg] 
	\\
	&
	\leq
	\exp\!
	\big(  
	\rho
	( 2a + \| B \|_{\HS(U,H)}^2 )
	(t-s)
	\big)
	\E
	\big[
	e^{ \rho \| Y_s \|_H^2 }
	\big]
	. 
	\end{split}
	\end{equation}
	The assumption that
	$ p \varepsilon \leq 2 \rho 
	\exp( - 2 (  b + \rho \| B \|_{\HS(U,H)}^2 ) T ) $
	and~\eqref{eq:first_estimate1}
	hence
	illustrate that
	for all  
	$ s \in [0,T] $,
	$ t \in [s,T] $
	we have that
	\begin{equation}
	\begin{split} 
\E \Big[ 
	\big \| \tfrac{\partial}{\partial x} 
	X_{s,t}^{ Y_s }  
	\big \|_{L(H)}^p
	\Big] 
	&
	\leq
	e^{ p { \bf C} ( t - s ) }
	\E\bigg[
	\exp\!
	\bigg( 
	2
	\rho
	e^{
		- 2 (
		b + \rho \| B \|_{\HS(U,H)}^2)
		T
	}
	\int_s^t 
	\| X_{s,u}^{ Y_s } \|_{ H_{ \nicefrac{1}{2} } }^2  
	\, du
	\bigg)
	\bigg] 
	\\
	&
	\leq
	\exp\!
	\big(  
	p { \bf C} ( t - s )
	+ 
		\rho
		( 2a + \| B \|_{\HS(U,H)}^2 )  
		(t-s)
	\big)
	\E
	\big[
	e^{ \rho \| Y_s  \|_H^2 }
	\big]
	. 
	\end{split}
	\end{equation}
	The proof of Lemma~\ref{lemma:inheritet}
	is hereby completed.
\end{proof}
\subsection{Strong error estimates for exponential Euler-type approximations}
\label{section:strongApriori}
In this subsection we employ the results from 
Subsections~\ref{subsection:Measure}
and~\ref{section:AprioriBounds}
to
derive 
in
Proposition~\ref{proposition:main_error_estimate}
an upper bound
for the 
strong error
between
the exact solution
of an SODE with additive noise and
given initial value
(see~\eqref{eq:EQ} below)
and its numerical approximation (see~\eqref{eq:some assumption} below).
\begin{proposition}
	\label{proposition:main_error_estimate}
	\sloppy 
	Assume Setting~\ref{setting:main},
	assume that $ \dim(H) < \infty $,
	let 
	$ T \in (0,\infty) $,  
	$ \theta \in \varpi_T $,
	$ a, b, { \bf C}, \rho \in [0,\infty) $,   
	$ C, c, p \in [1,\infty) $,  
	$ \gamma \in [0,1) $,
		$ \delta \in [0, \gamma ] $,  
	$ \kappa \in \R $,  
		$ B \in \HS(U,H) $,     
	$  \varepsilon \in [0, ( \nicefrac{ \rho }{ p } ) 
	\exp( - 2 ( b + \rho \| B \|_{\HS(U,H)}^2 ) T ) ] $,
	$ F \in \mathcal{C}^1( H, H ) $,
	$ \ff \in \M( \B(H), \B(H) ) $,
	$ \Phi \in \mathcal{C}( H, [0, \infty) ) $,
	assume for all 
	$ x, y \in H $ 
	that 
		$ \langle x, F(x) \rangle_H  
	\leq
	a + b \| x \|_H^2 $, 
	$ \langle F'(x) y, y \rangle_H \leq 
	( \varepsilon \| x \|_{H_{\nicefrac{1}{2}}}^2 + { \bf C}  ) \| y \|_H^2 
	+ \|y \|_{H_{\nicefrac{1}{2}}}^2 $,
	$
	\| F( x )- F( y )  \|_H 
	\leq 
	C  \| x - y \|_{H_\delta} ( 1 + \| x \|_{H_\kappa}^c + \| y \|_{H_\kappa}^c ) $,
	and
	$ \langle x, Ax + F( x + y ) \rangle_H
	\leq \Phi(y) ( 1 + \| x \|_H^2 ) $,
	let $ ( \Omega, \F, \P ) $
	be a probability space with a normal filtration
	$ ( \f_t )_{t \in [0,T]} $,
	let
	$ \xi \in \M( \f_0, \B(H) ) $,
	let $ (W_t)_{ t \in [0,T]} $
	be an
	$ \operatorname{Id}_U $-cylindrical
	$ ( \f_t )_{t\in [0,T]} $-Wiener process, 
		let 
	$ O \colon [0,T] \times \Omega \to H $
	be a
	stochastic process 
	w.c.s.p.\ 
	which satisfies
	for all
	$ t \in [0,T] $ 
	that 
	$ [ O_t ]_{\P, \mathcal{B}(H ) } = \int_0^t e^{(t-u)A} B \, dW_u $,
	and  
	let 
	$ \y \colon [0,T] \times \Omega \to H $ 
	and
	$ \O \colon [0,T] \times \Omega \to H $ 
	be  
	$ ( \f_t )_{ t \in [0,T] } $-adapted stochastic processes
	w.c.s.p.\ 
	which satisfy
	for all $ t \in [0,T] $ that
	\begin{equation} 
	\label{eq:some assumption}
	\P \bigg(
	\y_t
	=
	e^{t A}
	\xi
	+
	\int_0^t
	e^{ ( t - \llcorner u \lrcorner_\theta ) A } 
	\ff ( \y_{ \llcorner u \lrcorner_\theta } ) \, du
	+
	\O_t
	\bigg) = 1
	.
	\end{equation} 
	Then
	\begin{enumerate}[(i)]
		\item \label{item:Existence} 
		there exists a unique stochastic process
		$ X \colon [0, T] \times \Omega \to H $
		w.c.s.p.\ 
		which satisfies for all
		$ t \in [0,T] $
		that
	\begin{equation}
	\label{eq:EQ}
	X_{ t} 
	= 
	e^{ t A } \xi
	+
	\int_0^t e^{ ( t - u ) A } F( X_u  ) \, du + O_t
	,
	\end{equation}
	\item \label{item:Adaptedness}
	we have that
$ X $ 
	is $ ( \f_t )_{ t \in [0,T] } $-adapted,
	and
	\item \label{item:EstimateCrucial} 
	we have
		for all
		$ t \in [0,T] $ that  
			\begin{equation}
		\begin{split} 
		\label{eq:Needs}
		& 
		\|  \y_t - X_t \|_{ \L^p( \P; H ) } 
		\leq
		\sup\nolimits_{s \in [0,T]}
		\| \O_s - O_s \|_{ \L^p( \P; H ) } 
		\\
		&
			\quad
		+ 
		\tfrac{ C [ \max \{ T, 1 \} ]^2 }{ 1 - \gamma }
		\exp\!
		\big(  
		\big(
		{ \bf C} 
		+   
		\rho
		( 2a + \| B \|_{\HS(U,H)}^2 ) 
		\big)
		t
		\big)
		\bigg[ 
		\int_0^t 
		\E
		\big[
		e^{ \rho \|  \y_s -  \O_s + 
			O_s 
\|_H^2 
		}
		\big]
		\, ds
		\bigg]
		\\
		&
		\quad
		\cdot
		\Big\{
		[| \theta |_T]^{\gamma - \delta}
		\sup\nolimits_{s\in [0,T]}
		\| \ff (  \y_s ) \|_{ \L^{2p} (\P; H_{\gamma-\delta} )}
		+
		\sup\nolimits_{s\in [0,T]}
		\| \ff ( \y_s ) - F ( \y_s ) \|_{ \L^{ 2 p } ( \P; H ) }
		\\
		&
		\quad
		+
		\Big(   
		2
		[ | \theta |_T ]^{\gamma - \delta} 
		\sup\nolimits_{s\in [0,T]}
		\|
		\ff ( \y_s )
		\|_{\L^{4p}(\P; H)}
		+
		\sup\nolimits_{s\in [0,T]}
		\|  \O_s -  \O_{ \llcorner s \lrcorner_\theta } \|_{\L^{4p}(\P; H_\delta)}
		\\
		&
		\quad
		+
		[ | \theta |_T ]^{ \gamma - \delta }
		\| \xi \|_{\L^{4p}(\P; H_\gamma)}
		+
		\sup\nolimits_{s\in [0,T]}
		\|  \O_s - O_s \|_{\L^{4p}(\P; H_\delta)}
		\Big)
		\\
		&
			\quad
		\cdot
		\big[
		1
		+
		2 
		\sup\nolimits_{s\in [0,T]}
		\|
	 	 \y_s
		\|_{\L^{4p c}(\P; H_\kappa)}
		+
		\sup\nolimits_{s\in [0,T]}
		\|   \O_s - O_s  \|_{\L^{4p c}(\P; H_\kappa)}
		\big]^c
		\Big\}
		.  
		\!\!\!\!\!
		\end{split}
		\end{equation}
	\end{enumerate}
\end{proposition}
\begin{proof}[Proof of Proposition~\ref{proposition:main_error_estimate}]
	Throughout this proof  
	let
	$ \Sigma \subseteq \Omega $ 
	be the set which satisfies that
	\begin{equation} 
	\begin{split}
	\label{eq:NewY}
	\Sigma =
	&
	\bigg\{ \omega \in \Omega \colon 
	\bigg( 
	\forall \, t \in [0, T] \colon 
	\y_t ( \omega ) 
	=
	e^{t A}
	\xi ( \omega )
	+
	\int_0^t
	e^{ ( t - \llcorner u \lrcorner_\theta ) A } 
	\ff ( \y_{ \llcorner u \lrcorner_\theta } (\omega ) ) \, du
	+
	\O_t(\omega)
	\bigg)
	\bigg\} 
	,
	\end{split}
	\end{equation}
	let 
	$ \mathcal{Y} \colon 
	[0,T] \times \Omega \to H $
satisfy for all 
	$ t \in [0,T] $,
	$ \omega \in \Omega $  
	that
	\begin{equation}
	\label{eq:NewY1}
	\mathcal{Y}_t (\omega) =
	\begin{cases}
	\y_t(\omega) & \colon \omega \in \Sigma \\
	0 & \colon \omega \in ( \Omega \backslash \Sigma )
	,
	\end{cases}
	\end{equation}
	and let
	$ \mathcal{O} \colon 
	[0,T] \times \Omega \to H $
satisfy for all
	$ t \in [0,T] $,
	$ \omega \in \Omega $  
	that
	\begin{equation}
	\label{eq:NewY2}
	\mathcal{O}_t( \omega ) 
	= 
	\begin{cases} 
	 \O_t( \omega ) 
	& \colon \omega \in \Sigma
	\\
	-e^{tA} \xi(\omega)
	- 
	\int_0^t e^{(t-\llcorner u \lrcorner_\theta)A} 
	\ff (0) \, d u 
	&\colon 
	\omega \in ( \Omega \backslash \Sigma )
	.
	\end{cases}
	\end{equation}
Note that
the assumption that for all $ x, y \in H $ we have that
\begin{equation}
\label{eq:local Lipschitz} 
\| F( x )- F( y )  \|_H 
\leq 
C  \| x - y \|_{H_\delta} ( 1 + \| x \|_{H_\kappa}^c + \| y \|_{H_\kappa}^c ), 
\end{equation} 
the assumption that for all $ x, y \in H $ we have that
\begin{equation} 
\label{eq:Localy monotone}
\langle x, Ax + F( x + y ) \rangle_H
\leq \Phi(y) ( 1 + \| x \|_H^2 ),
\end{equation} 
and, e.g., \cite[Corollary~2.4]{JentzenLindnerPusnik2017c}
	(applies with
	$ H = H $, 
	$ \H = \H $,
	$ \values = \values $,
	$ A = A $,
	$ T = T $,  
	$ s = 0 $,
	$ C = C $, 
	$ c = c $, 
	$ \delta = \delta $, 
	$ \kappa = \kappa $,
	$ F = F $,
	$ \Phi = \Phi $,
	$ ( \Omega, \F, \P, ( \f_t )_{ t \in [0,T ] } ) = ( \Omega, \F, \P, ( \f_t )_{ t \in [0,T ] } ) $,
	$ \xi = \xi + O_0 $,
	$ O = O $
	in the setting of~\cite[Corollary~2.4]{JentzenLindnerPusnik2017c})
	justify items~\eqref{item:Existence} and~\eqref{item:Adaptedness}.
	In the next step we are going to use
	Corollary~\ref{corollary:main_error_estimate}
	and
	Lemma~\ref{lemma:inheritet}
	to verify~\eqref{eq:Needs}.
	For this observe that~\eqref{eq:local Lipschitz},
	\eqref{eq:Localy monotone}, 
	and, e.g.,
	\cite[Corollary~2.4]{JentzenLindnerPusnik2017c}
	(applies with
	$ H = H $, 
	$ \H = \H $,
	$ \values = \values $,
	$ A = A $,
	$ T = T $,  
	$ s = s $,
	$ C = C $,   
	$ c = c $,
	$ \delta = \delta $, 
	$ \kappa = \kappa $,
	$ F = F $,
	$ \Phi = \Phi $,
	$ ( \Omega, \F, \P, ( \f_t )_{ t \in [0,T ] } ) = ( \Omega, \F, \P, ( \f_t )_{ t \in [0,T ] } ) $,
	$ \xi = ( \Omega \ni \omega \mapsto x \in H ) $,
	$ O = O $
	for 
	$ s \in [0,T] $,
	$ x \in H $
	in the setting of~\cite[Corollary~2.4]{JentzenLindnerPusnik2017c})
	illustrate that
	there exist stochastic processes
	$ \mathcal{X}_{s,(\cdot)}^x
	=
	(\mathcal{X}_{s,t}^x)_{ t \in [s,T] }
	\colon 
	[s,T] \times \Omega
	\to H $,
	$ s \in [0,T] $,
	$ x \in H $,
	w.c.s.p.\ 
	which satisfy for all 
	$ s \in [0,T] $,
	$ t \in [s,T] $,
	$ x \in H $
	that
	$ \mathcal{X}_{s,(\cdot)}^x $
	is 
	$ ( \f_u )_{ u \in [s,T]} $-adapted
	and
		\begin{equation}
		\label{eq:omegawise_0}
		\mathcal{X}_{s,t}^x
		= 
		e^{ ( t - s ) A } x
		+
		\int_s^t e^{ ( t - u ) A } 
		F ( \mathcal{X}_{s,u}^x ) \, du 
		+ 
		O_t - e^{(t-s)A} O_s 
		.
		\end{equation}
		Moreover, note that~\eqref{eq:some assumption}
	and the fact that
	$ \y $ and $ \O $
	are 
	stochastic processes 
	w.c.s.p.\ 
	ensure that
	\begin{equation}
	\Sigma \in \F \qquad \text{and} \qquad \P(\Sigma)=1 .
	\end{equation}
	The fact that $ ( \f_t )_{t \in [0,T]} $ 
	is a normal filtration 
	and the fact that
	$ \y $ and $ \O $
	are 
	$ ( \f_t)_{ t \in [0,T] } $-adapted 
	therefore gives that
	\begin{enumerate}[(a)] 
		\item \label{item:Y adapted} we have that 
		$ \mathcal{Y} $ is 	$ ( \f_t )_{ t \in [0,T]} $-adapted, 
			\item\label{item:O adapted} we have that
			$ \mathcal{O} $ is $ ( \f_t )_{ t \in [0,T]} $-adapted,
		\item \label{item:negligible} we have for all $ t \in [0, T] $
		that
		$ \P( \mathcal{Y}_t = \y_t ) = 1 $, and
		\item  \label{item:negligible2} we have for all $ t \in [0, T] $
		that
		$ \P( \mathcal{O}_t = \O_t ) = 1 $.
	\end{enumerate} 
	In addition, note that~\eqref{eq:omegawise_0}
	gives that for all 
	$ t \in [0,T] $,
	$ \omega \in \Omega $
	we have that
	\begin{equation}
	\label{eq:omegawise_0b}
	\mathcal{X}_{0,t}^{ \xi(\omega) + O_0(\omega) }
	(\omega)
	= 
	e^{ t A }  \xi(\omega) 
	+
	\int_0^t e^{ ( t - u ) A } 
	F \big( \mathcal{X}_{0,u}^{ \xi(\omega) + O_0(\omega) }(\omega) \big) \, du 
	+ 
	O_t(\omega) 
	.
	\end{equation}
	Furthermore,
	observe that item~\eqref{item:Existence}
	ensures that for all 
	$ t \in [0,T] $,
	$ \omega \in \Omega $ we have that
	\begin{equation}
	\label{eq:omegawise_0c}
	X_t
	(\omega)
	= 
	e^{ t A } \xi(\omega) 
	+
	\int_0^t e^{ ( t - u ) A } F( X_u (\omega)   ) \, du + O_t (\omega) 
	.
	\end{equation}
	Combining this, 
	\eqref{eq:omegawise_0b},
	and, e.g., \cite[item~(i)
	of
	Corollary~2.4]{JentzenLindnerPusnik2017c}	  %
	(applies with
	$ H = H $, 
	$ \H = \H $,
	$ \values = \values $,
	$ A = A $,
	$ T = T $,  
	$ s = 0 $,
	$ C = C $,
	$ c = c $,  
	$ \delta = \delta $, 
	$ \kappa = \kappa $,
	$ F = F $,
	$ \Phi = \Phi $,
	$ ( \Omega, \F, \P, ( \f_t )_{ t \in [0,T ] } ) = 
	( \Omega, \F, \P, ( \f_t )_{ t \in [0,T ] } ) $,
	$ \xi = \xi + O_0 $,
	$ O = O $ 
	in the setting of~\cite[item~(i)
	of
	Corollary~2.4]{JentzenLindnerPusnik2017c})
	yields that for all
	$ t \in [0,T] $, 
	$ \omega \in \Omega $
	we have that
	\begin{equation} 
	\label{eq:Now relevant}
	X_t (\omega) 
	= 
	\mathcal{X}_{0,t}^{ \xi(\omega) + O_0(\omega) } ( \omega)	
	.
	\end{equation} 
	Moreover, 
	observe that~\eqref{eq:NewY}--\eqref{eq:NewY2}
	verify that
	for all
	$ t \in [0,T] $
	we have that
	\begin{equation}
	\label{eq:omegawise_2}
	\mathcal{Y}_t
	=
	e^{t A}
	\xi
	+
	\int_0^t
	e^{ ( t - \llcorner u \lrcorner_\theta ) A } 
	\ff ( \mathcal{Y}_{ \llcorner u \lrcorner_\theta } ) \, du
	+
	\mathcal{O}_t
	.
	\end{equation} 
Combining item~\eqref{item:negligible},
\eqref{eq:local Lipschitz},
	\eqref{eq:omegawise_0},
	\eqref{eq:Now relevant}, 
	and
	Corollary~\ref{corollary:main_error_estimate} 
	(applies with
	$ ( \Omega, \F, \P) = ( \Omega, \F, \P) $, 
	$ T = T $, 
	$ \theta = \theta $,
	$ C = C $, 
	$ c = c $,
	$ p = p $,
	$ \gamma = \gamma $, 
	$ \delta = \delta $,
	$ \iota = \gamma - \delta $,
	$ \kappa = \kappa $,
	$ \xi = \xi $,
	$ F = F $,
	$ \ff = \ff $,
	$ \O_s = \mathcal{O}_s $, 
	$ O_s = O_s $,
	$ X^x_{s,t} = \mathcal{X}^x_{s,t} $, 
	$ \y_s = \mathcal{Y}_s $,
	$ \zeta = \xi $
	for
	$ t \in [s,T] $,
	$ s \in [0,T] $,
	$ x \in H $
	in the setting of Corollary~\ref{corollary:main_error_estimate}) 
	therefore 
    justifies
	that 
	\begin{enumerate}[(A)]
	\item 
	we have for all
	$ s \in [0,T] $,
	$ t \in [s,T] $,
	$ \omega \in \Omega $
	that 
	$ H \ni x \mapsto \mathcal{X}_{s,t}^x(\omega)\in H 
	 $ is differentiable,
	\item we have for all $ t \in [0, T] $ that
	$
	\big( \Omega \ni  \omega \mapsto 
	\mathcal{X}_{0,t}^{ 
	\xi(   \omega)	 
	+ 
O_0(  \omega )  } (   \omega ) 
	\in H \big)
	\in \M( \F, \B(H) )
	$,
	\item we have for all $ t \in [0, T] $ that
	$
	 \big( [0,t] \times \Omega 
	\ni (s, \omega)
	\mapsto 
	\frac{ \partial }{ \partial x }
	\mathcal{X}_{s, t}^{ 
	\mathcal{Y}_s (\omega) 
	- 
	\mathcal{O}_s (\omega) + O_s (\omega) 
	}
	(\omega)
	\in L(H) 
	 \big)  
	\in \M( \B( [0, t] ) \otimes \F, \B(L(H)))
	$,
	and 
	\item we have for all $ t \in [0, T] $ that
	\begin{equation}
	\begin{split}
	\label{eq:good_estimate}
	&
	\|  \y_t - X_t\|_{ \L^p( \P; H ) } 
	=
	\|  \mathcal{Y}_t - \mathcal{X}_{0, t}^{ \xi + O_0 } \|_{ \L^p( \P; H ) } 
	\\
	&
	\leq
	\sup\nolimits_{s \in [0,T]}
	\| \mathcal{O}_s - O_s \|_{ \L^p( \P; H ) } 
	+ 
	\tfrac{ C \max \{ T, 1 \} }{ 1 - \gamma }
	\bigg[ 
	\int_0^t 
	\big \| \tfrac{\partial}{\partial x}
	\mathcal{X}_{s, t}^{ 
		\mathcal{Y}_s 
		- 
		\mathcal{O}_s 
		+ 
		O_s
	}
	\big \|_{ \L^{ 2 p } ( \P; L(H) ) }
	\, ds
	\bigg]
	\\
	&
	\quad
	\cdot
	\Big\{ 
	[| \theta |_T]^{\gamma - \delta}
	\sup\nolimits_{s\in [0,T]}
	\| \ff ( \mathcal{Y}_s ) \|_{ \L^{2p} (\P; H_{\gamma-\delta} )}
	+
	\sup\nolimits_{s\in [0,T]}
	\| \ff ( \mathcal{Y}_s ) - F ( \mathcal{Y}_s ) \|_{ \L^{ 2 p } ( \P; H ) }
	\\
	&
	\quad
	+
	\Big(  
	( 
	[ | \theta |_T ]^{1 - \delta}
	+ 
	[ | \theta |_T ]^{\gamma - \delta}
	)
	\sup\nolimits_{s\in [0,T]}
	\|
	\ff ( \mathcal{Y}_s )
	\|_{\L^{4p}(\P; H)}
	+
	\sup\nolimits_{s\in [0,T]}
	\| \mathcal{O}_s - \mathcal{O}_{ \llcorner s \lrcorner_\theta } \|_{\L^{4p}(\P; H_\delta)}
\\
&
\quad
+
[ | \theta |_T ]^{ \gamma - \delta }
\| \xi \|_{\L^{4p}(\P; H_\gamma)}
	+
	\sup\nolimits_{s\in [0,T]}
	\| \mathcal{O}_s - O_s \|_{\L^{4p}(\P; H_\delta)} 
	\Big)
	\\
	&
	\quad
	\cdot
	\big[
	1
	+
	2 
	\sup\nolimits_{s\in [0,T]}
	\|
	\mathcal{Y}_s
	\|_{\L^{4p c}(\P; H_\kappa)}
	+
	\sup\nolimits_{s\in [0,T]}
	\|   \mathcal{O}_s - O_s  \|_{\L^{4 p c}(\P; H_\kappa)} 
	\big]^c
	\Big\}
	.  
	\end{split}
	\end{equation}
	\end{enumerate}
	Moreover, 
	note that~\eqref{eq:omegawise_0}, 
	the fact that
	$ \mathcal{Y} $,
	$ \mathcal{O} $,
	and
	$ O $ 
	are
	$ ( \f_t )_{ t \in [0,T] } $-adapted
	stochastic processes 
	w.c.s.p.,
	the assumption that
	for all $ x,y \in H $ we have that
	$ \langle x, F(x) \rangle_H  
	\leq
	a + b \| x \|_H^2 $  
	and
	$ \langle F'(x) y, y \rangle_H \leq 
	( \varepsilon \| x \|_{H_{\nicefrac{1}{2}}}^2 + { \bf C}  ) \| y \|_H^2 
	+ \|y \|_{H_{\nicefrac{1}{2}}}^2 $, 
	and
	Lemma~\ref{lemma:inheritet}
	(applies with
	$ T = T $, 
	$ a = a $,
	$ b = b $,
	$ { \bf C } = { \bf C } $,
	$ \rho = \rho $,
	$ p = 2 p $,
	$ B = B $, 
	$ \varepsilon = \varepsilon $,
		$ F = F $,
	$ ( \Omega, \F, \P ) = ( \Omega, \F, \P ) $,
	$ ( \f_t )_{ t \in [0,T] }
	=
	( \f_t )_{ t \in [0,T] } $,
	$ (W_t)_{ t \in [0,T] } = ( W_t)_{ t \in [0,T]} $, 
	$ Y_s = \mathcal{Y}_s - \mathcal{O}_s + O_s $, 
	$ O_s = O_s $,
	$ X_{s,u}^x = \mathcal{X}_{s,u}^x $
	for 
	$ u \in [s,T] $,
	$ s \in [0,T] $,
	$ x \in H $
	in the setting of Lemma~\ref{lemma:inheritet}) 
	verify that for all 
	$ s \in [0,T] $,
	$ t \in [s,T] $
	we have that
	\begin{equation}
				\begin{split} 
				\label{eq:der_estimate}
				& 
				\E\Big[ 
				\big \| \tfrac{\partial}{\partial x} 
					\mathcal{X}_{s,t}^{ \mathcal{Y}_s - \mathcal{O}_s + O_s  
					}  
				\big\|_{L(H)}^{2p} 
				\Big]
				\leq
				\exp\!
				\big(  
				\big(
				2 p { \bf C} 
				+   
						\rho
						( 2a + \| B \|_{\HS(U,H)}^2 ) 
				\big)
				t
				\big)
				\E
				\big[
				e^{ \rho \| \mathcal{Y}_s - \mathcal{O}_s + 
					 O_s  
\|_H^2 }
				\big]
				. 
				\end{split}
				\end{equation}
	This and~\eqref{eq:good_estimate}  
	yield that for all
	$ t \in [0,T] $
	we have that
	\begin{equation}
	\begin{split} 
	&
	\|  \y_t - X_t\|_{ \L^p( \P; H ) } 
	\leq
	\sup\nolimits_{s \in [0,T]}
	\| \mathcal{O}_s - O_s \|_{ \L^p( \P; H ) } 
	\\
	& 
	\quad
	+ 
	\tfrac{ C [ \max \{ T, 1 \} ]^2 }{ 1 - \gamma } 
	\exp\!
	\big(  
	\big(
	{ \bf C} 
	+   
	\rho
	( 2a + \| B \|_{\HS(U,H)}^2 ) 
	\big)
	t
	\big)
	\bigg[ 
	\int_0^t 
	\big( 
	\E
	\big[
	e^{ \rho \| \mathcal{Y}_s - \mathcal{O}_s + 
	O_s  
	\|_H^2 }
	\big]
	\big)^{ \nicefrac{1}{(2p)} }
	\, ds
	\bigg]
	\\
	&
	\quad
	\cdot
	\Big\{
	[| \theta |_T]^{\gamma - \delta}
	\sup\nolimits_{s\in [0,T]}
	\| \ff ( \mathcal{Y}_s ) \|_{ \L^{2p} (\P; H_{\gamma-\delta} )}
	+
	\sup\nolimits_{s\in [0,T]}
	\| \ff ( \mathcal{Y}_s ) - F ( \mathcal{Y}_s ) \|_{ \L^{ 2 p } ( \P; H ) }
	\\
	&
	\quad
	+
	\Big(   
	2
	[ | \theta |_T ]^{\gamma - \delta} 
	\sup\nolimits_{s\in [0,T]}
	\|
	\ff ( \mathcal{Y}_s )
	\|_{\L^{4p}(\P; H)}
	+
	\sup\nolimits_{s\in [0,T]}
	\| \mathcal{O}_s - \mathcal{O}_{ \llcorner s \lrcorner_\theta } \|_{\L^{4p}(\P; H_\delta)}
	\\
	&
	\quad
	+
	[ | \theta |_T ]^{ \gamma - \delta }
	\| \xi \|_{\L^{4p}(\P; H_\gamma)} 
	+
	\sup\nolimits_{s\in [0,T]}
	\| \mathcal{O}_s - O_s \|_{\L^{4p}(\P; H_\delta)} 
	\Big)
	\\
	&
		\quad
	\cdot
	\big[
	1
	+
	2 
	\sup\nolimits_{s\in [0,T]}
	\|
	\mathcal{Y}_s
	\|_{\L^{4p c}(\P; H_\kappa)}
	+
	\sup\nolimits_{s\in [0,T]}
	\|   \mathcal{O}_s - O_s  \|_{\L^{4p c}(\P; H_\kappa)} 
	\big]^c
	\Big\}
	.  
	\end{split}
	\end{equation}
	Combining this and items~\eqref{item:negligible} 
	and~\eqref{item:negligible2}
	justifies item~\eqref{item:EstimateCrucial}.
	The proof of Proposition~\ref{proposition:main_error_estimate}
	is hereby completed.
\end{proof}
\section[Strong convergence rates with assuming finite exponential moments]{Strong convergence rates for space-time discrete exponential Euler-type approximations with assuming finite exponential moments}
\label{section:Strong convergence rates}
\subsection{Moment bounds for spatial spectral Galerkin approximations}
\label{section:AprioriBound}
In this subsection
we prove in 
Lemma~\ref{lemma:existence}
suitable 
a priori moment bounds for exact solutions of 
SODEs.
Corollary~\ref{corollary:AprioriExact} 
then establishes
uniform a priori moment bounds
for spectral Galerkin approximations 
of exact solutions of
semilinear SPDEs with additive noise.
\begin{lemma}
	\label{lemma:existence}
	Assume Setting~\ref{setting:main},
	assume that $ \dim(H) < \infty $,
	let
	$ T \in (0,\infty) $,
	$ a, b \in [0,\infty) $, 
	$ p \in [2, \infty) $,
	$ s \in [0,T] $, 
	$ B \in \HS(U,H) $,  
	$ F \in \mathcal{C}( H, H ) $,
	assume  for all
	$ x \in H $   
	that  
	$ \langle x, F(x) \rangle_H  
	\leq
	a + b \| x \|_H^2 $,
	let $ ( \Omega, \F, \P ) $ 
	be a probability space
with a normal filtration $  (\f_t)_{ t \in [0,T] } $,
	let $ (W_t)_{ t \in [0,T]} $
	be an
	$ \operatorname{Id}_U $-cylindrical
	$ (\f_t)_{ t \in [0,T] } $-Wiener process, 
	let
	$ \xi \in \M( \f_s, \mathcal{B}(H) ) $,
	let 
	$ O \colon [0, T] \times \Omega \to H $
	be a stochastic process
	w.c.s.p.\ 
	which satisfies for all $ t \in [0, T] $ that
	$ [ O_t ]_{\P, \mathcal{B}(H ) } 
	= 
	\int_0^t e^{ ( t - u ) A } B \, dW_u $, 
	and
	let
	$ X \colon [s,T] \times \Omega \to H $
	be an
	$ ( \f_t )_{ t \in [0,T] } $-adapted
	 stochastic process 
	 w.c.s.p.\ 
	which satisfies for all   
	$ t \in [s,T] $ that 
	\begin{equation}
	\begin{split}
	\P\bigg(
	X_t = e^{(t-s)A} \xi + \int_s^t e^{(t-u)A} F (  X_u  ) \, du
	+
	O_t - e^{(t-s)A} O_s
	\bigg)
	=
	1
	.
	\end{split} 
	\end{equation}
	Then 
	\begin{equation}
	\begin{split}
	\label{eq:Lp bound}
	\sup\nolimits_{t\in [s,T]}
	\E [  \| X_t \|_H^p ]
	&
	\leq
	\big( 
	\E [ \| \xi \|_H^p ]
	+
	2
	T
	\big[
	a 
	+
	\tfrac{ p - 1 }{ 2 } \| B \|_{\HS(U,H)}^2
	\big]^{\nicefrac{p}{2}}
	\big)
	\exp (
	( pb + p - 2 )T
	)
	.
	\end{split}
	\end{equation}
\end{lemma}
\begin{proof}[Proof of Lemma~\ref{lemma:existence}]
	Throughout this proof let
	$ \mathbb{U} \subseteq U $ be an orthonormal basis of $ U $.
	Note 
	that
	Lemma~\ref{lemma:Ito formula}
	(applies with
	$ T = T $,
	$ s = s $,
	$ B = B $, 
	$ ( \Omega, \F, \P  ) = ( \Omega, \F, \P ) $,
	$ (W_t)_{ t \in [0,T] } = (W_t)_{ t \in [0,T] } $,
	$ \xi = \xi $,
	$ Z_t = F(X_t) $,
	$ Y_t = X_t $,
	$ O = O $
	for  
	$ t \in [s, T] $ 
	in the setting of Lemma~\ref{lemma:Ito formula})
	yields
	that for all $ t \in [s, T] $ we have that
	\begin{equation}
	\begin{split}
	\label{eq:satisfies_deterministic}
	[ X_t ]_{\P, \mathcal{B}(H)} 
	=
	\bigg[
	\xi 
	+ 
	\int_s^t [ A X_u + F ( X_u  ) ] \, du
	\bigg ]_{\P, \mathcal{B}(H)} 
	+
	\int_s^t B \, dW_u 
	.
	\end{split}
	\end{equation}
	Furthermore, 
	observe that
	the fact that 
	$ X $ has 
	continuous sample paths 
	ensures that
	there exist $ ( \f_t )_{ t \in [s,T] } $-stopping times
	$ \tau_r \colon \Omega \to [s,T] $, 
	$ r \in (0, \infty) $, 
	which satisfy
	for all 
	$ r \in (0,\infty) $
	that
	\begin{equation}
	\label{eq:Tau}
	\tau_r = \inf ( \{T\} \cup 
	\{ t \in [s,T] \colon \| X_t \|_H \geq r \} ) 
	.
	\end{equation}
	Note that
	It\^o's formula,
	\eqref{eq:satisfies_deterministic},
	and~\eqref{eq:Tau}  
	illustrate 
	that for all 
	$ r \in (0, \infty) $,
	$ t \in [s,T] $
	we have that
	\begin{equation}
	\begin{split}
	\label{eq:conseqIto}
	&
	[
	\| X_{ \min \{ \tau_r, t \} } \|_H^p
	]_{ \P, \B(\R) } 
	=
	\bigg[
	\| \xi \|_H^p
	+
	\int_s^{ \min \{ \tau_r, t \} }
	p \| X_u \|_H^{p-2} 
	\langle X_u, A X_u + F(X_u) \rangle_H 
	\, du
	\bigg]_{ \P, \B(\R) }
	\\
	&
	\quad
	+
	\int_s^{ t } 
	p
	\1_{ \{  \tau_r \geq u \}}  
	\| X_u \|_H^{p-2} 
	\langle X_u, B \, d W_u \rangle_H 
	\\
	&
	\quad
	+
	\bigg[
	\tfrac{1}{2} 
	\int_s^{ \min \{ \tau_r, t \} } 
	\sum_{ { \bf u } \in \mathbb{U} }
	\big[
	p \| X_u \|_H^{p-2} \| B { \bf u } \|_H^2 
	+
	p ( p-2) \1_{ \{ X_u \neq 0 \} }
	\| X_u \|_H^{p-4}
	| \langle X_u, B { \bf u } \rangle_H |^2 
	\big]
	\,
	du 
	\bigg]_{ \P, \B(\R) }
	\\
	&
	\leq
	\bigg[
	\| \xi \|_H^p
	+
	\int_s^{ \min \{ \tau_r, t \} }
	p \| X_u \|_H^{p-2} 
	\langle X_u, A X_u + F(X_u) \rangle_H 
	\, du
	\bigg]_{ \P, \B(\R) }
	\\
	&
	\quad
	+
	\int_s^t 
	p
		\1_{ \{  \tau_r \geq u \}} 
	\| X_u \|_H^{p-2} 
	\langle X_u, B \, d W_u \rangle_H 
	+
	\bigg[
	\tfrac{p(p-1)}{2}
	\|B \|_{ \HS(U,H) }^2  
	\int_s^{ \min \{ \tau_r, t \} }  
	\| X_u \|_H^{p-2}  
	\,
	du
	\bigg]_{ \P, \B(\R) }
	.
	\end{split}
	\end{equation}
	Moreover, observe that 
	for all 
	$ r \in (0,\infty) $, 
	$ t \in [s,T] $ 
	we have that
	\begin{equation}
	\begin{split}
	&
	\int_s^t
	\1_{ \{ \tau_r \geq u \} }
	\| X_u \|_H^{2(p-2)} 
	\| ( U \ni v \mapsto \langle X_u, B( v ) \rangle_H \in \R ) \|_{\HS(U,\R)}^2
	\, d u
\\
&
	\leq
	\int_s^t
	\1_{ \{ \tau_r \geq u \} } 
	\| X_u \|_H^{2(p-1)}
	\| B \|_{\HS(U,H)}^2
	\, d u 
\\
&
\leq 
	\int_s^t
	r^{2(p-1)} 
	\| B \|_{\HS(U,H)}^2
	\, d u 
	\leq 
	\int_0^T
	r^{2(p-1)} 
	\| B \|_{\HS(U,H)}^2
	\, d u 
	< \infty .
	\end{split}
	\end{equation}
	Combining this, the assumption that
	for all $ x \in H $
	we have that
	$ \langle x, F(x) \rangle_H \leq a + b \| x \|_H^2 $, 
	\eqref{eq:conseqIto}, 
	Tonelli's theorem,
	and 
	Young's inequality  
	verifies
	that
	for all 
	$ r \in (0,\infty ) $, 
	$ t \in [s,T] $ 
	we have that
	\begin{equation}
	\begin{split}
	&
	\E[ \| \1_{ \{ \tau_r \geq t \} } X_t \|_H^p ]
	\leq
	\E [ ( \| \1_{ \{ \tau_r \geq t \} } X_{ \min \{ \tau_r, t \} } \|_H
	+
	\| \1_{ \{ \tau_r < t \} } X_{ \min \{ \tau_r, t \} } \|_H )^p ]
	=
	\E[ \| X_{ \min \{ \tau_r, t \} } \|_H^p ]
	\\
	&
	\leq
	\E[ \|  
	\xi  \|_H^p ]
	+
	p
	\E \bigg[
	\int_s^{ \min \{ \tau_r, t \} } 
	\|  
	X_u \|_H^{p-2} 
	\big( 
	\langle  X_u, A  X_u + F( X_u ) \rangle_H  
	+
	\tfrac{p-1}{2} \| B \|_{\HS(U,H)}^2
	\big)
	\, du 
	\bigg]
	\\
	&
	\leq
	\E[ \|  
	\xi  \|_H^p ]
	+	 
	p
	\E \bigg[
	\int_s^{ \min \{ \tau_r, t \} } 
	\|  
	X_u \|_H^{p-2} 
	\big(
	a + b \| X_u \|_H^2   
	+
	\tfrac{p-1}{2} \| B \|_{\HS(U,H)}^2
	\big)
	\, du
	\bigg]	
	\\
	&
	=
	\E[ \|  
	\xi  \|_H^p ]
	+	 
	p
	\E 
	\bigg[
	\int_s^t 
	\1_{ \{ \tau_r \geq u \} }
	\|  
	X_u \|_H^{p-2} 
	\big(
	a + b \| X_u \|_H^2   
	+
	\tfrac{p-1}{2} \| B \|_{\HS(U,H)}^2
	\big)
	\, du
	\bigg]
	\\
	&
	=
	\E[ \|  
	\xi  \|_H^p ]
	+	 
	p
	\int_s^t 
	\E \big[ 
	\1_{ \{ \tau_r \geq u \} }
	\|  X_u  \|_H^{p-2} 
	\big( 
	a
	+
	\tfrac{p-1}{2} \| B \|_{\HS(U,H)}^2
	\big)
	+
	b 
	\1_{ \{ \tau_r \geq u \} }
	\| X_u  \|_H^p
	\big]
	\, du
	\\
	&
	\leq
	\E[ \|  
	\xi  \|_H^p ]
	+	 
	p
	\int_s^t 
	\E\big[
	\tfrac{p-2}{p}
	\1_{ \{ \tau_r \geq u \} }
	\| X_u \|_H^p 
	+
	\tfrac{2}{p}
	\big( 
	a + \tfrac{p-1}{2} \| B \|_{\HS(U,H)}^2 
	\big)^{\nicefrac{p}{2}}
	+
	b
	\1_{ \{ \tau_r \geq u \} }
	\| X_u \|_H^p 
	\big]
	\, du 
	\\
	&
	=  
	\E[ \|  
	\xi  \|_H^p ]
	+
	(pb + p - 2)
	\int_s^t 
	\E [ 
	\1_{ \{ \tau_r \geq u \} }
	\| X_u \|_H^p  
	]
	\, du
	+
	2 (t-s) 
	\big(  
	a + \tfrac{p-1}{2} \| B \|_{\HS(U,H)}^2 
	\big)^{\nicefrac{p}{2}} 
	\\
	&
	\leq 
	\E[ \| 
	\xi  \|_H^p ]
	+
	(pb + p - 2)
	(t-s)r^p
	+
	2 (t-s)
	\big(
	a
	+
	\tfrac{ p-1 }{2} \| B \|_{\HS(U,H)}^2
	\big)^{\nicefrac{p}{2}}
	.
	\end{split}
	\end{equation}
	Gronwall's lemma therefore yields that for all
	$ r \in (0, \infty) $,
	$ t \in [s,T] $ we have that
	\begin{equation}
	\begin{split}
	\label{eq:GrangeGronwallEstimate}
	&
	\E [ \1_{ \{ \tau_r \geq t \} } \| X_t \|_H^p ]
	\leq
	\big( 
	\E [ \| \xi \|_H^p ]
	+
	2 ( t - s )
	\big[
	a
	+
	\tfrac{ p - 1 }{2} \| B \|_{\HS(U,H)}^2
	\big]^{\nicefrac{p}{2}}
	\big)
	\exp  (
	( pb + p - 2 )( t - s ) 
	)
	.
	\end{split}
	\end{equation}
	The fact that
	for all $ n \in \N $, $ t \in [0,T] $
	we have that
	$ \1_{ \{ \tau_n \geq t \} } 
	\leq
	\1_{ \{ \tau_{n+1} \geq t \} } $
	and the
	monotone convergence theorem
	hence justify~\eqref{eq:Lp bound}.
	The proof of
	Lemma~\ref{lemma:existence}
	is hereby completed.
\end{proof} 
\begin{corollary}
	\label{corollary:AprioriExact}
	Assume Setting~\ref{setting:main},
	let
	$ T \in (0,\infty) $,
	$ a, b \in [0, \infty) $,  
	$ p \in [1, \infty) $, 
		$ \beta \in [0, \nicefrac{1}{2}) $,
	$ \gamma, \eta_1 \in [0, \nicefrac{1}{2} + \beta ) $, 
	$ \eta_2 \in [\eta_1, \nicefrac{1}{2} + \beta ) $,
	$ \iota \in [ \eta_2, \nicefrac{1}{2}  + \beta  ) $, 
	$ \alpha_1 \in [0, 1 - \eta_1 ) $, 
	$ \alpha_2 \in [0, 1 - \eta_2 ) $,
	$ B \in \HS(U, H_\beta ) $,  
	$ F \in \mathcal{C}(  H_\gamma, H ) $,    
	$ (P_I)_{ I \in \mathcal{P}(\H) } \subseteq L(H) $ 
	satisfy for all 
	$ I \in \mathcal{P} (\H) $,
	$ x \in H $ 
	that
	$ P_I(x)
	= \sum_{h \in I} \langle h, x \rangle_H h $,
	assume for all 
	$ I \in \mathcal{P}_0 (\H) $,
	$ x \in P_I(H) $ 
	that
	$ \langle x, F(x) \rangle_H  
	\leq
	a + b \| x \|_H^2 $
	and
	\begin{equation}
	\label{assumption:quadratic growth}
	\Big[
	\sup\nolimits_{ v \in H_{ \max \{ \gamma, \eta_2 \} }   }
	\tfrac{ \| F(v) \|_{H  } }
	{ 1 + \| v \|_{H_{ \eta_2 } }^2 }
	\Big] 
	+
	\Big[
	\sup\nolimits_{ v \in H_{ \max \{ \gamma, \eta_1 \} }  }
	\tfrac{ \| F(v) \|_{H_{ - \alpha_2 } } }
	{ 1 + \| v \|_{ H_{\eta_1} }^2 }
	\Big]
	+
	\Big[
	\sup\nolimits_{ v \in H_{\gamma}  } 
	\tfrac{ \| F(v) \|_{H_{-\alpha_1} } }{ 1 + \| v \|_H^2 } 
	\Big]
	< \infty
	,
	\end{equation}
	\sloppy 
	let
	$ ( \Omega, \F, \P ) $
	be a probability space
	with a normal filtration $ (\f_t)_{ t \in [0,T] } $,
	let $ (W_t)_{ t \in [0,T]} $
	be an
	$ \operatorname{Id}_U $-cylindrical
	$ ( \f_t )_{ t \in [0,T] } $-Wiener process,
		let
	$ \xi \in \L^{4p}(\P|_{\f_0}; H_{ \iota } ) $
	satisfy 
	$ \E[ \| \xi \|_H^{8p} ] < \infty $, 
	and
	let 
	$ X^I \colon [0,T] \times \Omega \to P_I(H) $,
	$ I \in \mathcal{P}_0 (\H) $,
	and
	$ O^I \colon [0,T] \times \Omega \to P_I(H) $,
	$ I \in \mathcal{P}_0 (\H) $,
	be 
		$ ( \f_t )_{ t \in [0,T] } $-adapted
	stochastic processes 
	w.c.s.p.\ 
	which satisfy for all
	$ I \in \mathcal{P}_0 (\H) $,
	$ t \in [0,T] $ 
	that
	$ [ O_t^I ]_{\P, \mathcal{B}( P_I( H ) ) } = \int_0^t e^{(t-s)A} P_I B \, dW_s $
	and
	\begin{equation}
	X_t^I
	= 
	e^{tA} P_I \xi  
	+ 
	\int_0^t e^{(t-s)A} P_I F(X_s^I)
	\,
	ds
	+
	O_t^I 
	.
	\end{equation}
	Then 
	\begin{equation}
	\label{eq:Uniform Lp estimate}
	\sup\nolimits_{ I \in \mathcal{P}_0 (\H) }
	\sup\nolimits_{ t \in [0,T] }
	\| X_t^I \|_{ \L^p(\P; H_\iota) } < \infty .
	\end{equation}
\end{corollary}
\begin{proof}[Proof of Corollary~\ref{corollary:AprioriExact}]
Throughout this proof let
$ \mathcal{A}_I \colon P_I(H) \to P_I(H) $, $ I \in \mathcal{P}_0(\H) $,
satisfy for all
$ I \in \mathcal{P}_0(\H) $,
$ v \in P_I(H) $ that
$ \mathcal{A}_I v = A v $
and for every 
$ I \in \mathcal{P}_0(\H) $
let
$ ( \mathcal{H}_{I, r}, \langle \cdot, \cdot \rangle_{ \mathcal{H}_{I, r} }, \left \| \cdot \right\|_{ \mathcal{H}_{I, r} }) $, $ r \in \R $,
be a family of interpolation spaces associated to $ - \mathcal{A}_I $.
Note that the
Burkholder-Davis-Gundy-type inequality in Da Prato \& Zabczyk~\cite[Lemma~7.7]{dz92}
verifies
that
for all 
$ t \in [0,T] $,  
$ q \in [2, \infty) $
we have that
\begin{equation}
\begin{split}
\label{eq:Condition1}
&
\sup\nolimits_{I \in \mathcal{P}_0 (\H) }
\| O_t^I \|_{ \L^q(\P; H_\iota) }^2
\leq
\tfrac{q(q-1)}{2}
\sup\nolimits_{I \in \mathcal{P}_0 (\H) }
\int_0^t 
\| e^{(t-s)A} P_I B \|_{ \HS(U,H_\iota)}^2 \, ds
\\
&
\leq
\tfrac{q(q-1)}{2}
\int_0^t 
\| 
(-A)^{ \iota - \beta } e^{(t-s)A} \|_{L(H)}^2 
\| B \|_{ \HS(U,H_\beta)}^2
\, ds
\leq
\tfrac{q(q-1)}{2} 
\int_0^t 
(t-s)^{ 2 \beta - 2 \iota }
\| B \|_{ \HS(U,H_\beta)}^2
\, ds
\\
&
\leq
\tfrac{q(q-1)}{2} 
\tfrac{ t^{ 1 + 2 \beta - 2 \iota } }{ 1 + 2 \beta - 2 \iota }
\| B \|_{ \HS(U,H_\beta)}^2
< \infty
.
\end{split}
\end{equation}
Next observe that
the fact 
that
$ \xi \in \L^{8p}( \P|_{\f_0}; H ) $,
the assumption that
for all
	$ I \in \mathcal{P}_0 (\H) $,
$ x \in P_I(H) $ 
we have 
that
$ \langle x, F(x) \rangle_H  
\leq
a + b \| x \|_H^2 $,
and
Lemma~\ref{lemma:existence} 
(applies with
$ H = P_I(H) $,
$ \mathbb{H} = P_I( \mathbb{H} ) $,
$ \values = ( I \ni h \mapsto \values_h \in \R )$,
$ A = \mathcal{A}_I $,
$ ( H_s )_{ s \in \R } = ( \mathcal{H}_{I, s} )_{ s \in \R } $,
$ T = T $,
$ a = a $,
$ b = b $, 
$ p = 8 p $,
$ s = 0 $,
$ B = ( U \ni u \mapsto  P_I B(u) \in P_I(H) ) $,
$ F = ( P_I( H ) \ni x \mapsto P_I F(x) \in P_I(H) ) $,
$ ( \Omega, \F, \P ) 
=
( \Omega, \F, \P ) $,
$ (\f_t)_{ t \in [0,T] } = (\f_t)_{ t \in [0,T] } $,
$ (W_t)_{t \in [0,T]} = ( W_t )_{ t \in [0,T] } $,
$ \xi = ( \Omega \ni \omega \mapsto P_I \xi (\omega) \in P_I(H) ) $,
$ O = O^I $,
$ X = X^I $
for
$ I \in ( \mathcal{P}_0 (\H) \backslash \{ \emptyset \} ) $
in the setting of 
Lemma~\ref{lemma:existence})
give that
\begin{equation}
\label{eq:Condition2}
\sup\nolimits_{I \in \mathcal{P}_0 (\H) } \sup\nolimits_{ t \in [0,T] } 
\| X_t^I
\|_{\L^{8 p}(\P; H ) } 
<
\infty 
.
\end{equation}
Combining
the assumption that
$ \xi \in \L^{4p}( \P |_{\f_0}; H_\iota ) $,
\eqref{assumption:quadratic growth},
and~\eqref{eq:Condition1} 
with, e.g., \cite[Lemma~3.4]{JentzenLindnerPusnik2017c}
(applies with
$ H = H $,
$ \H = \H $,
$ \values = \values $,
$ A = A $,
$ ( \Omega, \F, \P) = ( \Omega, \F, \P ) $,
$ T = T $,
$ \beta = \nicefrac{1}{2} + \beta  $,
$ \gamma = \gamma $,
$ \xi = ( \Omega \ni \omega \mapsto P_I(\xi(\omega))
\in H_{ \nicefrac{1}{2}  + \beta  } ) $,
$ F = ( H_\gamma  \ni x \mapsto  
P_I F(x) \in H ) $,
$ \kappa = ( [0,T] \ni t \mapsto t \in [0,T] ) $,
$ Z = ( [0,T] \times \Omega \ni (t, \omega) \mapsto X_t^I(\omega) \in H_\gamma ) $,
$ O = ( [0,T] \times \Omega \ni (t, \omega) \mapsto O_t^I(\omega) \in H_{ \nicefrac{1}{2} + \beta  } ) $,
$ Y = ( [0,T] \times \Omega \ni (t, \omega) \mapsto X_t^I(\omega) \in H ) $,
$ p = p $, 
$ \rho = \eta_1 $,
$ \eta = \eta_2 $,
$ \iota = \iota $,
$ \alpha_1 = \alpha_1 $,
$ \alpha_2 = \alpha_2 $
for  
$ I \in \mathcal{P}_0 (\H) $
in the setting of~\cite[Lemma~3.4]{JentzenLindnerPusnik2017c})
therefore
justifies~\eqref{eq:Uniform Lp estimate}.
The proof of
Corollary~\ref{corollary:AprioriExact}
is hereby completed.
\end{proof}
\subsection[Strong error estimates for space-time discrete exponential Euler-type approximations]{Strong error estimates for space-time discrete truncated exponential Euler-type approximations}
\label{subsection:StrongConvergenceRates}
In this subsection we study numerical approximations for a class of 
semilinear SPDEs with additive noise and establish in 
Proposition~\ref{proposition:Exact_to_numeric_general} 
below strong convergence rates for truncated exponential Euler-type approximation processes 
$ ( \y_t^{\theta,I})_{ t \in [0,T]} $,
$ I \in \mathcal{P}_0(\H) $,
$ \theta \in \varpi_T $,
(see~\eqref{eq:scheme}
in Proposition~\ref{proposition:Exact_to_numeric_general} below)
 under (i) the assumption that the truncated exponential Euler-type approximations satisfy suitable exponential moment bounds and (ii) suitable approximatibility assumptions on the stochastic convolution process. Our proof of Proposition~\ref{proposition:Exact_to_numeric_general} employs 
 Proposition~\ref{proposition:main_error_estimate} 
 and Corollary~\ref{corollary:AprioriExact} 
 above as well as the elementary truncation error estimate in 
 Lemma~\ref{lemma:function_error_exact_conterpart} below.

\begin{lemma}
	\label{lemma:function_error_exact_conterpart}
	Assume Setting~\ref{setting:main},
		let
		$ ( \Omega, \F, \P ) $
		be a probability space,
		let
		$ ( V, \left \| \cdot \right \|_V ) $
		be an $ \R $-Banach space,
	let  
	$ \varsigma \in [0,\infty) $,
	$ p \in [1, \infty) $, 
	$ \alpha, c, h \in (0,\infty) $,
	$ Y \in \M( \F, \mathcal{B}(  V ) ) $, 
	$ r \in \M( \mathcal{B}( V ), \mathcal{B}( [0,\infty) ) ) $, 
	$ P \in L(H) $, 
	$ F \in \M( \mathcal{B}( V ),
	\mathcal{B}(H) ) $, 
 $ D \in \mathcal{B}( V ) $
 satisfy 
 $ \{
	v \in V \colon
	r( v )
	\leq 
	c
	h^{-\varsigma} \} \subseteq D $.
	Then we have 
	that
	\begin{equation}
	\begin{split} 
	\|
	\1_{ 	D
	}
	(Y) 
	\, 
	P
	F(Y  )
	-
	F( Y  )
	\|_{\L^p(\P; H)} 
	&
	\leq
	c^{- \alpha}
h^{\alpha \varsigma}
\| r( Y ) \|_{\L^{2p\alpha}(\P; \R)}^\alpha
\| P F ( Y  ) \|_{\L^{2p}(\P; H)}
\\
&
\quad
	+
\| ( P - \operatorname{Id}_H ) F(Y ) \|_{\L^p(\P; H)}
	.
	\end{split}
	\end{equation}
\end{lemma}
\begin{proof}[Proof
	of Lemma~\ref{lemma:function_error_exact_conterpart}]
	Observe that 
	the triangle inequality and
	H\"older's inequality verify that
	\begin{equation}
	\begin{split}
	\label{eq:local_monotonicity_implies}
	&
	\|
	\1_{ 	D 
	} 
	(Y ) 
	\, 
	P
	F(Y  )
	-
	F( Y  )
	\|_{\L^p(\P; H)}
	\\
	&
	\leq
	\|
	(
	\1_{ 	D 
	} 
	(Y ) 
	-
	1
	)  
	P
	F(Y  ) 
	\|_{\L^p(\P; H)}
	+
	\|
	P
	F(Y  )
	-
	F( Y )
	\|_{\L^p(\P; H)}
	\\
	&
	\leq
	\| \1_{ 	D 
	} 
	( Y ) 
	-
	1 \|_{\L^{2p}(\P; \R)}
	\| P F ( Y ) \|_{\L^{2p}(\P; H)}
	+
	\| ( P - \operatorname{Id}_H ) F( Y ) \|_{\L^p(\P; H)}
	.
	\end{split}
	\end{equation}
	Moreover, note that
	Markov's inequality
	yields 
	that
	\begin{equation}
	\begin{split}
	\| \1_{ 	D 
	} 
	( Y ) 
	-
	1 
	\|_{\L^{2p}(\P; \R)}
	&=
	\| \1_{ V \backslash D }(Y) \|_{\L^{2p}(\P; \R)}
	\leq 
	\| 
	\1_{ \{ r(Y) > ch^{-\varsigma} \} }
	\|_{\L^{2p}(\P; \R)}
	\\
	&
	=
	[
	\P ( |r( Y )|^{ 2 p \alpha } 
	>
	( c h^{- \varsigma } )^{ 2 p \alpha }  )
	]^{\nicefrac{1}{ ( 2p ) }}
	\leq
	(c h^{-\varsigma} )^{ - \alpha }
	( \E[ |r( Y )|^{ 2 p \alpha } ] )^{\nicefrac{1}{ ( 2 p ) }}.
	\end{split}
	\end{equation}
This and~\eqref{eq:local_monotonicity_implies}  
	give that 
	\begin{equation}
	\begin{split} 
	\|
	\1_{ 	D 
	} 
	( Y ) 
	\, 
	P
	F( Y )
	-
	F( Y )
	\|_{\L^p(\P; H)} 
	&
	\leq
	c^{ - \alpha }
	h^{ \alpha \varsigma }
	(
	\E[ | r( Y ) |^{ 2 p \alpha }]
	)^{\nicefrac{1}{ ( 2p ) }}
	\| P F ( Y ) \|_{\L^{2p}(\P; H)}
	\\
	& \quad	+
	\| ( P - \operatorname{Id}_H ) F( Y ) \|_{\L^p(\P; H)}
	.
	\end{split}
	\end{equation}
	The proof of Lemma~\ref{lemma:function_error_exact_conterpart} is hereby completed.
\end{proof}
\begin{lemma}
	\label{lemma:density_argument}
	Assume Setting~\ref{setting:main},
	let 
	$ C, c, \gamma \in [0, \infty) $, 
	$ \delta, \kappa \in [0,\gamma] $, 
	$ F \in \mathcal{C}( H_\gamma, H ) $,
	let 
	$ (P_I)_{ I \in \mathcal{P}(\H) } \subseteq L(H) $ 
	satisfy for all 
	$ I \in \mathcal{P} (\H) $,
	$ v \in H $ 
	that
	$ P_I( v )
	= \sum_{h \in I} \langle h, v \rangle_H h $,
	and assume for all
	$ I \in \mathcal{P}_0(\H) $,
	$ u, v \in P_I(H) $
	that
	$ \| P_I F( u ) - P_I F( v ) \|_H
	\leq
	C \| u - v \|_{H_\delta} ( 1 + \| u \|_{H_\kappa}^c + \| v \|_{H_\kappa}^c ) $.
	Then we have for all
	$ u, v \in H_\gamma $ that
	\begin{equation} 
	\label{eq:End results}
	\| F( u ) - F( v ) \|_H
	\leq
	C \| u - v \|_{H_\delta} ( 1 + \| u \|_{H_\kappa}^c + \| v \|_{H_\kappa}^c ) 
	.
	\end{equation}
\end{lemma}
\begin{proof}[Proof of Lemma~\ref{lemma:density_argument}]
	Throughout this proof let
	$ I_n \subseteq \H $, $ n \in \N $, 
	be sets which
	satisfy for all $ n \in \N $ that
	$ I_n \subseteq I_{n+1} $
	and
	$ \cup_{m \in \N } I_m = \H $.
	Note that the triangle inequality gives
	that
	for all
	$ m, n \in \N $,
	$ u, v \in H_\gamma $
	we have 
	that
	\begin{equation}
	\begin{split}
	\label{eq:Triangle}
	&
	\| F(u) - F(v) \|_V 
	\leq 
	\| F(u) - P_{I_m} F( u) \|_H
	+
	\| P_{I_m}  F( u )
	-
	P_{I_m} 
	F(P_{I_n} u )
	\|_H
	\\
	&
	+
	\| P_{I_m} 
	F(P_{I_n} u )
	-
	P_{I_m}
	F(P_{I_n} v ) 
	\|_H
	+
	\| P_{I_m} F( P_{I_n} v ) 
	-
	P_{I_m} 
	F( v ) 
	\|_H
	+
	\|
	P_{I_m} 
	F( v )
	-
	F(v)
	\|_H
	.
	\end{split}
	\end{equation}
	Next observe that 
	for all $ v \in H $ we have that
	\begin{equation} 
	\label{eq:limit1b}
	\limsup\nolimits_{n\to \infty}  
	\| v - P_{I_n} v \|_H 
	= 0 
	.
	\end{equation}
	This
	ensures that for all $ u, v \in H_\gamma $
	we have that
	\begin{equation}
	\begin{split}
	\label{eq:Estimate first two}
	\limsup\nolimits_{m \to \infty}
	\big(
	\| F(u) - P_{I_m} F( u) \|_H 
	+
	\|
	P_{I_m} 
	F( v )
	-
	F(v)
	\|_H
	\big) 
	=
	0.
	\end{split}
	\end{equation}
	In addition, observe that
	for all $ u \in H_\gamma $  we have that
	\begin{equation} 
	\label{eq:limit1a}
	\limsup\nolimits_{n\to \infty}  
	\| u - P_{I_n}u \|_{H_\gamma }
	= 0  
	.
	\end{equation}
	The
	assumption that $ F \in \mathcal{C}( H_\gamma, H ) $ 
	hence
	gives that for all 
	$ m \in \N $,
	$ u, v \in H_\gamma $
	we have that
	\begin{equation}
	\begin{split}
	\label{eq:Estimate second two}
	&
	\limsup\nolimits_{n \to \infty}
	\big( 
	\| P_{I_m}  F( u )
	-
	P_{I_m} 
	F(P_{I_n} u )
	\|_H
	+
	\| P_{I_m} F( P_{I_n} v ) 
	-
	P_{I_m} 
	F( v ) 
	\|_H 
	\big)
	\\
	&
	\leq
	\limsup\nolimits_{n \to \infty} 
	\big(
	\|  F( u )
	- 
	F(P_{I_n} u )
	\|_H
	+
	\| F( P_{I_n} v ) 
	- 
	F( v ) 
	\|_H 
	\big) 
	=
	0
	.
	\end{split}
	\end{equation}
	Moreover, note that
	the fact that	 
	$ \forall \, n \in \N $,
	$ u, v \in P_{I_n}(H) \colon 
	\| P_{I_n} F( u ) - P_{I_n} F( v ) \|_H
	\leq
	C \| u - v \|_{H_\delta} 
	( 1 + \| u \|_{H_\kappa}^c + \| v \|_{H_\kappa}^c ) $
	and the fact that
	$ \forall \, m \in \N $, 
	$ n \in ( [m, \infty) \cap \N ) $,
	$ u \in H \colon 
	\| P_{I_m} u \|_H
	=
	\| P_{I_m} P_{I_n} u \|_H 
	\leq 
	\| P_{I_n} u \|_H $ 
	yield that
	for all
$ m \in \N $, 
$ n \in ( [m, \infty) \cap \N ) $,
	$ u, v \in H_\gamma $
	we have that
	\begin{equation}
	\begin{split}
	\| P_{I_m} 
	F(P_{I_n} u )
	-
	P_{I_m}
	F(P_{I_n} v ) 
	\|_H
	&
	\leq
	\| P_{I_n} 
	F(P_{I_n} u )
	-
	P_{I_n}
	F(P_{I_n} v ) 
	\|_H
	\\
	&
	\leq 
	C
	\| P_{I_n} u - P_{I_n} v\|_{H_\delta}
	( 
	1 
	+ 
	\| P_{I_n} u \|_{H_\kappa}^c 
	+ 
	\| P_{I_n} v \|_{H_\kappa}^c 
	)
	.
	\end{split}
	\end{equation}
	The fact that
	$ \delta, \kappa \in [0, \gamma] $
	and~\eqref{eq:limit1a}
	therefore verify that
	for all 
	$ m \in \N $,
	$ u, v \in H_\gamma $ 
	we have 
	that
	\begin{equation}
	\begin{split} 
	\limsup\nolimits_{ 
		n \to \infty }
	\| P_{I_m} 
	F(P_{I_n} u )
	-
	P_{I_m}
	F(P_{I_n} v ) 
	\|_H
	&
	\leq
	C \| u - v \|_{H_\delta}
	( 1 + \| u \|_{H_\kappa}^c + \| v \|_{H_\kappa}^c ) 
	.
	\end{split} 
	\end{equation}
	Combining~\eqref{eq:Triangle} 
	and~\eqref{eq:Estimate second two}
	hence gives that
	for all
	$ m \in \N $,
	$ u, v \in H_\gamma $ we have that
	\begin{equation}
	\begin{split} 
	\|  
	F( u ) 
	-
	F ( v ) 
	\|_H
	&
	\leq
	\| F(u) - P_{I_m} F( u) \|_H 
	+
	C \| u - v \|_{H_\delta}
	( 1 + \| u \|_{H_\kappa}^c + \| v \|_{H_\kappa}^c )
	\\
	& 
	\quad  
	+
	\|
	P_{I_m} 
	F( v )
	-
	F(v)
	\|_H
	.
	\end{split} 
	\end{equation}
	This and~\eqref{eq:Estimate first two}
	justify~\eqref{eq:End results}.
	The proof of Lemma~\ref{lemma:density_argument}
	is hereby completed.
\end{proof}
\begin{proposition}
\label{proposition:Exact_to_numeric_general}
Assume Setting~\ref{setting:main},
let
$ T, \upnu, \varsigma, \alpha \in (0,\infty) $,
$ \aa, \iota, \rho \in [0, \infty) $,  
$ C, c, p \in [1, \infty) $, 
$ \beta \in [0,\nicefrac{1}{2}) $,
$ \gamma \in  
 [2 \beta, \nicefrac{1}{2} + \beta  ) $, 
$ \delta, \kappa \in [0,\gamma] $,  
$ \eta_1 \in [0, \nicefrac{1}{2}  + \beta  ) $,
$ \eta_2 \in [\eta_1, \nicefrac{1}{2}  + \beta  ) $, 
$ \alpha_1 \in [0, 1 - \eta_1 ) $,
$ \alpha_2 \in [0, 1 - \eta_2 ) $,   
$ B \in \HS(U, H_\beta ) $,
$  \varepsilon \in [0, ( \nicefrac{ \rho }{ p } ) 
\exp( - 2 ( \aa + \rho \| B \|_{\HS(U,H)}^2 ) T ) ] $,
$ F \in \mathcal{C}^1(  H_\gamma, H ) $, 
$ r \in \M ( \mathcal{B}(H_\gamma), \mathcal{B}( [0,\infty) ) ) $,
$ (D_h^I)_{ h\in (0,T], I \in \mathcal{P}_0(\H) } \subseteq \mathcal{B}(H_\gamma) $, 
let
$ \Phi \colon H \to [0, \infty) $
be a function,
let  
$ (P_I)_{ I \in \mathcal{P}(\H) } \subseteq L(H) $ 
satisfy for all 
$ I \in \mathcal{P}(\H) $,
$ x \in H $ 
that
$ P_I(x)
= \sum_{h \in I} \langle h, x \rangle_H h $,
assume for all
$ I \in \mathcal{P}_0(\H) $,  
$ h \in (0,T] $
that
$ \{
v \in P_I( H ) \colon
r( v )
\leq  
\upnu
h^{-\varsigma}
\}
\subseteq D_h^I $
and
$ ( P_I(H) \ni v \mapsto \Phi( v ) \in [0, \infty) )
\in \mathcal{C} ( P_I(H), [0, \infty) ) $,
assume for all  
$ I \in \mathcal{P}_0(\H) $,
$ x, y \in P_I( H ) $
that
$ \langle x, F(x) \rangle_H \leq \aa ( 1 + \| x \|_H^2 ) $,
$ \langle F'(x) y, y \rangle_H  \leq 
( \varepsilon  \| x \|_{H_{\nicefrac{1}{2}}}^2 
+ C ) \| y \|_H^2 
+ \|y \|_{H_{\nicefrac{1}{2}}}^2 $,
$
\| P_I( F( x ) - F( y ) ) \|_H
\leq
C \| x - y \|_{H_\delta} ( 1 + \| x \|_{H_\kappa}^c  + \| y \|_{H_\kappa}^c  ) $,
$ \langle x, A x + F(x+y) \rangle_H
\leq
\Phi(y) ( 1 + \| x \|_H^2 ) $,   
and
\begin{equation}
\label{eq:GivesMoments}
\Big[
\sup\nolimits_{ v \in H_{ \max \{ \gamma, \eta_2 \} } }
\tfrac{ \| F(v) \|_{H  } }
{ 1 + \| v \|_{H_{ \eta_2  } }^2 }
\Big] 
+
\Big[
\sup\nolimits_{ v \in H_{ \max \{ \gamma, \eta_1 \} }  }
\tfrac{ \| F(v) \|_{H_{ - \alpha_2 } } }
{ 1 + \| v \|_{ H_{ \eta_1 } }^2 }
\Big]
+
\Big[
\sup\nolimits_{ v \in H_{\gamma}  } 
\tfrac{ \| F(v) \|_{H_{-\alpha_1} } }{ 1 + \| v \|_H^2 } 
\Big]
< \infty,
\end{equation}
let
$ ( \Omega, \F, \P ) $
be a probability space with a normal filtration
$ ( \f_t )_{t \in [0,T]} $,
let $ (W_t)_{ t \in [0,T]} $
be an
$ \operatorname{Id}_U $-cylindrical
$ ( \f_t )_{t\in [0,T]} $-Wiener process,
let
$ \xi \in \L^{ 4 p \max \{ c, 2 \} }(\P |_{\f_0}; H_{ \max \{\gamma, \eta_2 \} } ) $
satisfy
$ \E[  \| \xi \|_H^{ 16 p } ] < \infty $, 
let 
$ X \colon [0,T] \times \Omega \to H_\gamma $
and
$ O \colon [0,T] \times \Omega \to H_\gamma $  
be $ ( \f_t )_{ t \in [0,T] } $-adapted 
stochastic processes 
w.c.s.p.\ 
which satisfy for all $ t \in [0,T] $ that
$ [ O_t ]_{\P, \mathcal{B}(H_\gamma) } = \int_0^t e^{(t-s)A} B \, dW_s $
and
\begin{equation}
\P\bigg(
X_t
= 
e^{tA} \xi  
+ 
\int_0^t e^{(t-s)A} F(X_s)
\,
ds
+
O_t
\bigg)
=
1,
\end{equation}
let
$ \y^{\theta, I}  \colon [0,T] \times \Omega \to P_I( H ) $,
$ \theta \in \varpi_T $,
$ I \in \mathcal{P}_0(\H) $, 
and
$ \O^{\theta, I } \colon [0,T] \times \Omega \to P_I( H ) $,
$ \theta \in \varpi_T $,
$ I \in \mathcal{P}_0(\H) $, 
be 
$ ( \f_t )_{t\in [0,T]} $-adapted
stochastic 
processes  
w.c.s.p.\
which satisfy for all
$ \theta \in \varpi_T $,
$ I \in \mathcal{P}_0(\H) $, 
$ t \in [0,T] $
that 
\begin{equation}
\label{eq:scheme}
\P \bigg(
\y_t^{\theta, I}
=
e^{tA}
P_I
\xi
+
\int_0^t
\1_{  D_{ |\theta|_T  }^I  }
\!
( \y^{\theta, I }_{\llcorner s \lrcorner_\theta })
\,
e^{(t-\llcorner s \lrcorner_\theta )A}
P_I F(
\y^{\theta, I }_{ \llcorner s \lrcorner_\theta  } 
) 
\,
ds
+
\O_t^{\theta, I}
\bigg) 
= 1
,
\end{equation}
and
assume for all  
$ \theta \in \varpi_T $,
$ I, \mathcal{I} \in \mathcal{P}_0(\H) $
with
$ I \subseteq \mathcal{I} $
that
\begin{equation} 
\label{eq:TimeRegularityNoise}
\sup\nolimits_{ s \in [0,T] } 
\| \O_s^{\theta, I} - \O_{ \llcorner s \lrcorner_\theta }^{\theta, I} 
\|_{\L^{4p}( \P; H_\delta ) } 
\leq 
C [ | \theta |_T]^\alpha 
,
\end{equation} 
\begin{equation} 
\label{eq:NoiseRates}
\sup\nolimits_{ s \in [0,T] } 
\| 
\O_s^{\theta, I} - P_{ \mathcal{I} } O_s
\|_{\L^{4p c}(\P; H_{ \max \{ \kappa, \delta \} } )}   
\leq 
C 
( \| P_{ \H \backslash I } (-A)^{-\iota}  \|_{L(H)}
+
[ | \theta |_T]^\alpha  )
,
\end{equation}
\begin{equation} 
\begin{split} 
\label{eq:ExpBound}
&
\sup\nolimits_{J, K\in \mathcal{P}_0(\H)} 
\sup\nolimits_{\vartheta \in \varpi_T} 
\int_0^T
\E\big[
\exp\!
\big(  
\rho
\|  
\y_s^{\vartheta, K}
-
\O_s^{\vartheta, K}
+
P_J O_s
+
e^{s A} P_{ J \backslash K } \xi
\|_H^2   
\big)
\big] 
\,
ds
< 
\infty
,
\end{split} 
\end{equation} 
\begin{equation} 
\begin{split}
\label{eq:SupBound1} 
\sup\nolimits_{J \in \mathcal{P}_0(\H)} 
\sup\nolimits_{\vartheta \in \varpi_T}
\sup\nolimits_{ s \in [0,T] }
\big[
\| P_J F ( \y_s^{\vartheta, J} ) \|_{\L^{4p}(\P; H_{  \gamma - \delta } )}
+
\| P_J F ( \y_s^{\vartheta, J} ) \|_{\L^{2p}(\P; H_{ \iota } )}
\big] 
< \infty, 
\end{split} 
\end{equation}
\begin{equation} 
\begin{split}
\label{eq:SupBound2} 
\text{and}
\quad
\sup\nolimits_{J\in \mathcal{P}_0(\H)} 
\sup\nolimits_{\vartheta \in \varpi_T}
\sup\nolimits_{ s \in [0,T] }
\big[ 
\| \y_s^{\vartheta, J} \|_{\L^{4 p c }( \P; H_\kappa ) } 
+
\| 
r( \y_s^{\vartheta, J } )
\|_{\L^{ \nicefrac{4p\alpha}{\varsigma} }(\P; \R) }
\big] < \infty 
.
\end{split} 
\end{equation} 
Then there exists $ \mathfrak{c} \in \R $
such that for all 
$ \theta \in \varpi_T $, 
$ I \in \mathcal{P}_0(\H) $
we have that
\begin{equation}
\label{eq:Theorem general result}
\sup\nolimits_{ t \in [0,T] }
\| X_t - \y^{\theta, I }_t \|_{\L^p(\P; H)}
\leq
\mathfrak{c} \big( \| P_{\H \backslash I } ( -A )^{- \min \{ \gamma - \delta, \iota \} } \|_{L(H)}
+
[ | \theta |_T ]^{ \min \{  \gamma - \delta, \alpha \} }
\big)
.
\end{equation}
\end{proposition}
\begin{proof}[Proof of Proposition~\ref{proposition:Exact_to_numeric_general}]
Throughout this proof let
$ O^I
 \colon [0, T]
\times \Omega \to P_I( H ) $,
$ I \in \mathcal{P}_0(\H) $,
be the 
$ ( \f_t )_{ t \in [0,T] } $-adapted
stochastic processes  
which satisfy
for all 
$ I \in \mathcal{P}_0(\H) $, 
$ t\in [0, T] $
that  
$ O_t^I = P_I O_t $, 
let
$ \mathcal{A}_I \colon P_I(H) \to P_I(H) $,
$ I \in \mathcal{P}_0(\H) $,
satisfy for all 
$ I \in \mathcal{P}_0(\H) $,
$ v \in P_I(H) $
that
$ \mathcal{A}_I v = A v $, 
for every
$ I \in \mathcal{P}_0(\H) $
let
$ ( \mathcal{H}_{I, s}, \langle \cdot, \cdot \rangle_{ \mathcal{H}_{I, s} },
\left \| \cdot \right\|_{ \mathcal{H}_{I, s} } ) $,
$ s \in \R $,
be a family of interpolation spaces
associated
to $ - \mathcal{A}_I $,
and let
$ I_m \in 
( \mathcal{P}_0(\H) \backslash \{ \emptyset \} ) $,
$  m \in \N  $,
be sets which satisfy  
$ \cup_{n\in \N} ( \cap_{m\in \{n+1, n+2,\ldots \}} I_m ) = \H $.
Note that
the fact that
for all
$ I \in \mathcal{P}_0(\H) $, $ x \in P_I(H) $
we have that
\begin{equation} 
\label{eq:finite coercivity}
\langle x, P_I F(x) \rangle_H \leq \aa ( 1 + \| x \|_H^2 ), 
\end{equation} 
the fact that for all
$ I \in \mathcal{P}_0(\H) $,
$ x, y \in P_I( H ) $
we have that 
$ \langle (P_I F)'(x) y, y \rangle_H  \leq 
( \varepsilon  \| x \|_{H_{\nicefrac{1}{2}}}^2 
+ C ) \| y \|_H^2 
+ \|y \|_{H_{\nicefrac{1}{2}}}^2 $,
$ \langle x, A x + P_I F(x+y) \rangle_H
\leq
\Phi(y) ( 1 + \| x \|_H^2 ) $,   
and
\begin{equation} 
\label{eq:local Lip}
\| P_I( F( x ) - F( y ) ) \|_H
\leq
C \| x - y \|_{H_\delta} ( 1 + \| x \|_{H_\kappa}^c  + \| y \|_{H_\kappa}^c  )
,
\end{equation} 
Proposition~\ref{proposition:main_error_estimate}
(applies with
$ H = P_{I_n}(H) $,
$ \mathbb{H} = P_{I_n} ( \mathbb{H} ) $,
$ \values = ( I_n \ni h \mapsto \values_h \in \R ) $,
$ A = \mathcal{A}_{I_n} $,
$ ( H_s )_{ s \in \R } = ( \mathcal{H}_{I_n, s} )_{ s \in \R } $, 
$ T = T $, 
$ \theta = \theta $,
$ a = \aa $,
$ b = \aa $,
$ {\bf C} = C $,
$ \rho = \rho $, 
$ C = C $,
$ c = c $,
$ p = p $, 
$ \gamma = \gamma $,
$ \delta = \delta $, 
$ \kappa = \kappa $, 
$ B = ( U \ni u \mapsto P_{I_n} B(u) \in P_{I_n}(H) ) $,
$ \varepsilon = \varepsilon $,
$ F = ( P_{I_n}(H)  \ni x \mapsto P_{I_n} F( x ) \in P_{I_n}(H) ) $,
$ \ff = ( P_{I_n}(H) \ni x \mapsto 
\1_{ D_{ |\theta |_T }^I }( x ) P_I F( x ) \in P_{I_n}(H) ) $,
$ \Phi = ( P_{I_n}(H) \ni x \mapsto \Phi(x) \in [0, \infty) ) $,
$ (\Omega, \F, \P ) = ( \Omega, \F, \P ) $,
$ ( \f_t )_{ t \in [0,T]} = ( \f_t )_{ t \in [0,T]} $,
$ \xi = ( \Omega \ni \omega \mapsto P_{I_n} \xi( \omega) \in P_{I_n}(H) )$,
$ ( W_t )_{t\in [0,T]} 
= 
( W_t )_{t \in [0,T]} $,
$ O = O^{I_n} $, 
$ \y = ( [0,T] \times \Omega \ni (t, \omega)
\mapsto \y^{\theta, I}_t( \omega )
\in P_{ I_n }(H) ) $, 
$ \O = ( [0,T] \times \Omega \ni (t, \omega) 
\mapsto \O^{\theta, I}_t( \omega )
-
e^{tA}
P_{I_n \backslash I } \xi(\omega)
\in P_{ I_n }(H) ) $,
$ X = \mathcal{X}^n $
for  
$ \theta \in \varpi_T $,
$ I \in \mathcal{P}( I_n ) $,
$ n \in \N $ 
in the setting of  
Proposition~\ref{proposition:main_error_estimate}),
and the triangle inequality  
verify that
\begin{enumerate}[(a)] 
	\item \label{item:help_exact_mild_general2} we have that there exist
	$ ( \f_t )_{ t \in [0,T] } $-adapted
	stochastic processes 
	$ \mathcal{X}^n \colon [0,T] \times \Omega \to P_{I_n}(H) $,
	$ n \in \N $,
w.c.s.p.\
	which satisfy
	for all 
	$ n \in \N $,
	$ t \in [0,T] $
	that
	\begin{equation} 
	\label{eq:help_exact_mild_general2}
	\mathcal{X}_t^n
	= 
	e^{t A} P_{I_n} \xi 
	+
	\int_0^t e^{(t-u)A} P_{I_n} F (  \mathcal{X}_{u}^n ) 
	\,
	du
	+
	O_t^{I_n}  
	\end{equation}  
	and
	\item \label{item:first_num_estimate}
we have for all
$ \theta \in \varpi_T $,
$ I \in \mathcal{P}_0(\H) $, 
$ n \in \N $,
$ t \in [0,T] $
with
$ I \subseteq I_n $
that
\begin{equation}
\begin{split}
\label{eq:first_num_estimate}
&
\|  \y_t^{\theta,I} - \mathcal{X}_t^n \|_{ \L^p( \P; H ) }  
\leq
\sup\nolimits_{s \in [0,T]}
\| \O_s^{\theta, I} - O_s^{I_n} \|_{ \L^p( \P; H ) } 
+
\| 
P_{ I_n \backslash I } \xi  
\|_{ \L^p( \P; H ) } 
\\
& 
+ 
\tfrac{ C [ \max \{ T, 1 \} ]^2
}{ 1 - \gamma }
\exp\!
\big( 
\big( 
C
+
\rho	( 2 \aa + \|  B \|_{\HS(U,H)}^2 )  
\big) 
T
\big) 
\\
&
\cdot
\bigg[ 
\int_0^T 										\E 									\big[
e^{ \rho \| \y_s^{\theta,I} - \O_s^{\theta, I} + O_s^{I_n}
	+e^{sA}
	P_{ I_n \backslash I }
	\xi   \|_H^2 }
\big]
\, ds
\bigg] 
\Big\{
[| \theta |_T]^{\gamma - \delta}
\sup\nolimits_{s\in [0,T]}
\| P_I F ( \y_s^{\theta, I} ) \|_{ \L^{2p} (\P; H_{\gamma-\delta} )}
\\
&
+
\sup\nolimits_{s\in [0,T]}
\| 
\1_{ D_{ |\theta|_T }^I } ( \y_s^{\theta, I} )
P_I F ( \y_s^{\theta, I} ) 
- 
P_{I_n} F ( \y_s^{\theta, I} ) 
\|_{ \L^{ 2 p } ( \P; H ) }
\\
& 
+
\Big(   
2
[ | \theta |_T ]^{\gamma - \delta}
\sup\nolimits_{s\in [0,T]}
\|
P_I
F( \y_s^{\theta, I} )
\|_{\L^{4p}(\P; H)}
+
\sup\nolimits_{s\in [0,T]}
\| \O_s^{\theta, I} - \O^{\theta, I}_{ \llcorner s \lrcorner_\theta } \|_{\L^{4p}(\P; H_\delta)}
\\
&
+
[ | \theta |_T ]^{ \gamma - \delta }
\|  \xi \|_{\L^{4p}(\P; H_\gamma)}
+
\sup\nolimits_{s\in [0,T]}
\| \O_s^{\theta, I} - O_s^{I_n} \|_{\L^{4p}(\P; H_\delta)}
+
\| P_{I_n \backslash I} \xi \|_{\L^{4p}(\P; H_\delta)}
\Big)
\\
&
\cdot
\big[
1
+
2 
\sup\nolimits_{s\in [0,T]}
\|
\y_s^{\theta, I}
\|_{\L^{4p c }(\P; H_\kappa)}
+
\sup\nolimits_{s\in [0,T]}
\| \O_s^{\theta, I} - O_s^{I_n} \|_{\L^{4p c }(\P; H_\kappa)}
+
\| P_{I_n \backslash I} \xi \|_{\L^{4p c}(\P; H_\kappa)}
\big]^c 
\Big\}
.  
\end{split}
\end{equation}
\end{enumerate}
Moreover, observe that the triangle inequality 
gives that for all  
$ \theta \in \varpi_T $,
$ I \in \mathcal{P}_0(\H) $,
$ n \in \N $,
$ t \in [0,T] $
we have 
that
\begin{equation}
\label{eq:apply_alekseev_grobner_full_general}
\begin{split}
&
\| 
\y^{\theta, I }_t 
- 
X_t 
\|_{\mathcal{L}^p(\P; H)}
\leq
\| \y_t^{\theta, I} - \mathcal{X}_t^n 
\|_{\mathcal{L}^p(\P; H)}
+ 
\| \mathcal{X}_t^n - X_t  
\|_{\mathcal{L}^p(\P; H)}
.
\end{split}
\end{equation}
Next note that~\eqref{eq:GivesMoments},
\eqref{eq:finite coercivity},
the fact that
$ \xi \in \L^{ 8 p }( \P|_{\f_0}; H_{ \max \{\gamma, \eta_2 \} } ) $,
the fact that  
$ \E[ \| \xi \|_H^{ 16 p } ] < \infty $,
and
Corollary~\ref{corollary:AprioriExact}
(applies with
$ T = T $,
$ a = \aa $,
$ b = \aa $,
$ p = 2p $,
$ \beta = \beta $,
$ \gamma = \gamma $,
$ \eta_1 = \eta_1 $,
$ \eta_2 = \eta_2 $,
$ \iota = \max \{ \gamma, \eta_2 \} $,
$ \alpha_1 = \alpha_1 $,
$ \alpha_2 = \alpha_2 $, 
$ B = B $,
$ F = F $,
$ P_I = P_I $,
$ ( \Omega, \F, \P ) 
=
(
\Omega, \F, \P ) $,
$ (\f_s)_{ s \in [0,T]} = (\f_s)_{ s \in [0,T]} $,
$ ( W_s )_{ s \in [0,T] } = ( W_s )_{ s \in [0,T] } $,
$ \xi = \xi $,
$ X^{ I_n }_t = \mathcal{X}^n_t $,
$ O^{I}_t = O^{I}_t $
for
$ t \in [0,T] $,
$ n \in \N $,
$ I \in \mathcal{P}_0(\H) $
in the setting of Corollary~\ref{corollary:AprioriExact})
illustrate that
\begin{equation} 
\label{eq:CrucialMoments!!}
\sup\nolimits_{n \in \N } \sup\nolimits_{ t \in [0,T]} \| \mathcal{X}_t^n \|_{\L^{2 p}(\P; H_\gamma)}
< \infty 
.
\end{equation}
In addition, observe that~\eqref{eq:local Lip},
Lemma~\ref{lemma:density_argument}
(applies with
$ C = C $,
$ c = c $,
$ \gamma = \gamma $,
$ \delta = \delta $,
$ \kappa = \kappa $,
$ F = F $,
$ P_I = P_I $
for 
$ I \in \mathcal{P}( \H ) $
in the setting of
Lemma~\ref{lemma:density_argument}), 
and the fact that
$ \gamma \geq \max \{ 2 \beta, \kappa, \delta \} $
yield that for all
$ R \in (0,\infty) $,
$ x, y \in H_\gamma $ 
with
$ \max \{ \| x \|_{H_\gamma}, \| y \|_{H_\gamma} \} \leq R $
we have that
\begin{equation}
\begin{split}
&
\| F( x ) - F( y ) \|_{ H_{ 2 \beta - \gamma } }
\leq
 \| ( - A )^{ 2 \beta - \gamma } \|_{L(H)}
 \| F( x ) - F( y ) \|_H
\\
&
\leq
C
\| ( - A )^{ 2 \beta - \gamma } \|_{L(H)}
\| x - y \|_{H_\delta}
( 1 + 2 ( \| ( -A )^{\kappa- \gamma } \|_{L(H)}  R )^c )
\\
&
\leq
C
\| ( - A )^{ 2 \beta - \gamma } \|_{L(H)}
\| (-A)^{\delta - \gamma} \|_{L(H)}
\| x - y \|_{H_\gamma}
( 1 + 2 ( \| ( -A )^{\kappa- \gamma } \|_{L(H)}  R )^c )
< \infty 
.
\end{split}
\end{equation}
Combining this,
\eqref{eq:help_exact_mild_general2},
\eqref{eq:CrucialMoments!!},   
and
the fact that
$ 2 \beta - \gamma \leq 0 $
with, e.g., \cite[Corollary~6.5]{JentzenPusnik2019Published}
(applies with
$ H = H $,
$ U = U $,
$ \H = \H $,
$ \lambda = \values $,
$ A = A $, 
$ \gamma = \gamma $,
$ T = T $,
$ p = 2 p $,
$ ( \Omega, \F, \P  )
 = ( \Omega, \F, \P ) $,
 $ (\F_t)_{ t \in [0,T]} = (\f_t)_{ t \in [0,T]} $
 $ \xi = \xi $,
 $ (W_t)_{ t \in [0,T] } 
 =
 (W_t)_{ t \in [0,T] } $, 
 $ \eta = 2 ( \gamma - \beta ) $,
 $ F = ( H_\gamma \ni x \mapsto F(x) \in H_{ 2 \beta - \gamma } ) $,
 $ B = ( H_\gamma \ni x \mapsto B \in 
 \HS(U, H_{ \beta } ) ) $,
 $ I_n = I_n $,
 $ X^n = \mathcal{X}^n $, 
 $ X^0 = X $,
 $ q = p $,
 $ K
 =
 C
 \| (- A )^{ 2 \beta - \gamma } \|_{L(H)}
 \| (-A)^{\delta - \gamma} \|_{L(H)} 
 ( 1 + 2 ( \| ( -A )^{\kappa- \gamma } \|_{L(H)} R )^c )
 $
 for 
 $ n \in \N $,
 $ R \in (0,\infty) $
in the setting of~\cite[Corollary~6.5]{JentzenPusnik2019Published}) 
ensures  that
\begin{equation}
\label{eq:galerkin_no_rate}
\limsup\nolimits_{   n \to \infty}
\big(
\sup\nolimits_{t \in [0,T] }
 \| 
\mathcal{X}_t^n 
- 
X_t  
  \|_{\mathcal{L}^p(\P; H_\gamma)}
\big)
= 0
.
\end{equation}
In addition,
note that
the assumption that
for all 
$ I \in \mathcal{P}_0(\H) $,
$ h \in (0,T] $
we have that
$ \{ v \in P_I(H) \colon r( v ) \leq \upnu h^{-\varsigma} \}
\subseteq
D_h^I
$
and
Lemma~\ref{lemma:function_error_exact_conterpart}
(applies with 
$ ( \Omega, \F, \P ) 
=
( \Omega, \F, \P ) $,
$ V = P_I( H ) $,   
$ \varsigma = \varsigma $,
$ p = 2 p $,
$ \alpha = \frac{\alpha}{\varsigma} $,
$ c = \upnu $,
$ h = |\theta|_T $, 
$ Y = ( \Omega \ni \omega \mapsto \y^{\theta, I}_t(\omega) \in P_I(H) ) $,
$ r = ( P_I(H) \ni x \mapsto r(x) \in [0, \infty) ) $, 
$ P = P_I $,
$ F = ( P_I(H) \ni x \mapsto P_{I_n} F(x) \in H ) $,
$ D = D_{ | \theta |_T }^I $
for 
$ \theta \in \varpi_T $,
$ I \in \mathcal{P}_0( \H ) $,
$ n \in \N $,
$ t \in [0,T] $
in the setting of 
Lemma~\ref{lemma:function_error_exact_conterpart}) 
verify that for all  
$ \theta \in \varpi_T $,
$ I \in \mathcal{P}_0(\H) $,
$ n \in \N $, 
$ t \in [0,T] $
with
$ I \subseteq I_n $
we have
that
\begin{equation}
\begin{split}
\label{eq:LipsEstimate}
&
\|
\1_{
	D_{ |\theta|_T }^I
}
(
\y_t^{\theta, I }
)
\,
P_I 
F( \y_t^{\theta, I } )
-
P_{I_n}
F( \y_{ t }^{\theta, I} )
\|_{\L^{ 2 p }(\P; H)}
\\
&
=
\|
\1_{
	D_{ |\theta|_T }^I
}
(
\y_t^{\theta, I }
)
\,
P_I 
(
P_{I_n}
F( \y_t^{\theta, I } ) )
-
P_{I_n}
F( \y_{ t }^{\theta, I} )
\|_{\L^{ 2 p }(\P; H)}
\\
&  
\leq 
| \upnu |^{- \nicefrac{\alpha}{\varsigma} }
[ |\theta|_T  ]^{ \alpha }
\| 
r( \y_t^{\theta, I } )
\|_{\L^{ 4 p \alpha / \varsigma }(\P; \R) }^{ \nicefrac{\alpha}{\varsigma} }
\| P_I P_{I_n} F ( \y_t^{\theta, I } ) \|_{\L^{4 p}(\P; H)}
\\
&
\quad 
+
\| ( P_I - \operatorname{Id}_H ) P_{I_n} 
F( \y_t^{\theta, I } ) \|_{\L^{ 2 p }(\P; H)}
\\
&
=
| \upnu |^{- \nicefrac{\alpha}{\varsigma} }
[ |\theta|_T ]^{ \alpha }
\| 
r( \y_t^{\theta, I } )
\|_{\L^{ 4 p \alpha / \varsigma }(\P; \R) }^{ \nicefrac{\alpha}{\varsigma} }
\| P_I F ( \y_t^{\theta, I } ) \|_{\L^{4 p}(\P; H)}
+
\|  P_{I_n \backslash I} F( \y_t^{\theta, I } ) \|_{\L^{ 2 p }(\P; H)}
.
\end{split}
\end{equation}
Moreover, note that for all 
$ \theta \in \varpi_T $,
$ I \in \mathcal{P}_0(\H) $,
$ n \in \N $ 
with
$ I \subseteq I_n $
we have that
\begin{equation} 
\begin{split}
\label{eq:Triangle1}
&
 \sup\nolimits_{t \in [0,T]}
\| 
P_{ I_n \backslash I } F ( \y_t^{\theta, I}) 
\|_{\L^{ 2 p }(\P; H)}
\leq
\| P_{ I_n \backslash I } (-A)^{-\iota} \|_{L(H)}
\sup\nolimits_{t \in [0,T]}
\| 
P_{ I_n \backslash I } F ( \y_t^{\theta, I}) 
\|_{\L^{ 2 p }(\P; H_\iota)}
\\
&
\leq 
\| P_{ \H \backslash I } (-A)^{-\iota} \|_{L(H)}
\sup\nolimits_{ J \in \mathcal{P}_0(\H) }
\sup\nolimits_{t \in [0,T]}
\| 
P_{ J } F ( \y_t^{\theta, I}) 
\|_{\L^{ 2 p }(\P; H_\iota)}
.
\end{split} 
\end{equation}
In addition, observe that for all 
$ I \in \mathcal{P}_0(\H) $, 
$ n \in \N $
with
$ I \subseteq I_n $
we have that
\begin{equation}
\begin{split}
\label{eq:last eq}  
\| P_{I_n \backslash I} \xi \|_{\L^{4p}(\P; H_\delta)}
\leq
\| P_{ \H \backslash I } ( - A)^{\delta - \gamma } \|_{L(H) }
\| \xi \|_{\L^{4p}(\P; H_\gamma)}
.
\end{split}
\end{equation}
Next note that~\eqref{eq:NoiseRates}
ensures that  
for all 
$ \theta \in \varpi_T $,
$ I \in \mathcal{P}_0(\H) $,
$ n \in \N $
with
$ I \subseteq I_n $
we have that
\begin{equation} 
\begin{split}
&
\sup\nolimits_{ s \in [0,T] }
\| \O_s^{\theta, I} - O_s^{I_n} \|_{\L^{p}(\P; H)}
\\
&\leq
C
\max \{ \| (-A)^{ - \max \{ \kappa, \delta \} } \|_{L(H)}, 1 \}
( \| P_{ \H \backslash I} ( - A)^{-\iota} \|_{L(H)} + [ |\theta|_T ]^\alpha ),
\end{split}
\end{equation}
\begin{equation} 
\begin{split}
\label{eq:rates3}
&
 \sup\nolimits_{ s \in [0,T] }
 \| \O_s^{\theta, I} - O_s^{I_n} \|_{\L^{4p}(\P; H_\delta)}
 \\
 &
 \leq
 C
 \max \{ \| (-A)^{ \delta - \max \{ \kappa, \delta \} } \|_{L(H)}, 1 \}
 \big(
 \| P_{ \H \backslash I} ( - A )^{-\iota} \|_{L(H)}
 +
 [ | \theta |_T ]^\alpha
 \big)
 ,
 \end{split}
 \end{equation}
 and
 \begin{equation} 
 \begin{split}
 &
 \sup\nolimits_{ s \in [0,T] }
 \| \O_s^{\theta, I} - O_s^{I_n} \|_{\L^{4p c}(\P; H_\kappa)}
 \\
 &
 \leq
 C
 \max \{ \| (-A)^{ \kappa - \max \{ \kappa, \delta \} } \|_{L(H)}, 1 \}
 ( \| ( - A)^{-\iota} \|_{L(H)} + [ \max\{ T, 1 \} ]^\alpha ).
 \end{split}
 \end{equation}
 Combining~\eqref{eq:TimeRegularityNoise},
 item~\eqref{item:first_num_estimate},
 and
 \eqref{eq:LipsEstimate}--\eqref{eq:last eq} 
hence
gives that
for all 
$ \theta \in \varpi_T $,
$ I \in \mathcal{P}_0(\H) $ 
we have 
that
\begin{align}
\nonumber
& 
\limsup\nolimits_{n \to \infty}
\big(  
\sup\nolimits_{ t \in [0,T] }
\| \y_t^{ \theta, I } - \mathcal{X}_{ t}^n \|_{ \L^p( \P; H ) } 
\big) 
\\
\nonumber
& 
\leq   
C
\max \{ \| (-A)^{ - \max \{ \kappa, \delta \} } \|_{L(H)}, 1 \}
\big( \| P_{ \H \backslash I } ( - A)^{-\iota} \|_{L(H)} + [ |\theta|_T ]^\alpha \big)
+
	\| 
P_{ \H \backslash I } \xi  
\|_{ \L^p( \P; H ) } 
\\
\nonumber
&
\quad 
+
\tfrac{ C [ \max \{ T, 1 \} ]^2
 }{ 1 - \gamma } 
						\exp\!
 \big(  
				\big(		 C  
										+  
												\rho
												( 2 \aa + \|   B \|_{\HS(U,H)}^2 ) 
										\big) T
										\big)
\\
\nonumber
&
\quad
\cdot 
\bigg[ 
\sup\nolimits_{n \in \N}
										\int_0^T
										\E
										\big[
										e^{ \rho \| \y_s^{\theta,I} - \O_s^{\theta,I} + O_s^{ I_n}   
									+ e^{sA} P_{I_n \backslash I} \xi			
											\|_H^2 
									}
										\big]
						\, ds
					\bigg] 
\Big\{
[ | \theta |_T ]^{\gamma - \delta} 
\sup\nolimits_{s\in [0,T]}
\| P_I F ( \y_s^{\theta, I} ) \|_{ \L^{2p} (\P; H_{\gamma-\delta} )}
\\
\nonumber
&
\quad
+
| \upnu |^{- \nicefrac{\alpha}{\varsigma} }
\sup\nolimits_{t \in [0,T]}
\| 
r( \y_t^{\theta, I } )
\|_{\L^{ 4 p \alpha / \varsigma }(\P; \R) }^{ \nicefrac{\alpha}{\varsigma} }
\sup\nolimits_{s \in [0,T]}\| P_I F ( \y_s^{\theta, I } ) \|_{\L^{4 p}(\P; H)}
[ |\theta|_T ]^{ \alpha } 
\\
&
\quad
+
\| P_{ \H \backslash I } (-A)^{-\iota} \|_{L(H)}
\sup\nolimits_{ J \in \mathcal{P}_0(\H) }
\sup\nolimits_{s \in [0,T]}
\| 
P_{ J } F ( \y_s^{\theta, I}) 
\|_{\L^{ 2 p }(\P; H_\iota)}
\\
\nonumber
&
\quad 
+
\Big( 
2
[| \theta |_T]^{\gamma - \delta}
\sup\nolimits_{s\in [0,T]}
\|
P_I
F( \y_s^{\theta, I} )
\|_{\L^{4p}(\P; H)}
+
C [ | \theta |_T ]^\alpha
+
[ | \theta |_T ]^{ \gamma - \delta }
\| \xi \|_{\L^{4p}(\P; H_\gamma)}
\\
\nonumber
& 
\quad
+
C
\max \{ \| (-A)^{ \delta - \max \{ \kappa, \delta \} } \|_{L(H)}, 1 \}
(
\| P_{ \H \backslash I } ( - A )^{-\iota}  \|_{L(H)}
+
[ | \theta |_T ]^\alpha
)
\\
\nonumber
&
\quad
+
\| P_{ \H \backslash I } ( - A)^{\delta - \gamma } \|_{L(H) }
\| \xi \|_{\L^{4p}(\P; H_\gamma)}
\Big)
\big[
1
+
2 
\sup\nolimits_{s\in [0,T]}
\|
\y_s^{\theta, I}
\|_{\L^{4p c}(\P; H_\kappa)} 
\\
\nonumber
&
\quad
+
C
\max \{ \| (-A)^{ \kappa - \max \{ \kappa, \delta \} } \|_{L(H)}, 1 \}
\big(
\| ( - A )^{-\iota} \|_{L(H)}
+
[ \max \{ T, 1 \} ]^\alpha
\big)
+ 
\| \xi \|_{\L^{4p c}(\P; H_\kappa)}
\big]^c
\Big\}
.
\end{align}
This verifies that 
for all 
$ \theta \in \varpi_T $,
$ I \in \mathcal{P}_0(\H) $ 
we have that
\begin{align}
\label{eq:almost done}
\nonumber
&
\limsup\nolimits_{n \to \infty}
\big( 
\sup\nolimits_{ t \in [0,T] }
\| \y_t^{ \theta, I } - \mathcal{X}_{ t}^n \|_{ \L^p( \P; H ) } 
\big) 
\\
\nonumber
&
\leq   
C
\max \{ \| (-A)^{ - \max \{ \kappa, \delta \} } \|_{L(H)}, 1 \}
\| ( - A )^{ - \iota + \min \{ \gamma- \delta, \iota \} } \|_{L(H)}
\| P_{ \H \backslash I } ( - A )^{ - \min \{ \gamma - \delta, \iota \} }  \|_{L(H)}
\\
\nonumber
&
\quad
+
C
\max \{ \| (-A)^{ - \max \{ \kappa, \delta \} } \|_{L(H)}, 1 \}  [ |\theta|_T ]^\alpha
\\
\nonumber
&
\quad
+
	\| 
P_{ \H \backslash I } 
(-A)^{ - \min \{  \gamma- \delta, \iota \} } 
\|_{L(H)}
\| 
( -A )^{ \min \{  \gamma- \delta, \iota \} + \delta - \gamma }
\|_{L(H)}
\| 
\xi  
\|_{ \L^p( \P; H_{ \gamma - \delta } ) } 
\\
\nonumber
&
\quad
+
\tfrac{ C [ \max \{ T, 1 \} ]^2
}{ 1 - \gamma } 
\exp\!
\big(  
\big(		 C  
+  
\rho
( 2 \aa + \|   B \|_{\HS(U,H)}^2 ) 
\big) T
\big)
\big[
1
+
2 
\sup\nolimits_{s\in [0,T]}
\|
\y_s^{\theta, I}
\|_{\L^{4p c}(\P; H_\kappa)} 
\\
\nonumber
&
\quad
+
C
\max \{ \| (-A)^{ \kappa - \max \{ \kappa, \delta \} } \|_{L(H)}, 1 \}
\big(
\| ( - A )^{-\iota} \|_{L(H)}
+
[ \max \{ T, 1 \} ]^\alpha
\big) 
+ 
\| \xi \|_{\L^{4p c}(\P; H_\kappa)}
\big]^c 
\\
\nonumber
&
\quad
\cdot
\bigg[ 
\sup\nolimits_{n \in \N}
\int_0^T
\E
\big[
e^{ \rho \| \y_s^{\theta,I} - \O_s^{\theta,I} + O_s^{ I_n} + e^{sA} P_{I_n \backslash I} \xi \|_H^2 }
\big]
\, ds
\bigg]
\Big\{
\sup\nolimits_{s\in [0,T]}
\| P_I F ( \y_s^{\theta, I} ) \|_{ \L^{2p} (\P; H_{\gamma-\delta} )}
[ | \theta |_T ]^{\gamma - \delta} 
\\
&
\quad 
+
| \upnu |^{- \nicefrac{\alpha}{\varsigma} }
\sup\nolimits_{t \in [0,T]}
\| 
r( \y_t^{\theta, I } )
\|_{\L^{ 4 p \alpha / \varsigma }(\P; \R) }^{ \nicefrac{\alpha}{\varsigma} }
\sup\nolimits_{s \in [0,T]}\| P_I F ( \y_s^{\theta, I } ) \|_{\L^{ 4 p }(\P; H)}
[ |\theta|_T ]^{ \alpha } 
\\
\nonumber
&\quad
+
\sup\nolimits_{ J \in \mathcal{P}_0(\H) }
\sup\nolimits_{s \in [0,T]}
\| 
P_{ J } F ( \y_s^{\theta, I}) 
\|_{\L^{ 2 p }(\P; H_\iota)}
\\
\nonumber
&
\quad
\cdot 
\| ( - A )^{ - \iota + \min \{ \gamma- \delta, \iota \} } \|_{L(H)}
\| P_{ \H \backslash I } ( - A )^{ - \min \{ \gamma - \delta, \iota \} }  \|_{L(H)}
\\
\nonumber
&
\quad 
+
2
\sup\nolimits_{s\in [0,T]}
\|
P_I
F( \y_s^{\theta, I} )
\|_{\L^{4p}(\P; H)}
[| \theta |_T]^{\gamma - \delta}
+
C [ | \theta |_T ]^\alpha
+
\| \xi \|_{\L^{4p}(\P; H_\gamma)}
[ | \theta |_T ]^{ \gamma - \delta }
\\
\nonumber
&
\quad
+
C
\max \{ \| (-A)^{ \delta - \max \{ \kappa, \delta \} } \|_{L(H)}, 1 \}
\| ( - A )^{ - \iota + \min \{ \gamma- \delta, \iota \} } \|_{L(H)}
\| P_{ \H \backslash I } ( - A )^{ - \min \{ \gamma - \delta, \iota \} }  \|_{L(H)}
\\
\nonumber
&
\quad 
+
C 
\max \{ \| (-A)^{ \delta - \max \{ \kappa, \delta \} } \|_{L(H)}, 1 \}
[ | \theta |_T ]^\alpha 
\\
\nonumber
&
\quad
+
\| (-A)^{ \delta - \gamma + \min \{ \gamma - \delta, \iota \} } \|_{L(H)}
\| P_{ \H \backslash I } ( - A)^{ - \min \{ \gamma - \delta, \iota \} } \|_{L(H) }
\| \xi \|_{\L^{4p}(\P; H_\gamma)}
\Big\} 
.
\end{align}
Moreover, note that~\eqref{eq:apply_alekseev_grobner_full_general}
and~\eqref{eq:galerkin_no_rate}
ensure that
for all
$ \theta \in \varpi_T $,
$ I \in \mathcal{P}_0(\H) $,  
$ t \in [0,T] $
we have that
\begin{equation} 
\begin{split}
&
\| 
\y^{\theta, I }_t 
- 
X_t 
\|_{\mathcal{L}^p(\P; H)}
\leq
\limsup\nolimits_{n \to \infty}
\big( 
\sup\nolimits_{ t \in [0,T] }
\| \y_t^{ \theta, I } - \mathcal{X}_{ t}^n \|_{ \L^p( \P; H ) } 
\big) 
.
\end{split}
\end{equation}
Combining
the fact that
$ \xi \in \L^{ 4 p c }( \P; H_\gamma ) $,
\eqref{eq:ExpBound}--\eqref{eq:SupBound2},
and~\eqref{eq:almost done} 
therefore
justifies~\eqref{eq:Theorem general result}.
The proof of Proposition~\ref{proposition:Exact_to_numeric_general}
is hereby completed.
\end{proof}
\section[Strong convergence rates without assuming finite exponential moments]{Strong convergence rates for space-time discrete tamed-truncated exponential Euler-type approximations without assuming finite exponential moments}
\begin{setting} 
\label{setting:StrongApriori}
\sloppy 
Assume Setting~\ref{setting:main},
let
$ T \in (0,\infty) $, 
$ a, b, \upnu \in [0,\infty) $,
$ \varsigma \in (0, \nicefrac{1}{18} ) $,
$ \epsilon \in (0, 1 ] $,
$ \beta \in [0, \nicefrac{1}{2}) $, 
$ \gamma \in [0, \nicefrac{1}{2}  + \beta ) $,
$ B \in  \HS(U, H_\beta) $, 
$ F \in \M( \B(H_\gamma), \B(H) ) $,
$ (D^I_h)_{h\in (0,T], I \in \mathcal{P}_0(\H)} \subseteq \mathcal{B}(H_\gamma) $,
let 
$ (P_I)_{ I \in \mathcal{P}(\H) } \subseteq L(H) $ 
satisfy for all   
$ I \in \mathcal{P}(\H) $,
$ x \in H $ 
that
$ P_I(x)
= \sum_{h \in I} \langle h, x \rangle_H h $,
assume
for all  
$ I \in \mathcal{P}_0(\H) $,  
$ h \in (0,T] $,
$ x \in D_h^I $ 
that
$ D_h^I \subseteq \{ v \in P_I ( H ) \colon \| B \|_{\HS(U, H )} +   \epsilon\|  v \|_H^2 \leq  \upnu   h^{- \varsigma }  \} $,
$ \max\{\| P_I F(x) \|_H,
\| B \|_{\HS(U, H  )}
\} 
\leq \upnu  h^{-\varsigma} $,
and 
$ \left< x,   F(x) \right>_H
\leq a + b \| x \|_H^2 $,
let
$ ( \Omega, \F, \P ) $
be a probability space with a normal filtration
$ ( \f_t )_{t \in [0,T]} $,
let $ (W_t)_{t\in [0,T]} $
be an $ \operatorname{Id}_U $-cylindrical  
$ ( \f_t )_{t\in [0,T]} $-Wiener process, 
let 
$ \xi \in \M(
\f_0, \mathcal{B}(H_\gamma)) $
satisfy
$ \E[ \exp( \epsilon \| \xi \|_{H}^2 ) ]<\infty $,
and let
$ \y^{\theta, I } \colon [0,T] \times \Omega \to P_I( H ) $,
$ \theta \in \varpi_T $,
$ I \in \mathcal{P}_0(\H) $, 
be 
$ ( \f_t )_{t\in [0,T]} $-adapted
stochastic 
processes 
w.c.s.p.\ 
which satisfy for all 
$ \theta \in \varpi_T $,
$ I \in \mathcal{P}_0(\H) $, 
$ t \in [0,T] $
that
$ \y_0^{\theta,I}= P_I \xi $
and
\begin{equation} 
\begin{split}    
\label{eq:AppProcesses}
[
\y_t^{\theta, I}  
]_{\P, \mathcal{B}( P_I(H) ) }
&=   
\big[
e^{(t-\llcorner t \lrcorner_\theta )A}
\y^{\theta, I }_{
	\llcorner t \lrcorner_\theta 
}  
+
\1_{  D_{ |\theta|_T }^I  }
\!
( \y^{\theta, I }_{\llcorner t \lrcorner_\theta })
\,
e^{(t-\llcorner t \lrcorner_\theta )A}
P_I F(
\y^{\theta, I }_{ \llcorner t \lrcorner_\theta  } 
) 
(
t - \llcorner t \lrcorner_\theta 
)  
\big ]_{\P, \mathcal{B}( P_I(H) ) }
\\
&
\quad 
+
\frac{
	\int_{ \llcorner t \lrcorner_\theta  }^t
	\1_{  D_{ |\theta|_T  }^I  }
	\!
	( \y^{\theta, I }_{\llcorner t \lrcorner_\theta })
	\,
	e^{(t- \llcorner t \lrcorner_\theta  )A}
	P_I B 
	\, 
	dW_s
}{
	1 + 
	\| 
	\int_{ \llcorner t \lrcorner_\theta }^t
	P_I B 
	\, 
	dW_s 
	\|_H^2
} 
.
\end{split}
\end{equation}
\end{setting}
\subsection[Finite exponential moments for tamed-truncated Euler-type approximations]{Finite exponential moments for tamed-truncated Euler-type approximations}
\label{subsection:StrongAprioriBoundNumeric}
\sloppy 
In this subsection we establish 
in Corollary~\ref{Corollary:full_discrete_scheme_convergence} 
below uniformly bounded exponential moments for the space-time discrete tamed-truncated exponential Euler-type approximation processes 
$ ( \y_t^{\theta, I} )_{ t \in [0,T] } $, 
$ \theta \in \varpi_T $,
$ I \in \mathcal{P}_0(\H) $,
(see~\eqref{eq:AppProcesses} above). 
Our proof of Corollary~\ref{Corollary:full_discrete_scheme_convergence} 
uses the exponential moment estimate in~\cite[Corollary~3.4]{JentzenPusnik2018Published}. 
We then employ 
Corollary~\ref{Corollary:full_discrete_scheme_convergence} to establish in Corollary~\ref{corollary:uniform_H_moment_bounds} below for every 
$ p \in (0,\infty) $ uniformly bounded 
$ \mathcal{L}^p $-moments for the considered approximation processes.
 Moreover, combining Corollary~\ref{corollary:uniform_H_moment_bounds} 
 with~\cite[Corollary~3.1]{JentzenLindnerPusnik2017b}
 and~\cite[Lemma~3.4]{JentzenLindnerPusnik2017c}
  allows us to establish in Corollary~\ref{corollary:AprioriNumApp} below for every $ p \in (0,\infty) $ strengthened uniformly bounded 
  $ \mathcal{L}^p $-moments for the considered approximation processes.
\begin{lemma}
	\label{lemma:Finite dimensional process Wn}
	Assume Setting~\ref{setting:Notation},
	let
	$ ( H, \langle \cdot, \cdot \rangle_H, \left \| \cdot \right\|_H ) $ 
	be a non-zero separable $ \R $-Hilbert space,
	let 
	$ ( U, \langle \cdot, \cdot \rangle_U, \left \| \cdot \right\|_U ) $ 
	be a separable $ \R $-Hilbert space,
	let
	$ \mathfrak{N} = [1, \dim(H)] \cap \N $, 
	let $ ( h_n )_{ n \in \mathfrak{N} } \subseteq H $
	be an orthonormal basis of $ H $,
	let 
	$ \H = \{ h_n \colon n \in \mathfrak{N} \} $,
	let $ B \colon U \to H $ be a linear function,
	let
	$ (P_I)_{ I \in \mathcal{P}(\H)} \subseteq L(H) $
	satisfy for all
	$ I \in \mathcal{P}(\H) $,
	$ v \in H $ that
	$ P_I(v) = \sum_{h \in I} 
	\langle h, v \rangle_H h $,
	for every
	$ n \in \mathfrak{N} $
	let
	$ \U_n \subseteq [ \operatorname{ker} 
	( P_{ \{ h_1, h_2, \ldots, h_n \} } B ) ]^\perp $
	be an orthonormal basis of 
	$ [ \operatorname{ker} ( P_{ \{ h_1, h_2, \ldots, h_n \} } B ) ]^\perp $, 
	assume for all 
	$ n \in ( \mathfrak{N} \backslash \{\sup(\mathfrak{N})\} ) $
	that
	$ \U_n \subseteq \U_{n+1} $, 
	and let
	$ ( \mathfrak{P}_I )_{I \in \mathcal{P}( \cup_{ n \in \mathfrak{N} } \U_n ) } \subseteq L( U ) $ 
	satisfy
	for all 
	$ I \in \mathcal{P}( \cup_{ n \in \mathfrak{N} } \U_n ) $,
	$ u \in U $ 
	that
	$ \mathfrak{P}_I u
	=
	\sum_{ \mathfrak{u} \in I }
	\langle \mathfrak{u}, u \rangle_U 
	\mathfrak{u} $.
	Then  
	there exists 
	$ \Gamma \colon \mathcal{P}_0(\H) \to \mathfrak{N} $
	which satisfies that
	\begin{enumerate}[(i)]
		\item \label{eq:Exist embedding}
		we have for all $ I \in \mathcal{P}_0(\H) $ that	
		$ [ \operatorname{ker}(P_I B) ]^\perp 
		\subseteq 
		\mathfrak{P}_{
			\U_{ \Gamma(I) }
		}(U) $,
		\item \label{item:Characterize} we have for all $ n \in \mathfrak{N} $ that
		$ \Gamma( \{ h_1, h_2, \ldots, h_n\} )
		\leq n $,
		and
		\item\label{item:Main stochastic integral} we have for all $ I \in \mathcal{P}_0(\H) $ 
		that
		$ P_I B 
		=  
		P_I B \mathfrak{ P}_{ \U_{\Gamma(I)} } 
		$.
	\end{enumerate} 
\end{lemma}
\begin{proof}[Proof of Lemma~\ref{lemma:Finite dimensional process Wn}]
	Throughout this proof
	let 
	$ \Gamma \colon \mathcal{P}_0(\H ) \to \N \cup \{ \infty \} $
	satisfy for all 
	$ I \in \mathcal{P}_0(\H ) $  
	that
	\begin{equation}
	\label{eq:first step}
	\Gamma( I ) 
	=  
	\inf 
	(
	\{
	n \in \mathfrak{N} 
	\colon 
	[ \operatorname{ker}(P_IB)]^\perp
	\subseteq 
	\mathfrak{P}_{ \U_n }(U)
	\}
	\cup \{ \infty \} ).
	\end{equation}
	Observe that for all $ n \in \mathfrak{N} $ we have that
	\begin{equation} 
	\label{eq:Necessary equality}
	[ \operatorname{ker}( P_{ \{h_1, h_2, \ldots, h_n \} } B ) ]^\perp 
	=
	\mathfrak{P}_{  \U_n } ( U ) 
	.
	\end{equation}  
	Moreover, note that for every 
	$ I \in \mathcal{P}_0(\H) $
	there exists $ n \in \mathfrak{N} $
	such that
	$ I \subseteq \{ h_1, h_2, \ldots, h_n\} $.
	This ensures that 
	for every $ I \in \mathcal{P}_0(\H) $
	there exists $ n \in \mathfrak{N} $
	such that
	\begin{equation} 
	\operatorname{ker}(
	P_{\{h_1, h_2, \ldots,h_n\}}B)
	\subseteq \operatorname{ker}(P_IB)
	.
	\end{equation} 
	This 
	and~\eqref{eq:Necessary equality} 
	give that
	for every 
	$ I \in \mathcal{P}_0(\H) $
	there exists
	$ n \in \mathfrak{N} $
	such that
	\begin{equation}
	\label{eq:item 1}
	[ \operatorname{ker}(P_I B) ]^\perp 
	\subseteq
	\mathfrak{P}_{\U_n}(U) 
	.
	\end{equation}
	Therefore,
	we obtain that
	for all
	$ I \in \mathcal{P}_0(\H) $
	we have that
	$ \Gamma(I) \in \mathfrak{N} $.
	Combining this, 
	\eqref{eq:first step}, 
	and
	\eqref{eq:Necessary equality}
	justifies items~\eqref{eq:Exist embedding}
	and~\eqref{item:Characterize}.
	Moreover, note that
	item~\eqref{eq:Exist embedding} 
	gives that
	for all 
	$ I \in \mathcal{P}_0(\H ) $
	we have that 
	\begin{equation} 
	P_I B 
	=  
	P_I B \mathfrak{ P}_{ \U_{ \Gamma(I) } }
	.
	\end{equation} 
	This gives
	item~\eqref{item:Main stochastic integral}.
	The proof of Lemma~\ref{lemma:Finite dimensional process Wn}
	is hereby completed.
\end{proof}
\begin{corollary}
	\label{corollary:Finite dimensional process Wn}
	\sloppy 
	Assume Setting~\ref{setting:Notation},
	let
	$ ( H, \langle \cdot, \cdot \rangle_H, \left \| \cdot \right\|_H ) $ 
	be a non-zero separable $ \R $-Hilbert space,
	let 
	$ ( U, \langle \cdot, \cdot \rangle_U, \left \| \cdot \right\|_U ) $ 
	be a separable $ \R $-Hilbert space, 
	let
	$ \mathfrak{N} = [1, \dim(H)] \cap \N $,
	let $ (h_N)_{N \in \mathfrak{N}} \subseteq H $
	be an orthonormal basis of $ H $,
	let 
	$ T \in (0, \infty) $,
	$ B \in \HS(U,H) $,
	let
	$ \mathbb{B} \in L(H, U) $
	satisfy for all
	$ v \in H $, $ u \in U $ that
	$ \langle B u, v \rangle_H
	=
	\langle u, \mathbb{B} v \rangle_U $, 
	let
	$ (P_N)_{ N \in \mathfrak{N} } \subseteq L(H) $
	satisfy for all
	$ N \in \mathfrak{N} $,
	$ v \in H $ 
	that
	$ P_N(v) = \sum_{n = 1}^N
	\langle h_n, v \rangle_H h_n $,
	for every $ N \in \mathfrak{N} $
	let 
	$ \U_N \subseteq 
	[ \operatorname{ker}( P_N B ) ]^\perp $
	be an orthonormal basis of
	$ [ \operatorname{ker}( P_N B ) ]^\perp $,
	assume for all 
	$ N \in ( \mathfrak{N} \backslash \{ \sup( \mathfrak{N} ) \} ) $
	that
	$ \U_N \subseteq \U_{N+1} $,
	let
	$ ( \mathfrak{P}_N )_{ N \in \mathfrak{N} } \subseteq L(U) $
	satisfy for all $ N \in \mathfrak{N} $,
	$ u \in U $
	that
	$ \mathfrak{P}_N u = \sum_{ \mathfrak{u} \in \U_N }
	\langle \mathfrak{u}, u \rangle_U \mathfrak{u} $,
	let $ ( \Omega, \F, \P ) $
	be a probability space,
	let
	$ (W_t)_{t\in [0,T]} $
	be an
	$ \operatorname{Id}_U $-cylindrical  
	Wiener process,
	and
	for every $ N \in \mathfrak{N} $ let 
	$ W^N \colon [0,T] \times \Omega \to P_N(H) $
	be a stochastic process
	w.c.s.p.\
	which
	satisfies for all  
	$ t \in [0,T] $ 
	that
	$ [ W_t^N ]_{\P, \B(P_N(H)) } = \int_0^t P_N B \, dW_s $.
	Then
	\begin{enumerate}[(i)]
		\item \label{item:computable covariance}
		we have for all $ N \in \mathfrak{N} $
		that
		$ P_N B \mathfrak{P}_N = P_N B $, 
		\item \label{item:W} we have for all $ N \in \mathfrak{N} $, 
		$ t \in [0,T] $ that
		$ [ W_t^N ]_{\P, \B(P_N(H)) } = \int_0^t P_N B  \mathfrak{P}_N \, dW_s $,
		and
		\item \label{item:Wiener process} we have for all 
		$ N \in \mathfrak{N} $
		that
		$ ( W_t^N )_{ t \in [0,T] } $
		is a
		$ ( P_N B \mathbb{B} |_{ P_N(H) } ) $-Wiener
		process. 
	\end{enumerate}
\end{corollary}
\begin{proof}[Proof of Corollary~\ref{corollary:Finite dimensional process Wn}]
	Throughout this proof 
	let
	$ ( \f_t )_{ t \in [0,T] } $
	be the normal filtration generated by $ (W_t)_{ t \in [0,T] } $.
	Observe that
	Lemma~\ref{lemma:Finite dimensional process Wn}
	(applies with
	$ H = H $,
	$ U = U $,
	$ \mathfrak{N} = \mathfrak{N} $,
	$ h_n = h_n $,
	$ B = B $, 
	$ P_{ \{ h_1, h_2, \ldots, h_n \} } = P_n $,
	$ \U_n = \U_n $, 
	$ \mathfrak{P}_{ \U_n } 
	= 
	\mathfrak{P}_n $
	for 
	$ n \in \mathfrak{N} $
	in the setting of
	Lemma~\ref{lemma:Finite dimensional process Wn})
	ensures that 
	for all
	$ N \in \mathfrak{N} $, 
	$ t \in [0,T] $
	we have 
	that
	\begin{equation}
	\label{eq:Change 1}
	P_N B 
	=
	P_N B 
	\mathfrak{P}_N 
	.
	\end{equation} 
	This 
	justifies items~\eqref{item:computable covariance}
	and~\eqref{item:W}. 
	Combining~\eqref{eq:Change 1}
	and, e.g, \cite[Lemma~3.2]{JentzenPusnik2018Published}
	(applies with
	$ H = P_N(H) $,
	$ U = U $,
	$ T = T $,
	$ Q = \operatorname{Id}_U $,
	$ (\Omega, \F, \P, ( \F_t )_{ t \in [0,T] } ) 
	=
	(\Omega, \F, \P, ( \f_t )_{ t \in [0,T] } ) $,
	$ (W_t)_{ t \in [0,T]} 
	=
	(W_t)_{ t \in [0,T]} $,
	$ R = ( U \ni u \mapsto P_N B (u) \in P_N(H) ) $,
	$ (\tilde W_t)_{ t \in [0,T] } = (W_t^N)_{ t \in [0,T] } $
	for 
	$ N \in \mathfrak{N} $
	in the setting
	of~\cite[Lemma~3.2]{JentzenPusnik2018Published})
	justifies
	item~\eqref{item:Wiener process}.
	The proof of Corollary~\ref{corollary:Finite dimensional process Wn}
	is hereby completed.
\end{proof} 
\begin{lemma}
	\label{lemma:Continuous adapted modification}
	Assume Setting~\ref{setting:main},
	let
	$ T \in (0,\infty) $, 
	$ \theta \in \varpi_T $,
	$ \beta \in [0, \nicefrac{1}{2}) $, 
	$ \gamma \in [0, \nicefrac{1}{2} + \beta ) $,
	$ B \in  \HS(U, H_\beta) $, 
	$ F \in \M( \B(H_\gamma), \B(H) ) $,
	$ D \in \mathcal{B}(H_\gamma) $,
	let
	$ ( \Omega, \F, \P ) $
	be a probability space with a normal filtration
	$ ( \f_t )_{t \in [0,T]} $,
	let $ (W_t)_{t\in [0,T]} $
	be an $ \operatorname{Id}_U $-cylindrical  
	$ ( \f_t )_{t\in [0,T]} $-Wiener process, 
	let
	$ \xi \in \M(
	\f_0, \mathcal{B}(H_\gamma)) $, 
	$ I \in \mathcal{P}_0(\H) $,
	$ P \in L(H) $ 
	satisfy for all  
	$ x \in H $
	that
	$ P(x)
	= \sum_{h \in I} \langle h, x \rangle_H h $,
	let
	$ \mathcal{W} \colon[0,T]\times\Omega\to P(H) $ 
	be a stochastic process 
	w.c.s.p.\ 
	which satisfies
	for all 
	$ t \in [0,T] $
	that 
	$ [ \mathcal{W}_t ]_{\P,\B(P(H))} = \int_0^t P B \, dW_s $,
	and let
	$ \y \colon [0,T] \times \Omega \to P( H ) $ 
	be 
	an
	$ ( \f_t )_{t\in [0,T]} $-adapted
	stochastic process
	which satisfies for all
	$ t \in [0,T] $
	that
	$ \y_0 = P \xi $
	and
	\begin{equation} 
	\begin{split}    
	\label{eq:AppProcesses 2}
	[
	\y_t   
	]_{\P, \mathcal{B}( P(H) ) }
	&=   
	\big[
	e^{(t-\llcorner t \lrcorner_\theta )A}
	\y_{
		\llcorner t \lrcorner_\theta 
	}  
	+
	\1_{  D } 
	( \y_{\llcorner t \lrcorner_\theta })
	\,
	e^{(t-\llcorner t \lrcorner_\theta )A}
	P  F(
	\y_{ \llcorner t \lrcorner_\theta  } 
	) 
	(
	t - \llcorner t \lrcorner_\theta 
	)  
	\big ]_{\P, \mathcal{B}( P (H) ) }
	\\
	&
	\quad 
	+
	\frac{
		\int_{ \llcorner t \lrcorner_\theta  }^t
		\1_{  D    } 
		( \y_{\llcorner t \lrcorner_\theta })
		\,
		e^{(t- \llcorner t \lrcorner_\theta  )A}
		P B 
		\, 
		dW_s
	}{
		1 + 
		\| 
		\int_{ \llcorner t \lrcorner_\theta }^t
		P  B 
		\, 
		dW_s 
		\|_H^2
	} 
	.
	\end{split}
	\end{equation}
	Then   
	there exists
	an $ ( \f_t )_{ t \in [0,T] } $-adapted
	stochastic process
	$ \X \colon [0,T] \times \Omega \to P(H) $
	w.c.s.p.\
	which satisfies
	that
	\begin{enumerate}[(i)]
		\item \label{item:initial} we have that $ \X_0 = P \xi $,
		\item \label{item:base recursion} we have for all $ t \in [0,T] $ that
		\begin{equation} 
		\begin{split} 
		\!\!\! 
		\X_t   
		&=   
		e^{(t-\llcorner t \lrcorner_\theta )A}
		\X_{
			\llcorner t \lrcorner_\theta 
		} 
		+
		\1_{  D } 
		( \X_{\llcorner t \lrcorner_\theta })
		\,
		e^{(t-\llcorner t \lrcorner_\theta )A}
		\bigg[ 
		P  F(
		\X_{ \llcorner t \lrcorner_\theta  } 
		) 
		(
		t - \llcorner t \lrcorner_\theta 
		)  
		+
		\frac{
			( \mathcal{W}_t - \mathcal{W}_{ \llcorner t \lrcorner_\theta } )
		}{
			1 + 
			\| 
			\mathcal{W}_t - \mathcal{W}_{ \llcorner t \lrcorner_\theta } 
			\|_H^2
		} 
		\bigg]
		,
		\end{split}
		\end{equation}
		\item \label{item:Same relation}
		we have for all $ t \in [0,T] $ that 
		\begin{equation} 
		\begin{split}    
		\label{eq:AppProcesses 3}
		[
		\X_t   
		]_{\P, \mathcal{B}( P(H) ) }
		&=   
		\big[
		e^{(t-\llcorner t \lrcorner_\theta )A}
		\X_{
			\llcorner t \lrcorner_\theta 
		}  
		+
		\1_{  D } 
		( \X_{\llcorner t \lrcorner_\theta })
		\,
		e^{(t-\llcorner t \lrcorner_\theta )A}
		P  F(
		\X_{ \llcorner t \lrcorner_\theta  } 
		) 
		(
		t - \llcorner t \lrcorner_\theta 
		)  
		\big ]_{\P, \mathcal{B}( P (H) ) }
		\\
		&
		\quad 
		+
		\frac{
			\int_{ \llcorner t \lrcorner_\theta  }^t
			\1_{  D    } 
			( \X_{\llcorner t \lrcorner_\theta })
			\,
			e^{(t- \llcorner t \lrcorner_\theta  )A}
			P B 
			\, 
			dW_s
		}{
			1 + 
			\| 
			\int_{ \llcorner t \lrcorner_\theta }^t
			P  B 
			\, 
			dW_s 
			\|_H^2
		} 
		,
		\end{split}
		\end{equation}
		and
		\item \label{item:modified}
		we have 
		for all
		$ t \in [0,T] $ that
		$ \P( \X_t = \y_t ) = 1 $.
	\end{enumerate}
\end{lemma}
\begin{proof}[Proof of Lemma~\ref{lemma:Continuous adapted modification}]
	Throughout this proof let
	$ \X \colon[0,T]\times\Omega\to P(H) $
	be the stochastic process
	which satisfies for all
	$ t \in [0,T] $ 
	that
	$ \X_0 = P\xi $
	and
	\begin{equation} 
	\begin{split} 
	\label{eq:Introduce mathcal X} 
	\X_t   
	&=   
	e^{(t-\llcorner t \lrcorner_\theta )A}
	\X_{
		\llcorner t \lrcorner_\theta 
	} 
	+
	\1_{  D } 
	( \X_{\llcorner t \lrcorner_\theta })
	\,
	e^{(t-\llcorner t \lrcorner_\theta )A}
	\bigg[ 
	P  F(
	\X_{ \llcorner t \lrcorner_\theta  } 
	) 
	(
	t - \llcorner t \lrcorner_\theta 
	)  
	+
	\frac{
		( \mathcal{W}_t - \mathcal{W}_{ \llcorner t \lrcorner_\theta } )
	}{
		1 + 
		\| 
		\mathcal{W}_t - \mathcal{W}_{ \llcorner t \lrcorner_\theta } 
		\|_H^2
	} 
	\bigg]
	.
	\end{split}
	\end{equation}
	Note that the fact that for all
	$ s \in [0,T] $ 
	we have that
	$ ( [s, T] \times H \ni ( t, x ) \mapsto e^{(t - s)A} x \in P(H) )
	\in \mathcal{C}([s,T], P(H) ) $,
	the fact that
$ \mathcal{W} $
	has 
	continuous sample paths, 
	and~\eqref{eq:Introduce mathcal X}
	ensure that
	$ \mathcal{X} $ has
	continuous sample paths. 
	Moreover, observe that the assumption that
	$ ( \f_t )_{t \in [0,T]} $ is a normal filtration
	and the assumption that
	for all
	$ t \in [0,T] $
	we have that 
	$ [ \mathcal{W}_t ]_{\P,\B(P(H))} = \int_0^t P B \, dW_s $
	yield that
	$ \mathcal{W} $
	is $ (\f_t)_{ t \in [0,T] } $-adapted.
	Combining this, 
	\eqref{eq:Introduce mathcal X},
	the fact that
	$ \xi \in \M(
	\f_0, \mathcal{B}(P(H))) $, 
	and the assumption that
	$ ( \f_t )_{t \in [0,T]} $ is a normal filtration
	therefore
	yields that
	$ \mathcal{X} $
	is $ (\f_t)_{ t \in [0,T] } $-adapted.
	This,
	\eqref{eq:Introduce mathcal X},
	and the fact that
$ \mathcal{X} $ has 
continuous sample paths 
	justify items~\eqref{item:initial}
	and~\eqref{item:base recursion}.
	Next note that the fact that
	$ \mathcal{X} $
	is $ (\f_t)_{ t \in [0,T] } $-adapted
	ensures that 
	for all $ t \in [0,T] $ we have that 
	\begin{equation}
	\begin{split}
	\bigg[ 
	\1_{  D } 
	( \X_{\llcorner t \lrcorner_\theta })
	\,
	e^{(t-\llcorner t \lrcorner_\theta )A}
	\frac{ 
		( \mathcal{W}_t - \mathcal{W}_{ \llcorner t \lrcorner_\theta } )
	}{
		1 + 
		\| 
		\mathcal{W}_t - \mathcal{W}_{ \llcorner t \lrcorner_\theta } 
		\|_H^2
	}  
	\bigg]_{\P, \B( P( H ) )}
	=
	\frac{
		\int_{ \llcorner t \lrcorner_\theta  }^t
		\1_{  D    } 
		( \X_{\llcorner t \lrcorner_\theta })
		\,
		e^{(t- \llcorner t \lrcorner_\theta  )A}
		P B 
		\, 
		dW_s
	}{
		1 + 
		\| 
		\int_{ \llcorner t \lrcorner_\theta }^t
		P  B 
		\, 
		dW_s 
		\|_H^2
	}
	.
	\end{split}
	\end{equation}
	Combining this and~\eqref{eq:Introduce mathcal X} 
	illustrates that for all
	$ t \in [0,T] $
	we have that 
	\begin{equation} 
	\begin{split}     
	\label{eq:recursion again}
	[
	\X_t   
	]_{\P, \mathcal{B}( P(H) ) }
	&=   
	\big[
	e^{(t-\llcorner t \lrcorner_\theta )A}
	\X_{
		\llcorner t \lrcorner_\theta 
	}   
	+
	\1_{  D } 
	( \X_{\llcorner t \lrcorner_\theta })
	\,
	e^{(t-\llcorner t \lrcorner_\theta )A}
	P  F(
	\X_{ \llcorner t \lrcorner_\theta  } 
	) 
	(
	t - \llcorner t \lrcorner_\theta 
	)  
	\big ]_{\P, \mathcal{B}( P (H) ) }
	\\
	&
		\quad 
	+
	\frac{
		\int_{ \llcorner t \lrcorner_\theta  }^t
		\1_{  D    } 
		( \X_{\llcorner t \lrcorner_\theta })
		\,
		e^{(t- \llcorner t \lrcorner_\theta  )A}
		P B 
		\, 
		dW_s
	}{
		1 + 
		\| 
		\int_{ \llcorner t \lrcorner_\theta }^t
		P  B 
		\, 
		dW_s 
		\|_H^2
	} 
	.
	\end{split}
	\end{equation}
	This justifies item~\eqref{item:Same relation}.
	Moreover, observe 
	that~\eqref{eq:AppProcesses 2},
	\eqref{eq:recursion again},
	and
	item~\eqref{item:initial}  
	assure that 
	for all $ t \in [0,T] $
	we have that
	\begin{equation} 
	\P(
	\X_t =
	X_t 
	) 
	=
	1 
	.
	\end{equation}
	This justifies item~\eqref{item:modified}.
	The proof of Lemma~\ref{lemma:Continuous adapted modification} is hereby completed.
\end{proof}
\begin{corollary}
	\label{Corollary:full_discrete_scheme_convergence}
	Assume Setting~\ref{setting:StrongApriori}.
	Then 
	\begin{equation}
	\begin{split} 
	\label{eq:exponential_moments} 
	& 
	\sup\nolimits_{ \theta \in \varpi_T }
	\sup\nolimits_{ I \in \mathcal{P}_0(\H) } 
	\sup\nolimits_{t\in [0,T]}
	\E\!\left[ 
	\exp \! \left( 
	\frac{    
		\epsilon  
	}
	{
		e^{ 2(b + \| B \|_{\HS(U, H )}^2 ) T }
	} \| \y^{ \theta, I}_t \|_H^2
	\right) 
	\right]
	< \infty.
	\end{split}
	\end{equation}
\end{corollary}
\begin{proof}[Proof of Corollary~\ref{Corollary:full_discrete_scheme_convergence}]
		\sloppy 
	Throughout this proof 
	let
	$ c = 2\max\{
	\epsilon a,
	\epsilon \| B \|_{\HS(U, H )}, \epsilon, \upnu, 1 \} $,
	let
	$ \mathfrak{N} = [1, \dim(H)] \cap \N $,
	let 
	$ h_n \in H $,
	$ n \in \mathfrak{N} $,  
	satisfy for all 
	$ m, n \in \N $
	that 
	$ h_m \neq h_n $
	and
	$ \mathbb{H} = \{ h_N \colon N \in \mathfrak{N} \} $,
	let
	$ \U_1 \subseteq [ \operatorname{ker} 
	( P_{ \{ h_1 \} } B ) ]^\perp $
	be an orthonormal basis of 
	$ [ \operatorname{ker} ( P_{ \{ h_1 \} } B ) ]^\perp $,
	for every 
	$ n \in ( [2, \infty) \cap \mathfrak{N} ) $ 
	let
	$ \U_n \subseteq [ \operatorname{ker} 
	( P_{ \{ h_1, h_2, \ldots, h_n \} } B ) ]^\perp $
	be an orthonormal basis of 
	$ [ \operatorname{ker} ( P_{ \{ h_1, h_2, \ldots, h_n \} } B ) ]^\perp $
	with 
	$ \U_{n-1} \subseteq \U_n $, 
		let 
		$ \mathcal{U} \subseteq U $
		be an orthonormal basis of $ U $
		with
		$ \mathcal{U} \supseteq \cup_{ n \in \N } \U_n $, 
	let
	$ \mathfrak{P}_I \in L( U ) $,
	$ I \in \mathcal{P}( \mathcal{U} ) $, 
	satisfy
	for all 
	$ I \in \mathcal{P}( \mathcal{U} ) $,
	$ u \in U $ 
	that
	$ \mathfrak{P}_I u
	=
	\sum_{ \mathfrak{u} \in I }
	\langle \mathfrak{u}, u \rangle_U 
	\mathfrak{u} $,
	and let
	$ \mathfrak{X}^{\theta, I, J} \colon [0,T] \times \Omega 
	\to P_I(H) $,
	$ \theta \in \varpi_T $,
	$ I \in \mathcal{P}_0(\H) $,
	$ J \in \mathcal{P}_0(\mathcal{U}) $,
	be $ ( \f_t )_{ t \in [0,T] } $-adapted
	stochastic processes
	which satisfy for all
	$ \theta \in \varpi_T $,
	$ I \in \mathcal{P}_0(\H) $,
	$ J \in \mathcal{P}_0(\mathcal{U}) $,
	$ t \in [0,T] $ that
	$ \mathfrak{X}_0^{\theta, I, J}= P_I \xi $
	and
	\begin{equation} 
	\begin{split}     
	[
	\mathfrak{X}_t^{\theta, I, J}  
	]_{\P, \mathcal{B}( P_I(H) ) }
	&=   
	\big[
	e^{(t-\llcorner t \lrcorner_\theta )A}
	\mathfrak{X}^{\theta, I, J}_{
		\llcorner t \lrcorner_\theta 
	}  
	+
	\1_{  D_{ |\theta|_T }^I  }
	\!
	( \mathfrak{X}^{\theta, I, J }_{\llcorner t \lrcorner_\theta })
	\,
	e^{(t-\llcorner t \lrcorner_\theta )A}
	P_I F(
	\mathfrak{X}^{\theta, I, J }_{ \llcorner t \lrcorner_\theta  } 
	) 
	(
	t - \llcorner t \lrcorner_\theta 
	)  
	\big ]_{\P, \mathcal{B}( P_I(H) ) }
	\\
	&
	\quad 
	+
	\frac{
		\int_{ \llcorner t \lrcorner_\theta  }^t
		\1_{  D_{ |\theta|_T  }^I  }
		\!
		( \mathfrak{X}^{\theta, I, J }_{\llcorner t \lrcorner_\theta })
		\,
		e^{(t- \llcorner t \lrcorner_\theta  )A}
		P_I B \mathfrak{P}_J
		\, 
		dW_s
	}{
		1 + 
		\| 
		\int_{ \llcorner t \lrcorner_\theta }^t
		P_I B  \mathfrak{P}_J
		\, 
		dW_s 
		\|_H^2
	}  
	.
	\end{split}
	\end{equation}
	Observe that
	Lemma~\ref{lemma:Continuous adapted modification}
	(applies with
	$ T = T $,
	$ \theta = \theta $,
	$ \beta = \beta $,
	$ \gamma = \gamma $,
	$ B = B \mathfrak{P}_J $,
	$ F = F $,
	$ D = D_{ | \theta |_T }^I $,
	$ ( \Omega, \F, \P ) = ( \Omega, \F, \P ) $,
	$ ( \f_t )_{ t \in [0,T] } = ( \f_t )_{ t \in [0,T] } $,
	$ ( W_t )_{ t \in [0,T] } = ( W_t )_{ t \in [0,T] } $,
	$ \xi = \xi $,
	$ I = I $,
	$ P = P_I $,
	$ \y^{\theta, I} = \mathfrak{X}^{\theta, I, J} $
	for
	$ \theta \in \varpi_T $,
	$ I \in \mathcal{P}_0(\H) $,
	$ J \in \mathcal{P}_0(\mathcal{U}) $
	in the setting of 
	Lemma~\ref{lemma:Continuous adapted modification})
	ensures that there exist
	$ ( \f_t )_{ t \in [0,T] } $-adapted
	stochastic processes 
	$ \X^{\theta, I, J} \colon [0,T] \times \Omega \to P_I(H) $,
	$ \theta \in \varpi_T $,
	$ I \in \mathcal{P}_0(\H) $,
	$ J \in \mathcal{P}_0(\mathcal{U}) $, 
	w.c.s.p.\
	which satisfy for all 
	$ \theta \in \varpi_T $,
	$ I \in \mathcal{P}_0(\H) $,
	$ J \in \mathcal{P}_0(\mathcal{U}) $,
	$ t \in [0, T] $ that
	$ \X_0^{\theta, I, J}= P_I \xi $
	and
	\begin{equation} 
	\begin{split}   
	\label{eq: introduce new X}  
	[
	\X_t^{\theta, I, J}  
	]_{\P, \mathcal{B}( P_I(H) ) }
	&=   
	\big[
	e^{(t-\llcorner t \lrcorner_\theta )A}
	\X^{\theta, I, J}_{
		\llcorner t \lrcorner_\theta 
	}  
	+
	\1_{  D_{ |\theta|_T }^I  }
	\!
	( \X^{\theta, I, J }_{\llcorner t \lrcorner_\theta })
	\,
	e^{(t-\llcorner t \lrcorner_\theta )A}
	P_I F(
	\X^{\theta, I, J }_{ \llcorner t \lrcorner_\theta  } 
	) 
	(
	t - \llcorner t \lrcorner_\theta 
	)  
	\big ]_{\P, \mathcal{B}( P_I(H) ) }
	\\
	&
	\quad 
	+
	\frac{
		\int_{ \llcorner t \lrcorner_\theta  }^t
		\1_{  D_{ |\theta|_T  }^I  }
		\!
		( \X^{\theta, I, J }_{\llcorner t \lrcorner_\theta })
		\,
		e^{(t- \llcorner t \lrcorner_\theta  )A}
		P_I B \mathfrak{P}_J
		\, 
		dW_s
	}{
		1 + 
		\| 
		\int_{ \llcorner t \lrcorner_\theta }^t
		P_I B  \mathfrak{P}_J
		\, 
		dW_s 
		\|_H^2
	}  
	.
	\end{split}
	\end{equation}
	Next note that
	Lemma~\ref{lemma:Finite dimensional process Wn}
	(applies with
	$ H = H $,
	$ U = U $,
	$ \mathfrak{N} = \mathfrak{N} $,
	$ h_n = h_n $,
	$ \H = \H $,
	$ B = (U \ni u \mapsto B(u) \in H ) $,
	$ P_I = P_I $,
	$ \U_n = \U_n $,
	$ \mathfrak{P}_J = \mathfrak{P}_J $
	for 
	$ I \in \mathcal{P}(\H) $,
	$ n \in \mathfrak{N} $,
	$ J \in \mathcal{P}( \cup_{ n \in \mathfrak{N} } \U_n ) $
	in the setting of
	Lemma~\ref{lemma:Finite dimensional process Wn})
	assures that
	there exists 
	$ \Gamma \colon \mathcal{P}_0(\H) \to \mathfrak{N} $
	which satisfies
	for all  
	$ I \in \mathcal{P}_0(\H) $ 
	that 
	\begin{equation} 
	\label{eq:identity some}
	P_I B   
	=  
	P_I B \mathfrak{ P}_{ \U_{ \Gamma(I) } } 
	.
	\end{equation}
	Combining~\eqref{eq:AppProcesses}
	and~\eqref{eq:identity some} 
	illustrates that for all 
	$ \theta \in \varpi_T $,
	$ I \in \mathcal{P}_0(\H) $, 
	$ t \in [0,T] $ 
	we have that
	\begin{equation} 
	\begin{split}   
	\label{eq:CorrectScheme} 
	[
	\y_t^{\theta, I}  
	]_{\P, \mathcal{B}( P_I(H) ) }
	&=   
	\big[
	e^{(t-\llcorner t \lrcorner_\theta )A}
	\y^{\theta, I }_{
		\llcorner t \lrcorner_\theta 
	}  
	+
	\1_{  D_{ |\theta|_T  }^I  }
	\!
	( \y^{\theta, I }_{\llcorner t \lrcorner_\theta })
	\,
	e^{(t-\llcorner t \lrcorner_\theta )A}
	P_I F(
	\y^{\theta, I }_{ \llcorner t \lrcorner_\theta  } 
	) 
	(
	t - \llcorner t \lrcorner_\theta 
	)  
	\big ]_{\P, \mathcal{B}( P_I(H) ) }
	\\
	&
	\quad 
	+
	\frac{
		\int_{ \llcorner t \lrcorner_\theta  }^t
		\1_{  D_{ |\theta|_T  }^I  }
		\!
		( \y^{\theta, I }_{\llcorner t \lrcorner_\theta })
		\,
		e^{(t- \llcorner t \lrcorner_\theta  )A}
		P_I B \mathfrak{ P}_{ \U_{ \Gamma(I) } }
		\, 
		dW_s
	}{
		1 + 
		\| 
		\int_{ \llcorner t \lrcorner_\theta }^t
		P_I B \mathfrak{ P}_{ \U_{ \Gamma(I) } }
		\, 
		dW_s 
		\|_H^2
	} 
	.
	\end{split}
	\end{equation}
	This and~\eqref{eq: introduce new X} 
	ensure that for all
	$ \theta \in \varpi_T $,
	$ I \in \mathcal{P}_0(\H) $,
	$ t \in [0,T] $
	we have that
	\begin{equation}
	\label{eq:important equal}
	\y^{\theta, I} 
	=
	\X^{\theta, I, \U_{\Gamma(I)} }
	.
	\end{equation}
	In addition, note that for all 
	$ I \in \mathcal{P}_0(\H) $,
	$ h \in (0,T] $
	we have that
	\begin{equation}
	\label{eq:Needed1}
	D_h^I 
	\subseteq 
	\{ v \in P_I( H ) \colon 
	\| B \|_{\HS(U, H )} +   \epsilon\|  v \|_H^2 \leq  \upnu   h^{- \varsigma }  \}
	\subseteq 
	\{ v \in H \colon \| B \|_{\HS(U, H )} +   \epsilon\|  v \|_H^2 \leq  c  h^{- \varsigma }  \}
	.
	\end{equation}
	Furthermore, observe
	that
	for all 
		$ I \in \mathcal{P}_0(\H) $,
	$ h \in (0,T] $,   
	$ x \in D_h^I $
	we have that
	\begin{equation} 
	\begin{split} 
	\label{eq:Needed20}
	\max\{\| P_I F(x) \|_H,
	\| 
	P_I B \mathfrak{ P}_{ \U_{ \Gamma(I) } }
	\|_{\HS(U, H  )}
	\} 
	\leq 
	\max\{\| P_I F(x) \|_H,
	\|  B \|_{\HS(U, H  )}
	\} 
	\leq \upnu  h^{-\varsigma}
	\leq 
	c  h^{-\varsigma}
	.
	\end{split} 
	\end{equation}
	Moreover,  note that
	the fact that 
	for all 
	$ I \in \mathcal{P}_0(\H) $,
	$ h \in (0,T] $
	we have that
	$ D_h^I \subseteq P_I(H) $
	illustrates
	that
	for all 
		$ I \in \mathcal{P}_0(\H) $,
	$ h \in (0,T] $,
	$ x \in D_h^I $
	we have 
	that
	\begin{equation}
	\label{eq:Needed3}
	\left< x, P_I F(x) \right>_H
	=
	\left< x, F(x) \right>_H
	\leq 
	a + b \| x \|_H^2
	.
	\end{equation}
	Combining
	this and~\eqref{eq:CorrectScheme}--\eqref{eq:Needed20}  
	with~\cite[Corollary~3.4]{JentzenPusnik2018Published}
	(applies with
	$ H = H $,
	$ U = U $,
	$ \H = \H $,
		$ \mathbb{U} = \mathcal{U} $,
	$ \lambda = \values $,
	$ A = A $,
	$ T = T $,
	$ \gamma = \gamma $,
	$ \delta = \varsigma $,
	$ ( \Omega, \F, \P, ( \F_t )_{ t \in [0,T] } )
	=
	( \Omega, \F, \P, ( \f_t )_{ t \in [0,T] } ) $,
	$ (W_t)_{ t \in [0,T] }
	=
	(W_t)_{ t \in [0,T] } $, 
	$ \xi = \xi $,
	$ F = F $,
	$ B = ( H_\gamma \ni x \mapsto B \in \HS(U,H) ) $,
	$ D_h^I = D_h^I $,
	$ P_I = P_I $, 
	$ \hat P_J = 
	\mathfrak{P}_J $, 
	$ \vartheta = \| B \|_{\HS(U,H)}^2 $,
	$ b_1 = a $,
	$ b_2 = b $,
	$ \varepsilon = \epsilon $,
	$ \varsigma = \varsigma $,
	$ c = c $,
	$ Y^{\theta, I, J } = \X^{\theta, I, J} $
	for
	$ \theta \in \varpi_T $,
	$ I \in \mathcal{P}_0(\H) $, 
	$ J \in \mathcal{P}_0( \mathcal{U} ) $,
	$ h \in (0,T] $   
	in the setting of~\cite[Corollary~3.4]{JentzenPusnik2018Published}) 
	yields that
	\begin{equation}
	\begin{split} 
	\label{eq:exponential_moments2} 
	& 
	\sup\nolimits_{ \theta \in \varpi_T }
	\sup\nolimits_{ I \in \mathcal{P}_0(\H) } 
	\sup\nolimits_{t\in [0,T]}
	\E\!\left[ 
	\exp \! \left(
	\frac{    
		\epsilon \| \y^{ \theta, I}_t \|_H^2
	}
	{
		e^{ 2(b + \epsilon \| B \|_{\HS(U, H )}^2 ) t }
	} 
	\right) 
	\right]
	< \infty.
	\end{split}
	\end{equation}
	In addition, note that
	the fact that $ \epsilon \leq 1 $
	assures that for all $ t \in [0,T] $ we have that
	\begin{equation}
	\frac{    
		\epsilon  
	}
	{
		e^{ 2(b + \epsilon \| B \|_{\HS(U, H )}^2 ) t }
	} 
	\geq
	\frac{    
		\epsilon  
	}
	{
		e^{ 2(b + \epsilon \| B \|_{\HS(U, H )}^2 ) T }
	}
	\geq
	\frac{    
		\epsilon  
	}
	{
		e^{ 2(b + \| B \|_{\HS(U, H )}^2 ) T }
	}
	.
	\end{equation}
	This and~\eqref{eq:exponential_moments2}  
	justify~\eqref{eq:exponential_moments}.
	The proof of Corollary~\ref{Corollary:full_discrete_scheme_convergence}
	is hereby completed.
\end{proof}
\begin{corollary}
	\label{corollary:uniform_H_moment_bounds}
	Assume Setting~\ref{setting:StrongApriori}
	and
	let  
	$ p \in (0,\infty) $. 
	Then we have that
	\begin{equation}
	\sup\nolimits_{I \in \mathcal{P}_0(\H)}
	\sup\nolimits_{\theta \in \varpi_T}
	\sup\nolimits_{t\in [0,T]}
	\| \y_t^{ \theta, I } \|_{\L^p(\P; H)}
	< \infty .
	\end{equation}
\end{corollary}
\begin{proof}[Proof of Corollary~\ref{corollary:uniform_H_moment_bounds}]
	Throughout this proof let
	$ N \in 
	( [ \frac{p}{2} , \frac{p}{2}  + 1) \cap \N ) $.
	Observe that 
	Corollary~\ref{Corollary:full_discrete_scheme_convergence}
	yields that there exists $ M \in [0,\infty) $
	such that for all 
	$ \theta \in \varpi_T $,
	$ I \in \mathcal{P}_0(\H) $,
	$ t \in [0,T] $,
	$ \varepsilon \in (0,
		\epsilon 
		\exp( - 2(b + \| B \|_{\HS(U, H )}^2 ) T )
	] $
	we have that
	\begin{equation}
	\begin{split} 
	\label{eq:First step estimate}
	&  
	\E \big[ 
	\exp \! \big( 
		\varepsilon \| \y^{ \theta, I}_t \|_H^2 
	\big) 
	\big]
	\leq M.
	\end{split}
	\end{equation}
	In addition, note that
	Young's inequality ensures that
	for all  
	$ x \in (0, \infty) $ 
	we have that
	\begin{equation}
	\begin{split} 
	x^{ \nicefrac{p}{2} }
	&=  
	x^{(N-1)(N- ( p/2 ) )} 
	x^{N( ( p/2 ) -N+1)}  \leq  (N-\tfrac{p}{2}) x^{N-1} + (\tfrac{p}{2}-N+1) x^N
	\\
	&\leq N x^{N-1} +x^N
	= (N!)  \Big(   \tfrac{  x^{N-1}  }{ (N-1)! }  +  \tfrac{  x^N } { N! } \Big)
	\leq (N!) e^x
	.
	\end{split}
	\end{equation}
Therefore, we obtain that for all 
$ \theta \in \varpi_T $,
$ I \in \mathcal{P}_0(\H) $,
$ t \in [0,T] $ 
we have that
\begin{equation}
\E\!\left[ 
 \big|
\varepsilon \| \y^{ \theta, I}_t \|_H^2 
 \big|^{ \nicefrac{p}{2} }
\right] 
\leq 
(N!)
\E\big[ 
\exp \! \big( 
\varepsilon \| \y^{ \theta, I}_t \|_H^2 
\big)
\big] 
.
\end{equation}
This 
and~\eqref{eq:First step estimate}  
	give that
	there exists $ M \in [0, \infty) $
	such that 
	for all 
	$ \theta \in \varpi_T $,
	$ I \in \mathcal{P}_0(\H) $,
	$ t \in [0,T] $,
	$ \varepsilon \in (0,
	\epsilon 
	\exp( - 2(b + \| B \|_{\HS(U, H )}^2 ) T )
	] $
	we have that
	\begin{equation}
	\begin{split} 
	\left(
	\E\!\left[ 
	\big| 
	  \varepsilon  \| \y_t^{\theta, I} \|_H^2  
	\big|^{\nicefrac{p}{2}} 
	\right]
	\right)^{ \nicefrac{2}{p} }
	\leq
	(
	( N! )
	M
	)^{ \nicefrac{2}{p}  }
	.
	\end{split}
	\end{equation}
	This concludes the proof of Corollary~\ref{corollary:uniform_H_moment_bounds}.
\end{proof} 
\begin{corollary}
	\label{corollary:AprioriNumApp}
	Assume Setting~\ref{setting:StrongApriori},
	let  
	$ p \in (0, \infty) $,
	$ \eta_1 \in [0, \nicefrac{1}{2}  + \beta  ) $,
	$ \eta_2 \in [ \eta_1, \nicefrac{1}{2} + \beta  ) $,  
	$ \iota \in [\eta_2, \nicefrac{1}{2} + \beta  ) $,   
	$ \alpha_1 \in [0, 1 - \eta_1 ) $,
	$ \alpha_2 \in [0, 1 - \eta_2 ) $,  
	and assume that 
	$ \E [ \| \xi \|_{ H_{ \iota } }^{4 \max \{p, 1\} } ] < \infty $
	and
	\begin{equation}
	\label{eq:AssumptionF}
	\Big[
	\sup\nolimits_{ v \in H_{ \max \{ \gamma, \eta_2 \} }  }
	\tfrac{ \| F(v) \|_{H  } }
	{ 1 + \| v \|_{H_{ \eta_2 } }^2 }
	\Big]
	+
	\Big[
	\sup\nolimits_{ v \in H_{ \max \{ \gamma, \eta_1 \} }  }
	\tfrac{ \| F(v) \|_{H_{ - \alpha_2 } } }
	{ 1 + \| v \|_{ H_{\eta_1} }^2 }
	\Big]
	+
	\Big[
	\sup\nolimits_{ v \in H_{\gamma} } 
	\tfrac{ \| F(v) \|_{H_{-\alpha_1} } }{ 1 + \| v \|_H^2 } 
	\Big]
	< \infty
	.
	\end{equation}
	Then  
	we have that
	\begin{equation}
	\label{eq:AprioriBound}
	\sup\nolimits_{ \theta \in \varpi_T }
	\sup\nolimits_{I \in \mathcal{P}_0(\H)} \sup\nolimits_{ t \in [0,T] }
	\| \y_t^{ \theta, I }  \|_{\L^{ p }(\P; H_\iota )}
	<
	\infty 
	. 
	\end{equation} 
\end{corollary}
\begin{proof}[Proof of Corollary~\ref{corollary:AprioriNumApp}]
	Throughout this proof let 
	$ ( \mathbb{G}_t )_{ t \in [0, T] } $
	be the normal filtration
	generated by 
	$ ( W_t )_{ t \in [0,T] } $,
	let $ \mathbb{U} $ be an orthonormal 
	basis of $ U $, 
	and let
	$ \O^{\theta, I} \colon [0,T] 
	\times \Omega \to P_I(H) $, 
		$ \theta \in \varpi_T $,
	$ I \in \mathcal{P}_0(\H) $,
	be stochastic processes 
	which satisfy for all 
		$ \theta \in \varpi_T $, 
	$ I \in \mathcal{P}_0(\H) $,
	$ t \in [0,T] $ 
	that 
	\begin{equation} 
	\O_t^{\theta, I} 
	= 
	\y_t^{\theta, I} 
	- 
	\bigg( 
	e^{tA} P_I \xi
	+ 
	\int_0^t
	\1_{ D_{|\theta|_T}^I }
	\!
	( \y_{\llcorner s \lrcorner_\theta}^{\theta, I} )
	\,
	e^{(t-\llcorner s \lrcorner_\theta)A}  
	P_I F( \y_{\llcorner s \lrcorner_\theta}^{\theta, I} )
	\,
	ds 
	\bigg) 
	.
	\end{equation} 
	Observe that~\cite[Corollary~3.1]{JentzenLindnerPusnik2017b}
	(applies with 
	$ H = H $,
	$ U = U $,
	$ \H = \H $,
	$ \values = \values $, 
	$ A = A $, 
	$ \beta = \beta $,
	$ T = T $, 
	$ ( \Omega, \F, \P, ( \f_t )_{ t \in [0,T]}  ) =
	( \Omega, \F, \P, ( \mathbb{G}_t )_{ t \in [0,T]} ) $,
	$ ( W_t )_{ t \in [0,T] } 
	=
	( W_t )_{ t \in [0,T] } $,  
	$ B = B $,
	$ \mathbb{U} = \mathbb{U} $,
	$ P_I = P_I $,  
	$ \hat P_{ \mathbb{U} } = \operatorname{Id}_U $,
	$ \chi^{ \theta, I, \mathbb{U} } =
	( [0,T] \times \Omega \ni (t, \omega)
	\mapsto 
	\1_{ D_{ |\theta|_T}^I } \! ( \y_t^{ \theta, I } ( \omega ) ) 
	\in [0,1] 
	) $,
	$ \O^{ \theta, I, \mathbb{U} } = \O^{ \theta, I } $, 
	$ p = \max \{p, 1\} $, 
	$ \gamma = \iota $ 
	for  
	$ \theta \in \varpi_T $,
	$ I \in \mathcal{P}_0(\H) $ 
	in the setting of~\cite[Corollary~3.1]{JentzenLindnerPusnik2017b}) 
	yields that
	\begin{equation} 
	\label{eq:Apriori1}
	\sup\nolimits_{\theta \in \varpi_T }
	\sup\nolimits_{ I \in \mathcal{P}_0(\H) }
	\sup\nolimits_{ t \in [0,T] }
	\| \O_t^{ \theta, I } \|_{ \L^{4 \max \{p, 1\} }( \P; H_{ \iota } )} 
	< \infty 
	.
	\end{equation}
	Next note that
	Corollary~\ref{corollary:uniform_H_moment_bounds}
	(applies with
	$ p = 8 \max \{p, 1\} $
	in the setting of Corollary~\ref{corollary:uniform_H_moment_bounds})
	verifies that
	\begin{equation}
	\label{eq:Apriori2}
	\sup\nolimits_{ \theta \in \varpi_T }
	\sup\nolimits_{I \in \mathcal{P}_0(\H)} \sup\nolimits_{ t \in [0,T] }
	\| \y_t^{\theta, I}
	\|_{\L^{8 \max \{p, 1\} }(\P; H ) }
	< \infty  
	.
	\end{equation}
	Combining
	this, 
	\eqref{eq:AssumptionF},
	and~\eqref{eq:Apriori1} 
	with, e.g., \cite[Lemma~3.4]{JentzenLindnerPusnik2017c}
	(applies with
	$ H = H $,
	$ \H = \H $,
	$ \values = \values $,
	$ ( \Omega, \F, \P) = ( \Omega, \F, \P ) $,
	$ T = T $,
	$ \beta = \nicefrac{1}{2}  + \beta $,
	$ \gamma = \gamma $,
	$ \xi = ( \Omega \ni \omega \mapsto P_I(\xi(\omega)) \in H_{ \nicefrac{1}{2} + \beta } ) $,
	$ F = ( H_\gamma \ni x \mapsto \1_{ D_{ | \theta |_T }^I }(x) P_I F(x) \in H ) $,
	$ \kappa = ( [0,T] \ni t \mapsto \llcorner t \lrcorner_\theta \in [0,T] ) $, 
	$ Z = ( [0,T] \times \Omega \ni (t,\omega)
	\mapsto \y_{ \llcorner t \lrcorner_\theta }^{ \theta, I }
	(\omega) \in H_\gamma ) $,
	$ O = ( [0,T] \times \Omega \ni (t, \omega ) 
	\mapsto 
	\O_t^{ \theta, I}(\omega) 
	\in H_{ \nicefrac{1}{2} + \beta } ) $, 
	$ Y = ( [0,T] \times \Omega \ni (t, \omega) \mapsto 
	\y_t^{ \theta, I }( \omega ) \in H ) $,
	$ p = \max \{p, 1\} $,
	$ \rho = \eta_1 $,
	$ \eta = \eta_2 $,
	$ \iota = \iota $,
	$ \alpha_1 = \alpha_1 $,
	$ \alpha_2 = \alpha_2 $
	for  
	$ \theta \in \varpi_T $,
	$ I \in \mathcal{P}_0(\H) $ 
	in the setting of~\cite[Lemma~3.4]{JentzenLindnerPusnik2017c})
	yields that
	\begin{equation} 
	\sup\nolimits_{ \theta \in \varpi_T }
	\sup\nolimits_{I \in \mathcal{P}_0(\H)} \sup\nolimits_{ t \in [0,T] }
	\| \y_t^{ \theta, I }  \|_{\L^{ \max \{p, 1\} }(\P; H_\iota )}
	<
	\infty 
	. 
	\end{equation}
	H\"older's inequality
	therefore	 
	justifies~\eqref{eq:AprioriBound}.
	The proof of 
	Corollary~\ref{corollary:AprioriNumApp}
	is hereby completed.
\end{proof}
\subsection[Strong error estimates for tamed-truncated Euler-type approximations]{Strong error estimates for tamed-truncated Euler-type approximations}
\label{subsection:StrongRates}
In this subsection we establish the main result
of this article
in Theorem~\ref{theorem:Exact_to_numeric} below.
To do so, we first
prove
an elementary exponential
moment estimate in
Lemma~\ref{lemma:chi2_estimate}. 
Combining 
Corollaries~\ref{Corollary:full_discrete_scheme_convergence}--\ref{corollary:AprioriNumApp},
Lemma~\ref{lemma:chi2_estimate},
and~\cite[Corollaries~3.2--3.4]{JentzenLindnerPusnik2017b} 
allows us to 
apply
Proposition~\ref{proposition:Exact_to_numeric_general}
to derive
in 
Theorem~\ref{theorem:Exact_to_numeric}
strong convergence rates
for the numerical approximations
$ ( \y_t^{\theta,I})_{ t \in [0,T]} $,
$ \theta \in \varpi_T $,
$ I \in \mathcal{P}_0(\H) $,
(see~\eqref{eq:Scheme} below)
for a general class
of 
semilinear
SPDEs
with additive noise
and a possibly non-globally monotone nonlinearity.
Moreover, in
Corollary~\ref{corollary:Exact_to_numeric}
we briefly present and prove a simplified version of 
Theorem~\ref{theorem:Exact_to_numeric}.
\begin{lemma}
\label{lemma:chi2_estimate}
Assume Setting~\ref{setting:main},
let 
$ T \in (0, \infty) $,
$ B \in \HS( U, H ) $,
let
$ ( P_I )_{ I \in \mathcal{P}_0(\H ) } \subseteq L(H) $
satisfy for all  
$ I \in \mathcal{P}_0(\H) $,
$ v \in H $ that
$ P_I(v) = \sum_{ h \in I } 
\langle h, v \rangle_H h $,
	let $ ( \Omega, \F, \P ) $ be a probability space, 
	and
	let 
	$ (W_t)_{t\in [0,T]} $
	be an $ \operatorname{Id}_U $-cylindrical
	Wiener process.
	Then we have 
	for all $ t \in [0, T] $
	with $ 2 t  \| B \|_{\HS(U, H)}^2 < 1 $
	that
	\begin{equation}
	\sup\nolimits_{ I \in \mathcal{P}_0(\H) }
	\E \big[ e^{\| \int_0^t e^{(t-s)A} P_I B   
		\,  
		dW_s \|_H^2} \big] 
	\leq 
	\tfrac{2}{ 1 
			- 
			4t^2 
			\| B\|_{\HS(U, H)}^4 
		}
	.
	\end{equation}
\end{lemma}
\begin{proof}[Proof of Lemma~\ref{lemma:chi2_estimate}]
Throughout this proof let
$ \mathbb{U} \subseteq U $ 
be an orthonormal basis of $ U $,
let $ ( \f_t )_{ t \in [0,T] } $
be the normal filtration generated by $ (W_t)_{ t \in [0,T] } $, 
and let
$ O^I \colon [0,T] \times \Omega \to P_I(H) $,
$ I \in \mathcal{P}_0(\H) $,
be $ ( \f_t )_{ t \in [0,T] } $-adapted
stochastic processes
w.c.s.p.\ 
which satisfy for all 
$ I \in \mathcal{P}_0(\H ) $,
$ t \in [0,T] $ 
that
$ [ O_t^I ]_{ \P, \B( P_I(H) ) }= \int_0^t P_I e^{(t-s)A} B \, dW_s $.
Observe that for all 
$ I \in \mathcal{P}_0(\H) $,
$ t \in [0,T] $ 
we have that
\begin{equation}
[ O_t^I ]_{ \P, \B( P_I(H)) } 
=
\bigg[ 
\int_0^t A O_s^I \, ds
\bigg ]_{ \P, \B( P_I(H)) }
+
\int_0^t P_I B \, dW_s
.
\end{equation}
	 It\^o's formula therefore yields that for all 
	 $ p \in [2, \infty) $,
	$ I \in \mathcal{P}_0(\H ) $,
	$ t \in [0,T] $ 
	we have that
	\begin{multline}
	\label{eq:ItoNoise}
	[  \| O_t^I \|_H^p ]_{ \P, \B( \R ) } 
	=
	\bigg[ 
	\int_0^t p \| O_s^I \|_H^{p-2}
	\langle O_s^I, A O_s^I \rangle_H \, ds
	\bigg ]_{ \P, \B( \R ) }
	+
	\int_0^t p \| O_s^I \|_H^{p-2}
	\langle O_s^I, B \, dW_s \rangle_H
	\\
	+
	\bigg[ 
	\tfrac{1}{2}
	\int_0^t
	\sum_{ { \bf u }  \in \mathbb{U} }
	\big[
	p 
	\| O_s^I \|_H^{p-2}
	\| B { \bf u }  \|_H^2
	+
	p(p-2)
	\1_{ \{ O_s^I \neq 0 \} }
	\| O_s^I \|_H^{p-4}
	| \langle O_s^I, B { \bf u }  \rangle_H |^2
	\big]
	\, ds
	\bigg ]_{ \P, \B( \R ) }
	.
	\end{multline}
	Moreover, note that
	the
	Burkholder-Davis-Gundy-type inequality in Da Prato \& Zabczyk~\cite[Lemma~7.7]{dz92} 
	verifies that
	for all  
	$ p \in [2, \infty) $,
	$ I \in \mathcal{P}_0(\H ) $,
	$ t \in [0,T] $ 
	we have that
	\begin{equation}
	\begin{split}
	&
	\int_0^t 
	\E \big[
	\| O_s^I \|_H^{2(p-2)} 
	\| ( U \ni u \mapsto \langle O_s^I, B( u ) \rangle_H \in \R ) \|_{\HS(U,\R)}^2
	\big]
	\, d s
	\\
	&
	\leq
	\int_0^t
	\E \big[  
	\| O_s^I \|_H^{2(p-1)}
	\| B \|_{\HS(U,H)}^2
	\big]
	\, d s
	=
	\| B \|_{\HS(U,H)}^2
	\int_0^t
	\| O_s^I \|_{ \L^{2(p-1) }(\P; H ) }^{2(p-1)}
	\, ds
	\\
	&
	\leq 
	\| B \|_{\HS(U,H)}^2
	\int_0^t
	[ (p-1) ( 2p - 3 ) ]^{(p-1)} 
	\bigg[  
	\int_0^s
	\| P_I e^{(s-u)A} B \|_{ \HS(U,H) }^2
	\, du
	\bigg]^{ (p-1) }
	\, ds
	\\
	&
	\leq 
	\| B \|_{\HS(U,H)}^2
	\int_0^t
	[ (p-1) ( 2p - 3 ) ]^{(p-1)}
	\bigg[  
	\int_0^s
	\| e^{(s-u)A} \|_{L(H)}^2
	\| B \|_{ \HS(U,H) }^2
	\, du
	\bigg]^{ (p-1) }
	\, ds
	\\
	&
	\leq 
	\| B \|_{\HS(U,H)}^{ 2 p }
	[ (p-1) ( 2p - 3 ) ]^{(p-1)}
	\int_0^t
	\bigg[  
	\int_0^s 
	\, du
	\bigg]^{ ( p - 1 ) }
	\, ds
	< \infty .
	\end{split}
	\end{equation}
	Combining~\eqref{eq:ItoNoise},
	the fact that
	for all
	$ x \in H_1 $
	we have that
	$ \langle x, Ax \rangle_H
	=
	-
	\| x \|_{ H_{ \nicefrac{1}{2} } }^2
	\leq 0 $,
	Cauchy-Schwarz's inequality,
	and
	Tonelli's theorem
	therefore gives that for all 
	$ p \in [2, \infty) $,
	$ I \in \mathcal{P}_0(\H ) $,
	$ t \in [0,T] $
	we have that
	\begin{equation}
	\begin{split}
	&
	\E[ \| O_t^I \|_H^p ]  
	\leq 
	\tfrac{1}{2}
	\E
	\bigg[ 
	\int_0^t
	\sum_{ { \bf u }  \in \mathbb{U} }
	\big[
	p 
	\| O_s^I \|_H^{p-2}
	\| B { \bf u }  \|_H^2
	+
	p(p-2)
	\1_{ \{ O_s^I \neq 0 \} }
	\| O_s^I \|_H^{p-2}
	\| B { \bf u }  \|_H^2
	\big]
	\, ds
	\bigg] 
	\\
	&
	=
	\tfrac{1}{2}
	\| B \|_{ \HS(U,H)}^2
	\E
	\bigg[ 
	\int_0^t 
	\big[
	p 
	\| O_s^I \|_H^{p-2} 
	+
	p(p-2)
	\1_{ \{ O_s^I \neq 0 \} }
	\| O_s^I \|_H^{p-2} 
	\big]
	\, ds
	\bigg] 
	\\
	&
	=
	\tfrac{1}{2}
	\| B \|_{ \HS(U,H)}^2 
	\int_0^t 
	\E
	\big[
	p 
	\| O_s^I \|_H^{p-2} 
	+
	p(p-2) 
	\| O_s^I \|_H^{p-2} 
	\big]
	\, ds 
	=
	\tfrac{ p (p-1) \| B \|_{ \HS(U,H)}^2 }{2}
	\int_0^t 
	\E
	\big[ 
	\| O_s^I \|_H^{p-2}  
	\big]
	\, ds 
	.
	\end{split}
	\end{equation} 
	This ensures that for all 
	$ I \in \mathcal{P}_0(\H) $,
	$ n \in \N $,
	$ t_0 \in [0,T] $ 
	we have that
	\begin{equation}
	\begin{split}
	\label{eq:same_exp}
	\E [ \| O_{ t_0 }^I \|_H^{ 2 n } ]
	&
	\leq
	\tfrac{ 2 n (2 n -1) \| B \|_{ \HS(U,H)}^2 }{2} 
	\int_0^{ t_0 }
	\E
	\big[ 
	\| O_s^I \|_H^{2(n-1)}  
	\big]
	\, ds 
	\\
	&
	\leq 
	\tfrac{ ( 2 n )! \| B \|_{ \HS(U,H)}^{2 n} }{2^n}  
	\int_0^{ t_0 }
	\int_0^{t_1}
	\cdots
	\int_0^{t_{n-1}}
	\, dt_n
	\cdots
	dt_2
	\, 
	dt_1
	=
	\tfrac{ ( 2 n )! \| B \|_{ \HS(U,H)}^{2 n} }{2^n n!}  
	t_0^n
	.
	\end{split}
	\end{equation}
	Moreover, note that for all $ x \in [0, \infty) $ we have that
	$ 
	e^x \leq 2  \smallsum_{n=0}^\infty
	\tfrac{x^{2n}}{(2n)!} 
	$
	(see, e.g., Hutzenthaler et al.\  \cite[Lemma~2.4]{HutzenthalerJentzenWang2017Published}).
	Combining this, \eqref{eq:same_exp},
	Tonelli's theorem,
	and the fact that
	for all
	$ n \in \N $
	we have that 
	$ (4n)! \leq 2^{4n} [(2n)!]^2 $
	gives that for all  
	$ I \in \mathcal{P}_0(\H) $,
	$ t \in [0, \infty) $
	with 
	$ 2 t  \| B \|_{\HS(U, H)}^2 < 1 $ 
	we have that
	\begin{equation}
	\begin{split}
	\E\Big[ 
	e^{ \| \int_0^t e^{(t-s)A} P_I B  
		\, 
		dW_s \|_H^2 }
	\Big]
	&
	=
	\E\big[
	e^{ \| O_t^I \|_H^2 }
	\big]
	\leq
	2 \,
	\E \Big[
		\sum\nolimits_{n=0}^\infty
		\tfrac{   \| 
			O_t^I \|_H^{4 n} }{(2n)!} 
			\Big]
	=
	2 
	\sum\nolimits_{n=0}^\infty
	\tfrac{ \E [ \| 
	O_t^I  \|_H^{4 n}]}{(2 n)!} 
	\\
	&
	\leq
	2 
	\sum\nolimits_{n=0}^\infty
	\tfrac{ (4n)! \| B\|_{\HS(U, H)}^{4n} 
		t^{2n}}{ [(2n)!]^2 2^{2n}   } 
	\leq
		2 
		\sum\nolimits_{n=0}^\infty
		 2^{2n} \| B\|_{\HS(U, H)}^{4n} t^{2n}
	\\
	&
	=
	2 
	\sum\nolimits_{n=0}^\infty
	\big( 
	4
	\| B\|_{\HS(U, H)}^4    
	t^2 
	\big)^n
	=  
	\tfrac{2}{ 1 
		- 
		4t^2 
		\| B\|_{\HS(U, H)}^4 
	}
	.
	\end{split}
	\end{equation}
	This
	concludes 
	the proof of Lemma~\ref{lemma:chi2_estimate}.
\end{proof}
\begin{theorem} 
	\label{theorem:Exact_to_numeric}
Assume Setting~\ref{setting:main},
\sloppy 
let  
$ T, \upnu \in (0,\infty) $, 
$ \varsigma \in (0, \nicefrac{1}{18} ) $, 
$ \aa \in [0, \infty) $,
$ C, c, p \in [1,\infty) $,  
$ \beta \in [0, \nicefrac{1}{2}) $,
$ \gamma \in [ 2 \beta, \nicefrac{1}{2} + \beta ) \cap (0,\infty) $,
$ \delta \in ( \gamma - \nicefrac{1}{2}, \gamma) \cap [0, \infty) $, 
$ \kappa \in [0, \gamma] \cap [0, \nicefrac{1}{2} + \beta - \gamma + \delta ) $,
$ \eta_0 = 0 $,
$ \sigma, \nu, \eta_1 \in [0, \nicefrac{1}{2} + \beta ) $,
$ \eta_2 \in [\eta_1, \nicefrac{1}{2} + \beta ) $,
$ \alpha_1 \in [0, 1 - \eta_1 ) $, 
$ \alpha_2 \in [0, 1 - \eta_2 ) $, 
$ \alpha_3 = 0 $,
$ B \in  \HS(U, H_\beta) $,
$ \epsilon \in (0, \exp( - 2 ( \aa + \| B \|_{\HS(U, H )}^2 ) T ) ] $,
$ \varepsilon \in [0, 
\frac{1 }{ 16 p  }
   \exp( - 2 ( \aa + \| B \|_{\HS(U, H )}^2 ) T )
   \min \{ \epsilon \exp( -2( \aa + \| B \|_{\HS(U, H )}^2 ) T )  ,
\nicefrac{ 1 }{ ( 8 
	\max \{ \| B \|_{\HS(U, H )}^2, 1 \}
	\max \{ T, 1 \}
	)^2
}
\} ) $,
$ F \in \mathcal{C}^1(  H_\gamma, H ) $,
$ r \in \M ( \mathcal{B}(H_\gamma), \mathcal{B}( [0,\infty) ) ) $, 
$ (D_h^I)_{h\in (0,T], I \in \mathcal{P}_0(\H)} \subseteq \mathcal{B}(H_\gamma) $,
let
$ \Phi \colon H \to [0, \infty) $
be a function,
let 
$ ( P_I )_{ I \in \mathcal{P}(\H) } \subseteq L(H) $
satisfy for all  
$ I \in \mathcal{P}(\H) $,
$ x \in H $
that
$ P_I(x)
= \sum_{h \in I} \langle h, x \rangle_H h $,
assume for all
$ I \in \mathcal{P}_0(\H) $, 
$ h \in (0,T] $ 
that 
$ D_h^I = \{ v \in P_I( H ) \colon r(v) \leq  \upnu  h^{- \varsigma }  \} $
and
$ ( P_I(H) \ni v \mapsto \Phi(v) \in [0, \infty) )
\in 
\mathcal{C}( P_I(H), [0, \infty) ) $,
assume for all  
$ I \in \mathcal{P}_0(\H) $, 
$ h \in (0,T] $,
$ x \in D_h^I $
that
$ \max\{\| P_I F(x) \|_H,
\| B \|_{\HS(U, H  )}
\} 
\leq \upnu h^{-\varsigma} $,
assume for all    
$ I \in \mathcal{P}_0(\H) $,
$ x, y \in P_I( H ) $ 
that   
$ \| B \|_{\HS(U, H )} + \epsilon \| x \|_H^2 \leq r(x) \leq 
C ( 1 + \| x \|_{ H_\nu }^2 ) $,
$ \langle x, F(x) \rangle_H \leq \aa ( 1 + \| x \|_H^2 ) $,
$ \langle F'(x) y, y \rangle_H  \leq 
( \varepsilon \| x \|_{H_{\nicefrac{1}{2}}}^2 
+ C  ) \| y \|_H^2 
+ \|y \|_{H_{\nicefrac{1}{2}}}^2 $,
$ \| P_I( F( x ) - F( y ) ) \|_H
\leq
C \| x - y \|_{H_\delta} 
( 1 + \| x \|_{H_\kappa}^c  + \| y \|_{H_\kappa}^c ) $,
$ \langle x, A x + F(x+y) \rangle_H
\leq
\Phi(y) ( 1 + \| x \|_H^2 ) $,   
and
\begin{equation}
\Big[
\sup\nolimits_{ J \in \mathcal{P}_0(\H) } 
\sup\nolimits_{ v \in P_J(H) }
\tfrac{ \| P_J F(v) \|_{H_{ \gamma - \delta }} }{
	1 + \| v \|_{H_{\sigma}}^2	
} 
\Big]
+
\sum_{i=0}^2
\Big[
\sup\nolimits_{ v \in H_{ \max \{ \gamma, \eta_i \} } } 
\tfrac{ \| F(v) \|_{H_{ - \alpha_{i+1} } } }{ 1 + \| v \|_{H_{\eta_i}}^2 } 
\Big]
<
\infty
,
\end{equation}
let
$ ( \Omega, \F, \P ) $
be a probability space with a normal filtration
$ ( \f_t )_{t \in [0,T]} $,
let $ (W_t)_{t\in [0,T]} $
be an $ \operatorname{Id}_U $-cylindrical  
$ ( \f_t )_{t\in [0,T]} $-Wiener process, 
let
$ \xi \in \L^{ 32 p c \max \{  ( \gamma - \delta ) / \varsigma, 1 \} }(\P|_{ \f_0 } ;H_{ \max \{ \eta_2, \sigma, \nu, \gamma \} } ) $
satisfy 
$ \E[ \exp ( \epsilon \| \xi \|_H^2 ) ] < \infty $,
let 
$ X \colon [0,T] \times \Omega \to H_\gamma $
be an 
$ (\f_t )_{t\in [0,T]} $-adapted
stochastic process
w.c.s.p.\
which satisfies for all $ t \in [0,T] $ that
\begin{equation}
[ X_t ]_{\P, \B(H_\gamma)} 
= 
\bigg[
e^{tA} \xi 
+
\int_0^t e^{(t-s)A} F(X_s) \, ds
\bigg]_{\P, \B(H_\gamma)} 
+
\int_0^t e^{(t-s)A} B \, dW_s,
\end{equation}
and
let
$ \y^{\theta, I } \colon [0,T] \times \Omega \to P_I( H ) $,
$ \theta \in \varpi_T $,
$ I \in \mathcal{P}_0(\H) $, 
be 
$ ( \f_t )_{t\in [0,T]} $-adapted
stochastic 
processes 
which satisfy for all 
$ \theta \in \varpi_T $,
$ I \in \mathcal{P}_0(\H) $, 
$ t \in [0,T] $
that 
$ \y_0^{\theta,I}= P_I \xi $
and
\begin{equation} 
\begin{split}  
\label{eq:Scheme}  
[
\y_t^{\theta, I}  
]_{\P, \mathcal{B}( P_I(H) ) }
&=   
\big[
e^{(t-\llcorner t \lrcorner_\theta )A}
\y^{\theta, I }_{
	\llcorner t \lrcorner_\theta 
}  
+
\1_{  D_{ |\theta|_T  }^I  }
\!
(\y^{\theta, I }_{\llcorner t \lrcorner_\theta })
\,
e^{(t-\llcorner t \lrcorner_\theta )A}
P_I F(
\y^{\theta, I }_{ \llcorner t \lrcorner_\theta  } 
) 
(
t - \llcorner t \lrcorner_\theta 
)  
\big ]_{\P, \mathcal{B}( P_I(H) ) }
\\
&
\quad 
+
\frac{
	\int_{ \llcorner t \lrcorner_\theta  }^t
	\1_{  D_{ |\theta|_T  }^I  }
	\!
	(\y^{\theta, I }_{\llcorner t \lrcorner_\theta })
	\,
	e^{(t- \llcorner t \lrcorner_\theta  )A}
	P_I B 
	\, 
	dW_s
}{
	1 + 
	\| 
	\int_{ \llcorner t \lrcorner_\theta }^t
	P_I B 
	\, 
	dW_s 
	\|_H^2
} 
.
\end{split}
\end{equation}  
Then there exists 
$ \mathfrak{c} \in \R $
such that for all  
$ \theta \in \varpi_T $, 
$ I \in \mathcal{P}_0(\H) $
we have that
\begin{equation}
\label{eq:Main convergence estimate}
\sup\nolimits_{ t \in [0,T] }
\| X_t - \y^{\theta, I }_t \|_{\L^p(\P; H)}
\leq
\mathfrak{c} \big( \| P_{\H \backslash I } ( -A )^{ \delta - \gamma } \|_{L(H)}
+
[| \theta |_T]^{ \gamma - \delta  }
\big)
.
\end{equation} 
\end{theorem}
\begin{proof}[Proof of Theorem~\ref{theorem:Exact_to_numeric}]
	Throughout this proof 
	let 
	$ \rho \in (0, \infty) $ satisfy that
	\begin{equation}
	\varepsilon p e^{ 2 ( \aa + \| B \|_{\HS(U, H )}^2 ) T }
	\leq
	\rho 
	<
	\tfrac{1}{16}
	\min \! \Big \{
	\epsilon e^{ - 2 ( \aa + \| B \|_{\HS(U, H )}^2 ) T },
	\tfrac{ 1 }{ 
		(
		8 
	\max \{ \| B \|_{\HS(U, H )}^2, 1 \}
	\max \{ T, 1 \}
	)^2
	}
	\Big\}
	.
	\end{equation}
	Note that
	Lemma~\ref{lemma:Continuous adapted modification}
	(applies with
	$ T = T $,
	$ \theta = \theta $,
	$ \beta = \beta $,
	$ \gamma = \gamma $,
	$ B = B $,
	$ F = F $,
	$ D = D_{ |\theta|_T }^I $,
	$ ( \Omega, \F, \P ) = ( \Omega, \F, \P ) $,
	$ ( \f_t )_{ t \in [0,T] } = ( \f_t )_{ t \in [0,T] } $,
	$ ( W_t )_{ t \in [0,T] } = ( W_t )_{ t \in [0,T] } $,
	$ \xi = \xi $,
	$ I = I $,
	$ P = P_I $,
	$ \y = \y^{\theta, I} $
	for
	$ \theta \in \varpi_T $,
	$ I \in \mathcal{P}_0(\H) $
	in the setting of 
	Lemma~\ref{lemma:Continuous adapted modification})
	verifies that there
	exist
	 $ ( \f_t )_{ t \in [0,T] } $-adapted
	stochastic processes
	$ \X^{\theta, I} \colon [0,T] \times \Omega \to P_I(H) $,
	$ \theta \in \varpi_T $,
	$ I \in \mathcal{P}_0(\H) $,
	w.c.s.p.\ 
	which satisfy
	for all
	$ \theta \in \varpi_T $,
	$ I \in \mathcal{P}_0(\H) $,
	$ t \in [0,T] $ 
	that
	$ \P( \X_t^{\theta, I} = \y_t^{\theta, I} ) = 1 $
 	and
		\begin{equation} 
		\begin{split}     
		\label{eq:less assumptions}
		[
		\X_t^{\theta, I}   
		]_{\P, \mathcal{B}( P_I(H) ) }
		&=   
		\big[
		e^{(t-\llcorner t \lrcorner_\theta )A}
		\X_{
			\llcorner t \lrcorner_\theta 
		}^{\theta, I}  
		+
		\1_{  D_{ |\theta|_T }^I } 
		( \X_{\llcorner t \lrcorner_\theta }^{\theta, I} )
		\,
		e^{(t-\llcorner t \lrcorner_\theta )A}
		P_I F(
		\X_{ \llcorner t \lrcorner_\theta  }^{\theta, I} 
		) 
		(
		t - \llcorner t \lrcorner_\theta 
		)  
		\big ]_{\P, \mathcal{B}( P_I (H) ) }
		\\
		&
		\quad 
		+
		\frac{
			\int_{ \llcorner t \lrcorner_\theta  }^t
			\1_{  D_{ |\theta|_T }^I } 
			( \X_{\llcorner t \lrcorner_\theta }^{\theta, I} )
			\,
			e^{(t- \llcorner t \lrcorner_\theta  )A}
			P_I B 
			\, 
			dW_s
		}{
			1 + 
			\| 
			\int_{ \llcorner t \lrcorner_\theta }^t
			P_I  B 
			\, 
			dW_s 
			\|_H^2
		} 
		.
		\end{split}
		\end{equation}
	Next
	let
	$ \O^{\theta, I} \colon [0,T] 
	\times \Omega \to P_I(H) $, 
		$ \theta \in \varpi_T $,
	$ I \in \mathcal{P}_0(\H) $,
	be stochastic processes 
	which satisfy for all 
		$ \theta \in \varpi_T $, 
	$ I \in \mathcal{P}_0(\H) $,
	$ t \in [0,T] $ 
	that 
	\begin{equation} 
	\label{eq:Identity_Clarity}
	\O_t^{\theta, I} 
	= 
	\X_t^{\theta, I} 
	- 
	\bigg( 
	e^{tA} P_I \xi
	+ 
	\int_0^t
	\1_{ D_{|\theta|_T}^I }
	\!
	( \X_{\llcorner s \lrcorner_\theta}^{\theta, I} )
	\,
	e^{(t-\llcorner s \lrcorner_\theta)A}  
	P_I F( \X_{\llcorner s \lrcorner_\theta}^{\theta, I} )
	\,
	ds 
	\bigg) 
	.
	\end{equation} 
	We intend to verify Theorem~\ref{theorem:Exact_to_numeric} 
	through an application of 
	Proposition~\ref{proposition:Exact_to_numeric_general} 
	(applies with 
	$ \alpha=\gamma-\delta $,
	$ \iota = \gamma - \delta $,
	$ \y^{\theta, I } = \X^{\theta, I} $,
	$ \O^{\theta, I} = \O^{\theta, I} $
	for 
		$ \theta \in \varpi_T $, 
	$ I \in \mathcal{P}_0(\H) $
	in the setting of Proposition~\ref{proposition:Exact_to_numeric_general}).
	For this we now verify the hypotheses~\eqref{eq:TimeRegularityNoise}--\eqref{eq:SupBound2}  
	in Proposition~\ref{proposition:Exact_to_numeric_general}.
	Observe that~\eqref{eq:less assumptions}  
	and~\eqref{eq:Identity_Clarity}
	give that
	for all 
		$ \theta \in \varpi_T $,
	$ I \in \mathcal{P}_0(\H) $, 
	$ t \in [0,T] $ 
	we have that
	\begin{equation}
	\begin{split}
	\label{eq:SPlit1}
	[
	\O_t^{\theta, I}
	]_{\P, \mathcal{B}( P_I(H) ) } 
	& 
	= 
	[
	e^{(t-\llcorner t \lrcorner_\theta )A} 
	\O_{ \llcorner t \lrcorner_\theta }^{\theta, I}
	]_{\P, \mathcal{B}( P_I(H) )}  
	+ 
	\frac{
		\int_{ \llcorner t \lrcorner_\theta }^t
		\1_{ D_{ | \theta |_T }^I }
		\!
		(
		\X_{ \llcorner t \lrcorner_\theta }^{\theta, I}   )
		\,
		e^{(t-\llcorner t \lrcorner_\theta )A}   
		P_I B 
		\, 
		dW_s
	}{
		1 + 
		\|
		\int_{\llcorner t \lrcorner_\theta}^t  
		P_I B 
		\, 
		dW_s
		\|_H^2
	}   
	.
	\end{split}
	\end{equation}
	This and~\cite[Corollary~3.2]{JentzenLindnerPusnik2017b}
	(applies with
	$ H = H $,
	$ U = U $,
	$ \H = \H $,
	$ \values = \values $,
	$ A = A $,
	$ \beta = \beta $,
	$ T = T $,
	$ ( \Omega, \F, \P, ( \f_t )_{ t \in [0,T] } ) = ( \Omega, \F, \P, ( \f_t )_{ t \in [0,T] } ) $,
	$ ( W_t )_{ t \in [0,T] } = ( W_t )_{ t \in [0,T] } $, 
	$ B = B $,
	$ P_I = P_I $, 
	$ \hat P_{ \mathbb{U} } = \operatorname{Id}_U $,
	$ \chi^{\theta, I, \mathbb{U} } = ( [0,T] \times \Omega \ni (t, \omega) 
	\mapsto 
	\1_{  D_{ |\theta|_T  }^I  }
	\!
	( \X^{\theta, I }_t ( \omega ) ) \in [0,1] ) $, 
	$ \O^{ \theta, I, \mathbb{U}  } = \O^{ \theta, I } $,
	$ p = 4 p $, 
	$ \gamma = \delta $,
	$ \rho = \gamma - \delta $
	for 
	$ \theta \in \varpi_T $,
	$ I \in \mathcal{P}_0(\H) $
	in the setting of~\cite[Corollary~3.2]{JentzenLindnerPusnik2017b})
	yield that there 
	exists $ \mathfrak{C} \in \R $
	which satisfies 
	that for all
	$ \theta \in \varpi_T $
	we have
	that
	\begin{equation}
	\begin{split}
	\label{eq:no3}
	&
	\sup\nolimits_{ I \in \mathcal{P}_0(\H) }
	\sup\nolimits_{ s \in [0,T] } 
	\| \O_s^{\theta, I} - \O_{ \llcorner s \lrcorner_\theta }^{\theta, I} 
	\|_{\L^{4p}( \P; H_\delta ) } 
	\leq 
		\mathfrak{C}
		\,
		[| \theta |_T]^{ \gamma - \delta } 
		.
	\end{split}
	\end{equation}
	Moreover, note that 
	the fact that $ \gamma < \nicefrac{1}{2} + \beta $
	ensures that
	there exists an
	$ ( \f_t )_{ t \in [0,T] } $-adapted 
	stochastic process
	$ O \colon [0,T] \times \Omega \to H_\gamma $
	w.c.s.p.\ 
	which satisfies for all 
	$ t \in [0,T] $ that
	\begin{equation} 
	\label{eq:involve O}
	[ 
	O_t
	]_{\P, \B(H_\gamma)} = \int_0^t e^{ (t-s) A } B \, dW_s 
	\end{equation} 
		(cf., e.g., \cite[Lemma~5.5]{JentzenLindnerPusnik2017c}).
	Next let 
	$ O^I  \colon [0,T] 
	\times \Omega \to P_I(H) $, 
	$ I \in \mathcal{P}_0(\H) $, 
	be stochastic processes 
	which satisfy for all 
	$ I \in \mathcal{P}_0(\H) $, 
	$ t \in [0,T] $ 
	that
	\begin{equation} 
	\label{eq:some equality}
	O_t^I = P_I O_t
	.
	\end{equation} 
	Observe that~\eqref{eq:some equality}
	and 
	H\"older's inequality give
	that for all 
	$ \theta \in \varpi_T $,
	$ I, J \in \mathcal{P}_0(H) $,
	$ s \in [0,T] $
	we have that
	\begin{equation}
	\begin{split}
	\label{eq:Exp0}
	& 
	\E\big[
	\exp\!
	\big(  
	\rho
	\|  
	\X_s^{\theta, J}
	-
	\O_s^{\theta, J}
	+
	O_s^I
	+
	e^{sA} P_{ I \backslash J } \xi 
	\|_H^2   
	\big)
	\big]  
	\\
	&
	\leq
	\E\big[
	\exp\!
	\big(
	4  
	\rho
	(
		\|  
	\X_s^{\theta, J}
	\|_H^2
	+
	\|
	\O_s^{\theta, J}
	\|_H^2
	+
	\|
	O_s^I
	\|_H^2 
	+
	\|
	\xi  
	\|_H^2    
	)
	\big)
	\big] 
	\\
	&
	\leq
	\big[
	\E [
	\exp
	 (
	16
	\rho 
	\|  
	\X_s^{\theta, J}
	\|_H^2 
	 )
	 ] 
 	\,
	\E [
	\exp
	 (
	16
	\rho
	\|
	\O_s^{\theta, J}
	\|_H^2 
	 )
	 ]
 	\,
	\E [
	\exp
	 (
	16
	\rho
	\|
	O_s^I
	\|_H^2 
	 )
	 ]  
 	\,
	\E [
	\exp
	 (
	16
	\rho
	\|
	\xi
	\|_H^2 
	 )
	 ]
	\big]^{\nicefrac{1}{4}} 
	.
	\end{split}
	\end{equation}
	Moreover, note that
	the assumption that
	for all    
	$ I \in \mathcal{P}_0(\H) $,
	$ x \in P_I( H ) $ 
	we have that   
	\begin{equation} 
	\label{eq:contintion on D}
	\| B \|_{\HS(U, H )} + \epsilon \| x \|_H^2 \leq r( x ) \leq 
	C ( 1 + \| x \|_{ H_\nu }^2 )
	\end{equation} 
	and the assumption that for all
		$ I \in \mathcal{P}_0(\H) $,
	$ h \in (0,T] $
	we have that 
	\begin{equation} 
	\label{eq:AssIndicator}
	 D_h^I = \{ v \in P_I( H ) \colon r(v) \leq  \upnu  h^{- \varsigma }  \} 
	 \end{equation}
	ensure that
	for all
	$ I \in \mathcal{P}_0( \H ) $,
	$ h \in (0,T] $
	we have that 
	\begin{equation} 
	\label{eq:AssIndicator3}
	\{ v \in P_I(H) \colon 
	C ( 1 + \| v \|_{ H_\nu }^2 ) \leq \upnu h^{-\varsigma} \}
	\subseteq 
	D_h^I \subseteq \{ v \in P_I( H ) \colon  \| B \|_{\HS(U, H )} +  \epsilon \|  v \|_H^2 \leq  \upnu  h^{- \varsigma }  \}
	.
	\end{equation} 
	Combining 
	this, 
the assumption that
for all 
$ I \in \mathcal{P}_0(\H) $, 
$ h \in (0,T] $, 
$ x \in D_h^I $
we have that
\begin{equation} 
\label{eq:AssIndicator2} 
\max\{\| P_I F(x) \|_H,
\| B \|_{\HS(U, H  )}
\} 
\leq \upnu h^{-\varsigma},
\end{equation} 
	the assumption that $ \E[ \exp ( \epsilon \| \xi \|_H^2 ) ] < \infty $,
	the fact that $ \xi \in \M( \f_0, \B(H_\gamma) ) $,
 	the assumption that
 	for all
 	$ I \in \mathcal{P}_0(\H) $, 
 	$ x \in P_I(H) $
 	we have that
 	\begin{equation}
 	\label{eq:coercivity use} 
 	\langle x, F(x) \rangle_H \leq \aa ( 1+  \| x \|_H^2 ),
 	\end{equation} 
	the fact that
	$ 16 \rho \leq \epsilon 
	\exp( - 2 ( \aa + \| B \|_{\HS(U, H )}^2 ) T ) $,
	\eqref{eq:less assumptions},
	and 
	Corollary~\ref{Corollary:full_discrete_scheme_convergence} 
	(applies with
	$ T = T $, 
	$ a = \aa $,
	$ b = \aa $,
		$ \upnu = \upnu $,
	$ \varsigma = \varsigma $, 
	$ \epsilon = \epsilon $,
	$ \beta = \beta $,
	$ \gamma = \gamma $,
	$ B = B $,
		$ F = F $,
		$ D_h^I = D_h^I $,
	$ P_I = P_I $,	
	$ ( \Omega, \F, \P ) = ( \Omega, \F , \P ) $,
	$ ( \f_t )_{ t \in [0,T] } = ( \f_t )_{ t \in [0,T] } $,
	$ ( W_t )_{ t \in [0,T] } = ( W_t )_{ t \in [0,T] } $,
	$ \xi = \xi $,
	$ \y^{ \theta, I } = \X^{ \theta, I } $
	for
	$ \theta \in \varpi_T $,
	$ I \in \mathcal{P}_0(\H) $, 
	$ h \in (0,T] $
	in the setting of Corollary~\ref{Corollary:full_discrete_scheme_convergence}) 
	verifies that
	\begin{equation}
	\begin{split}
	\label{eq:Exp1}
		&
		\sup\nolimits_{\theta \in \varpi_T}
		\sup\nolimits_{J \in \mathcal{P}_0(H) }
		\sup\nolimits_{s \in [0,T]} 
		\E [
		\exp
		 (
		16 \rho
		\|  
		\X_s^{\theta, J} 
		\|_H^2   
		 )
		 ] 
		< 
		\infty 
		.
	\end{split}
	\end{equation}
	In addition,
	note that
	the fact that
	$ 16 \rho <
	\nicefrac{1}{ ( 8 \max \{ \| B \|_{\HS(U, H )}^2, 1 \} \max \{ T, 1 \} )^2 } $,
	\eqref{eq:SPlit1},
	and~\cite[Corollary~3.4]{JentzenLindnerPusnik2017b}
	(applies with 
	$ H = H $,
	$ U = U $,
	$ \H = \H $,
	$ \values = \values $,
	$ A = A $,
	$ \beta = \beta $,
	$ T = T $,
	$ ( \Omega, \F, \P, ( \f_t )_{ t \in [0,T] } ) = ( \Omega, \F, \P, ( \f_t )_{ t \in [0,T] } ) $,
	$ ( W_t )_{ t \in [0,T] } = ( W_t )_{ t \in [0,T] } $, 
	$ B = B $,
	$ P_I = P_I $, 
	$ \hat P_{ \mathbb{U} } = \operatorname{Id}_U $,
	$ \chi^{\theta, I, \mathbb{U}} = ( [0,T] \times \Omega \ni (t, \omega) 
	\mapsto 
	\1_{  D_{ |\theta|_T  }^I  }
	\!
	( \X^{\theta, I }_t ( \omega ) ) \in [0,1] ) $, 
	$ \O^{ \theta, I, \mathbb{U} } = \O^{ \theta, I } $,
	$ \varepsilon = 16 \rho $
	for  
		$ \theta \in \varpi_T $,
	$ I \in \mathcal{P}_0(\H) $
	in the setting of~\cite[Corollary~3.4]{JentzenLindnerPusnik2017b}) 
	assure that
	\begin{equation}
	\begin{split}
	\label{eq:Exp2}
	&
	\sup\nolimits_{\theta \in \varpi_T}
	\sup\nolimits_{J \in \mathcal{P}_0(H) }
	\sup\nolimits_{s \in [0,T]} 
	\E [
	\exp
	 (
     16 \rho
	\|  
	\O_s^{\theta, J} 
	\|_H^2   
	 )
	 ] 
	< 
	\infty 
	.
	\end{split}
	\end{equation}
	Furthermore, note that
	Lemma~\ref{lemma:chi2_estimate}
	(applies with
	$ T = T $,
	$ B = ( U \ni u \mapsto 4 \sqrt{ \rho } Bu \in H ) $,
	$ P_I = P_I $,
	$ ( \Omega, \F, \P ) = ( \Omega, \F, \P ) $,
	$ ( W_t )_{ t \in [0,T] } = ( W_t )_{ t \in [0,T] } $
	for
	$ I \in \mathcal{P}_0(\H) $
	in the setting of Lemma~\ref{lemma:chi2_estimate})
	yields that 
	for all  
	$ I \in \mathcal{P}_0(\H) $,
	$ s \in [0, T] $
	with
	$ 32 \rho s \| B \|_{\HS(U, H)}^2 < 1 $
	we have that
	\begin{equation} 
	\label{eq:ExpMomEst}
	\E [
	\exp
	 (
	16  
	\rho
	\|  
	O_s^{I} 
	\|_H^2   
	 )
	 ]  
	\leq 
	\tfrac{2}{ 1 
		- 
		1024 \rho^2 s^2 
		\| B\|_{\HS(U, H)}^4 
	}
	.
	\end{equation}
	\sloppy 
	Next observe that the fact that 
	for all
	$ x \in [0, \infty) $
	we have that
	$ x < 2 e^x $
	gives that
	$ 4 T \| B \|_{\HS(U, H )}^2  < 2 e^{ 4 T \| B \|_{\HS(U, H )}^2 } $.
	This yields that
	$ 2 T \| B \|_{\HS(U, H )}^2   e^{ - 4 T \| B \|_{\HS(U, H )}^2 } 
	< 
	1 $.
	Therefore, we obtain that
	\begin{equation} 
	\begin{split}
	32 \rho T \| B \|_{\HS(U,H)}^2 
	&
	\leq     
	\tfrac{ 32 \epsilon 	T \| B \|_{\HS(U,H)}^2  }{ 16 } 
	e^{ - 2 T ( \aa + \| B \|_{\HS(U, H )}^2 )  }
	\leq 
	2 
	T
	\| B \|_{\HS(U,H)}^2
	e^{ - 4 T ( \aa + \| B \|_{\HS(U, H )}^2 )  } 
	\\
	&
	\leq 
	2
	T \| B \|_{\HS(U,H)}^2
	e^{ - 4  T \| B \|_{\HS(U, H )}^2   }
	< 
	1
	.
	\end{split}
	\end{equation} 
	This and~\eqref{eq:ExpMomEst}
	give that
	\begin{equation} 
	\label{eq:Exp3}
	\sup\nolimits_{I \in \mathcal{P}_0(H) } 
	\sup\nolimits_{ s \in [0,T]}
	\E [
	\exp
	 (
	16  
	\rho
	\|  
	O_s^{I} 
	\|_H^2   
	 )
	 ] 
	< \infty
	.
	\end{equation}
	Combining
	this, 
	\eqref{eq:Exp0},
	\eqref{eq:Exp1},
	\eqref{eq:Exp2}, 
	the assumption that
	$ \E[ \exp ( \epsilon \| \xi \|_H^2 ) ] < \infty $,
	and the fact that
	$ 16 \rho < \epsilon $
	illustrates that
	\begin{equation}
	\begin{split}
	\label{eq:no5} 
	&
	\sup\nolimits_{ I, J \in \mathcal{P}_0(\H) }
	\sup\nolimits_{ \theta \in \varpi_T }
	\sup\nolimits_{ s \in [0,T] } 
	\E\big[
	\exp\!
	\big(  
	\rho
	\|  
	\X_s^{\theta, J}
	-
	\O_s^{\theta, J}
	+
	O_s^I
	+
	e^{ s A } P_{ I \backslash J } \xi
	\|_H^2   
	\big)
	\big]  
	< \infty
	.
	\end{split}
	\end{equation}
	Next observe that
	the fact that
	$ \xi \in \L^{ 32 p c \max \{  ( \gamma - \delta ) / \varsigma, 1 \} }(\P; H_{ \max \{ \eta_2, \sigma, \nu, \gamma \} } ) $, 
	the fact that
	\begin{equation}
	\label{eq:Addition2}
	\Big[
	\sup\nolimits_{ v \in H_{ \max \{ \gamma, \eta_2 \} }  }
	\tfrac{ \| F(v) \|_{H  } }
	{ 1 + \| v \|_{H_{ \eta_2 } }^2 }
	\Big] 
	+
	\Big[
	\sup\nolimits_{ v \in H_{ \max \{ \gamma, \eta_1 \} }  }
	\tfrac{ \| F(v) \|_{H_{ - \alpha_2 } } }
	{ 1 +  \| v \|_{ H_{\eta_1} }^2 }
	\Big]
	+
	\Big[
	\sup\nolimits_{ v \in H_{\gamma} } 
	\tfrac{ \| F(v) \|_{H_{-\alpha_1} } }{ 1 + \| v \|_H^2 } 
	\Big]
	< \infty,
	\end{equation}
		the assumption that $ \E[ \exp ( \epsilon \| \xi \|_H^2 ) ] < \infty $,
\eqref{eq:less assumptions}, 
	\eqref{eq:AssIndicator3}--\eqref{eq:coercivity use},
 the fact that
 $ \epsilon \leq 
 \exp( -2(\aa + \| B \|_{\HS(U, H )}^2 )T ) $,
	and 
	Corollary~\ref{corollary:AprioriNumApp}
	(applies with 
	$ T = T $,
	$ a = \aa $,
	$ b = \aa $,
	$ \upnu = \upnu $,
	$ \varsigma = \varsigma $,
	$ \epsilon = \epsilon $,
	$ \beta = \beta $,
	$ \gamma = \gamma $,
	$ B = B $, 
	$ F = F $,
	$ D_h^I = D_h^I $, 
	$ P_I = P_I $,
	$ ( \Omega, \F, \P ) = ( \Omega, \F, \P ) $,
	$ ( \f_t )_{ t \in [0,T] } = ( \f_t )_{ t \in [0,T] } $,
	$ ( W_t )_{ t \in [0,T] } = ( W_t )_{ t \in [0,T] } $,
	$ \xi = \xi $,
	$ \y^{\theta, I} = \X^{\theta, I} $,
	$ p = 8 p  c 
	\max \{ \nicefrac{ ( \gamma - \delta ) }{ \varsigma }, 1 \} $,
	$ \eta_1 = \eta_1 $,
	$ \eta_2 = \eta_2 $,
	$ \iota = \max \{ \eta_2, \sigma, \nu, \gamma \} $, 
	$ \alpha_1 = \alpha_1 $,
	$ \alpha_2 = \alpha_2 $
	for   
	$ h \in (0,T] $,
	$ \theta \in \varpi_T $,
	$ I \in \mathcal{P}_0(\H) $
	in the setting of 
	Corollary~\ref{corollary:AprioriNumApp}) 
	verify that 
	\begin{equation}
	\label{eq:AssSatisfied000}
	\sup\nolimits_{\theta \in \varpi_T}
	\sup\nolimits_{I\in \mathcal{P}_0(\H)} 
	\sup\nolimits_{ t \in [0,T] }  
	\| \X_t^{\theta, I} \|_{\L^{8 p c \max \{ 
			 ( \gamma - \delta ) / \varsigma, 1 \} }( \P; H_{ \max \{ \eta_2, \sigma, \nu, \gamma \} } ) }  
	< \infty 
	.
	\end{equation} 
	Combining this 
	and the fact that
	$
	\sup\nolimits_{ I \in \mathcal{P}_0(\H) } 
	\sup\nolimits_{ v \in P_I(H) }
	\big( 
	\nicefrac{   ( \| P_I F(v) \|_{H_{ \gamma - \delta }} )  }{
		( 1 + \| v \|_{H_{\sigma}}^2 )	
	}
\big)  
< \infty 
$
	yields that
	\begin{equation}
	\begin{split} 
	&
	\sup\nolimits_{ \theta \in \varpi_T }
	\sup\nolimits_{ I \in \mathcal{P}_0(\H) }
	\sup\nolimits_{ t \in [0,T] } 
\| P_I F ( \X_t^{\theta, I } ) \|_{\L^{4p c \max \{  
		 ( \gamma - \delta ) / \varsigma, 1 \} }( \P; H_{ \gamma-\delta } ) }
	\\
	&
	\leq
	\bigg[
	\sup_{ I \in \mathcal{P}_0(\H) } 
\sup_{ v \in P_I(H) }
	\tfrac{ \| P_I F(v) \|_{H_{ \gamma-\delta }} }{
		1 + \| v \|_{H_{ \sigma }}^2	
	} 
	\bigg]
	\bigg[
	1
	+
	\sup_{ \theta \in \varpi_T }
	\sup_{ I \in \mathcal{P}_0(\H) }
	\sup_{ t \in [0,T] }
	\| \X_t^{\theta,I} \|_{\L^{8 p c \max \{ 
			 ( \gamma - \delta ) / \varsigma, 1 \} }( \P; H_{ \sigma } ) }^2
	\bigg] 
	< \infty
	.
	\end{split}
	\end{equation}
	This and~\eqref{eq:AssSatisfied000} assure that
	\begin{equation}
	\label{eq:AssSatisfied}
	\sup\nolimits_{\theta \in \varpi_T}
	\sup\nolimits_{I\in \mathcal{P}_0(\H)} 
	\sup\nolimits_{ t \in [0,T] } 
	\big[
	\| P_I F ( \X_t^{\theta, I } ) \|_{\L^{4 p}( \P; H_{ \gamma-\delta } ) }
	+
	\| P_I 
	F ( \X_t^{\theta, I } ) \|_{\L^{2 p}( \P; H_{  \gamma - \delta } ) }
	\big]
	< \infty 
	\end{equation} 
	and
	\begin{equation}
	\label{eq:Impo}
	\sup\nolimits_{ I \in \mathcal{P}_0(\H) }
	\sup\nolimits_{ \theta \in \varpi_T }
	\sup\nolimits_{ t \in [0,T] }
	\big[
	\|  
	\X_t^{\theta, I}
	\|_{ \L^{  4 p c } (\P; H_\kappa) }
	+
	\|  
	 \X_t^{\theta, I}
	\|_{ \L^{  2 \max \{ 
			4 p ( \gamma - \delta ) / \varsigma, 1 \}  } (\P; H_\nu) }
	\big]
	< \infty
	.
	\end{equation}
	In addition, note that~\eqref{eq:contintion on D}
	and~\eqref{eq:Impo}
	verify that
	\begin{equation} 
	\label{eq:R_integrable}
	\sup\nolimits_{ I \in \mathcal{P}_0(\H) }
	\sup\nolimits_{ \theta \in \varpi_T }
	\sup\nolimits_{ t \in [0,T] } 
	\big[ 
	\|  
	\X_t^{\theta, I}
	\|_{ \L^{  4 p c } (\P; H_\kappa) }
	+
	\|  
	r( \X_t^{\theta, I} )
	\|_{ \L^{  
			 4 p ( \gamma - \delta ) / \varsigma  } (\P; \R) }
		 \big] 
	< \infty
	.
	\end{equation}
		Moreover, 
		observe that~\eqref{eq:AssIndicator3}
	and Markov's inequality
	ensure
	that for all 
	$ \theta \in \varpi_T $,
	$ I \in \mathcal{P}_0(\H) $,
	$ h \in (0,T] $,
	$ t \in [0,T] $
	we have that 
	\begin{equation}
	\begin{split}
	\label{eq:FastGrowth}
	&
	\big\| 
	1
	-
	\1_{  D_h^I  } 
	( \X^{\theta, I }_{\llcorner t \lrcorner_\theta } )
	\big\|_{ \L^{ 4 p c }( \P; \R ) }
	=
	\big\|
	\1_{  P_I(H) \backslash D_h^I  } 
	( \X^{\theta, I }_{\llcorner t \lrcorner_\theta } )
	\big\|_{ \L^{ 4 p c }( \P; \R ) }
	\leq 
	\big\|
	\1_{  
		\{
		C ( 1 + \| \X^{\theta, I }_{\llcorner t \lrcorner_\theta} \|_{ H_\nu }^2 ) > \upnu h^{ - \varsigma } \} %
	} 
	\big\|_{ \L^{ 4 p c }( \P; \R ) }
	\\
	&
	\leq 
	\big[
	\P\big(
	| C ( 1 + \| \X^{\theta, I }_{\llcorner t \lrcorner_\theta} \|_{ H_\nu }^2 ) |^{  \nicefrac{ 4 p c ( \gamma - \delta ) }{ \varsigma } } 
	>
	( \upnu h^{ - \varsigma } )^{ \nicefrac{ 4 p c ( \gamma - \delta ) }{ \varsigma } }
	\big)
	\big]^{ \nicefrac{1}{ ( 4 p c ) } }
	\\
	&
	\leq
	( \upnu h^{ - \varsigma } )^{- \nicefrac{ ( \gamma - \delta ) }{ \varsigma } }
	\big(
	\E\big[
	| C ( 1 + \| \X^{\theta, I }_{\llcorner t \lrcorner_\theta} \|_{ H_\nu }^2 ) |^{  \nicefrac{ 4 p c ( \gamma - \delta ) }{ \varsigma } }
	\big]
	\big)^{ \nicefrac{ 1 }{ ( 4 p c ) } }
	\\
	&
	=
	| \upnu |^{ - \nicefrac{ ( \gamma - \delta ) }{ \varsigma } } h^{  \gamma - \delta }
	C^{ \nicefrac{ ( \gamma - \delta ) }{ \varsigma } }
	\big\|
	1 + \| \X^{\theta, I }_{\llcorner t \lrcorner_\theta} \|_{ H_\nu }^2   
	\big\|_{ \L^{   
			4 p c ( \gamma - \delta ) / \varsigma } ( \P, \R ) }^{ \nicefrac{ ( \gamma - \delta ) }{ \varsigma } }
	\\
	&
	\leq
	| \upnu |^{ - \nicefrac{ ( \gamma - \delta ) }{ \varsigma } } h^{  \gamma - \delta }
	C^{ \nicefrac{ ( \gamma - \delta ) }{ \varsigma } }
	\big(
	1 + \| \X^{\theta, I }_{\llcorner t \lrcorner_\theta} 
	\|_{ \L^{ 2
			\max \{
			 4 p c  ( \gamma - \delta )
			/ \varsigma ,
			1 
			\}  	 
		} ( \P, H_\nu ) }^{ 2 } 
	\big)^{ \nicefrac{ ( \gamma - \delta ) }{ \varsigma } }
	.
	\end{split}
	\end{equation}
	Combining
	this
	and~\eqref{eq:Impo}
	illustrates that
	there exists
	$ \mathbf{C} \in [1, \infty) $
	which satisfies that
	for all 
	$ \theta \in \varpi_T $,
	$ I \in \mathcal{P}_0(\H) $,
	$ t \in [0,T] $
	we have that
 \begin{equation} 
 \big\| 
	1
	-
	\1_{  D_{ | \theta |_T }^I  } 
	( \X^{\theta, I }_{\llcorner t \lrcorner_\theta } )
	\big\|_{ \L^{ 4 p c }( \P; \R ) } 
	\leq \mathbf{C} \, [ | \theta |_T ]^{ \gamma - \delta } 
	.
	\end{equation} 
	This, \eqref{eq:SPlit1},
	and~\cite[Corollary~3.3]{JentzenLindnerPusnik2017b}
	(applies with
	$ H = H $,
	$ U = U $,
	$ \H = \H $,
	$ \values = \values $,
	$ A = A $,
	$ \beta = \beta $,
	$ T = T $,
	$ ( \Omega, \F, \P, ( \f_t )_{ t \in [0,T] } ) = ( \Omega, \F, \P, ( \f_t )_{ t \in [0,T] } ) $,
	$ ( W_t )_{ t \in [0,T] } = ( W_t )_{ t \in [0,T] } $,  
	$ B = B $,
	$ P_I = P_I $,
	$ \hat P_{ \mathbb{U} } = \operatorname{Id}_U $,
	$ \chi^{\theta, I, \mathbb{U}} = ( [0,T] \times \Omega \ni (t, \omega) 
	\mapsto 
	\1_{  D_{ |\theta|_T  }^I  }
	\!
	( \X^{\theta, I }_t ( \omega ) ) \in [0,1] ) $, 
	$ \O^{ \theta, I, \mathbb{U} } = \O^{ \theta, I } $,
	$ p = 4 p c $, 
	$ C = \mathbf{C} $, 
	$ \gamma = \max \{ \delta, \kappa \} $,
	$ \eta = \gamma - \delta $,
	$ \rho = \gamma - \delta $, 
	$ O = O $
	for
	$ \theta \in \varpi_T $,
	$ I \in \mathcal{P}_0(\H) $  
	in the setting of~\cite[Corollary~3.3]{JentzenLindnerPusnik2017b}) 
	illustrate that
	there exists 
	$ \mathscr{C} \in \R $
	which satisfies that
	for all 
	$ I, J \in \mathcal{P}_0(\H) $
	with $ I \subseteq J $ we have that
	\begin{equation}
	\begin{split}
	\label{eq:no6}
	&
	\sup\nolimits_{s \in [0,T]} 
	\| \O_s^{\theta, I} - O_s^{J} \|_{\L^{ 4 p c }(\P; H_{ \max \{ \delta, \kappa \} } )}  
	\leq
	\mathscr{C} 
	\big(
	\| P_{ \H \backslash I } ( -A)^{\delta - \gamma} \|_{L(H)}
	+
	[ |\theta|_T ]^{ \gamma - \delta }
	\big) 
	.
	\end{split}
	\end{equation}
Moreover, observe that~\eqref{eq:Identity_Clarity}
and
the fact that 
$ ( \X_t^{\theta, I})_{ t \in [0,T] } $,
$ \theta \in \varpi_T $,
$ I \in \mathcal{P}_0(\H) $,
are 
$ (\f_t)_{ t \in [0,T] } $-adapted
stochastic processes 
w.c.s.p.\
ensure that
$ (\O_t^{\theta, I})_{ t \in [0,T] } $,
$ \theta \in \varpi_T $,
$ I \in \mathcal{P}_0(\H) $,
are 
$ (\f_t)_{ t \in [0,T] } $-adapted
stochastic processes 
w.c.s.p.
	This,
	the assumption that
	for all 
	$ I \in \mathcal{P}_0(\H) $,
	$ x, y \in P_I( H ) $ 
	we have that   
	$ \langle F'(x) y, y \rangle_H  \leq 
	( \varepsilon \| x \|_{H_{\nicefrac{1}{2}}}^2 
	+ C  ) \| y \|_H^2 
	+ \|y \|_{H_{\nicefrac{1}{2}}}^2 $,
	$ \| P_I( F( x ) - F( y ) ) \|_H
	\leq
	C \| x - y \|_{H_\delta} 
	( 1 + \| x \|_{H_\kappa}^c  + \| y \|_{H_\kappa}^c ) $,
	and
	$ \langle x, A x + F(x+y) \rangle_H
	\leq
	\Phi(y) ( 1 + \| x \|_H^2 ) $,   
	the fact that
	$  \varepsilon 
	\leq 
	\frac{ \rho }{ p } 
	\exp( - 2 ( \aa + \| B \|_{\HS(U, H )}^2 ) T ) $,
	the fact that
	$ \xi \in \L^{ 4 p \max \{c,2\} }( \P |_{\f_0}; H_{ \max \{\gamma, \eta_2 \} } ) $,
	the fact that 
	$ \E \big[ \| \xi\|_H^{ 16 p } ] < \infty $,
	\eqref{eq:less assumptions}, 
	\eqref{eq:Identity_Clarity},
	\eqref{eq:no3},
	\eqref{eq:involve O},
	\eqref{eq:AssIndicator},
	\eqref{eq:coercivity use},
	\eqref{eq:no5}, 
	\eqref{eq:Addition2},
	\eqref{eq:AssSatisfied},
	\eqref{eq:R_integrable}, 
	\eqref{eq:no6},
	and
	Proposition~\ref{proposition:Exact_to_numeric_general}
	(applies with
	$ T = T $,
	$ \upnu = \upnu $,
	$ \varsigma = \varsigma $,
	$ \alpha = \gamma - \delta $, 
	$ \aa = \aa $,
	$ \iota = \gamma - \delta $,
	$ \rho = \rho $, 
	$ C = \max \{ C, \mathfrak{C}, \mathscr{C} \} $, 
	$ c = c $,
	$ p = p $,
	$ \beta = \beta $,
	$ \gamma = \gamma $, 
    $ \delta = \delta $,
	$ \kappa = \kappa $,
	$ \eta_1 = \eta_1 $,
	$ \eta_2 = \eta_2 $,
	$ \alpha_1 = \alpha_1 $,
	$ \alpha_2 = \alpha_2 $,
$ B = B $, 
	$ \varepsilon = \varepsilon $,
		$ F = F $,
	$ r = r $,
		$ D_h^I = D_h^I $,
	$ \Phi = \Phi $,
	$ P_I = P_I $,
	$ ( \Omega, \F, \P)
	=
	( \Omega, \F, \P ) $,
	$ ( \f_t )_{ t \in [0,T] }  = ( \f_t )_{ t \in [0,T] }  $,
	$ ( W_t )_{ t \in [0,T] } 
	=
	( W_t )_{ t \in [0,T] } $,
	$ \xi = \xi $,
	$ X = X $,
	$ O = O $,
	$ \y^{\theta, I} = \X^{ \theta, I} $,
	$ \O^{ \theta, I } = \O^{ \theta, I } $
	for
	$ \theta \in \varpi_T $,
	$ I \in \mathcal{P}_0(\H) $,
	$ h \in (0,T] $ 
	in the setting of
	Proposition~\ref{proposition:Exact_to_numeric_general})
	therefore justify~\eqref{eq:Main convergence estimate}.
	The proof of Theorem~\ref{theorem:Exact_to_numeric}
	is hereby completed.
\end{proof}
\begin{corollary} 
	\label{corollary:Exact_to_numeric} 
	\sloppy 
	Assume Setting~\ref{setting:main},
	let  
 	$ T \in (0,\infty) $,
 		$ \varsigma \in (0, \nicefrac{1}{18} ) $, 
$ \aa \in [0, \infty) $,
	$ C, c, p \in [1,\infty) $,
	$ ( C_\varepsilon )_{ \varepsilon \in (0, \infty) } \subseteq [0, \infty) $,
$ \beta \in [0, \nicefrac{1}{2}) $,
$ \gamma \in [ 2 \beta, \nicefrac{1}{2} + \beta ) \cap (0,\infty) $,
	$ \delta \in ( \gamma - \nicefrac{1}{2}, \gamma) \cap [0, \infty) $, 
	$ \kappa \in [0, \gamma] 
	\cap [0, \nicefrac{1}{2} + \beta - \gamma + \delta  ) $,
	$ \eta_0 = 0 $,
	$ \sigma, \nu, \eta_1 \in [0, \nicefrac{1}{2} + \beta ) $, 
	$ \eta_2 \in [\eta_1, \nicefrac{1}{2} + \beta ) $,
	$ \alpha_1 \in [0, 1 - \eta_1 ) $, 
	$ \alpha_2 \in [0, 1 - \eta_2 ) $, 
	$ \alpha_3 = 0 $, 
	$ B \in  \HS(U, H_\beta) $, 
	$ F \in \mathcal{C}^1(  H_\gamma, H ) $,
	let
	$ \Phi \colon H \to [0, \infty) $
	be a function,
    let $ ( P_I )_{ I \in \mathcal{P}(\H) } \subseteq L(H) $ 
    satisfy
for all 
	$ I \in \mathcal{P}(\H) $,
	$ x \in H $
	that
	$ P_I(x)
	= \sum_{h \in I} \langle h, x \rangle_H h $,
	assume for all 
	$ I \in \mathcal{P}(\H) $ 
	that 
	$ ( P_I(H) \ni v \mapsto \Phi(v) \in [0, \infty) )
	\in 
	\mathcal{C}( P_I(H), [0, \infty) ) $,
	assume for all    
	$ I \in \mathcal{P}_0(\H) $,
	$ x, y \in P_I( H ) $, 
	$ \varepsilon \in (0, \infty) $
	that   
		$ \langle x, F(x) \rangle_H \leq \aa ( 1 + \| x \|_H^2 ) $,
	$ \langle F'(x) y, y \rangle_H  \leq 
	( \varepsilon \| x \|_{H_{\nicefrac{1}{2}}}^2 
	+ C_\varepsilon ) \| y \|_H^2 
	+ \|y \|_{H_{\nicefrac{1}{2}}}^2 $,
	$ \| F( x ) - F( y ) \|_H
	\leq
	C \| x - y \|_{H_\delta} 
	( 1 + \| x \|_{H_\kappa}^c  + \| y \|_{H_\kappa}^c ) $,
	$ \langle x, A x + F(x+y) \rangle_H
	\leq
	\Phi(y) ( 1 + \| x \|_H^2 ) $,   
	and
		\begin{equation}
	\label{eq:Addition0}
	\bigg[
	\sup\nolimits_{ J \in \mathcal{P}_0(\H) } 
	\sup\nolimits_{ v \in P_J(H) }
	\Big\{
	\tfrac{ \| P_J F(v) \|_H }{
		1 + \| v \|_{H_\nu}^2	
	} 
 + 
  \tfrac{ \| P_J F(v) \|_{H_{ \gamma - \delta }}  }{
		1 + \| v \|_{H_\sigma}^2	
	} 
\Big\} 
	\bigg]
	+
	\sum_{i=0}^2
	\Big[
	\sup\nolimits_{ v \in H_{ \max \{ \gamma, \eta_i \} } } 
	\tfrac{ \| F(v) \|_{H_{ - \alpha_{i+1} } } }{ 1 + \| v \|_{H_{\eta_i}}^2 } 
	\Big]
	<
	\infty
	,
	\end{equation}
	let
	$ ( \Omega, \F, \P ) $
	be a probability space with a normal filtration
	$ ( \f_t )_{t \in [0,T]} $,
	let $ (W_t)_{t\in [0,T]} $
	be an $ \operatorname{Id}_U $-cylindrical  
	$ ( \f_t )_{t\in [0,T]} $-Wiener process, 
	let 
	$ \xi \in \L^{ 32 p c \max \{  ( \gamma - \delta ) / \varsigma, 1 \} }(\P|_{ \f_0 }; H_{ \max \{ \eta_2, \sigma, \nu, \gamma \} } ) $
	satisfy 
	$ \inf_{ \epsilon \in (0,\infty) } \E[ \exp ( \epsilon \| \xi \|_H^2 ) ] < \infty $,
	let 
	$ X \colon [0,T] \times \Omega \to H_\gamma $
	be an 
	$ (\f_t )_{t\in [0,T]} $-adapted
	stochastic process 
	w.c.s.p.\ 
	which satisfies for all $ t \in [0,T] $ that
\begin{equation}
[ X_t ]_{\P, \B(H_\gamma)} 
= 
\bigg[
e^{tA} \xi 
+
\int_0^t e^{(t-s)A} F(X_s) \, ds
\bigg]_{\P, \B(H_\gamma)} 
+
\int_0^t e^{(t-s)A} B \, dW_s,
\end{equation}
	and
	let
	$ \y^{\theta, I } \colon 
	[0,T] \times \Omega \to P_I( H ) $,
	$ \theta \in \varpi_T $,
	$ I \in \mathcal{P}_0(\H) $, 
	be 
	$ ( \f_t )_{t\in [0,T]} $-adapted
	stochastic 
	processes 
	which satisfy for all 
	$ \theta \in \varpi_T $,
	$ I \in \mathcal{P}_0(\H) $, 
	$ t \in [0,T] $
	that 
	$ \y_0^{\theta,I}= P_I \xi $
	and
	\begin{align}  
	\label{eq:Scheme2}
	\nonumber  
	[
	\y_t^{\theta, I}  
	]_{\P, \mathcal{B}( P_I(H) ) }
	&=   
	\big[
	e^{(t-\llcorner t \lrcorner_\theta )A}
	\y^{\theta, I }_{
		\llcorner t \lrcorner_\theta 
	}   
	+ 
	\1_{ \{ 1 + \|  \y^{\theta, I }_{\llcorner t \lrcorner_\theta } \|_{H_\nu}^2 \leq [ | \theta |_T ]^{- \varsigma } \} }
	e^{(t-\llcorner t \lrcorner_\theta )A}
	P_I F(
	\y^{\theta, I }_{ \llcorner t \lrcorner_\theta  } 
	) 
	(
	t - \llcorner t \lrcorner_\theta 
	)  
	\big ]_{\P, \mathcal{B}( P_I(H) ) }
	\\
	&
	\quad 
	+
	\frac{
		\int_{ \llcorner t \lrcorner_\theta  }^t 
		\1_{ \{ 1 + \|  \y^{\theta, I }_{\llcorner t \lrcorner_\theta } \|_{H_\nu}^2 \leq [ | \theta |_T ]^{- \varsigma } \} }
		e^{(t- \llcorner t \lrcorner_\theta  )A}
		P_I B 
		\, 
		dW_s
	}{
		1 + 
		\| 
		\int_{ \llcorner t \lrcorner_\theta }^t
		P_I B 
		\, 
		dW_s 
		\|_H^2
	} 
	.
	\end{align}  
	Then there exists 
	$ \mathfrak{c} \in \R $
	such that for all 
	$ \theta \in \varpi_T $, 
	$ I \in \mathcal{P}_0(\H) $
	we have that
	\begin{equation}
	\label{eq:Further simplified convergence estimate}
	\sup\nolimits_{ t \in [0,T] }
	\| X_t - \y^{\theta, I }_t \|_{\L^p(\P; H)}
	\leq
	\mathfrak{c} \big( \| P_{\H \backslash I } ( -A )^{ \delta - \gamma } \|_{L(H)}
	+
	[| \theta |_T]^{ \gamma - \delta  }
	\big)
	.
	\end{equation}  
\end{corollary}
\begin{proof}[Proof of Corollary~\ref{corollary:Exact_to_numeric}]
	Throughout this proof
	let
 $ D_h^I \in \mathcal{P}(H) $,
 $ h \in (0,T] $, 
 $ I \in \mathcal{P}_0(\H) $,
 be the sets which satisfy
 for all 
  $ I \in \mathcal{P}_0(\H) $,
 $ h\in (0,T] $
 that		
 \begin{equation} 
 \label{eq:need 1}
 D_h^I = \{ v \in P_I(H) \colon 1 + \| v \|_{ H_\nu }^2 \leq h^{-\varsigma} \},
 \end{equation} 
 let
  $ \epsilon \in (0, 
  \exp( - 2 ( \aa + \| B \|_{\HS(U, H )}^2 ) T ) ] $, 
  $ \varepsilon, \mathbf{C} \in (0,\infty) $ satisfy that
\begin{equation} 
\begin{split} 
\label{eq:master C}
\mathbf{C}
&
= 
\max \{ C_\varepsilon, 1 \}
\max \{ \| B \|_{\HS(U, H )}, 1 \}
\max \{ \| (-A)^{-\nu} \|_{L(H) }^2, 1 \}
\\
&
\quad 
+
\max\!
\Big\{ 
\sup\nolimits_{ J \in \mathcal{P}_0(\H) }
\sup\nolimits_{ v \in 
P_J(H) 
}
\tfrac{ \| P_J F(v) \|_H }{
	1 + \| v \|_{H_{\nu}}^2	
},
\| B \|_{\HS(U, H  )}
\Big\},
\end{split}
\end{equation} 
	\begin{equation}
	\label{eq:bound on varepsilon}
	 \varepsilon < \tfrac{ \exp(- 2 ( \aa +  \| B \|_{\HS(U, H )}^2 ) T ) }{ 16 p }
	\min \! \Big \{ \epsilon 
	\exp ( -2( \aa + \| B \|_{\HS(U, H )}^2 ) T ),
	\tfrac{ 1 }{ ( 8 
		\max \{ \| B \|_{\HS(U, H )}^2, 1 \}
		\max \{ T, 1 \}
		)^2
	}
	\Big \}, 
	\end{equation}
and
$ \E[ \exp ( \epsilon \| \xi \|_H^2 ) ] < \infty $,
and
let
$ r \colon H_\gamma \to [0, \infty) $
satisfy 
 for all 
 $ v \in H_\gamma $ 
 that
\begin{equation} 
\begin{split} 
& 
r( v ) =
\begin{cases}
\mathbf{C} ( 1 + \| v \|_{H_\nu}^2 ) 
&\colon v \in H_{ \max \{ \nu, \gamma \} }\\
0 &\colon  
v \in ( H_\gamma \backslash H_{\max \{ \nu, \gamma \} } )
.
\end{cases}
\end{split} 
\end{equation} 
	Observe that,	
	e.g.,
	Becker et al.\ \cite[Lemma~5.3]{BeckerGessJentzenKloeden2017} 
	(applies with
	$ V = H_{ \max \{ \nu, \gamma \} } $,
	$ W = H_\gamma $,
	$ ( S, \mathcal{S} ) = ( [0,\infty), \B( [0,\infty) ) ) $,
	$ \Psi = r $ 
	in the setting of  
	Becker et al.\ \cite[Lemma~5.3]{BeckerGessJentzenKloeden2017})
	ensures that
	\begin{equation} 
	\label{eq:Measurable} 
	r \in \M( \B(H_\gamma), \B( [0, \infty) ) ) 
	.
	\end{equation} 
Next note that for all
$ x \in H_{ \max \{ \nu, \gamma \} } $ we have that
\begin{equation}
\begin{split} 
\label{eq:Ass1}
&
\| B \|_{\HS(U, H )} + \epsilon \| x \|_H^2 
\leq 
\max \{ \| B \|_{\HS(U, H )}, \epsilon \}
( 1 + \| x \|_H^2 )
\\
&
\leq
\max \{ \| B \|_{\HS(U, H )}, \epsilon \}
\max \{ \| (-A)^{-\nu} \|_{L(H) }^2, 1 \}
( 1 + \| x \|_{H_\nu}^2 )
\leq 
\mathbf{C} 
( 1 + \| x \|_{ H_\nu }^2 ) 
=
r(x)
.
\end{split} 
\end{equation} 
	Moreover, observe that for all
		$ I \in \mathcal{P}_0(\H) $,
	$ h \in (0,T] $ 
	we have that 
	\begin{equation}  
	D_h^I = \{ v \in P_I( H ) \colon r(v) \leq  \mathbf{C}  h^{- \varsigma }  \}
	.
	\end{equation} 
	This,
	\eqref{eq:Measurable},
	and, 
	e.g., 
	Andersson et al.\ \cite[Lemma~2.2]{AnderssonKurniawanJentzen2016}
	(applies with
	$ V_0 = H_\gamma $,
	$ V_1 = P_I(H) $
		for
		$ I \in \mathcal{P}_0( \H ) $
	in the setting of 
	Andersson et al.\ \cite[Lemma~2.2]{AnderssonKurniawanJentzen2016})
	assure that for all
	 $ I \in \mathcal{P}_0(\H) $,
 	$ h \in (0,T] $
 we have that
 \begin{equation} 
 \label{eq:Measurable Hgamma}
 D_h^I \in \B(H_\gamma).
 \end{equation}
	Furthermore, note that~\eqref{eq:need 1}
	and~\eqref{eq:master C}
	give that 
	for all 
		$ I \in \mathcal{P}_0(\H) $,
	$ h \in (0,T] $,  
	$ x \in D_h^I $
	we have that 
	\begin{equation} 
	\begin{split}
	\label{eq:Ass3}
	&
	\max\{\| P_I F(x) \|_H,
	\| B \|_{\HS(U, H  )}
	\} 
	\\
	&
	\leq 
	\max\!
	\Big\{
	\Big(
	\sup\nolimits_{ J \in \mathcal{P}_0(\H) }
	\sup\nolimits_{ v \in 
P_J(H)	
}
	\tfrac{ \| P_J F(v) \|_H }{
		1 + \| v \|_{H_{\nu}}^2	
	} 
	\Big)
	 ( 1 + \| x \|_{H_{\nu}}^2  )
	,
	\| B \|_{\HS(U, H  )}
	\Big\}
	\\
	& 
	\leq 
	\max\!
	\Big\{ 
	\sup\nolimits_{ J \in \mathcal{P}_0(\H) }
	\sup\nolimits_{ v \in 
		P_J(H)
	}
	\tfrac{ \| P_J F(v) \|_H }{
		1 + \| v \|_{H_{\nu}}^2	
	},
	\| B \|_{\HS(U, H  )}
	\Big\} 
	(
	1
	+
	\| x \|_{H_{\nu}}^2
	)
	\leq 
	\mathbf{C} h^{-\varsigma}.
	\end{split} 
	\end{equation}
	Combining this, 
	\eqref{eq:Addition0}, 
	\eqref{eq:bound on varepsilon}, 
	\eqref{eq:Measurable}--\eqref{eq:Measurable Hgamma}, 
	the fact that
	$ \E[ \exp ( \epsilon \| \xi \|_H^2 ) ] < \infty $, 
	the assumption that for all    
	$ I \in \mathcal{P}_0(\H) $,
	$ x, y \in P_I( H ) $ 
	we have that   
	$ \langle F'(x) y, y \rangle_H  \leq 
	( \varepsilon \| x \|_{H_{\nicefrac{1}{2}}}^2 
	+ C_\varepsilon ) \| y \|_H^2 
	+ \|y \|_{H_{\nicefrac{1}{2}}}^2 $,
	$ \langle x, F(x) \rangle_H \leq \aa ( 1 + \| x \|_H^2 ) $,
	$ \| F( x ) - F( y ) \|_H
	\leq
	C \| x - y \|_{H_\delta} 
	( 1 + \| x \|_{H_\kappa}^c  + \| y \|_{H_\kappa}^c ) $,
	and
	$ \langle x, A x + F(x+y) \rangle_H
	\leq
	\Phi(y) ( 1 + \| x \|_H^2 ) $,
	and 
	Theorem~\ref{theorem:Exact_to_numeric}
	(applies with
	$ T = T $,
	$ \upnu = \mathbf{C} $,
	$ \varsigma = \varsigma $,
	$ \aa = \aa $,
	$ C = \mathbf{C} $,
	$ c = c $,
	$ p = p $,
	$ \beta = \beta $,
	$ \gamma = \gamma $, 
	$ \delta = \delta $,
	$ \kappa = \kappa $,
	$ \eta_0 = \eta_0 $,
	$ \sigma = \sigma $,
	$ \nu = \nu $,
	$ \eta_1 = \eta_1 $,
	$ \eta_2 = \eta_2 $,
	$ \alpha_1 = \alpha_1 $,
	$ \alpha_2 = \alpha_2 $,
	$ \alpha_3 = \alpha_3 $,
	$ B = B $,
	$ \epsilon = \epsilon $,
	$ \varepsilon = \varepsilon $,
	$ F = F $, 
	$ r = r $,
	$ D_h^I = D_h^I $, 
	$ \Phi = \Phi $,
	$ P_I = P_I $,
	$ ( \Omega, \F, \P ) = ( \Omega, \F, \P ) $,
	$ ( \f_t )_{ t \in [0,T] } = ( \f_t )_{ t \in [0,T] } $,
	$ ( W_t )_{ t \in [0,T] } = ( W_t )_{ t \in [0,T] } $,
	$ \xi = \xi $,
	$ X = X $,
	$ \y^{\theta, I} = \y^{ \theta, I } $
	for
		$ \theta \in \varpi_T $,
	$ I \in \mathcal{P}_0(\H) $,
	 $ h \in (0,T] $
	in the setting of 
	Theorem~\ref{theorem:Exact_to_numeric})
	justifies~\eqref{eq:Further simplified convergence estimate}.  
	The proof of Corollary~\ref{corollary:Exact_to_numeric}
	is hereby completed.
\end{proof}
\section[Strong convergence rates for stochastic Burgers equations]{Strong convergence rates for space-time discrete approximations of stochastic Burgers equations}
\label{section:Burgers}
In this section we illustrate
Corollary~\ref{corollary:Exact_to_numeric}  
in the case of
stochastic Burgers equations.
For this we combine some of the regularity results  in~\cite{JentzenLindnerPusnik2017c}
with
Corollary~\ref{corollary:Exact_to_numeric}  
to prove
in
Corollary~\ref{corollary:Burgers} 
strong convergence 
for the
numerical approximations
$ ( \y_t^{\theta, I} )_{ t \in [0,T] } $,
$ \theta \in \varpi_T $,
$ I \in \mathcal{P}_0(\H) $,
(see~\eqref{scheme:full_discrete}
below) 
of the 
solution of an additive trace-class noise driven stochastic Burgers equation
 (see~\eqref{eq:Burgers equation} below).
Finally,
Corollary~\ref{corollary:BurgersFinal}
presents the findings from 
Corollary~\ref{corollary:Burgers} 
in a further simplified setting. 
\begin{corollary}
	\label{corollary:Burgers}
Assume Setting~\ref{setting:Notation}, 
let
$ T, c_0 \in ( 0, \infty ) $,
$ c_1 \in \R $,
$ \varsigma \in (0, \nicefrac{1}{18}) $,  
$ p \in [1, \infty) $,
$ \beta \in (0, \nicefrac{1}{2} ) $,
$ \gamma \in ( [ \max \{ \nicefrac{1}{2}, 2 \beta \},   \nicefrac{1}{2}  + \beta ) \backslash \{ \nicefrac{1}{2}, \nicefrac{3}{4} \} ) $,
let
$ \lambda 
\colon
\mathcal{B}( (0,1) )
\rightarrow [0, 1] $
be the Lebesgue-Borel
measure on
$ (0,1) $,
let
$ (H, \langle \cdot, \cdot \rangle_H,
\left \| \cdot \right \|_H) 
=
(L^2( \lambda; \R), 
\langle \cdot, \cdot 
\rangle_{L^2( \lambda;\R)},
\left \| \cdot \right \|_{L^2( \lambda; \R)}
) $,  
let
$ ( e_n )_{ n \in \N } \subseteq H $
satisfy
for all
$ n \in \N $ 
that  
$ e_n = [ ( \sqrt{2} \sin( n \pi x ) )_{x \in (0,1) } ]_{ \lambda, \B(\R ) } $,
let $ \H \subseteq H $ satisfy that 
$ \H = \{ e_n \colon n \in \N \} $, 
\sloppy 
let
$ A \colon D( A ) \subseteq H \to H $ 
be the linear operator which satisfies
$ D(A) =  \{
v \in H \colon \sum_{ n=1 }^\infty  | n^2 
\langle e_n, v \rangle_H  | ^2 <  \infty
\} $ 
and  
$ \forall \, v \in D(A) \colon A v = - c_0 \sum_{ n = 1 }^\infty \pi^2 n^2 \langle 
e_n, v \rangle _H e_n $, 
let
$ (H_r, \langle \cdot, \cdot \rangle_{H_r}, \left \| \cdot \right \|_{H_r} ) $, 
$ r \in \R $, 
be a family of interpolation spaces associated to $ -A $,
for every  
$ v \in W^{1,2}( (0,1), \R ) $
let
$ \partial v    
\in H $
satisfy for all
$ \varphi  
\in 
\mathcal{C}_{cpt}^\infty( (0,1), \R ) $
that
$ \langle \partial  v, 
[ \varphi ]_{ \lambda, \mathcal{B}(\R )} \rangle_{H}
=
-
\langle v , 
[ \varphi' ]_{ \lambda, \mathcal{B}(\R )}
\rangle_{H} $,
let   
$ B \in \HS(H, H_\beta ) $,
let  
$ F \colon H_{\nicefrac{1}{2}} \to H $
satisfy 
for all
$ v \in H_{ \nicefrac{1}{2} } $ 
that
$ F(v) =
c_1
v \partial v $,
let
$ (P_I)_{ I \in \mathcal{P}(\H) } \subseteq L(H) $ 
satisfy
for all 
$ I\in \mathcal{P}(\H) $,  
$ v \in H $ 
that
$ P_I(v) =\sum_{h \in I} \left< h ,v \right>_H h $,
	let
	$ ( \Omega, \F, \P ) $
	be a probability space with a normal filtration
	$ ( \f_t )_{t \in [0,T]} $, 
	let
	$ (W_t)_{t\in [0,T]} $
	be an
	$ \operatorname{Id}_H $-cylindrical 
	$ ( \f_t )_{t\in [0,T]} $-Wiener process,
	let 
	$ \xi \in \L^{32p \max \{   ( 2 \gamma - 1 )  / (2\varsigma), 1  \} } (\P|_{\f_0}; H_\gamma ) $  
	satisfy 
	$ \inf_{ \epsilon \in (0,\infty) } \E[ \exp( \epsilon \| \xi \|_H^2 ) ] < \infty $,  
	let
	$ X  \colon [0,T] \times \Omega \to H_\gamma $
	be an $ ( \f_t )_{ t \in [0,T] } $-adapted 
	stochastic process
	w.c.s.p.\
	which satisfies
	for all
	$ t \in [0,T] $ that
	\begin{equation}
	\label{eq:Burgers equation} 
	[ X_t ]_{\P, \B(H_\gamma)} 
	= 
	\bigg[
	e^{tA} \xi 
	+
	\int_0^t e^{(t-s)A} F(X_s) \, ds
	\bigg]_{\P, \B(H_\gamma)} 
	+
	\int_0^t e^{(t-s)A} B \, dW_s,
	\end{equation}
	and
	let 
	$ \y^{\theta, I}\colon [0, T] \times \Omega \to P_I(H) $,
	$ \theta \in \varpi_T $, 
	$ I \in \mathcal{P}_0(\H) $, 
	be  
	$ (\f_t)_{t\in [0,T]} $-adapted
	stochastic processes
	which satisfy for all 
	$ \theta \in \varpi_T $, 
	$ I \in \mathcal{P}_0 (\H) $,
		$ t \in (0,T] $ 
	that
	$ \y^{\theta, I }_0 = P_I(\xi) $ and
	\begin{align}
	\label{scheme:full_discrete}
	\nonumber
	[ \y_t^{\theta, I } ]_{\P, \mathcal{B}( P_I( H ) )} 
	&= 
	\Big[ 
	e^{(t-\llcorner t \lrcorner_\theta)A} 
	\y_{
		\llcorner t \lrcorner_\theta
	}^{\theta, I } 
	+
	\,
	\1_{  \{ 
	1 + \| \y^{\theta,I }_{\llcorner t \lrcorner_{\theta}} \|_{ H_{ \nicefrac{1}{2} } }^2
		\leq  
		[|\theta|_T]^{-\varsigma }   \} }
	e^{(t-\llcorner t \lrcorner_\theta)A} 
	P_I
	F(
	\y_{ \llcorner t \lrcorner_\theta }^{\theta, I }
	) \,
	(
	t - \llcorner t \lrcorner_\theta
	)   
	\Big]_{\P, \mathcal{B}( P_I( H ) )}  
	\\
	&
	\quad
	+
	\frac{
		\int_{ \llcorner t \lrcorner_\theta }^t
		\1_{  \{ 
			1 + \| \y^{\theta,I }_{\llcorner t \lrcorner_{\theta}} \|_{ H_{ \nicefrac{1}{2} } }^2
			\leq  
			[|\theta|_T]^{-\varsigma }   \} }
		e^{(t-\llcorner t \lrcorner_\theta)A} 
		P_I
		B 
		\, dW_s
	}{
	1 + 
	\|
	\int_{ \llcorner t \lrcorner_\theta}^t 
	P_I
	B
	\, dW_s
	\|_H^2
	} 
	.
	\end{align}
	Then there exists 
	$ C \in \R $
	such that for all 
	$ \theta \in \varpi_T $,
	$ I \in \mathcal{P}_0 (\H) $
	we have that
	\begin{equation}
	\label{eq:Even further estimate}
	\sup\nolimits_{ t \in [0,T] }
	\| X_t - \y^{\theta, I }_t \|_{\L^p(\P; H)}
	\leq
	C
	\big( \| P_{\H \backslash I } ( -A )^{ (\nicefrac{1}{2}) - \gamma } \|_{L(H)}
	+
	[| \theta |_T]^{  \gamma - (\nicefrac{1}{2}) }
	\big)
	.
	\end{equation}
\end{corollary}
\begin{proof}[Proof of Corollary~\ref{corollary:Burgers}]
	Throughout this proof let
	$ \Phi \colon H \to [0, \infty) $
	satisfy 
	for all 
	$ w \in H $
	that
	\begin{equation}
	\Phi( w ) = 
	\begin{cases}
	\tfrac{ 3 | c_1 |^2 }{ 8  | c_0 |  } 
	\bigg[   
	\sup_{ u \in H_{\nicefrac{1}{2} } \backslash \{ 0 \} }
	\tfrac{ \| u \|_{ L^\infty(\lambda; \R) } }{ \| u \|_{ H_{ \nicefrac{1}{2} } } }
	+
	\sup_{ u \in H_{\nicefrac{1}{2} } \backslash \{ 0 \} }
	\tfrac{ \| u \|_{ L^4(\lambda; \R) }^2 }{ \| u \|_{ H_{ \nicefrac{1}{2} } }^2 } 
	\bigg]^2
	(
	1
	+ 
	\| w \|_{ H_{ \nicefrac{1}{2} } }^2
	)^2
	&\colon w \in H_{ \nicefrac{1}{2} } \\
	0 &\colon w \in ( H \backslash H_{ \nicefrac{1}{2} }  )
	.
	\end{cases}
	\end{equation}
	We intend to verify
	Corollary~\ref{corollary:Burgers}
	through an application of 
	Corollary~\ref{corollary:Exact_to_numeric}.
	For this note that, e.g., \cite[item~(ii) 
	of Lemma~4.13]{JentzenLindnerPusnik2017c} 
	yields that
	for all 
	$ v, w \in H_\gamma \subseteq H_{ \nicefrac{1}{2} } $  
	we have that 
	\begin{equation}
	\begin{split}
	\label{eq:LocLips}
	&
	\| F(v) - F(w) \|_H 
	\leq 
	\tfrac{ | c_1 | }{ \sqrt{3} \, c_0 }
	(   
	\| v \|_{H_{ \nicefrac{1}{2} } }   
	+
	\| w \|_{H_{ \nicefrac{1}{2} } }
	)
	\| v - w \|_{H_{ \nicefrac{1}{2} } }    
	.
	\end{split}
	\end{equation}
	In addition, observe that, e.g., \cite[Lemma~4.19]{JentzenLindnerPusnik2017c} 
	and the fact that
	$ H_\gamma \subseteq H_{ \nicefrac{1}{2} } $
	continuously give that
	\begin{enumerate}[(a)] 
	\item \label{item:C1 F} we have that
	$ F \in \mathcal{C}^1(H_\gamma, H) $
	and 
	\item \label{item:C1 F1} we have that  
	there exists $ C \in (0, \infty) $
	which satisfies for all 
	$ \varepsilon \in (0,\infty) $,
	$ v, w \in H_\gamma \subseteq H_{ \nicefrac{1}{2} } $ 
	that
	\begin{equation} 
	\begin{split} 
	&
	\langle F'(w) v, v \rangle_H 
	\leq  
	\varepsilon
	\| w \|_{H_{\nicefrac{1}{2}}}^2    \| v \|_H^2 
	+
	\tfrac{C}{ \varepsilon^2 } \| v\|_H^2
	+
	\| v \|_{H_{\nicefrac{1}{2}}}^2
	.
	\end{split}
	\end{equation}	
	\end{enumerate}
	Furthermore,
	note that
	the fact that
	$ 0 \leq \gamma - \frac{1}{2}  < \frac{1}{2} $,
	the fact that
	$ \gamma \neq \frac{3}{4} $,
	and, e.g., \cite[Lemma~4.20]{JentzenLindnerPusnik2017c} 
	(applies with
	$ \alpha = \gamma - \frac{1}{2} $
	in the setting of~\cite[Lemma~4.20]{JentzenLindnerPusnik2017c}) 
	ensure that
	\begin{equation}
	\begin{split}
	\label{eq:Change again}
	\sup\nolimits_{ I \in \mathcal{P}_0(\H) }
	\sup\nolimits_{ v \in H_\gamma \backslash \{ 0 \} }
	\!
	\Big( 
	\tfrac{ \| P_I F(v) \|_{H_{ \gamma - (\nicefrac{1}{2} )  } } }
	{ \| v \|_{ H_{ \gamma  } }^2 }
	\Big)
	<
	\infty 
	.
	\end{split}
	\end{equation}
	Moreover,
 observe that, e.g., \cite[Lemma~4.20]{JentzenLindnerPusnik2017c} 
	(applies with
	$ \alpha = 0 $
	in the setting of~\cite[Lemma~4.20]{JentzenLindnerPusnik2017c}) 
	verifies that  
	\begin{equation} 
	\label{eq:F_B_estimate}
	\sup\nolimits_{ I \in \mathcal{P}_0(\H) }
	\sup\nolimits_{ v \in H_{ \nicefrac{1}{2} } \backslash \{0\} }
	\!
	\Big(
	\tfrac{ \| P_I F(v) \|_H }{
		\| v \|_{ H_{ \nicefrac{1}{2} } }^2	
	} 
\Big) 
	<
	\infty
	.
	\end{equation}
	In addition, note that, e.g., \cite[Lemma~4.23]{JentzenLindnerPusnik2017c} 
	verifies that  
	for all 
	$ I \in \mathcal{P}_0(\H) $, 
	$ x \in P_I(H) $ 
	we have
	that 
	\begin{equation} 
	\langle x, F(x) \rangle_H = 0 
	.
	\end{equation}
	Furthermore, observe that, e.g., \cite[Corollary~4.22]{JentzenLindnerPusnik2017c} 
	(applies with 
	$ \alpha_1 = \alpha_1 $,
	$ \alpha_2 = \alpha_2 $
	for 
	$ \alpha_1 \in ( \nicefrac{3}{4}, \infty) $,
	$ \alpha_2 \in ( \nicefrac{1}{4}, \nicefrac{1}{2} ] $
	in the setting of~\cite[Corollary~4.22]{JentzenLindnerPusnik2017c}) 
	yields that for all  
	$ \alpha_1 \in ( \nicefrac{3}{4}, \infty) $,
	$ \alpha_2 \in ( \nicefrac{1}{4}, \nicefrac{1}{2} ] $
	we have that
	\begin{equation}
	\label{eq:EnoughRegular}
	\Big[ 
	\sup\nolimits_{ v \in H_{ \nicefrac{1}{2} } \backslash \{0\} }
	\tfrac{ \| F(v) \|_{H  } }
	{ \| v \|_{H_{ \nicefrac{1}{2} } }^2 }
	\Big]
	+
	\Big[
	\sup\nolimits_{ v \in H_{ 
			\nicefrac{1}{2}   } \backslash \{ 0 \} }
	\tfrac{ \| F(v) \|_{H_{ - \alpha_2 } } }
	{ \| v \|_{ H_{  \nicefrac{ ( 1 - \alpha_2 ) }{3} } }^2 } 
	\Big]
	+
	\Big[ 
	\sup\nolimits_{ v \in H_{ \nicefrac{1}{2} } \backslash
		\{0\} } 
	\tfrac{ \| F(v) \|_{H_{-\alpha_1} } }{ \| v \|_H^2 }
	\Big]
	< \infty 
	.
	\end{equation}  
	Moreover,
	note that, e.g., \cite[Corollary~4.24]{JentzenLindnerPusnik2017c} 
	(applies with
	$ \iota = \nicefrac{1}{2} $,
	$ v = v $,
	$ w = w $
	for
	$ v, w \in H_{ \nicefrac{1}{2} } $
	in the setting of~\cite[Corollary~4.24 ]{JentzenLindnerPusnik2017c}) 
	assures that for all 
	$ v, w \in H_{ \nicefrac{1}{2} } $
	we have that
	\begin{equation}
	\begin{split}
	\label{eq:GeneralCoercivity}
	\langle v, F( v + w ) \rangle_H 
	&
	\leq
	\Phi(w) ( 1 +  \| v \|_H^2 )
	-
	\langle v, A v \rangle_H  
	.
	\end{split}
	\end{equation}
	Combining
	this, 
	the
	assumption 
	that
	$ \inf_{ \epsilon \in (0, \infty) } 
	\E[ \exp( \epsilon \| \xi \|_H^2 ) ] < \infty $, 
	items~\eqref{item:C1 F} and~\eqref{item:C1 F1},
	\eqref{eq:involve O}, 
	\eqref{eq:LocLips},
	and~\eqref{eq:Change again}--\eqref{eq:EnoughRegular}
	with 
	Corollary~\ref{corollary:Exact_to_numeric}
	(applies with
    $ ( H, \langle \cdot, \cdot \rangle_H,
    \left \| \cdot \right \|_H ) 
    =
    ( H, \langle \cdot, \cdot \rangle_H,
    \left \| \cdot \right \|_H ) $,
    $ ( U, \langle \cdot, \cdot \rangle_U,
    \left \| \cdot \right \|_U ) 
    =
    ( H, \langle \cdot, \cdot \rangle_H,
    \left \| \cdot \right \|_H ) $,
    $ \H = \H $,
    $ \values_{e_n} = - c_0 \pi^2 n^2 $,
    $ A = A $,
    $ H_r = H_r $,
	$ T = T $,
	$ \varsigma = \varsigma  $,
	$ \aa = 0 $,
	$ C = \max \{ 1, \nicefrac{|c_1|}{c_0} \} $,
	$ c = 1 $,
	$ p = p $,
	 $ C_\varepsilon = \nicefrac{ C }{ \varepsilon^2 } $,
	 $ \beta = \beta $,
	 $ \gamma = \gamma $, 
	$ \delta = \nicefrac{1}{2} $,
	$ \kappa = \nicefrac{1}{2} $,
	$ \sigma = \gamma $, 
	$ \nu = \nicefrac{1}{2} $,
	$ \eta_1 = \nicefrac{ ( 1 - \alpha_2 ) }{3} $,
	$ \eta_2 = \nicefrac{1}{2} $,
	$ \alpha_1 = \alpha_1 $,
	$ \alpha_2 = \alpha_2 $,
	$ B = B $,  
	$ F = ( H_\gamma \ni x \mapsto F(x) \in H ) $,
	$ \Phi = \Phi $,
	$ P_I = P_I $,
	$ ( \Omega, \F, \P ) 
	=
	( \Omega, \F, \P ) $,
	$ ( \f_t )_{ t \in [0,T] } 
	= ( \f_t )_{ t \in [0,T] } $,
	$ (W_t)_{ t \in [0,T] } 
	=
	(W_t)_{ t \in [0,T] } $,
	$ \xi = \xi $,
	$ X = X $,
	$ \y^{\theta, I} = \y^{\theta, I} $
	for 
$ n \in \N $,
$ r \in \R $,
$ \varepsilon \in (0, \infty) $,
$ \alpha_2 \in ( \nicefrac{1}{4}, \nicefrac{1}{2} ) $,
$ \alpha_1 \in ( \nicefrac{3}{4}, \nicefrac{(2+\alpha_2)}{3} ) $, 
	$ \theta \in \varpi_T $,
	$ I \in \mathcal{P}_0(\H) $
	in the setting of Corollary~\ref{corollary:Exact_to_numeric})
	therefore
	justifies~\eqref{eq:Even further estimate}.
	The proof of Corollary~\ref{corollary:Burgers}
	is hereby completed.
\end{proof}
\begin{corollary}
	\label{corollary:BurgersFinal}
	\sloppy 
	Assume Setting~\ref{setting:Notation},
	let  
	$ T, \varepsilon, c_0 \in (0,\infty) $,
	$ c_1 \in \R $,
	$ \varsigma  \in (0, \nicefrac{1}{18}) $, 
	$ p \in [1,\infty) $,
	$ \beta \in (0, \nicefrac{1}{2} ] $,  
	$ \gamma \in [\nicefrac{1}{2}, \nicefrac{1}{2}  + \beta ) $,
	let
	$ \lambda 
	\colon
	\mathcal{B}( (0,1) )
	\rightarrow [0, 1] $
	be the Lebesgue-Borel
	measure on
	$ (0,1) $,
let
	$ (H, \langle \cdot, \cdot \rangle_H,
	\left \| \cdot \right \|_H)  
	=
	(L^2( \lambda; \R),
	\langle \cdot, \cdot 
	\rangle_{L^2( \lambda; \R)},
	\left \| \cdot \right \|_{L^2( \lambda; \R)} ) $,  
	let     
	$ (e_n)_{ n \in \N} \subseteq H $
	satisfy 
	for all $ n \in \N $
	that
	$ e_n = [ 
	( \sqrt{2} \sin( n \pi x ) )_{x \in (0,1) } 
	]_{ \lambda, \B(\R ) } $, 
	let
	$ A \colon D( A ) \subseteq H \to H $ 
	be the linear operator which satisfies
	$ D(A) =  \{
	v \in H \colon \sum_{ n=1 }^\infty  | n^2 
	\langle e_n, v \rangle_H  | ^2 <  \infty
	 \} $ 
	and  
	$ \forall \, v \in D(A) \colon A v = - c_0 \sum_{ n = 1 }^\infty \pi^2 n^2 \langle 
	e_n, v \rangle _H e_n $, 
	let
	$ (H_r, \langle \cdot, \cdot \rangle_{H_r}, \left \| \cdot \right \|_{H_r} ) $, $ r\in \R $, be a family of interpolation spaces associated to $ -A $, 
	for every  
	$ v \in W^{1,2}( (0,1), \R ) $
	let
	$ \partial v    
	\in H $
	satisfy for all 
	$ \varphi  
	\in 
	\mathcal{C}_{cpt}^\infty( (0,1), \R ) $
	that
	$ \langle \partial  v, 
	[ \varphi ]_{\lambda, \mathcal{B}(\R )} \rangle_H
	=
	-
	\langle v, 
	[ \varphi' ]_{ \lambda, \mathcal{B}(\R )}
	\rangle_H $,
	let
	$ B \in \HS(H, H_\beta) $,
	$ \xi \in H_{ \nicefrac{1}{2} + \beta } $,
	let  
	$ F \colon H_{\nicefrac{1}{2}} \to H $ 
satisfy 
	 for all   
	 $ v \in H_{ \nicefrac{1}{2} } $ 
	 that
	 $ F( v ) =
	 c_1
	 v \partial v 
	 $,
	let 
	$ ( P_N )_{N \in \N } \subseteq L(H) $ 
	satisfy
	for all 
	$ N \in \N $, 
	$ v \in H $ 
	that
	$ P_N( v ) = \sum_{n=1}^N \langle e_n ,v \rangle_H e_n $, 
	let
	$ ( \Omega, \F, \P ) $
	be a probability space with a normal filtration
	$ ( \f_t )_{t \in [0,T]} $,
	let
	$ (W_t)_{t\in [0,T]} $
	be an
	$ \operatorname{Id}_H $-cylindrical 
	$ ( \f_t )_{t\in [0,T]} $-Wiener process,  
	and
	let 
	$ \y^{\theta, N} \colon [0, T] \times \Omega \to P_N(H) $,
	$ \theta \in \varpi_T $, 
	$ N \in \N $, 
	be  
	$ (\f_t)_{t\in [0,T]} $-adapted
	stochastic processes 
	which satisfy for all 
	$ \theta \in \varpi_T $, 
	$ N \in \N $,
	$ t \in (0, T] $ 
	that
	$ \y^{\theta, N}_0 = P_N(\xi) $ and
	\begin{equation} 
	\begin{split} 
	\label{eq:lastEq}
	&[ \y_t^{\theta, N } ]_{\P, \mathcal{B}(P_N(H))} 
	=
	\frac{
		\int_{ \llcorner t \lrcorner_\theta }^t
		\1_{  \{ 
			1
			+
			\|   
			\y^{\theta, N}_{\llcorner t \lrcorner_{\theta}}  \|_{ H_{ \nicefrac{1}{2} } }^2 
			\leq 
			[|\theta|_T]^{- \varsigma }   \}}
		e^{(t-\llcorner t \lrcorner_\theta)A} 
		P_N
		B 
		\, dW_s
	}{
		1 + 
		\|
		\int_{ \llcorner t \lrcorner_\theta}^t 
		P_N
		B 
		\, dW_s
		\|_H^2
	} 
\\
& 
+
	\Big[
	e^{(t-\llcorner t \lrcorner_\theta)A}  
	\y_{
		\llcorner t \lrcorner_\theta
	}^{\theta, N} 
	+
	\,
	\1_{  \{ 
		1
		+
		\|   
		\y^{\theta, N}_{\llcorner t \lrcorner_{\theta}}  \|_{ H_{ \nicefrac{1}{2} } }^2 
		\leq 
		[|\theta|_T]^{- \varsigma }   \}}
		e^{(t-\llcorner t \lrcorner_\theta)A} 
	P_N
	F(
	\y_{ \llcorner t \lrcorner_\theta }^{\theta, N}
	) \,
	(
	t - \llcorner t \lrcorner_\theta
	) 
	\Big]_{\P, \mathcal{B}(P_N(H))}  
	.
	\end{split}
	\end{equation} 
	Then  
	\begin{enumerate}[(i)]
		\item \label{item:SolutionExists}
		there exists 
		an up to indistinguishability unique $ ( \f_t )_{ t \in [0,T] } $-adapted 
		stochastic process
		$ X  \colon [0,T] \times \Omega \to H_\gamma $
		w.c.s.p.\ 
		such that 
		for all
		$ t \in [0,T] $ 
		we have that
		\begin{equation}
		[ X_t ]_{\P, \B(H_{ \gamma } )} 
		= 
		\bigg[
		e^{tA} \xi 
		+
		\int_0^t e^{(t-s)A} F(X_s) \, ds
		\bigg ]_{\P, \B(H_{ \gamma })} 
		+
		\int_0^t e^{(t-s)A} B \, dW_s
		\end{equation} 
		and
	\item \label{item:AppConverges} 
	there exists $ C \in \R $
	such that for all 
	$ \theta \in \varpi_T $,
	$ N \in \N $
	we have that
	\begin{equation}
	\sup\nolimits_{ t \in [0,T] }
	\| X_t - \y^{\theta, N }_t \|_{\L^p(\P; H)}
	\leq
	C
	\big(
	N^{ ( \varepsilon - 2 \beta ) }
	+
	[| \theta |_T]^{ ( \beta - \varepsilon ) }
	\big)
	.
	\end{equation}
	\end{enumerate}
\end{corollary}
\begin{proof}[Proof of Corollary~\ref{corollary:BurgersFinal}]
Observe that~\cite[Theorem~5.10]{JentzenLindnerPusnik2017c}
(applies with 
$ T = T $,   
$ \varepsilon = \nicefrac{1}{2} + \beta - \pmb{\gamma} $, 
$ c_0 = c_0 $,
$ c_1 = c_1 $,
$ \beta = \beta $, 
$ \gamma = \pmb{\gamma} $,
$ H = H $,
$ e_n = e_n $,
$ A = A $,
$ H_r = H_r $,
$ ( \Omega, \F, \P ) 
=
( \Omega, \F, \P ) $,
$ ( \f_t )_{ t \in [0,T] } 
=
( \f_t )_{ t \in [0,T] } $,
$ ( W_t )_{ t \in [0,T] } 
=
 ( W_t )_{ t \in [0,T] } $,
$ B = B $, 
$ \xi = ( \Omega \ni \omega \mapsto \xi \in H_{  \nicefrac{1}{2} + \beta  } ) $ 
for 
$ r \in \R $,
$ n \in \N $,
$ \pmb{\gamma} \in [ \nicefrac{1}{2}, \nicefrac{1}{2} 
+ \beta ) $
in the setting of~\cite[Theorem~5.10]{JentzenLindnerPusnik2017c}) 
yields that  
	there exist  
  up to modification unique $ ( \f_t )_{ t \in [0,T] } $-adapted 
stochastic processes
$ X^{ \pmb{\gamma} } 
\colon 
[0,T] \times \Omega \to H_{ \pmb{\gamma} } $,
$ \pmb{\gamma} \in [ \nicefrac{1}{2}, \nicefrac{1}{2} + \beta ) $,  
w.c.s.p.\ 
which satisfy
for all
$ \pmb{\gamma} \in [ \nicefrac{1}{2}, \nicefrac{1}{2} + \beta ) $,
$ t \in [0,T] $ that
\begin{equation}
\label{eq:ExistenceLast}
[ X_t^{ \pmb{\gamma} } ]_{\P, \B(H_{ \pmb{\gamma} } )} 
= 
\bigg[
e^{tA} \xi 
+
\int_0^t e^{(t-s)A} F( X_s^{ \pmb{\gamma} } ) \, ds
\bigg ]_{\P, \B(H_{ \pmb{\gamma} })} 
+
\int_0^t e^{(t-s)A} B \, dW_s
.
\end{equation} 	
This justifies item~\eqref{item:SolutionExists}.
In the next step we note that for all 
$ \iota \in (0,\infty) $,
$ N \in \N $,
$ v \in H $ we have that
\begin{equation} 
\begin{split}
&
\| ( \operatorname{Id}_H - P_N ) ( -A )^{-\iota} v \|_H^2
=
|c_0|^{ - 2 \iota }
\sum_{n=N+1}^\infty 
(\pi^2 n^2)^{- 2 \iota}
| \langle v, e_n \rangle_H |^2
\\
&
\leq
|c_0|^{ - 2 \iota }
( \pi^2 N^2 )^{-2\iota} 
\sum_{n=N+1}^\infty  
| \langle v, e_n \rangle_H |^2
\leq
|c_0|^{ - 2 \iota }
( \pi^2 N^2 )^{-2\iota} 
\| v \|_H^2 
.
\end{split}
\end{equation} 
This yields that
for all 
$ \iota \in (0,\infty) $,
$ N \in \N $ 
we have that
\begin{equation}
\| ( \operatorname{Id}_H - P_N ) (-A)^{-\iota} \|_{L(H)}
\leq
 | c_0 |^{ - \iota }  
\pi^{-2\iota}
N^{-2\iota}
\leq 
 | c_0 |^{ - \iota } 
N^{-2\iota}
.
\end{equation}
The fact that
for all
$ \theta \in \varpi_T $, 
$ \epsilon \in (0, \infty) $
we have that
$ [ | \theta |_T ]^{ \beta 
	- 
	(\nicefrac{ \epsilon }{ 2 } ) }
\leq
T^{ \nicefrac{\epsilon}{2} } 
[ | \theta |_T ]^{ ( \beta - \epsilon ) } $,
\eqref{eq:ExistenceLast},
and
Corollary~\ref{corollary:Burgers}
(applies with
$ T = T $,
$ c_0 = c_0 $,
$ c_1 = c_1 $,
$ \varsigma = \varsigma $,
$ p = p $, 
$ \beta = \beta - \frac{ \epsilon }{ 4 } $,
$ \gamma = \frac{1}{2} + \beta - \frac{\epsilon}{2} $, 
$ H = H $,
$ e_n = e_n $, 
$ A = A $,
$ H_r = H_r $, 
$ B = B $,
$ F = F $,
$ P_{ \{ e_1, e_2, \ldots, e_n \} } = P_n $,
$ ( \Omega, \F, \P) 
=
( \Omega, \F, \P ) $,
$ (\f_t)_{ t \in [0,T] } =  (\f_t)_{ t \in [0,T] } $,
$ ( W_t )_{ t \in [0,T] } = ( W_t )_{ t \in [0,T] } $,   
$ \xi = ( \Omega \ni \omega \mapsto \xi \in H_{  ( \nicefrac{1}{2} )  + \beta - ( \nicefrac{\epsilon}{2} ) } ) $,
$ X = ( [0,T] \times \Omega \ni (t,\omega) \mapsto 
X_t^{ ( \nicefrac{1}{2} ) + \beta - ( \nicefrac{\epsilon}{2} ) } 
( \omega )
\in H_{ ( \nicefrac{1}{2} ) 
	+ \beta 
	- ( \nicefrac{\epsilon}{2} ) 
} ) $,
$ \y^{ \theta, \{ e_1, e_2, \ldots, e_n \} } 
=
\y^{\theta, n } $ 
for  
$ r \in \R $,
$ \theta \in \varpi_T $,
$ n \in \N $, 
$ \epsilon \in ( (0, 2 \beta ) \backslash 
\{ 2 \beta - \nicefrac{1}{2} \} ) $ 
in the setting of 
Corollary~\ref{corollary:Burgers}) 
therefore
justify item~\eqref{item:AppConverges}.
The proof of Corollary~\ref{corollary:BurgersFinal}
is hereby completed.
\end{proof} 
\subsubsection*{Acknowledgements}
This project has been partially supported through the SNSF-Research project $ 200021\_156603 $ ''Numerical 
approximations of nonlinear stochastic ordinary and partial differential equations''
and through the Deutsche Forschungsgesellschaft (DFG) 
via RTG 2131 \textit{High-dimensional Phenomena in Probability – Fluctuations and Discontinuity}.
\newpage
\bibliographystyle{acm}
\bibliography{../../../Bib/bibfile}

\end{document}